\documentclass[reqno]{amsart}
\usepackage{etex}
\usepackage[english]{babel}
\usepackage[usenames,dvipsnames]{xcolor}
\usepackage{xcolor}
\usepackage{amsmath,amsthm,amssymb,amsfonts,xspace,tensor}
\usepackage[all]{xy}
\usepackage{mathtools}
\usepackage{dsfont}
\usepackage[bookmarksnumbered,colorlinks]{hyperref}
\usepackage{graphicx}
\usepackage{enumerate}
\usepackage[export]{adjustbox}
\usepackage{mathrsfs}
\usepackage[normalem]{ulem}
\usepackage{todonotes}
\usepackage{multirow}
\usepackage{booktabs}

\def\bb#1\eb{{#1}} %
\def\br#1\er{{#1}} %
\def\bw#1\ew{{#1}} %
\def\soutE#1{}

\newcommand{\df}{{\rm d}}
\newcommand{\R}{\mathds R}
\newcommand{\N}{\mathds N}

\newcommand{\LL}{\mathds L}
\newcommand{\sstk}{\ensuremath{\rm SSTK}\xspace}
\newcommand{\eps}{\varepsilon}
\newcommand{\gam}[1]{\ensuremath{\tensor{\Gamma}{^k_i_j}(#1)}}
\newcommand{\de}{\mathrm d}

\newcommand{\MC}{M_{crit}} 
\newcommand{\MM}{M_{mild}} 
\newtheorem{thm}{Theorem}[section]
\newtheorem{prop}[thm]{Proposition}
\newtheorem{lemma}[thm]{Lemma}
\newtheorem{cor}[thm]{Corollary}
\theoremstyle{definition}
\newtheorem{defi}[thm]{Definition}

\newtheorem{exe}[thm]{Example}

\newtheorem{rem}[thm]{Remark}
\newtheorem{convention}[thm]{Convention}

\title[Wind Finslerian structures and spacetimes]{Wind Finslerian structures:\\ from Zermelo's navigation to \\  the causality of spacetimes}

\author[E. Caponio]{Erasmo Caponio}
\address{Dipartimento di Meccanica, Matematica e Management, \hfill\break\indent
	Politecnico di Bari, Via Orabona 4,\hfill\break\indent
	70125, Bari, Italy}
\email{erasmo.caponio@poliba.it}
\thanks{EC is a member of and has been partially supported during this research  by the ``Gruppo Nazionale per l'Analisi Matematica, la Probabilit\`a e le loro Applicazioni'' (GNAMPA) of the ``Istituto Nazionale di Alta Matematica (INdAM)''. \bw Partial supports by  PRIN 2017JPCAPN {\em Qualitative and quantitative aspects of nonlinear PDEs} and \ew \br  by the project MTM2013-47828-C2-1-P (Spanish Ministry of Economy
	and Competitiveness and European Regional Development Fund, ERDF)  \bw are  \ew acknowledged. \er}

\author[M. A. Javaloyes]{Miguel Angel Javaloyes}
\address{Departamento de Matem\'aticas, \hfill\break\indent
	Universidad de Murcia, \hfill\break\indent
	Campus de Espinardo,\hfill\break\indent
	30100 Espinardo, Murcia, Spain}
\email{majava@um.es}
\thanks{\br MAJ was partially supported by MICINN/FEDER project PGC2018-097046-B-I00 funded by MCIN/ AEI /10.13039/501100011033/ FEDER ``Una manera de hacer Europa'' and Fundaci\'on S\'eneca project with reference 19901/GERM/15. \er This work is a result of the activity developed within the framework of the Programme in
 	Support of Excellence Groups of the Regi\'on de Murcia, Spain, by Fundaci\'on S\'eneca, Science and Technology Agency of the Regi\'on de Murcia.}

\author[M. S\'anchez]{Miguel S\'anchez}
\address{Departamento de Geometr\'{\i}a y Topolog\'{\i}a, Facultad de Ciencias,
	\hfill\break\indent \br \& IMAG
	(Centro de Excelencia Mar\'\i a de Maeztu) \er \hfill\break\indent
	Universidad de Granada,\hfill\break\indent
	Campus Fuentenueva s/n,
	\hfill\break\indent 18071 Granada, Spain}
\email{sanchezm@ugr.es}
\thanks{MS has been partially supported by the  project \br  MTM2016-78807-C2-1-P funded by MCIN/ AEI /10.13039/501100011033/ FEDER ``Una manera de hacer Europa'' and by the framework of IMAG-Mar\'{\i}a de Maeztu grant CEX2020-001105-M funded by MCIN/AEI/ 10.13039/50110001103. \er} 



\begin{document}

	\begin{abstract}
		The notion of {\em wind Finslerian structure} $\Sigma$ is developed;
		this is a generalization of Finsler metrics   (and Kropina ones)   where the indicatrices at
		the tangent spaces may not contain the zero vector. In the
		particular case that these indicatrices are ellipsoids, called
		here {\em wind Riemannian structures}, they admit a double
		interpretation which provides: (a) a model for classical {\em Zermelo's
			navigation problem} even when the trajectories of  the moving objects (planes, ships) are
		influenced by  {\em strong} winds or streams, and (b) a natural
		description of the {\em causal structure} of relativistic
		spacetimes endowed with a non-vanishing Killing vector field $K$
		({\em SSTK splittings}), in terms of Finslerian elements. These
		elements can be regarded as conformally invariant Killing initial
		data on a  partial  Cauchy hypersurface. The links between both
		interpretations as well as the possibility to improve the results
		on one of them using the other viewpoint are stressed.
		
		The wind Finslerian structure $\Sigma$ is described in terms of two
		(conic, pseudo) Finsler metrics, $F$ and $F_l$, the former  with a
		convex indicatrix and the  latter  with a concave one. Notions such as balls and geodesics are extended to $\Sigma$.
		Among the applications, we obtain the solution of Zermelo's navigation  with arbitrary  time-independent   wind,
		metric-type properties for $\Sigma$ (distance-type arrival function, completeness, existence of minimizing,  maximizing or closed geodesics), as well as description of spacetime elements (Cauchy developments,  black hole horizons) in terms of Finslerian elements in  Killing initial data.  A general Fermat's principle of independent interest for arbitrary spacetimes, as well as its applications to \sstk spacetimes and Zermelo's navigation, are  also provided. 
	\end{abstract}
	
	%
	%
	
	\maketitle
	
	\newpage
	
	\tableofcontents
	
	\newpage
	
	\section{Introduction}
	Among the classic and recent applications of Finsler metrics,
	Randers ones can be linked to two quite different problems. The first one is Zermelo navigation problem 
	that  was considered for the first time in \cite{Ze31}.
	It consists in determining the trajectories which minimize the flight time of an airship
	(or of any other object capable of a certain maximum speed and moving in a wind or a current). 
	Zermelo determined the differential equations of the optimal trajectories in  dimensions  $2$ and $3$ (the so-called {\em navigation equations}). 
	The problem was then considered  by Levi-Civita, Von Mises, Caratheodory, Mani\`a \cite{Levi-C31, Mises31, Carath67, Mania37} becoming a classical problem in optimal control theory. 
	Randers metrics turned then out  to appear naturally in the problem of navigation under a mild  time-independent    wind  \cite{Sh03, BaRoSh04}. 
	
	The second one is the description of the  conformal  geometry  of spacetimes  $(\R\times M,g)$ endowed with a timelike Killing vector field $K$ (the so-called {\em standard stationary spacetimes}). This is an important class of spacetimes: for example, the region outside the ergosphere in   Kerr's solution  to   Einstein's equations  is of this type  and, more generally, the region outside the horizon of any  black hole should be so,  at sufficiently late times  (see \cite[\S 14.4]{ludvig}).   
	Also in this case, Randers metrics  arise naturally on $M$,
	encoding the  causality  of the  spacetime
	\cite{CapJavSan10}.

	In both cases, the interpretation of a
	Randers metric as a Riemannian one $(M,g_R)$ ``with a displaced
	unit ball'' becomes apparent: the displacement is caused by the
	vector field $W$ which represents the wind  in  the case of Zermelo's
	problem, and which is constructed in a conformally invariant way from 
	the {\em
		lapse} $\Lambda=-g(K,K)$ and the {\em shift}
	$\omega=g(K,\cdot)|_M$ in the case of spacetimes. It is
	remarkable that   Randers metrics  provide a natural way to go
	from the navigation problem to spacetimes, and vice versa. 
	
	In both
	problems, however, there is a neat restriction: the wind must be
	mild ($g_R(W,W)<1$) and, accordingly, the lapse of the spacetime
	must be positive ($\Lambda>0$); otherwise, the displaced unit
	ball  would not contain the zero vector, making  to collapse  the classical
	Finslerian description. Nevertheless, both problems are
	natural without such restrictions and, in fact, they become even
	more geometrically interesting then. Under a  {\em strong }  wind or
	current, the moving object (a Zeppelin or   a   plane in the air, a ship
	in  the ocean,   
	or even sound rays in the presence of a wind \cite{GW10,GW11}) may   face 
	both, regions which cannot be reached and others that can be
	reached but must be abandoned by, say, the compelling wind.
	Analogously, the change in the sign of the lapse $\Lambda$ means
	that the causal character of the Killing vector field $K$ changes
	from timelike to spacelike and, so, one might find a {\em Killing
		horizon}, which is an especially interesting type of relativistic
	hypersurface \cite{ChCoHe12, MarsReiris}. The correspondence between
	navigation and spacetimes becomes now even more appealing:
	although the description of the movement of the  navigating  object is
	non-relativistic, the set of points that can be reached at each
	instant of time becomes naturally described by the causal future
	of an event in
	the spacetime, and the latter may exhibit some of the known subtle
	possibilities in relativistic fauna: horizons, no-escape  regions (black
	holes) and so on.
	
	Our   aim   here is to show that  both Zermelo navigation  in the air or the sea, represented by a Riemannian manifold $(M,g_R)$,  with  time-independent  wind  $W$,  and the geometry of  a spacetime  $(\R\times M,g)$,  with a non-vanishing Killing vector field  $K$,   can still be described by a generalized Finsler structure  $(M,\Sigma)$,  that we call  {\em wind Riemannian}.    Roughly, $\Sigma$ is the hypersurface of the tangent bundle $TM$ which contains the maximum velocities of the \bw \soutE{mobile} \br moving \er object \ew in all the points and all directions, i.e., each $\Sigma_p\subset T_pM$ is obtained by adding the wind $W_p$ to the $g_R$-unit sphere at $p$, the latter representing the maximum possible velocities developed by the engine 
	of the \bw \soutE{mobile} \br moving \er object \ew at $p$ with respect 
	to the air or sea. 
	
	By using this structure, we  can  interpret Zermelo navigation    as a  problem about geodesics  whatever the  strength of the   wind is  and 
	we give  sufficient conditions   for the existence of   a solution minimizing or maximizing  travel time   (Theorem~\ref{compactcase}). These are based on an assumption, called {\em w-convexity} 
	which is satisfied if the wind Riemannian structure  $\Sigma$  is geodesically complete. Clearly, this might hold  also when $M$ is not compact, a case in that the so-called {\em common compact support} hypothesis in Filippov's theorem, applied  to the time-optimal control problem describing Zermelo navigation, does not hold (see \cite[Th. 10.1]{agrachev} and \cite[p. 52]{userres}). For example,  our techniques can also be used to prove existence of a solution in a (possibly unbounded) open subset of a manifold $M$, provided that the wind is mild in   its boundary and the boundary is convex  (Theorem~\ref{lake} and Remark~\ref{unboundomain}).

	As mentioned above, wind Riemannian structures allow us to describe also the causal structure  of a spacetime $(\R\times M,g)$ endowed with a non-vanishing Killing vector field \br $\partial_t$ \er which is everywhere transverse to the spacelike hypersurfaces $S_t=\{t\}\times M$. We name this type of spacetimes  {\em standard  with a space-transverse
		Killing vector field},  abbreviated in {\it \sstk splitting}.   They are endowed with a $t$-independent metric
	\begin{equation*}
		g=-\Lambda dt^2+2\omega dt + g_0
	\end{equation*}
	(see Definition~\ref{dsstk} and Proposition~\ref{psstk} for accurate
	details), so the Killing vector field is $K=\partial_t$. Even though  \sstk splittings are commonly used
	in General Relativity (see for example \cite{MarsReiris} and
	references therein),     we do not know any previous  systematic study of their causal structure,  so, this is carried out here with full depth.   Of course, \sstk splittings  include standard stationary spacetimes (i.e. the case in that $K$ is timelike or, equivalently, $\Lambda=-g(\partial_t,\partial_t)$ is positive) and also  asymptotically flat spacetimes admitting a Killing vector field which is only asymptotically timelike (which,  sometimes in the literature on Mathematical Relativity,  are also called stationary spacetimes, see for example \cite[Definition 12.2]{ludvig}).  The spacetime viewpoint will be crucial to solve technical problems about wind Riemannian structures. 
	
	The point at which Zermelo navigation and the causal geometry of an \sstk splitting are more closely related is  Fermat's principle. We prove here a Fermat's principle in a very general setting which is then refined when the ambient spacetime is  an  \sstk splitting. 
	Classical Fermat's principle, as established by Kovner \cite{Kovner90} and Perlick \cite{Pe90}, characterizes lightlike 
	pregeodesics 
	as the critical points of the arrival functional for smooth lightlike curves joining a prescribed point $z_0$ and a {\em timelike} curve $\alpha$. 
	However, the case when $\alpha$ is not timelike becomes also very interesting for different purposes. First, of course, this completes the mathematical development of the problem. 
	In particular, the proof of the result here, Theorem~\ref{lema:fermatprinc}  (plus 
	further extensions there),  refines all previous approaches. However, this result and its strengthening to \sstk spacetimes (Theorem~\ref{fp},  Corollary~\ref{fptimelike}),  admit interpretations for Zermelo's navigation, as well as for spacetimes (arrival at a Killing horizon) and even for the classical Riemannian viewpoint (Remark~\ref{eriemann}).   
	\br Specifically, \er \soutE{Concretely,} about Zermelo's  navigation, the case when the arrival curve $\alpha$ is not timelike corresponds  to a target point which lies in a zone of critical or strong wind ($g_R(W,W)\geq 1$). Thus, Fermat's principle 
	can be interpreted  as a variational principle for a generalized Zermelo's navigation problem, in the sense that navigation paths are  the  critical  (rather than only local minimum) points of the time of navigation. 
	
	About the technical framework of variational calculus, we would like to emphasize   that the travel time minimizing paths  between two given points $x_0, y_0\in M$  are the curves $\sigma$ connecting $x_0$ to $y_0$ which minimize the functional 
	$$\sigma\mapsto \int_{\sigma} \frac{g_R(\dot\sigma,\dot\sigma)}{g_R(\dot\sigma, W) + \sqrt{h (\dot\sigma,\dot\sigma)}},$$
	where 
	\begin{equation}\label{hagain}
		h(v,v):=(1-g_R(W,W))g_R(v,v)+g_R(v,W)^2 
	\end{equation}
	is a signature changing tensor on $M$ which is Riemannian on  the region of mild wind, Lorentzian  on the region of strong wind, while in the region of critical wind (i.e., at the points $p\in M$ where $g_R(W_p,W_p)=1$), it is degenerate.  On the region of critical or strong wind, this functional  is defined (and positive) only for curves whose velocities belong to a conic sub-bundle of $TM$ (see Proposition~\ref{Riemannclass} and  Proposition~\ref{randerskropinagen}).  
	This  constraint on the admissible velocities  plus the signature changing characteristic of $h$ make it difficult the use of  direct methods. Actually, we are able to prove the existence of a minimum by using  Lorentzian  results  about the existence of {\em limit curves} (see Definition~\ref{limitcurvedef}, Lemma~\ref{limitcurve}) in the \sstk splitting that can be associated  with  a data set $(M,g_R,W)$ for Zermelo navigation (Theorem~\ref{tfermatSSTK}).  What is more,  focusing only on the minimizing problem (or the optimal time control problem) is, in our opinion,  somehow reductive of the rich geometrical features  of  Zermelo navigation.   For example, Caratheodory abnormal geodesics \cite[\S 282]{Carath67} (see Section~\ref{ss6.3}) are interpreted here as both, 
	lightlike pregeodesics of  $h$ (up to a finite number of instants where the velocity vanishes) or  exceptional geodesics of the wind Riemannian structure 
	(Definition~\ref{windgeodesic}).

	In our study, we will proceed even from a more general viewpoint. We will move the indicatrix of any Finsler metric by using an arbitrary vector field $W$ and call the so-obtained hypersurface $\Sigma$ a 
	{\em wind Finslerian structure}.  We provide a thorough study of such a structure, which is then strengthened for wind Riemannian structures    thanks to the correspondence with conformal classes of \sstk splittings.    Of course,  wind Finslerian structures   
	generalize the class of all Finsler manifolds     because    the zero vector is allowed to belong to or to be outside  each hypersurface  $\Sigma_p=T_pM\cap \Sigma$.   Remarkably,   the correspondence   between \sstk splittings and wind Riemannian structures allows us to study the latter, including some``singular''  Finslerian geometries (such as
	the well-known Kropina metrics, 
	where the $0$ vector belongs to the indicatrix   $\Sigma_p$) in terms of the
	corresponding (non-singular) \sstk splitting.

	\smallskip
	
	\noindent    Next, we give a brief description of each section, which may serve as a guide for the reader.    In Section~\ref{section:windFinsler}, we start by
	introducing   wind Finslerian
	structures on a manifold. These will be defined in terms of a
	hypersurface $\Sigma$ of $TM$, satisfying a transversality condition  which
	provides a strongly convex compact hypersurface $\Sigma_p$ at each point $p\in M$,  called {\em wind Minkowskian structure}. This structure  plays the role of
	indicatrix,  although it might not surround the origin $0_p\in T_pM$.   An obvious example appears  when  the
	indicatrix  bundle  of a Finsler manifold $F_0$ is displaced along a vector
	field $W$  and any such $\Sigma$ can be constructed from some $F_0, W$ (clearly not univocally determined, even though a natural choice can be done), see Proposition~\ref{pzerm}.    The intrinsic analysis of $\Sigma$ shows: 
	\begin{quote}
		{\em      Any wind Finslerian structure $\Sigma$ can be described in
			terms of two {\em conic pseudo-Finsler} metrics $F$ and $F_l$, the
			former $F$ (resp. the latter $F_l$) defined on all $M$ (resp. in
			the region $M_l$ of {\em strong wind}, i.e., whenever the zero
			vector is not enclosed by $\Sigma$) with:
			
			{\em (i)} domain $A_p\subset T_pM$ at each $p\in M$ (resp. each
			$p\in M_l)$ equal to the interior of the conic region of $T_pM$
			determined by the half lines from the origin to $\Sigma_p$, and
			
			{\em (ii) } indicatrix the part of $\Sigma_p$ that is convex  $\Sigma_p^+$ 
			(resp. concave  $\Sigma_p^-$)  with respect to the position vector ---so that $F$ becomes a {\em conic Finsler metric} and
			$F_l$  a {\em Lorentzian Finsler metric} {\em (Proposition~\ref{possibwind},  Figure~\ref{dis0})}.
			
			\smallskip
			
			\noindent Moreover, $\Sigma$ admits general notions of lengths and
			balls {\em (Definitions~\ref{sigmadmissible}, \ref{sigmaballs})}, which
			allows us to define geodesics {\em (Definitions~\ref{extremizing}, \ref{windgeodesic})}, recovering the usual
			geodesics for  $F$ and $F_l$ } (
		Theorem~\ref{extregeo}). 
	\end{quote}
	Remarkably, we introduce the notion of {\em c-ball} in order to
	define geodesics  directly for $\Sigma$.  These balls  are intermediate  between open and
	closed balls. They make sense even in the Riemannian
	case  (Example~\ref{ex:riemcballs}),
	allowing a well motivated notion of convexity, namely, 
	w-convexity (Proposition~\ref{pgc}, Definition~\ref{strongconvex}).
	
	Especially, we focus  on the case when $\Sigma$ is a {\em wind
		Riemannian structure} (Section~\ref{windRiemann}).   The link with Zermelo's problem becomes
	apparent: $F$ describes the maximum velocity that the ship can
	reach in each direction and $F_l$ the minimum one. In this case,
	the conic pseudo-Finsler metrics $F, F_l$ can be described
	naturally in terms of  the data $g_R$ and $W$  (Proposition~\ref{Riemannclass}),
	and a generalization of the Zermelo/Randers
	correspondence is carried out: now Randers metrics appear for mild
	wind ($g_R(W,W)<1$), the pair $(F, F_l)$ for strong wind  ($ g_R(W,W)>1$), and Kropina metrics for
	the case of critical wind ($g_R(W,W)=1$).  In particular, $F$ becomes
	a Randers-Kropina metric in the region of non-strong wind (Definition~\ref{dranderskropina}, Proposition~\ref{randerskropinagen}).
	
	\smallskip
	
	\noindent In Section~\ref{windFermat}, our aim is to describe the
	correspondence between the wind Riemannian structures and the
	(conformal classes)  of \sstk splittings. 
	The existence of a unique {\em Fermat structure}, i.e., a wind
	Riemannian structure $\Sigma$ naturally associated  with the
	conformal class $[(\Lambda, \omega, g_0)]$ of an \sstk
	splitting, 
	is characterized in
	Theorem~\ref{tfermatSSTK}.
	Moreover, the equivalence between  these conformal \sstk -classes,
	and the description of a wind Riemannian structure either with
	Zermelo-type elements (i.e., in terms of a Randers-Kropina metric
	or a pair   of metrics $(F,F_l)$) or with its explicit Riemannian
	metric and wind (i.e., the pair $(g_R,W)$) is analyzed  in detail,
	see the summary in Fig.~\ref{rela}.
	In Subsection~\ref{sh} we identify and interpret the 
	(signature-changing) metric $h$ in \eqref{hagain}, 
	which becomes
	Riemannian when $\Lambda>0$, Lorentzian of coindex $1$  when
	$\Lambda<0$ and degenerate otherwise.  In particular, on its causal  (timelike or lightlike)  vectors in $TM_l$,
	it holds
	\begin{equation}\label{ehffl}
		h(v,v)=\frac{1}{4}  (1-g_0(W,W))^2(F-F_l)^2 (v),
	\end{equation}
	(see \eqref{eh}, Corollary~\ref{hcases}). As mentioned above, the metric $h$ will turn out essential  for describing  certain solutions of the Zermelo navigation problem.    We emphasize that, even though $h$  has a
	natural interpretation from the spacetime viewpoint
	(Proposition~\ref{gh}), 
	its importance would be difficult to discover from the purely
	Finslerian viewpoint (that is, from an expression such as
	\eqref{ehffl}). Summing up:
	\begin{quote} {\em  Any wind Riemannian structure $(g_R,W)$ becomes equivalent to an \sstk conformal class $[(g_0, \omega, \Lambda)]$.
			The spacetime interpretation allows us  to  reveal  elements (as the
			metric $h$ in \eqref{hagain},  \eqref{ehffl},  \eqref{eh})  and to find illuminating interpretations which will become essential for the analysis of Finslerian properties as well as for the solution of technical problems there. }
	\end{quote}
	We end with a subsection where
	the  fundamental tensors of $F$ and $F_l$ are computed
	explicitly  and discussed
	---in particular, \br  this makes it \er possible  to check the Finslerian character of
	the former and the Lorentzian Finsler one of the latter.

	\smallskip
	
	\noindent  About Sections~\ref{kroran} and \ref{generalcase}, recall first that the  main theorems of this paper deal with an exhaustive
	correspondence between the causal properties of an \sstk splitting
	and the metric-type properties of wind Riemannian structures.
	These theorems will become fundamental from both, the spacetime viewpoint (as important relativistic properties are characterized) and the viewpoint of navigation and wind Riemannian structures (as sharp characterizations on the existence of critical points/ geodesics are derived by applying the spacetime machinery).
	For the convenience of the reader, they are obtained gradually in
	Sections~\ref{kroran} and \ref{generalcase}.
	
	In Section
	\ref{kroran}, we consider the case when the Killing  field  $K$  of the $\sstk$ spacetime  is causal
	or, consistently, when the Fermat
	structure  has  (pointwise)  either mild or critical 
	wind. In this case, the Lorentzian Finsler metric $F_l$ is not
	defined, and the conic Finsler metric $F$ becomes a
	Randers-Kropina one. We introduce the $F$-separation  $d_F$  in a way
	formally analogous to the (non-necessarily symmetric) distance of
	a Finsler manifold. But, as the curves connecting each pair of
	points must be admissible now (in the Kropina region, the
	velocity of the curves must be included in the open half tangent
	spaces where $F$ can be applied), one may have, for example,
	$d_F(x,x)=+\infty$ for some $x\in M$.  In any case,  the
	chronological relation $\ll$ of the \sstk splitting can be
	characterized in terms of $d_F$ (Proposition~\ref{bolas}), and this
	allows us to prove that $d_F$  is still continuous outside the
	diagonal (Theorem~\ref{tcontdf}). The main result, Theorem~\ref{kropinaLadder}, yields a full characterization of the
	possible positions of the \sstk splitting in the so-called {\em
		causal ladder of spacetimes} in terms of the properties of $d_F$.
	This extends the results for stationary spacetimes in
	\cite{CapJavSan10}, and they  are applicable to relativistic spacetimes
	as the pp-waves (Example~\ref{ex_ppwave}). A nice  straightforward
	consequence  is a version of Hopf-Rinow  Theorem for the $F$-separation of any Randers-Kropina metric 
	(Corollary~\ref{cRandersKropinaHopfRinow}).
	
	In Section ~\ref{generalcase} the general case when there is no restriction on $K$ (i.e., a strong wind is permitted) is
	considered.  
	\soutE{Now,  there is definitively no any element
		similar to a distance. However} \bw In this case, \ew  our definitions of balls and
	geodesics are enough for a full description of the causal ladder  of the spacetime. 
	In fact, the chronological and causal futures, $I^+(z_0),
	J^+(z_0)$, of any \sstk -point $z_0\in \R\times M$ can be described
	in terms of the $\Sigma$-balls and c-balls in $M$  (Proposition~\ref{bolas2}).  Moreover, the
	horismotically related points (those in $J^+(z_0)\setminus
	I^+(z_0)$) are characterized by the existence of extremizing
	geodesics (Corollary~\ref{horismos}).  This  leads to  a complete description of the geodesics of an \sstk splitting in terms of the geodesics of its Fermat structure (Theorem~\ref{existenceofbolasNO}, Corollary~\ref{lightgeo2}, see also Fig.~\ref{geos}).  
	In order to characterize the
	closedness of $J^+(z_0)$ (Proposition~\ref{lreflect}), as well as to
	carry out some other technical steps, we require a result of
	independent interest about 
	limit curves (Lemma~\ref{limitcurve}). 
	This machinery allows us to prove our structural
	Theorem~\ref{generalK} which, roughly, means:
	\begin{quote} {\em
			Any \sstk splitting $(\R\times M,g)$ is stably causal and it will
			have further causality properties when some appropriate properties
			of the balls or geodesics of the corresponding Fermat structure
			$(M,\Sigma)$ hold. In particular, $(\R\times M,g)$ is causally
			continuous iff a natural property of symmetry holds for the closed
			balls of $(M,\Sigma)$, it is causally simple iff $(M,\Sigma)$ is
			w-convex and it is globally hyperbolic iff the intersections
			between the forward and backward closed $\Sigma$-balls are
			compact. Moreover, the fact that the slices $S_t=\{t\}\times M$  are
			Cauchy hypersurfaces is equivalent to the (forward and backward)
			geodesic completeness of $(M,\Sigma)$.}\end{quote}
	
	\noindent Section~\ref{causaltoRiemann} is devoted to the applications of the SSTK viewpoint to the geometry
of wind Riemannian structures. This follows the spirit of \cite{CapJavSan10, FlHeSa13}  in the stationary case, where the spacetime viewpoint allows one to find properties for Randers  metrics (see \cite[Section 5]{CapJavSan10}  or, for example, the section 5.3.4 in \cite{FlHeSa13}). In the case of wind Riemannian structures this viewpoint  becomes crucial due to
the appearance of certain singularities in the Finslerian elements. Indeed, it  offers  neat interpretations which permit to solve technical problems and serves as a guide for different developments. 
	Subsection~\ref{ss6.1} develops direct consequences of the previous results:  (1) a full characterization of the
	$\Sigma$-geodesics  as either (a) geodesics for $F$ or $F_l$, or (b) lightlike pregeodesics of $-h$ in the region of strong wind, up to isolated points of vanishing velocity  (Theorem~\ref{minimizers};
	the last possibility refines the result for any wind Finslerian
	structure in Theorem~\ref{extregeo})  and (2) a characterization of
	completeness and w-convexity in the spirit of Hopf-Rinow theorem
	(Proposition~\ref{c63}). However, in Subsection~\ref{ss6.2} a subtler
	application on $(M,\Sigma)$ is developed. Indeed, the same
	spacetime may split as an \sstk in two different ways (Lemma~
	\ref{fsplitting}), yielding
	two different Fermat structures (Proposition~\ref{changedf}). These
	structures share some properties  intrinsic to the \sstk spacetime and their consequences for the wind Riemannian structures associated with each splitting are analyzed.  
	In  Subsection~\ref{ss6.3}  we introduce  
	a relation of {\em weak precedence} $\preceq$ (resp. {\em precedence} $\prec$) between pairs of points in $(M, \Sigma)$ defined by the existence of a connecting {\em wind curve} (resp. an {\em $F$-wind curve}), namely, 
	a curve with velocity included in the region
	(resp. the interior of the region) allowed by $\Sigma$. Such a relation can be  characterized as the projection of the causal (resp. chronological) relation
	on the corresponding \sstk  (Proposition~\ref{esalva}).  This  allows us  to prove results on existence of
	minimizing  and maximizing  connecting geodesics    (Theorems~\ref{compactcase} and  \ref{lake},     Theorem~\ref{maxZermelo}) and of closed geodesics for
	$(M,\Sigma)$ (Theorem~\ref{prop:closedgeo}). In particular,  Theorems~\ref{compactcase},~\ref{maxZermelo} and Corollaries~\ref{compactcase2},~\ref{rsummaryZ}
	provide the full  solution to Zermelo
	navigation problem:
	\begin{quote}
		{\em For any wind Riemannian structure,  the solutions of Zermelo problem are pregeodesics of $\Sigma$. 
			The metric $-h$ in \eqref{eh} defines a natural relation of weak precedence  $\preceq$ (resp. precedence $\prec$) which determines if a point $x_0$ can be connected with a
			second one $y_0$ by means of a  wind 
			(resp. $F$-wind) curve; when the wind is strong,  i.e. $M=M_l$, $-h$ becomes Lorentzian on all $M$ and the relation of weak precedence (resp. precedence) coincides with the natural causal (resp. chronological) 
			of $-h$. Then:

			(a) if
			$x_0\preceq y_0$,  $x_0\neq y_0$,   and the c-balls are closed (i.e. $(M,\Sigma)$ is w-convex) then there exists a geodesic  of $(M,\Sigma)$ of minimum $F$-length from $x_0$ to $y_0$
			(which is also a lightlike pregeodesic of $-h$  when $x_0\not \prec   y_0 $);
			
			(b) if  $x_0\preceq y_0$, $x_0\neq y_0$,  the wind is strong  and $-h$ is globally hyperbolic on all $M$  then there exists   a  geodesic  of $(M,\Sigma)$ of maximum $F_l$-length
			from $x_0$ to $y_0$ (which  is also a lightlike pregeodesic of $-h$  if $x_0\not \prec y_0$).}
	\end{quote}
	The possibility of the existence of maximal solutions as well as of solutions which are limits of minimal and maximal ones was pointed out by Caratheodory in \cite{Carath67}  (see  the discussion at  part  (2) below Corollary~\ref{rsummaryZ}). 
	We stress that our result interprets geometrically all of them as geodesics. In particular, the limits of minimal and  maximal ones correspond (up to isolated points) to the lightlike  pregeodesics  of $-h$. 
	We would like to emphasize that the  accuracy of
	most of our results for wind Riemannian structures relies on
	their correspondence  to   \sstk splittings  (see, e.g. Proposition~\ref{rappendix}).    Nevertheless, some
	of these results might be  extendible to general wind Finslerian 
	ones.\footnote{ Indeed, in the case
		of the correspondence of Randers metrics with stationary
		spacetimes  already developed in \cite{CapJavSan10},  some of the properties obtained by using the spacetime viewpoint  could be extended to any
		Finslerian manifold (see for example \cite{Matvee13} or compare \cite[Th. 4,10]{CapJavSan10} with \cite[Theorem A]{TanSab12}).  Thus, the results for the wind Riemannian case might serve as a guide for a further development of  wind  Finslerian  structures \bw as explained in Section~\ref{conclusions}. \ew}
	
	\smallskip
	
	\noindent  In Section~\ref{further1},  we discuss  Fermat's principle, which constitutes a topic of  interest in its own right.  After an introductory motivation in Subsection~\ref{Fermatprinciple1}, in Subsection~\ref{Fermatprinciple2} we prove our Generalized Fermat's principle valid for causally arbitrary arrival curves (Theorem~\ref{lema:fermatprinc}). Moreover, we also develop an extension to the case when the trial curves are timelike with a prescribed proper time (instead of lightlike with necessarily 0 proper time, Corollary~\ref{timelikefermat}) as well as a first application to two purely Riemannian variational problems (Corollary~\ref{cor:appRieFermat}).  In Subsection~\ref{Fermatprinciple3} Generalized Fermat's principle is refined for Zermelo trajectories in SSTK spacetimes, providing a variational interpretation of the geodesics of any wind Riemannian structure (Theorem~\ref{fp}, Corollary~\ref{fptimelike}).

	\smallskip
	
	\noindent  In Section~\ref{further2}, we go further
	in the description of causal elements of \sstk splittings.
	Indeed,  in Subsection~\ref{develops}, 
	Cauchy developments and horizons of subsets included in a slice
	$S_t$ are described accurately in terms of the Fermat structure
	(Proposition~\ref{cauchydevhor}). As a nice consequence,   in Subsection~\ref{s8.2}  the results on
	differentiability of horizons for spacetimes can be now applied to
	obtain results on  smoothability of the Randers-Kropina separation
	$d_F$ to a subset (Proposition~\ref{smoothdF}), so extending results  in \cite{ChFuGH02}  for the Riemannian case and in \cite[\S
	5.4]{CapJavSan10} for the Randers one.
	In the last part (Subsection~\ref{ss6.5}), 
	we also introduce and develop the
	notion of $K$-horizon for any wind Finslerian structure. In
	particular, such horizons  allow us  to describe  the regions where the
	ship in Zermelo's navigation cannot enter (or from where it cannot
	escape). Accordingly, from the spacetime viewpoint, it provides  a
	description of black hole regions from the {\em Killing initial
		data} (KID)   on a Riemannian manifold
	$(M,g_0)$ for any \sstk splitting (see  \cite{CaMars, BeCh,Mae, MarsReiris}). Notice that these data appear naturally  in the initial
	value problem for  the Einstein equation, and include our $\Lambda$ and $
	\omega$ (usually denoted $N$  and $Y$ in Physics literature,
	the latter regarded eventually as a vector field). Given the initial data, the \sstk
	splitting is called its {\em (infinite) Killing development} \cite[Definition
	2]{MarsReiris}. When the  initial data are  well posed (namely, they satisfy conditions of compatibility with matter in the sense of \cite[Definition 2]{CaMars}), the
	Cauchy development of $S$ will include the unique maximal globally
	hyperbolic spacetime obtained as a solution of the Einstein
	equation.  Our results on Cauchy
	developments  make it possible  to determine these regions,  as well as possible black hole horizons,   in terms of the
	Fermat structure.
	
	\smallskip
	
	\noindent   Finally,  in Section~\ref{conclusions} some conclusions \soutE{and prospects} are summarized. \bw  Moreover, we survey a few  recent works, related to the topics  of the present paper, trying to provide to the interested reader a brief guide about developments and further investigations. \ew    
	
	 Due to the big number of notions here introduced,  an appendix containing  a list of symbols and definitions used throughout the paper is given for the  reader's convenience.

	\section{Wind Finslerian structures}\label{section:windFinsler}
	\subsection{Wind Minkowskian structures on vector spaces} Let us begin by  recalling the classical notion of
	Minkowski norm.
	\begin{defi}\label{min}
		Let $V$ be a   real  vector space of finite dimension $m\geq
		1$.  We say that a continuous non-negative function
		$F:V\rightarrow [0,+\infty)$ is a Minkowski norm \index{Minkowski norm $F$} if
		\begin{enumerate}[(i)]
			\item it is positive and smooth away from the zero vector,
			\item it is positively homogeneous, namely, $F(\lambda v)=\lambda F(v)$ for every $\lambda> 0$  and $v\in V$,
			\item for any $v\in V\setminus  \{0\}$, its {\it fundamental tensor} $g_v$, defined as
			\begin{equation}\label{fundten}
				g_v(u,w)= \frac{1}{2}\frac{\partial^2}{\partial s\partial t}F(v+tu+sw)^2|_{t=s=0}
			\end{equation}
			for any $u,w\in V$, is positive definite.
		\end{enumerate}
	\end{defi}
	The {\it  indicatrix} of $F$ is defined as the subset $\Sigma_F=\{v\in V: F(v)=1\}$. Observe that $\Sigma_F$ is a  {\em strongly convex} smooth  hypersurface embedded in $V$, in the sense that its   second fundamental form  $I\!I$ 
	with respect to one (and then all) transversal vector field is definite ---in the remainder, we choose the orientation of the transverse so that $I\!I$ will be positive definite, as usual.
	Notice that, in general,  any (connected) compact,  strongly convex hypersurface $\Sigma$ embedded in $V$ must be a topological sphere (the Gauss map with respect to any auxiliary scalar product would yield a diffeomorphism) and both, 
	$\Sigma$ and the bounded region $B$ determined by $\Sigma$, are {\em strictly convex} in the usual sense
	(i.e. $\Sigma$ touches every hyperplane tangent to it only at the tangency point and lies in one of the two half-spaces determined by the hyperplane,
	and $B$ satisfies that
	the segment between any two points in $\bar B=B\cup \Sigma$ is contained in $B$, except at most its endpoints). When $0\in B$, a Minkowski norm is
	uniquely determined having $\Sigma$ as indicatrix just by putting
	$F(v)=1/\lambda(v) $
	for all $v\in V\setminus \{0\}$,
	where $\lambda(v)\in \R$ is the unique positive number such that $\lambda(v)v\in \Sigma$  (see for example \cite[Prop. 2.3]{JavSan11}).

	If the indicatrix $\Sigma_0$ of a given Minkowski norm is translated, one obtains another strongly convex smooth hypersurface $\Sigma$ that determines a new Minkowski norm whenever $0$ still belongs to the new bounded region $B$. As explained in the Introduction, 
	this process of generating Minkowski norms is used pointwise  in  Zermelo's navigation problem  and one obtains (see Fig.~\ref{dis0}): 
	\begin{prop}
		Let $\Sigma_0$ be the indicatrix of a Minkowski norm. The translated indicatrix $\Sigma=\Sigma_0+W$ defines  a Minkowski norm if and only   if  $ F_0(-W)<1$.  
	\end{prop}
	This is a restriction of ``mild wind'' in Zermelo's problem; 
	so, let us consider now
	the case in that
	$ F_0(-W) \geq 1$.
	In this case, the zero vector is not contained in the  open bounded region $B$ delimited  by the translated indicatrix $\Sigma$  and, as a consequence,  $\Sigma$ does  not define a classical Finsler metric.
	Indeed, not all the rays departing from the zero  vector  must intersect  $\Sigma$ and, among the intersecting ones, those intersecting transversely will cross  $\Sigma$ twice, and those intersecting non-transversely 
	will intersect  only once, see Fig.~\ref{dis0}.
	\begin{figure}[h]
		\includegraphics[scale=1,center]{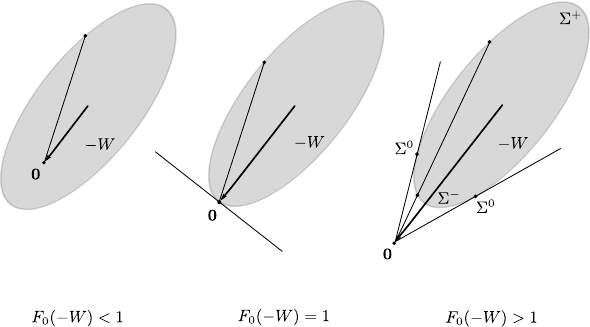}
		\caption{Wind Minkowskian structures}\label{dis0}
	\end{figure}
	The above discussion motivates the following definition.
	\begin{defi}\label{min2}
		A \emph{wind Minkowskian structure} on a real vector space $V$
		of dimension $m\geq 2$ (resp. $m=1$)  is a   compact,
		connected,  strongly convex,
		smooth hypersurface $\Sigma$
		embedded in $V$  (resp. a set of two points $\Sigma =
		\{\lambda_- v_0, \lambda_+ v_0\}$, $ \lambda_- < \lambda_+$, for
		some $v_0\in V \setminus \{0\}$). The bounded open domain $B$ (resp.
		the open segment $\{t\lambda_- v_0 + (1-t) \lambda_+ v_0: t\in
		(0,1)\}$) enclosed by $\Sigma$ will be called the {\em unit ball}
		of the wind Minkowskian structure.
	\end{defi}
	As an abuse of language, $\Sigma$   may also be  said  the {\em unit sphere} or the {\em indicatrix} of the wind Minkowskian structure.
	In order to study wind Minkowskian structures, it is convenient to
	consider the following generalization of Minkowski norms (see
	\cite{JavSan11} for a detailed study). 
	\begin{defi}\label{defA}
		Let $A\subset V
		\setminus \{0\}$ be  an  open {\it conic subset}, in the sense
		that if $v\in A$, then $\lambda v\in A$ for every $\lambda>0$.\footnote{Notice that, if an open conic subset $A$ contains the zero vector then   $A=V$. 
			As we will be especially  interested in the case $A\neq V$, in the remainder the $0$ vector will \bw always be \ew removed from $A$ for convenience. For comparison with the results in \cite{JavSan11}, notice that $A\cup \{0\}$ 
			will \bw always be \ew  convex in the following sections, even though one does not need to assume this a priori.}
		We say that a function $F:A\rightarrow [0,+\infty)$ is
		a {\it conic pseudo-Minkowski norm} if it satisfies $(i)$ and
		$(ii)$ in Definition~\ref{min}  (see \cite[Definition
		2.4]{JavSan11}). Moreover,  if \soutE{for any $v\in A$,} the fundamental tensor $g_v$ defined in \eqref{fundten} 
		is positive definite \bw for any $v\in A$, \ew  then $F$ is said a {\it conic Minkowski norm} while  if it 
		has coindex $1$  then $F$ is said a {\it  Lorentzian norm}.
	\end{defi}
	Of course, any conic pseudo-Minkowski norm  can be extended continuously to $0$ whenever $0$ does not lie in the closure in $V$ of the indicatrix and this is natural in the case  $A=V\setminus \{0\}$; in particular, Minkowski norms can be seen as  conic pseudo-Minkowski norms.

	According to these definitions, there are three different possibilities  for a wind Minkowskian structure.
	\begin{prop}\label{possibwind}
		Let $\Sigma$ be a wind Minkowskian structure in $V$ and $B$ its unit ball. 
		\begin{enumerate}[(i)]
			\item If $0\in B$, then $\Sigma$ is the indicatrix of a Minkowski
			norm. \item If $0\in \Sigma$, then $\Sigma$ is the indicatrix of a
			conic Minkowski norm with domain $A$ equal to an (open) half
			vector space. \item If $0\notin \bar B$, then define $A\subset
			V\backslash \{0\}$ as the interior of the set which includes all
			the rays starting at 0 and crossing $\Sigma$;  then $A$ is a
			(convex) conic open set and, when $m\geq 2$,  two conic
			pseudo-Minkowski norms $F, F_l$ with domain $A$ can be
			characterized as follows:
			\begin{enumerate}
				\item[(a)] each one of their  indicatrices is a connected part of $A\cap \Sigma$,  and
				\item[(b)] $F$ is a conic Minkowski norm and $F_l$, a  Lorentzian norm.
			\end{enumerate}
			Moreover, $F<F_l$ on all $A$,  both  pseudo-Minkowski  norms  can be extended  continuously  to  the closure $A_E$ of $A$ in $V\setminus \{0\}$  and both extensions coincide on the boundary of $A_E$.
		\end{enumerate}
		We will say that $\Sigma$ in each one of the previous cases is,   respectively,  a Minkowski norm, a  Kropina type norm or a strong (or proper)  wind Minkowskian  structure.
	\end{prop}
	\begin{proof}
		Parts $(i)$ and $(ii)$ are an easy consequence of \cite[Theorem 2.14]{JavSan11}. For part $(iii)$,
		if a ray  from zero meets $\Sigma$ transversely, it will cut $\Sigma$ in two points whereas if it is tangent to $\Sigma$ there will be a unique cut point. Then we can divide $\Sigma$ in three disjoint regions $\Sigma=\Sigma^-\cup \Sigma^0\cup \Sigma^+$, 
		where $\Sigma^-$ and $\Sigma^+$ are the sets of
		the points where the rays departing from $0$ cut $\Sigma$ transversely, first in $\Sigma^-$ and then in $\Sigma^+$, and $\Sigma^0$  is the set of  points where the rays from zero are tangent to  $\Sigma$   (see Fig.~\ref{dis0}).  
		The rays cutting $\Sigma^-\cup \Sigma^+$ generate the open subset $A\subset V$;
		recall that the  compactness and  strong convexity of $\Sigma$ imply both, the arc-connectedness of $\Sigma^-$  and $\Sigma^+$,  and the  convexity of $A$, ensuring $\it (a)$.
		Moreover, $\Sigma^-$ defines a Lorentzian norm $F_l$, since the restriction of its fundamental tensor  $g_v$ to the tangent hypersurface to $\Sigma^-$ is negative definite and $g_v$-orthogonal to $v$ \cite[Prop. 2.2]{JavSan11}
		(recall that this restriction coincides,  up to a negative constant,  with the second fundamental form of $\Sigma^-$ with respect to the opposite to the position vector,  \cite[Eq.  (2.5)]{JavSan11}).
		Analogously,
		$\Sigma^+$ defines a conic norm $F$ (thus completing $\it (b)$) and, by the choice of $\Sigma^+$, one has  $F<F_l$. Finally, observe that the points of $\Sigma^0$ lie necessarily in the boundary of $A_E$  
		since the rays from zero are tangent to $\Sigma$ (which is strictly convex, in particular); moreover, $\Sigma^0$ lies in the boundary of both $\Sigma^+$ and $\Sigma^-$, which ensures the properties of the extension.
	\end{proof}
	\begin{rem}
		Observe that, in general, a converse of Proposition~\ref{possibwind}  (namely,  whether  a wind Minkowski norm is determined by a conic Minkowski  norm   $F$ and a Lorentzian norm $F_l$ defined both in an open conic subset $A\subset V$, 
		such that $F$ and $F_l$ can be continuously extended to  $A_E$ and the extensions coincide)  would require further hypotheses in order to ensure  that
		the closures in $A_E$ of the indicatrices of $F$ and $F_l$   glue  smoothly at  their intersection with  the boundary of  $A_E$.
	\end{rem}
	\begin{convention}
		As a limit case  $m=1$ of Proposition
		\ref{possibwind}
		and, thus, $\Sigma=\{ \lambda_-v_0, \lambda_+ v_0\}$, one
		has naturally a Minkowski norm or a Kropina norm  (the latter identifiable to a norm with domain only a half line)   when $0\in B$ or
		$0\in \Sigma$, resp. When $0\not \in \bar B$, choose $v_0\in B$ and assume $(0<)\lambda_- (<1) < \lambda_+$.  Then, define $\Sigma^+ = \{ \lambda_+
		v_0 \} $ (resp. $\Sigma^- = \{ \lambda_- v_0 \}$), as the
		indicatrix of a conic  Minkowski  norm, which will also be regarded as
		Lorentzian  norm  in the case of $\Sigma^-$ ($\Sigma^+, \Sigma^-$
		are clearly independent of the chosen vector $v_0$).
	\end{convention}

	\subsection{Notions  on manifolds  and characterizations} Let $M$ be a  smooth $m$-dimensional manifold\footnote{ Manifolds are always assumed to be Hausdorff and paracompact. 
		However, the latter can be deduced from the existence of a Finsler metric (as then the manifold $M$ will admit a reversible one, and $M$ will be metrizable) as well as from the existence of a wind Finsler structure 
		(as in this case the centroid vector field is univocally defined, and $M$ will admit a Finsler metric, see Proposition~\ref{pzerm}  below). },  $TM$ its tangent bundle and
	$\pi:TM\rightarrow M$ the natural projection. Let us recall that a
	{\it Finsler metric} in $M$ is a  continuous  function
	$F:TM\rightarrow [0,+\infty)$ smooth away from the zero section
	and  such that $F_p=F|_{T_pM}$ is a Minkowski norm for every $p\in
	M$. Analogously, a {\it  conic Finsler metric, conic
		pseudo-Finsler metric or a Lorentzian Finsler metric} is a smooth
	function $F:A\rightarrow [0,+\infty)$, where $A$ is a {\em conic}
	open subset of $TM \setminus\mathbf{0}$ (i.e., each $A\cap T_pM$ is a
	conic subset) such that $F_p=F|_{A\cap T_pM}$ is,
	respectively, a conic Minkowski norm,  a conic pseudo-Minkowski
	norm  or a Lorentzian norm.
	\begin{defi}\label{windStruct}
		A smooth  (embedded)  hypersurface $\Sigma \subset TM$ is a
		\emph{wind Finslerian structure} on the manifold $M$ if, for every
		$p\in M$: (a) $\Sigma_p:=\Sigma\cap T_pM$  defines  a wind
		Minkowskian structure in $T_pM$,
		and (b) for each $v\in \Sigma_p$, $\Sigma$ is
		transversal to the vertical space $\mathcal V(v) \equiv T_v(T_pM)$
		in $TM$.  In this case, the pair $(M,\Sigma)$ is a {\em wind Finslerian manifold}.    Moreover, we will  denote by  $B_p$  the unit ball of each $\Sigma_p$;
		while  the {\em (open) domain} $A$ of the wind Finslerian structure
		will be  the union  of the sets $A_p\subset T_pM, p\in M$,
		where $A_p$ is defined as $A_p=T_pM \setminus \{0\}$ if
		$0\in B_p$  and  by parts $(ii)$ and $(iii)$ of Proposition
		\ref{possibwind} otherwise.
		\begin{rem}\label{r2.7}
			For a standard Finsler structure  $F\colon TM\to [0,+\infty)$, the
			indicatrix $\Sigma_F= \{v\in TM: F(v)=1\}$ is a wind Finslerian
			structure. In fact, (a) follows trivially, and (b) holds because,
			otherwise,  being $F$ smooth on $TM\setminus\mathbf{0}$,  $\mathcal V(v)$ would lie in the kernel of $dF_v$, in
			contradiction with the homogeneity of $F$ in the
			direction
			$v$. Notice that this property of transversality (b) also holds
			for the indicatrix of any conic Finsler or Lorentzian Finsler metric
			defined on $A\subset TM$ (while (a) does not).
			
			R.L. Bryant  \cite{Bryant02} defined a generalization of Finsler metrics also as a hypersurface. The proof of Proposition~\ref{windConseq} below shows that this notion is clearly related to the notion of {\em conic Finsler metric} 
			used here (even though, among other differences, in his definition $\Sigma$ must be radially transverse and it may be non-embedded and non-compact). 
		\end{rem}
		
		\begin{prop} 
			The wind Finslerian structure $\Sigma$ is  closed as a subset of $TM$, 
			and foliated by spheres.  Moreover,  
			the union of all
			the unit balls $B_p, \, p\in M$, as well as $A$, are open in $TM$. If
			$M$ is connected and $m\geq 2$ (resp. $m=1$), then $\Sigma$ is
			connected (resp. $\Sigma$ has two connected parts, each one
			naturally diffeomorphic  to $M$).
		\end{prop}
		\begin{proof}  For the  first sentence, recall that the property (a) of Definition~\ref{windStruct}
			implies that $\Sigma$ is foliated by topological spheres $S^{m-1}$
			and each $p\in M$ admits a  neighborhood $U$ such that $\Sigma
			\cap \pi^{-1}(U)$ is compact and homeomorphic to $U\times
			S^{m-1}$. 
			Indeed, for each chart $(U,\phi)$ around some $p\in M$, one can take the natural bundle chart $\phi^U:TU\  \rightarrow \phi(U)\times \R^n$ and choose a vector $o_p\in T_pM$ inside the inner domain of $\Sigma_p$.  We can assume by taking $U$ smaller if necessary that $(\phi^U)^{-1}(x,o^*_p)$ is in the inner domain of $\Sigma_{\phi^{-1}(x)}$ for all $x\in \phi(U)$, where the superscript $^*$ means the associated linear coordinates on  $T_pM$. Then the one-to-one map:
			\begin{align*}
				\Psi\colon& \phi(U)\times \Sigma_p\times \R^+\rightarrow TU\setminus \{(\phi^U)^{-1}((\phi(U),o_p^*))\}, \\ 
				&(x,v_p,\lambda)\mapsto (\phi^U)^{-1}(x, o^*_p + \lambda [(v_p-o_p)]^*)
			\end{align*}
			is a homeomorphism because of the invariance of domain theorem.
			Now, for each $(x,v_p)$ there exists a unique $\lambda(x,v_p)\in\R^+$ such that $(x,v_p, \lambda(x,v_p))\in \Psi^{-1}(\Sigma)$ and $\lambda(x,v_p)$ varies continuously with $x$ and $v_p$. Thus, as $\Sigma_p$ is a topological sphere, the required foliation of $\Sigma\cap TU$ is obtained. 
			For the last assertion, notice that, otherwise, any two
			non-empty disjoint open subsets that  covered  $\Sigma$  would 
			project onto open subsets of $M$ with a non-empty intersection
			$W$, in contradiction with the connectedness of $\Sigma_p$ at each
			$p\in W$ (for $m=1$, $M$ admits a non-vanishing vector field $V$,
			so that each two points in $\Sigma_p$ can be written now as $\lambda_-(p)V_p,\lambda_+(p)V_p$,
			with  $\lambda_-<\lambda_+$ on all $M$, thus  $p\rightarrow \lambda_-(p) V_p, \lambda_+(p) V_p$  yield the required diffeomorphisms with $M$).
		\end{proof}
	\end{defi}
	\begin{defi}\label{ae} 
		Let $(M,\Sigma)$ be a wind Finslerian manifold.
		The {\em region of critical wind} (resp.  {\em mild wind}) 
		is $$\MC =\{p\in M: 0_p \in \Sigma_p \} \quad  (\hbox{resp. } \MM = \{p\in M: 0_p\in B_p\}),$$ 
		and the {\em properly wind Finslerian region} or  {\em region of strong wind} is   
		$$M_l:=\{p\in M: 0_p\notin \bar{B}_p\}.$$ 
		The {\em (open) conic domain}  of the associated Lorentzian Finsler metric  $F_l$  is $$A_l:=\pi^{-1}(M_l)\cap A.$$ 
		Let $\mathbf{0}$ be the 0-section of $TM$.  The   {\em extended domain} of $F_l$ is 
		$$
		A_E:= \left( \hbox{Closure of} \; A_l \; \hbox{in} \; TM_l\setminus \mathbf{0} \right) 
		\cup \{0_p \in T_pM: 
		p \in \MC 
		\}.
		$$
	\end{defi}
	The zero vectors $0_p$  (with $p\in \MC$)  are included in $A_E$ for convenience (see Convention~\ref{caestar}). 
	In the region of strong wind, the convention on $A_E$ is consistent with  Proposition~\ref{possibwind}-(iii); moreover,    $A_l\subset TM\setminus  \mathbf{0}$, $A_l\subseteq A$ and,  whenever $p\in M_l$,
	$A\cap T_pM=A_l\cap T_pM$.
	
	\begin{prop}\label{windConseq}
		Any wind Finslerian structure $\Sigma$ in $M$ determines the
		conic pseudo-Finsler metrics $F:A\rightarrow [0,+\infty)$
		and $F_l:A_l\rightarrow [0,+\infty)$ in $M$ and $M_l$ respectively  (the latter when $M_l\neq \emptyset$)  characterized by the properties:
		\begin{enumerate}[(i)]
			\item $F$ is a conic Finsler metric with indicatrix included in $\Sigma \cap A$,
			\item $F_l$ is a Lorentzian Finsler metric with indicatrix included in $\Sigma \cap A_l$
		\end{enumerate}
		Moreover,  $F< F_l$ on $A_l$,   both  $F_l$ and $F$ can be extended  continuously   to the boundary of $A_l$ in $TM_l\setminus \mathbf{0}$ (i.e., $A_E\setminus \mathbf{0}$),
		and both extensions coincide in this boundary. 
	\end{prop}
	\begin{proof}
		From Proposition~\ref{possibwind}, we have to prove just the smoothability of $F, F_l$ in $A$, by using both,
		the smoothness of $\Sigma$ and its transversality. Let $v\in A_p
		\cap \Sigma$, and consider  the ray
		$\{\lambda v: \lambda >0\}$ (recall that $v\neq 0$). This ray is transversal
		to $\Sigma_p$ and,  because of the property of transversality of
		$\Sigma$, it is transversal to $\Sigma$ in $TM$ too. This property
		holds also for some open connected neighborhood $U^\Sigma$ of $v$
		in $A\cap \Sigma$, where  $U^\Sigma_{\pi(v')}(:=U^\Sigma\cap T_{\pi(v')}M)$ will  be either  strongly  convex (thus
		defining $F$) or  strongly  concave (defining $F_l$)   towards $0_{\pi(v')}$, for all $v'\in U^{\Sigma}$.   Moreover,  the map:
		$$
		\psi:  (0,+\infty) \times U^\Sigma \rightarrow TM \quad \quad
		(t,w)\mapsto tw
		$$
		is injective and smooth.  Even
		more, $d\psi$ is bijective at each point $(1,w), w\in U^\Sigma$,
		because of transversality, and it is also bijective at any
		$(\lambda, w), \lambda >0,$ because the homothety $H_\lambda:
		TM\rightarrow TM$ maps $U^\Sigma$ in the hypersurface $\lambda
		U^\Sigma$ which is also transversal to the radial direction.
		Summing up, $\psi$ is a diffeomorphism onto its image $U^{TM}\subset
		TM$, and the  inverse
		$$
		\psi^{-1}:  U^{TM}  \rightarrow (0,+\infty) \times U^\Sigma
		$$
		maps each $v$ in either $(F(v),v/F(v))$ or in
		$(F_l(v),v/F_l(v))$, depending on the convexity or concaveness of
		$U^\Sigma_{\pi(v)}$, $v\in U^\Sigma$,  proving consistently the smoothness of $F$ or $F_l$.
	\end{proof}
	\begin{prop}\label{ptf}
		Let $\Sigma$ and $W$ be, resp., a wind Finslerian structure and a
		(smooth) vector field on $M$. Then,  $\Sigma + W := \{ v+W_{\pi(v)}: v\in \Sigma \}$ is a wind Finslerian structure on $M$.
	\end{prop}
	\begin{proof}  The translation $T_W: TM
		\rightarrow TM, v\mapsto v +W_{\pi (v)}$, is a bundle isomorphism of
		$TM$; so, it preserves the properties of smoothness and
		transversality of $\Sigma$.
	\end{proof}
	In particular, the
	translation of the indicatrix $\Sigma_{F_0}$ of any standard Finsler metric  $F_0$   along   $W$
	is a wind Finslerian structure $\Sigma$.  In this case, the associated conic pseudo-Finsler metrics $F$ and $F_l$ can be determined as follows. 
	\begin{prop} Let $F_0$ be a Finsler metric  and $W$  be a smooth vector field on $M$. Then the translation of the indicatrix of $F_0$ by $W$ is a wind Finslerian structure whose conic pseudo-Finsler metrics  are determined as the solutions $Z(v)$ of the equation
		\begin{equation}\label{zerm}
			F_0  \left(\frac{v}{Z(v)}-W\right)=1.
		\end{equation}
	\end{prop}
	\begin{proof} 
		Clearly equation \eqref{zerm} corresponds to a translation by $W$ of the indicatrix of $F_0$ (see also the definition of the Zermelo metric $Z$ in
		\cite{Sh03}).
		The convexity of the indicatrix of $F_0$ implies that this equation will
		have a unique positive solution $Z(v)$ for any $v\in TM\setminus\mathbf{0}$ if $F_0(-W)<1$,  no solution  or only a positive  one if $F_0(-W)=1$, no solution or two positive ones if $F_0(-W)>1$.
	\end{proof}
	Conversely:
	\begin{prop}\label{pzerm}  Any wind Finslerian structure $\Sigma$ can be obtained as the displacement $\Sigma_{F_0} + W$ of the
		indicatrix $\Sigma_{F_0}$ of a Finsler metric $F_0$ along some
		vector field $W$. Moreover, $W$ can be chosen such that each $W_p$
		is the centroid  of $\Sigma_p$.
	\end{prop}
	\begin{proof}
		Even if this proof can be carried out by choosing a family of vector fields $W_i$ defined in some open subset with this property, whose  existence is trivial, and then doing a convex sum in all 
		the manifold with the help of a partition of unity, we will prove in fact that the vector field provided by the centroid is smooth. 
		For this
		aim, we can actually assume that $\Sigma$ is the indicatrix of a
		standard Finsler metric $F$ defined on some open subset $U$ of
		$\R^m$ (notice that (i) the smoothability of $W$ is a local
		property, (ii) if a vector $w_p$ belongs to the open ball $B_p$
		enclosed by $\Sigma$, this property will hold for any vector field
		$W$ extending $w_p$ in some neighborhood of $p$, so that
		Propositions~\ref{ptf} and \ref{windConseq} can be
		claimed, and (iii) the translation $T_W$ also translates the
		centroids). Let $S^{m-1}$ be the canonical unit sphere in $\R^m$
		with volume element $d\Omega$.
		So, the natural $x^i$-coordinate of the centroid
		$W_p$ is computed as:
		\begin{align}
			x^i(W_p)=&\int_{S^{m-1}}\int_0^{1/F_p(u)}r x^i(u) r^{m-1}\df r\df \Omega\Big/ \int_{S^{m-1}}\int_0^{1/F_p(u)} r^{m-1}\df r\df \Omega\nonumber \\
			=&\frac{m}{m+1}\int_{S^{m-1}}\frac{x^i(u)}{F_p(u)^{m+1}}\df \Omega\Big/ \int_{S^{m-1}} \frac{\df \Omega}{F_p(u)^{m}}\label{ec}
		\end{align}
		and its smoothness follows from the smooth variation of the
		integrands with $p$.
	\end{proof}
	\begin{exe}[Role of transversality]\label{rematrans}
		The smoothness of $W$ relies on the smoothness of $F$
		in \eqref{ec} and, thus, the transversality of $\Sigma$ imposed in
		the assumption (b) of Definition~\ref{windStruct} becomes
		essential. Figure~\ref{nontransv} shows a $1$-dimensional counterexample if the transversality
		condition is not imposed.
		Notice also that, as the absence of transversality would lead to non-smooth metrics, then this would lead to non-smooth
		\sstk splittings in the next Section~\ref{windFermat}. The well-known exotic properties of the chronological and causal
		futures and pasts of spacetimes with non-smooth metrics (see for example  \cite{ChrGra12}) would be related to exotic properties of  $\Sigma$. 
		\begin{figure}[h]
			\includegraphics[scale=1,center]{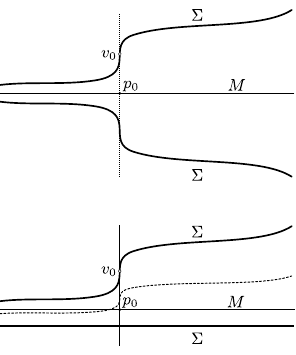}
			\caption{In the top figure, $\Sigma$ is a smooth
				hypersurface of $TM\equiv \R^2$ consisting of two curves which
				intersect the vertical space at $(p_0,v_0)$  (depicted as a vertical line)
				non-transversely.   So $\Sigma$ satisfies the property (a) in
				Definition~\ref{windStruct} and (as the curves are symmetric with
				respect to the zero section of $TM$) it determines
				continuously a scalar product in the tangent space at each
				$p\in M\equiv\R$. Nevertheless, the failure of (b) implies that
				this product does not vary smoothly with respect to $p$ and, so, $\Sigma$
				does not determine a (smooth) Riemannian metric on $M$. In the
				second figure, changing the lower
				curve  by a horizontal line, one
				obtains at each tangent space a wind Minkowskian structure varying
				continuously (but not smoothly) with the point. Moreover, the vector
				field determined by the centroids (the dashed curve)  is not differentiable at  $p_0$.}\label{nontransv}
		\end{figure}
	\end{exe}
	\begin{defi}\label{dreversewind}
		Let $\Sigma$ be a wind Finslerian structure  on $M$. Then,
		\[\tilde{\Sigma}:=-\Sigma:=\{v\in TM: -v\in\Sigma\}\]
		is   the {\it reverse wind Finslerian structure} of $\Sigma$.
	\end{defi}
	Obviously,  $\tilde \Sigma$ is a wind Finslerian structure too and, from the definition, one gets easily the following.
	\begin{prop}\label{reversewind}
		Given a wind Finslerian structure $\Sigma$, the conic Finsler  metric  $\tilde{F}$ and the Lorentzian Finsler one $\tilde{F}_l$ associated with the reverse wind Finslerian structure $-\Sigma$ are the (natural) reverse conic 
		pseudo-Finsler metrics of $F$ and $F_l$, that is,  the domains of $\tilde{F}$ and  $\tilde{F}_l$ are, respectively,
		$\tilde{A}=-A=\{v\in TM: -v\in A\}$ and $\tilde{A}_l=-A_l=\{v\in TM: -v\in A_l\}$ and they are defined as $\tilde{F}(v)=F(-v)$ for every $v\in \tilde{A}$ and $\tilde{F}_l(v)=F_l(-v)$ for every $v\in \tilde{A}_l$.
	\end{prop}
	\subsection{Wind lengths and balls} In order to deal with curves, the following conventions will be useful.
	\begin{convention}\label{caestar}  For any wind Finslerian structure $\Sigma$ we
		extend  $F$ and $F_l$ to $A\cup A_E$ as follows. 
		First, consistently with  Proposition~\ref{windConseq}, 
		$F$ and $F_l$ are regarded as continuously extended to the boundary of $A_l$ in $TM_l\setminus \mathbf{0}$. $F_l$ is extended as equal to $+\infty$ on $A$ in the regions of  mild  and critical wind  i.e. on the set 
		$\{v\in A_p: 0_p\in \bar B_p\}$
		(that is, $F_l$ is equal to $+\infty$ on  the vectors where $F$ has been  defined and $F_l$  has not). Finally, we define $F_l$  and $F$  as equal to $1$  on the set of critical wind zeroes
		(i.e., the set $\{0_p: 0_p\in \Sigma_p\}$, 
		which was included in the definition of $A_E$, Definition~\ref{ae}). Notice that neither this choice of  $F_l$ and $F$   on the critical wind region nor any other can ensure  their   continuity; however, $F_l$ and $F$ are continuous 
		on $A\cup (A_E\setminus\mathbf{0})$. 
		We also use natural notation such as
		$ (A_l)_p=  A_l \cap T_pM$,  $(A_E)_p=  A_E\cap T_pM$.
		
		To  understand this choice, recall first that the necessity to extend $A$ to $A_E$ in the critical and strong wind regions comes from the fact that all the indicatrices $\Sigma_p$ should be contained in $A_E$. 
		In the critical region,  $\Sigma_p\setminus \{0_p\}$ lies in $A$ and, so, in the domain of $F$. Therefore, it is not strange to include  $0_p$  in $A_E$ so that $F_l$ is defined on this vector and, 
		obviously, the choice  $F_l(0_p)=F(0_p)=1$   comes from the fact that $0_p$ lies in the indicatrix  and in the boundary of $A_E$. 
		A further support for these choices will come from the viewpoint of spacetimes, as the vectors in $A\cup A_E$ are those which can be obtained as the projection of a lightlike vector in the  spacetime. 

		As usual, a {\em piecewise smooth curve} $\gamma$ will be defined in a compact interval $I=[a,b]$, and it will be smooth except in a finite number of breaks $t_i\in I$, $i\in\{1,\dots, k\}$, where it is continuous and its 
		one-sided  derivatives are well defined\footnote{Even though typically, all the curves will be defined on a compact interval $I$, when necessary all the following notions can be used for non-compact $I$.
			In this case, one assumes that  the restriction of $\gamma$ to compact subintervals of $I$ satisfies the stated property, and it is natural to impose additionally that the images of the breaks $\{\gamma(t_i)\}$ do not accumulate.};
		its reparametrizations will be assumed  also piecewise smooth   and with positive  one-sided  derivatives  (so that, for example a piecewise smooth geodesic with proportional one-sided derivatives at each break pointing in the same direction can be reparametrized as smooth geodesics),  unless otherwise   specified. 
		
	\end{convention}
	
	\begin{defi}\label{sigmadmissible}
		Let $\Sigma$ be a wind Finslerian structure with associated  pseudo-Finsler metrics  $F$ and $F_l$   and consider a  piecewise smooth  curve  $\gamma:[a,b]\subset\R\rightarrow M$, $a<b$. 
		
		(i) $\gamma$ is {\em $\Sigma$-admissible}  if its left and right derivatives $ \dot\gamma(s^-), \dot\gamma(s^+)$ belong to  $A\cup A_E$  at every  $s\in [a,b]$.  
		Analogously, $\gamma$ is {\em  $F$-admissible} if  $\dot\gamma(s^{\pm})\in A$,  for each $s\in [a,b]$.  Accordingly, a vector field $V$ on $M$ is $\Sigma$-admissible (resp. $F$-admissible) if $V_p\in A\cup A_E$  for each $p\in M$,  
		(resp. $V_p\in A$  for each $p\in M$).
		
		(ii) A $\Sigma$-admissible curve $\gamma$ is a {\em wind curve} if 
		\begin{equation} \label{ewindcurve} 
			F(\dot\gamma(s))\leq 1 \leq F_l(\dot\gamma(s)) \quad \quad \forall s\in[a,b],  \end{equation} 
		and an $F$-admissible wind curve will be called just {\em $F$-wind curve}.

		(iii)  A $\Sigma$-admissible curve $\gamma$ is a {\em regular  curve} if its  one-sided  derivatives
		can vanish only at isolated points (which can be   regarded as 
		break points, even though the curve may be smooth there), and it is
		a {\em strictly regular  curve} if its  one-sided  derivatives (and, thus, its velocity outside the breaks)
		cannot vanish at any point. 
		
		(iv)  The {\em wind lengths} of a 
		$\Sigma$-admissible  curve $\gamma$  (not necessarily a wind curve)   are defined as
		\begin{align*}
			\ell_F(\gamma) & =\int_a^b F(\dot\gamma)ds\  \big(\!\!\in  (0,   +\infty] \big), & \ell_{F_l}(\gamma)&=\int_a^b F_l(\dot\gamma)ds\  \big(\!\!\in (0,+\infty]\big).
		\end{align*}
	\end{defi}
	Obviously, from  \eqref{ewindcurve} we get:
	\begin{prop}
		If  $\gamma$ is a wind curve then 
		\begin{equation}\label{elength}
			\ell_F(\gamma|_{[a',b']})\leq b'-a' \leq \ell_{F_l}(\gamma|_{[a',b']}), \quad\text{for all $a\leq a'<b'\leq b$.}
		\end{equation}
	\end{prop}
	We will use this and other natural properties (as the fact that the concatenation of two wind curves $\gamma_1, \gamma_2$ such that $\gamma_1(b_1)=\gamma_2(a_2)$ is another wind curve) with no further mention.
	\begin{rem} \label{remlength} 
		Wind curves   collect the intuitive idea of Zermelo's navigation problem, namely: the possible velocities  attained  by the moving object are those satisfying
		the inequalities in \eqref{ewindcurve}  (observe that in the region $M\setminus M_l$, the inequalities in \eqref{ewindcurve} reduce to $F(\dot\gamma(s))\leq 1$). 
		These  velocities never include $0_p$ if $p\in M_l$ and must include $0_p$ if $0_p\in\overline{B}_p$,  which happens iff $p\in M\setminus M_l$, even though, by convenience, we have excluded  $0_p$ from $A_p$ if $p\in\MM$ and included it in the extended domain $A_E$ when $p\in \MC$.  
		The reason to exclude  $0_p$ from $A_p$  when $p\in \MM$ is just to emphasize the different role of the zero vector in this region and in $\MC$  (as well as avoiding problems of differentiability 
		with $F$).\footnote{If the reader felt more comfortable, he/she  could redefine $A$ by adding $\{0_p:p\in \MM\}$ with no harm. In the part of spacetimes, the so redefined subset $A$ could be interpreted as the set which contains the
			projections of all the (future-pointing) timelike vectors, and $A\cup A_E$ as the set which contains the projections of the causal vectors.  However,  the reader should take into account that the fundamental tensor of a pseudo-Finsler metric
			is not well-defined in the zero section.}  
		In fact, in order to connect points  by means of curves included in $\MM$, one can avoid to use  velocities that vanish (and this may be convenient for purposes such as reparametrizing the curve at constant speed; 
		such an assumption is frequent in Riemannian Geometry too). However,  as in the case of Riemannian Geometry, the vanishing of the velocity in subsets with accumulation points leads to bothering problems about its reparametrizations. So, we will consider the solutions of Zermelo's problem as regular wind curves (allowing the velocity to vanish in isolated points), and we will ensure the existence of such solutions  (see  Corollary~\ref{rsummaryZ}). 
		Observe also that  the continuity of $F\circ \dot\gamma$ and $F_l\circ\dot\gamma$ has to be checked only when $\dot\gamma(s)$ is equal to a zero of the critical region 
		(see Proposition~\ref{pleche}-(ii))  and, in this case, $F$ and $F_l$ are defined as equal to $1$ there. A further explanation of this 
		choice is provided in Example~\ref{exazero} below, where two 
		paradigmatic examples of curves with Kropina's zeroes in the derivatives are given.

	\end{rem}
	\begin{exe}\label{exazero}
		Let $\R^2$ be endowed with the Kropina norm $F(x,y)=(x^2+y^2)/x$ defined in $A=\{(x,y)\in \R^2: x>0\}$. Then the curve $\alpha:[0,1]\rightarrow \R^2$, $\alpha(t)=(t^2,t^3)$, satisfies that 
		$\dot\alpha(0)=(0,0)$ and $\lim_{t\rightarrow 0}F(\dot\alpha(t))=0$.  Clearly, the reparametrization of this curve as an $F$-unit curve is not differentiable at  $t=0$.   In fact, this kind of curves   was excluded in the mild region. 
		However, consider the indicatrix of $F$ as a curve, take   the part which is $\Sigma$-admissible  and reparametrize it as an $F$-unit curve. In such a way, \soutE{you} \bw we \ew get a curve $\gamma:[a,b]\rightarrow \R^2$ whose derivative is zero in 
		the two end-points,   but with $F\circ \dot\gamma$  constantly equal to $1$. This second kind of curves is the main reason for including the zero in the Kropina region in the domains of $F$ and $F_l$. 
		Observe that if \soutE{you want}  \bw one wants \ew to exclude the first kind of curves, it is enough to require the continuity of $F\circ\dot\gamma$ in the definition of 
		wind curves in every 
		smooth piece. 
	\end{exe}
	
	Let $x_0,x_1 \in M$ and let us denote by $C^{\Sigma}_{x_0,x_1}$  (resp. $C^{A}_{x_0,x_1}$,   $\Omega^A_{x_0,x_1}$)  the set of the  wind curves (resp. $F$-wind curves, $F$-admissible curves)  between $x_0$ and $x_1$ (each curve $\gamma$ defined in a possibly different interval  $[a_\gamma,b_\gamma]$). 
	
	Following \cite{JavSan11}, we introduce the following notions. 
	\begin{defi}\label{from41}
		Given  a conic pseudo-Finsler  metric $F\colon A\subset TM\to [0,+\infty)$, the {\em Finslerian separation},  also called  {\em $F$-separation},  $d_F\colon M\times M\to [0,+\infty]$  is defined as $d_F(p,q)=\inf_{\gamma\in \Omega^A_{x_0,x_1}}\ell_F(\gamma)$ if $\Omega^A_{x_0,x_1}\neq \emptyset$ otherwise $d_F(p,q)=+\infty$.  
		By using the Finslerian separation two families of subsets of $M$   can be introduced: for any $x_0\in M$ and $r\in (0,+\infty)$, set $B_F^+(x_0, r)=\{y\in M: d_F(x_0,y)<r\}$ 
		and $B^-_F(x_0, r)=\{y\in M:d_F(y,x_0)<r\}$. Moreover, 
		a conic pseudo-Finsler metric is said {\em  Riemannianly lower bounded} on an open subset $D$ of $M$ if 
		there  exists a  Riemannian metric $g_0$ on $M$ such that $F(v)\geq \sqrt{g_0(v,v)}$, for all $v\in TD\cap A$.
	\end{defi}
	As $F$ and $F_l$ are continuously extendible to $A_E\setminus\mathbf{0}$, we immediately get, by homogeneity, that they are Riemannianly lower bounded on, respectively, $M$ and $M_l$. By \cite[Proposition 3.13]{JavSan11}, the collections $B_F^\pm(x_0, r)$ of a Riemannianly  lower bounded conic pseudo-Finsler $F$ constitute a basis for the topology of $D$,  thus we have:  
	\begin{prop}\label{riemlowerbound}
		The collections of $B_F^\pm(x_0, r)$ (resp $B_{F_l}^\pm(x_0, r)$) constitute a basis for the topology of $M$ (resp. $M_l$). \end{prop} 
	Some cautions, however, must be taken. For example, the Finslerian separation  of the conic Finsler metric $F$  may be discontinuous; in fact, the conic Finsler metric in 
	\cite[Example 3.18]{JavSan11} exhibits this property  (see also Section~\ref{kroran} below).    We refer to
	\cite[Section 3.5]{JavSan11} for a summary of  the  properties satisfied by the Fnslerian separation. 
	
	In order to work with the full geometry associated with  $\Sigma$ we also introduce the following new collections of subsets of $M$.
	\begin{defi}\label{sigmaballs}
		Let $x_0\in M$ and $r >  0$. The {\em forward} (resp. {\em backward}) {\em wind balls}  of center $x_0$ and radius $r$ 
		associated with the wind Finslerian structure $\Sigma$ are:
		\begin{align*}
			&B^+_{\Sigma}(x_0,r)=\{x\in M:\ \exists\   \gamma\in C^{\Sigma}_{x_0, x}, \text{ s.t. }   r=b_\gamma-a_\gamma \, \text{and} \;  \ell_F(\gamma)<r<\ell_{F_l}(\gamma)\},\\
			&B^-_{\Sigma}(x_0,r)=\{x\in M:\ \exists\   \gamma\in C^{\Sigma}_{x, x_0}, \text{ s.t. }  r=b_\gamma-a_\gamma \, \text{and} \;  \ell_F(\gamma)<r<\ell_{F_l}(\gamma)\}.\\
			\intertext{ being  {\em the closed balls}  $\bar{B}^\pm_{\Sigma}(x_0,r)$ their closures.  Moreover, the (forward, backward) {\em c-balls} are defined as:} 
			&\hat{B}^+_{\Sigma}(x_0,r)=\{x\in M:\ \exists\  \gamma\in C^{\Sigma}_{x_0, x},\text{ s.t. } r=b_\gamma-a_\gamma \, \text{(so,} \;  \ell_F(\gamma)\leq r\leq\ell_{F_l}(\gamma) )  \}
			,\\
			&\hat{B}^-_{\Sigma}(x_0,r)=\{x\in M:\ \exists\ \gamma\in C^{\Sigma}_{x, x_0},\text{ s.t. } r=b_\gamma-a_\gamma \, \text{(so,} \; \ell_F(\gamma)\leq r\leq\ell_{F_l}(\gamma))  \} 
		\end{align*}
		for  $r> 0$ and, by convention for $r=0$, $\hat{B}^\pm_{\Sigma}(x_0,0)=x_0$.
	\end{defi}
	Recall that, consistently with our conventions, if $0_{x_0}\in \Sigma_{x_0}$ then $x_0 \in \hat B^\pm_\Sigma(x_0,r)$ for all $r\geq 0$ 
	(this will be interpreted naturally in  the description of the causal future of a point in an \sstk see, e.g.
	Proposition~\ref{bolas2}).   
	\begin{prop}
		If a wind Finslerian
		structure comes from a Finsler one  then  the
		sets $B^+_{\Sigma}(x_0,r)$ and $B^-_{\Sigma}(x_0,r)$,  $r>0$, 
		coincide with the standard forward and backward  open  balls
		centred at $x_0$.  
	\end{prop}
	\begin{proof}
		Just take into account that the assumption is equivalent to $0\in B_p$, for all $p\in
		M$ and, according to Convention~\ref{caestar}, $F_l(v)=+\infty$, for all $v\in  A=  TM \setminus  \mathbf{0}$. 
	\end{proof}
	\begin{exe}\label{ex:riemcballs}
		$\hat B^+_{\Sigma}(x_0,r)$ and $\hat B^-_{\Sigma}(x_0,r)$  do not coincide   in general with the closures  $\bar{B}_{\Sigma}^+(x_0,r)$ and $\bar{B}_{\Sigma}^-(x_0,r)$. This may occur even when $\Sigma$ comes from a Riemannian metric
		(in $\R^2\setminus \{(1,0)\}$, $\hat{B}^+_{\Sigma}(0,2)$ is not closed);   another simple example (using a strong wind Minkowskian structure) can be seen in Fig.~\ref{eee}. 
		In fact, as we will see, the closedness of the c-balls will be related with the convexity of the manifold.
		\begin{figure}[h]
			\includegraphics[scale=0.4, center]{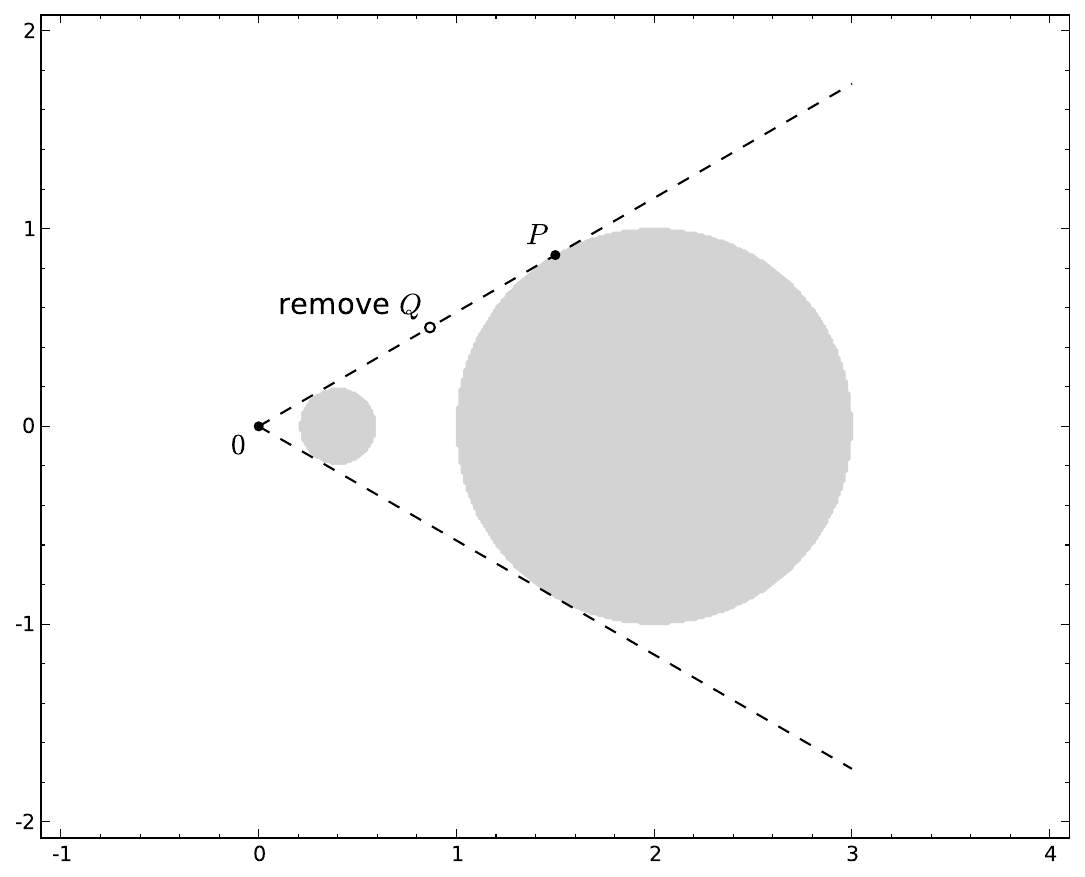}
			\caption{A wind Minkowskian structure $\Sigma$  in $\R^2\setminus\{Q\}$. The shaded regions represent the wind balls $B^+_{\Sigma}(0,1/5)$ and $B^+_{\Sigma}(0,1)$ which satisfy  $\hat B^+_{\Sigma}(0,1/5)=\bar B^+_{\Sigma}(0,1/5)$ but
				$ P\in  \bar B^+_{\Sigma}(0,1)\setminus \hat B^+_{\Sigma}(0,1)$.}\label{eee}
		\end{figure}
	\end{exe}
	The next  three  propositions provide a better understanding of
	$\hat{B}^\pm_{\Sigma}(x_0,r)$.  Before them, we will prove a
	technical lemma, which stresses the importance of transversality (recall  Example~\ref{rematrans}).  
	\begin{lemma}\label{lleche}
		Let $\Sigma$ be a wind Finslerian structure on $M$ and $p\in M$
		such that $0_p\in \Sigma_p$, and let
		$\gamma: (-\varepsilon,\varepsilon)\rightarrow M$ be a smooth  $F$-admissible curve such that $\gamma(0)=p$. Then,  reducing $\varepsilon$ if necessary,  the surface
		$$
		T_\gamma=\{\lambda \dot\gamma(s): \lambda\in\R,
		s\in(-\varepsilon,\varepsilon)\}
		$$
		is embedded in $TM$ 
		and it is transverse to $\Sigma$. 
		Moreover, 
		if $\varepsilon$ is small enough, 
		a smooth function $(-\varepsilon,\varepsilon)\ni s \mapsto \lambda (s)\in
		\R$ is obtained by requiring that each $\lambda(s)\dot\gamma(s)$  be 
		the point in $\Sigma_{\gamma(s)}\cap T_\gamma =
		\{\lambda_1\dot\gamma(s), \lambda_2 \dot\gamma(s)\}$ with smaller $|
		\lambda_i |, i=1,2$.
	\end{lemma}
	\begin{proof}  Clearly, $T_\gamma$ is embedded and  it cuts $\Sigma$
		transversely  in two points because $\dot\gamma(s)\in A$ for every $s\in (-\varepsilon,\varepsilon)$  (with $\varepsilon>0$ small enough).   So, $T_\gamma$ fulfils the required
		property of transversality and, moreover, $\Sigma\cap T_\gamma$ is
		composed by two connected one-dimensional  smooth  submanifolds $\rho_1$,
		which contains $0_p$, and $\rho_2$. 
		The parameter  $s$ of $\gamma$ can be chosen as a natural
		coordinate for $\rho_1$. In this coordinate, the inclusion of
		$\rho_1$ in $TM$ is the smooth map $s\mapsto \lambda(s)
		\dot\gamma(s)$, so that the map $\lambda$ is  smooth.
	\end{proof}
	
	\begin{prop} \label{pleche}   With the above notation: 
		\begin{enumerate}[(i)] 
			\item  Let $(M,\Sigma)$, $p\in M$, $0_p\in\Sigma_p$ and 
			$\gamma: (-\varepsilon,\varepsilon)\rightarrow M$, $\gamma(0)=p$, smooth and $F$-admissible, as in the previous lemma. 
			Then, 
			$\ell_{F_l}(\gamma|_{[0,\varepsilon']})=+\infty$ for all $0<\varepsilon' \leq \varepsilon$.  As a consequence, for each $r_0>0$ there exists $\varepsilon_0\leq \varepsilon$ such that
			$\gamma((0,\varepsilon_0])\subset B^+_{\Sigma}( p ,r)$ for all $r\geq r_0$. 
			
			\item If a smooth curve $\gamma\colon[a,b]\to M$ is $\Sigma$-admissible  and strictly regular,  then  $F\circ \dot\gamma$ and  $F_l\circ \dot\gamma$ are continuous (the latter as a map from $[a,b]$ to $(0,+\infty]$). 
			
			\item    A $\Sigma$-admissible curve $\gamma: [a,b]\rightarrow M$ satisfies $\ell_{F_l}(\gamma)= +\infty$ if  $F_l(\dot\gamma(s_0))=+\infty$ at some $s_0\in [a,b]$. The converse holds when $\gamma$ is strictly regular.  
			\item  For any $\Sigma$-admissible curve, 
			\begin{equation}\label{elength2}
				\ell_F(\gamma)\leq \ell_{F_l}(\gamma) 
			\end{equation}
			with equality iff   $\dot\gamma(s) \in A_E\setminus A$.    Moreover, for a wind curve satisfying the equality in \eqref{elength2},  $F_l(\dot\gamma)=F(\dot\gamma) \equiv 1$  everywhere.  
		\end{enumerate}
	\end{prop}
	\begin{proof}
		$(i)$ Choose any sequence $\varepsilon_k\searrow 0$ in
		$(0,\varepsilon)$. Clearly, we have 
		$\gamma(\varepsilon_k)\rightarrow  p$ and
		$\ell_F(\gamma|_{[0,\varepsilon_k]})\searrow 0$; so, it is enough to
		prove that $\ell_{F_l}(\gamma|_{[0,\varepsilon_k]})=+\infty$ for all $k$.
		From the definition of $\lambda (s)$ in Lemma~\ref{lleche} and $F_l$, we have
		\begin{equation}\label{contFl}
			F_l(\dot\gamma(s)) = \left\{ \begin{array}{lcrl} 1/\lambda(s) & &
				\hbox{if} & \lambda(s) >0 \\
				+\infty & & \hbox{if} & \lambda(s) \leq 0 \\ \end{array} \right.
		\end{equation}
		As $\lambda (0)=0$ and $\lambda$ is smooth around $0$
		$$\int_0^{\varepsilon_k}\frac{ds}{\lambda(s)}=+\infty ,$$
		and all the assertions follow directly.
		
		$(ii)$ Observe that $F\circ\dot \gamma$ is always continuous in this case and $F_l\circ \dot\gamma$ can be discontinuous in $s_0\in [a,b]$, only when $\gamma(s_0)$ belongs to $\MC$. Moreover, in this case, $\gamma$ has to be $F$-admissible 
		in a neighborhood of $s_0$ because it is smooth and strictly regular. Then applying  Lemma~\ref{lleche}  in order to get \eqref{contFl} to the reparametrization $\tilde{\gamma}(s)=\gamma(s-s_0)$, we conclude. 
		
		$(iii)$ Necessarily, $\gamma(s_0)$ must belong either to $\MC$, and the part $(i)$ applies  (recall that, being $F_l(\dot\gamma(s_0))=+\infty$, $\dot\gamma(s_0)\in A_{\gamma(s_0)}$ and $\gamma$ must be $F$-admissible and smooth in a right or  a left 
		neighborhood of $s_0$),  or to $\MM$ and $F_l(\dot\gamma)=+\infty$  in  some neighborhood of $s_0$. For the converse, notice that at least one   of the smooth pieces of $\gamma$ has to be of infinite $F_l$-length 
		then, necessarily, $F_l(\dot\gamma(s_0))=+\infty$ for at least one point $s_0\in [a,b]$ otherwise the $F_l$-length of such a piece would be finite by part $(ii)$.  
		
		$(iv)$ Apply Proposition~\ref{windConseq} and Convention~\ref{caestar}.
	\end{proof}
	\begin{rem}\label{remlengthbis}
		$F$-admissible curves  are always strictly regular 
		$\Sigma$-admissible ones. For these curves,
		$\ell_{F_l}(\gamma)$ may be  infinite  even in the case of an $F$-admissible curve contained in $M_l$ except at one endpoint, see Proposition~\ref{pleche}. The role of strict regularity becomes apparent from the discussion in Convention~\ref{caestar} 
		(see also Proposition~\ref{prop220} below).
	\end{rem}
	
	Notice that wind  curves  depend on reparametrizations. However, the following result suggests that this is not a relevant  restriction, at least when the velocities do not vanish; it also provides a control on the possible reparametrizations.

	\begin{prop} \label{prop220}
		Let $\gamma$ be a   piecewise  smooth $\Sigma$-admissible curve   such that, in each interval where $\gamma$ is smooth,    $F\circ \dot\gamma$ is continuous and $F_l\circ \dot\gamma$   is  either  infinite at some point or continuous.  Then, $\gamma$ admits a  (piecewise smooth)  reparametrization $\tilde \gamma : [0,r_0] \rightarrow M$ 
		as a wind curve and, necessarily then, $\ell_F(\gamma)\leq r_0 \leq \ell_{F_l}(\gamma)$. 
		Moreover, $r_0$ can be chosen equal  to any value of
		$[\ell_F(\gamma), \ell_{F_l}(\gamma)]$ if 
		$\ell_{F_l}(\gamma)<+\infty$,  and,  any value of $[\ell_F(\gamma),  +\infty)$ otherwise.  In particular, this applies for any strictly regular $\Sigma$-admissible curve and, therefore, for any $F$-admissible curve. 
	\end{prop}
	
	\begin{proof} 
		We can assume that $\gamma$ is smooth because the piecewise smooth case trivially follows from this.  Put $\tilde \gamma(r)=\gamma(s(r))$. The reparametrization $s(r)$ as a wind curve is characterized by 
		$$
		F(\dot\gamma(s(r))) \dot s(r) \leq 1 \leq F_l(\dot\gamma(s(r))) \dot s(r).
		$$
		As  $F\circ \dot \gamma$ is continuous, we can first reparametrize $\gamma$ with $F(\dot\gamma)\equiv 1$. Clearly, this  gives also a parametrization of $\gamma$ as a wind curve. In order to prove the last part of the proposition let  
		us distinguish three  cases: 
		
		(a)  If  $F_l(\dot\gamma)<+\infty$ at all the points  then, by assumptions, $F_l\circ \dot \gamma$ is  continuous and the family of reparametrizations, defined by 
		$\dot r_\lambda(s)= \lambda F_l(\dot\gamma(s))+ (1-\lambda) F(\dot\gamma(s)), \lambda\in [0,1]$, is enough to obtain  all the required values of $r_0$. 
		
		(b)  If $F_l(\dot\gamma(\bar s))=+\infty$, for some $\bar s \in [a,b]$, and $F_l\circ\dot\gamma$ is continuous (as a map assuming values in $(0,+\infty]$) everywhere, then, by Proposition~\ref{pleche}-(iii), $\ell_{F_l}(\gamma)=+\infty$ and the  conclusion  follows modifying  the  expression of $\dot r_\lambda$ in case (a) 
		by substituting $F_l(\dot\gamma(s))$ with  $\varphi_\lambda (F_l(\dot\gamma(s)))$,  $\lambda \in [0,1)]$,  where:
		$$
		\varphi_\lambda (t)= \left\{ 
		\begin{array}{ll}
			t & \hbox{if} \; t\leq 1/(1-\lambda)\\
			\phi_0(t-1/(1-\lambda)) + 1/(1-\lambda) & \hbox{if} \; t\in \left(1/(1-\lambda), 2+1/(1-\lambda)\right)\\
			1+1/(1-\lambda) & \hbox{if} \; t \geq 2+1/(1-\lambda)
		\end{array}
		\right. $$
		being $\phi_0:[0,2]\rightarrow [0,1]$  any  curve with $\phi_0(t)\leq t , t\in[0,2]$ that connects smoothly the graphs of $t\mapsto t$ for $t\leq 0$ and of $t\mapsto 1$ for $t\geq 2$,  and recalling that we have assumed  $F(\dot\gamma)=1$. 
		
		(c) Finally, if  $F_l(\dot\gamma(\bar s))=+\infty$, for some $\bar s \in [a,b]$, 
		then $\gamma$ must be strictly regular  in a neighbourhood $[a',b']$ of $\bar s$ and then, by Proposition~\ref{pleche}-(ii), $F_l\circ \dot\gamma$ must be continuous in $[a',b']$. 
		Therefore,  as in case (b), we can change the parametrization of   $\gamma$ only on the interval $[a',b']$  to get all the  values $r_0\in [\ell_F(\gamma), +\infty)$ also in this case. 
		
	\end{proof}
	\begin{prop}\label{pclosurecballs} For any wind Finslerian structure $\Sigma$ and $r>0$:
		\begin{align*}
			&B^+_{\Sigma}(x_0,r)\subset \hat B^+_{\Sigma}(x_0,r)\subset \bar{B}_{\Sigma}^+(x_0,r),\\
			&B^-_{\Sigma}(x_0,r)\subset\hat B^-_{\Sigma}(x_0,r)\subset \bar{B}_{\Sigma}^-(x_0,r).
		\end{align*}
		Thus, the closures of
		${B}_{\Sigma}^+(x,r)$ and $\hat {B}_{\Sigma}^+(x,r)$ are equal.
	\end{prop}
	\begin{proof}
		The first inclusions follow trivially from  the definitions.  Let $x\in \hat{B}^+_{\Sigma}(x_0,r)$ and consider  a  wind  curve
		$\gamma\colon [a,b]\to M$ from $x_0$ to $x$ such that
		$\ell_F(\gamma)\leq r=b-a\leq \ell_{F_l}(\gamma)$. If the two inequalities held strictly, there would be nothing to prove. Otherwise, consider the following cases:
		
		(a) $\ell_F(\gamma) = r = \ell_{F_l}(\gamma)$  (in particular, $\dot\gamma(s)\in A_E\setminus
		A$ for all $s$ and $F(\dot\gamma)\equiv1$,  recall Remark~\ref{remlengthbis}(2)). 
		Choose any $F$-admissible
		vector field $V$  such that $F(V)\equiv 1$  defined in some
		neighborhood $U$ of $x$;  notice that the integral curves of $V$ are wind curves.  Take a smaller
		neighborhood $U'$ and some $\varepsilon>0$ so that the flow of $V$ is
		defined in $[0,\varepsilon]\times U'$ and $\gamma([b-\varepsilon,
		b])\subset U'$. Choose $\{s_n\}\nearrow b$ and consider the curve
		$\gamma_n$ obtained by concatenating $\gamma|_{[a,s_n]}$ and the
		integral curve $\rho_n: [0,\varepsilon_n]\rightarrow M$ of $V$
		starting at $\gamma(s_n)$, 
		where $\varepsilon_n :=b-s_n>0$. 
		By construction, $\ell_F(\rho_n) = \varepsilon_n= \ell_F(\gamma|_{[b -s_n,b]})$ and $\ell_F(\gamma_n) =r<\ell_{F_l}(\gamma_n)$.  So, choosing some close $\varepsilon'_n<\varepsilon_n$, the lengths of the corresponding restriction of 
		$\gamma_n$ allow us to write $\rho_n(\varepsilon'_n)\in
		B^+_{\Sigma}(x_0,r)$ and  $\rho_n(\varepsilon'_n)\to x$,  as
		required.


		(b)  $\ell_F(\gamma) = r < \ell_{F_l}(\gamma)$. 
		Just notice that 
		the points $\gamma(b-\varepsilon)$ will
		belong to $B^+_{\Sigma}(x_0,r)$ for small $\varepsilon$.

		(c) $\ell_F(\gamma) < r \leq \ell_{F_l}(\gamma)$.
		Extending  $\gamma$ beyond $b$ by concatenating an
		$F$-admissible piece, the points in the extension close to $x$
		will belong to $B^+_{\Sigma}(x_0,r)$. 
		
	\end{proof}
	
	
	Finally, an interpretation of the c-balls is provided for the classical Finsler case. Notice that, in this case, the restriction for a piecewise smooth  curve  to be  ``wind'' is just to assume that its speed is not bigger 
	than 1 (in order to travel not faster than the maximum allowed speed) and the velocity not to be 0 (by convenience, see Remark~\ref{remlength} (3)); so, there are no relevant  restrictions from a practical viewpoint. 
	
	\begin{prop} \label{pgc}
		Let $(M,F)$ be a connected Finsler manifold and  $\Sigma$ its indicatrix, regarded as a wind Finslerian structure with  forward and backward balls
		$B^+_{F}(x_0,r) (=B^+_{\Sigma}(x_0,r))$ and $B^-_{F}(x_0,r) (= B^-_{\Sigma}(x_0,r))$. The following assertions are equivalent:
		
		\begin{enumerate}[(i)]
			\item  $\hat B^+_{\Sigma}(x_0,r) = \bar B^+_{F}(x_0,r)$ for all $x_0\in M, r>0$.
			
			\item  $\hat B^-_{\Sigma}(x_0,r) = \bar B^-_{F}(x_0,r)$ for all $x_0\in M, r>0$.
			
			\item $(M,F)$ is (geodesically) convex, i.e., any pair of  points $(p,q)$ can be connected by a geodesic of length equal to the Finsler distance $d_F(p,q)$.
		\end{enumerate}
	\end{prop}
	
	\begin{proof}
		We will consider only the equivalence between $(i)$ and $(iii)$, as the  convexity of $F$ is equivalent to the convexity of its reverse metric $\tilde{F}$.
		
		$(iii)\Rightarrow (i)$. Otherwise, there exists some $x_1\in
		\bar B^+_{F}(x_0,r)\setminus \hat B^+_{\Sigma}(x_0,r)$ and, by the continuity of the distance, $d_F(x_0,x_1)=r$. But no curve of length equal to $r$ can  join these points, which contradicts geodesic convexity.
		
		$(i)\Rightarrow (iii)$. Straightforward from the definitions (recall that when the wind Finslerian structure is Finsler, $F_l=+\infty$  and a minimizing curve must be a geodesic).
	\end{proof}

	\subsection{Geodesics}
	We aim now to introduce a notion of geodesic for a wind Finslerian structure which recovers the standard one for $F$ and $F_l$. As the radius   corresponding to each $v\in \Sigma\cap (A_E\setminus A)$ is not transversal
	to $\Sigma$, $\Sigma$ does not carry a globally defined smooth contact form such that the flow of its associated  Reeb vector field is compatible with the geodesic flow of both $F$ and $F_l$ (compare with \cite[Section 2]{Bryant02}). 
	Thus, we start by defining extremizing geodesics of $\Sigma$ by unifying  local extremizing properties of both type of geodesics as follows. 
	\begin{defi}\label{extremizing}
		Let $(M,\Sigma)$ be a wind Finslerian manifold.  A   wind  curve $\gamma: [a,b]\to M$,  $a<b$,   is called a
		{\em unit extremizing geodesic} if
		\begin{equation}\label{eunitgeodesic}
			\gamma(b)\in \hat{B}_\Sigma^+(\gamma(a),b-a)\setminus B_\Sigma^+(\gamma(a),b-a) .  \end{equation}
		We will say that $\gamma$ is an {\em extremizing geodesic (resp. pregeodesic)} if it is an affine  (resp. arbitrary,  according to the end  of Convention~\ref{caestar})  
		reparametrization  of a unit extremizing geodesic. 
		
	\end{defi}
	
	Some elementary properties of these geodesics are the following.  
	
	\begin{prop} \label{disjunction} 
		Let $(M,\Sigma)$ be a wind Finsler structure. 
		\begin{enumerate}[(i)]
			\item  If $\gamma$ is a unit extremizing geodesic of $(M,\Sigma)$, then: 
			
			\begin{enumerate}[(a)]
				\item   its restriction $\gamma_{|[a',b']}$ to any subinterval $[a',b'], a\leq a'<b'\leq b$ is also a  unit  extremizing geodesic. In particular,
				\[\gamma(t)\in \hat{B}_\Sigma^+(\gamma(a),t-a)\setminus B_\Sigma^+(\gamma(a),t-a) \]
				for every $t\in [a,b]$;

				\item   at least one of the following two properties holds: 
				\begin{equation}\label{mini}
					\gamma(b)\notin B^+_\Sigma(\gamma(a),\ell_F(\gamma)),
				\end{equation}
				\begin{equation}\label{maxi}
					\gamma(b)\notin B^+_\Sigma(\gamma(a),\ell_{F_l}(\gamma)) \quad \quad (\hbox{with } \quad \ell_{F_l}(\gamma)<+\infty).
				\end{equation}
				Moreover, in the first case, $F(\dot\gamma)\equiv 1$  everywhere  and in the second  one  $F_l(\dot\gamma)\equiv 1$ everywhere. 
			\end{enumerate}
			

			
			\item    If a wind curve $\gamma$ satisfies \eqref{mini} (resp. \eqref{maxi}), then  this same property holds for the restriction $\gamma_{|[a',b']}$. 
			Moreover, if $\gamma$ is also strictly regular  or, more generally, it satisfies the hypotheses in Proposition~\ref{prop220},  then it is an extremizing pregeodesic. 
			
			%
			
			
			\smallskip
			
			\item   If a constant curve $\gamma_{x_0}(t)=x_0$ for all $t\in [a,b]$ is a (unit) extremizing geodesic then $x_0\in \MC$. In this case, $\gamma_{x_0}$ will be called an {\em extremizing exceptional} geodesic. 
		\end{enumerate}
	\end{prop}
	\begin{proof} $(i)$ 
		For $(a)$, assume by contradiction  that \eqref{eunitgeodesic} is violated in a subinterval so that  $\gamma(b') \in B_\Sigma^+(\gamma(a'),b'-a')$  (recall that  \eqref{elength} holds). 
		So, there will exist a wind curve $\tilde\gamma: [a',b']\rightarrow M$ satisfying both strict inequalities in \eqref{elength}, and so will do the concatenation $\gamma_1\colon [a,b]\to M$ of 
		$\gamma|_{[a,a']}, \tilde\gamma$ and $\gamma|_{[b,b']}$ defined as
		
		\[\gamma_{1}(t)=\begin{cases}\gamma(t)&\text{if $t\in [a,b]\setminus[a',b']$,}\\
			\tilde\gamma(t)&\text{if $t\in [a',b']$,}
		\end{cases}
		\]
		in contradiction with \eqref{eunitgeodesic} (for all the interval $[a,b]$).
		
		
		For $(b)$, notice that by the assumptions,
		\begin{equation}\label{tont}\gamma(b)\not\in  B_\Sigma^+(\gamma(a),b-a)\end{equation} 
		and 
		$\ell_F(\gamma)\leq b-a\leq \ell_{F_l}(\gamma)$. Moreover, \eqref{tont} implies that at least one of the inequalities must be an equality; so,  replace  $b-a$  with  $\ell_F(\gamma)$ or $ \ell_{F_l}(\gamma)$  in \eqref{tont}.
		For the last assertion,  observe that, among the points where $F\circ \dot\gamma$ and $F_l\circ \dot\gamma$ can be different from $1$, they are continuous except in the (finite set of) breaks.  
		
		$(ii)$ The curve $\tilde\gamma$ which shows that \eqref{mini} (resp. \eqref{maxi}) does not hold for $\gamma_{|[a',b']}$, can be concatenated (as in $(a)$ above)  to obtain the contradiction that neither this property could hold for $\gamma$. 
		
		For the last assertion, reparametrize $\gamma$ as a wind curve with domain $[0,\ell_F(\gamma)]$ or $[0, \ell_{F_l}(\gamma)]$, (see Proposition~\ref{prop220}), and \eqref{eunitgeodesic} must hold. 
		

		$(iii)$ By our conventions, $\gamma_{x_0}$ is  $\Sigma$- admissible only when $x_0\in \MC$.
	\end{proof}
	Proposition ~\ref{disjunction} $(i)$-$(b)$ suggests that  extremizing pregeodesics satisfy minimization or maximization properties. Let us  introduce  a natural variational setting. 
	\begin{defi}\label{variations} 
		Let $\gamma\colon [a,b]\to M$ be a   wind  curve between $x_0$ and $x_1$,  and assume that $\{a=t_0\leq  \ldots\leq t_n=b\}$ is a subset of the interval $[a,b]$ such that $\gamma_{|[t_{i-1}, t_i]}$ is smooth, for each $i\in\{1,\ldots, n\}$.  Let  
		$C^{\Sigma}_{x_0,x_1}[a,b]:= \{\rho\in C^{\Sigma}_{x_0,x_1}:\text{ $\rho$  is defined on $[a,b]$}\}$, and analogously let
		$C^A_{x_0,x_1}[a,b]=\{\rho\in C^{A}_{x_0,x_1}:\text{ $\rho$  is defined }$ $\text{on $[a,b]$}\}$.
		A {\em  (proper)    wind  variation} of $\gamma$ is a continuous map $\psi\colon (-\eps,\eps)\times[a,b]\to M$,  such that $\psi=\psi(s,t)$ is a $C^2$ map on $(-\eps,\eps)\times [t_{i-1},t_i]$, $\psi(0,\cdot)=\gamma$ and for each 
		$s\in (-\eps,\eps)$, $\psi_s\colon=\psi(s,\cdot)\in C^{\Sigma}_{x_0,x_1}[a,b]$. A    wind  variation will be said an   {\em $F$-wind variation} if  $\psi_s\in C^{A}_{x_0,x_1}[a,b]$, for each $s\in (-\eps,\eps)\setminus\{0\}$.  
	\end{defi}
	Observe that,  according to Definition~\ref{variations},  any wind variation of an $F$-wind curve must be $F$-wind  (reducing $\eps$ if necessary). 
	\begin{exe}
		The wind restriction for a variation may be somewhat subtle.  Consider, for example, the case in that the wind Finsler manifold is just a Riemannian one, and one is looking for wind variations $\psi$ of a unit extremizing geodesic 
		$\gamma:[a,b]\rightarrow M$. Of course, such  geodesics are just the minimizing geodesics for the Riemannian manifold parametrized by arc length. For the variation $\psi$ we must impose
		$F(\dot \psi_s)\leq 1$, and $\ell_F(\psi_s)\leq \ell_F(\gamma)$.   So, a non-trivial wind variation can exist only when $\gamma(b)$ is the first conjugate point of $\gamma$. The non-existence of such a variation before the first conjugate point 
		means implicitly that $\gamma$ minimizes strictly among nearby  curves. Clearly, one can consider also geodesics parametrized at a different speed $c$: in the case $c<1$ wind variations are equal to classical variations 
		but, in the case $c>1$, the geodesic is not a wind curve and, so, no wind variation is defined. 
	\end{exe}  
	
	The following result suggests that the question of maximization / minimization becomes somewhat subtle.
	\begin{lemma} \label{lpanocha}
		Let $\gamma\in C^{\Sigma}_{x_0,x_1}[a,b]$ be a unit extremizing geodesic
		satisfying \eqref{mini},  hence 
		$
		\ell_F(\gamma) = b-a \leq \ell_{F_l}(\gamma).
		$ 
		If there exists a wind curve 
		$\alpha \in C^{\Sigma}_{x_0,x_1}[a,b]$ s.t. 
		$\ell_F(\alpha)<\ell_F(\gamma)$ then $\alpha$ is a unit extremizing geodesic satisfying \eqref{maxi} but not \eqref{mini}, i.e.:  
		$$
		\ell_F(\alpha)<\ell_F(\gamma) = b-a =\ell_{F_l}(\alpha)\leq \ell_{F_l}(\gamma).
		$$
	\end{lemma}
	\begin{proof}
		Being $\alpha$ a wind curve, \eqref{elength} holds and, so,  $b-a =\ell_{F_l}(\alpha)$ because, otherwise, 
		$x_1\in B_\Sigma^+(x_0,b-a)$.
	\end{proof}
	Of course, a dual version of the result  holds 
	for the case that $\gamma$ satisfies \eqref{maxi}.
	\begin{exe}   We emphasize that such an $\alpha$ can exist in some particular cases. In fact, notice first that, for any wind Minkowskian norm $\Sigma$ on $\R^n$, the unit extremizing geodesics are the  straight lines with velocity constantly equal 
		to any vector of $\Sigma$. Now,
		consider a strong wind Minkowskian example $(\R^2, \Sigma)$ obtained by the displacement of the usual unit sphere by the (constant) vector $(2,0)$, and construct a wind Finslerian cylinder $(S^1\times \R, \Sigma)$ by identifying each $(x,y)$ with $(x+1,y)$,
		see Fig.~\ref{cylinder2}. Choose  the  Minkowskian wind c-ball $\hat B_\Sigma^+((0,0),r_0) \subset \R^2$ with radius $r_0=1/2$. As the natural Euclidean diameter (as a subset of the Euclidean space $\R^2$) of  $\hat B_\Sigma^+((0,0),r_0)$ is 1, its projection  in $S^1\times \R$ identifies the points $p_1=(1/2,0), p_2=(3/2,0)$ in a single one $p_C$. Then, the univocally determined unit extremizing geodesics $\beta_1, \beta_2: [0,r_0]\rightarrow \R^2$ from $(0,0)$ to $p_1, p_2$ (resp.),
		project onto geodesics of $S^1\times \R$ which play the role of $\alpha$ and $\gamma$ in Lemma~\ref{lpanocha}.
		\begin{figure}[h]
			\includegraphics[scale=0.7, center]{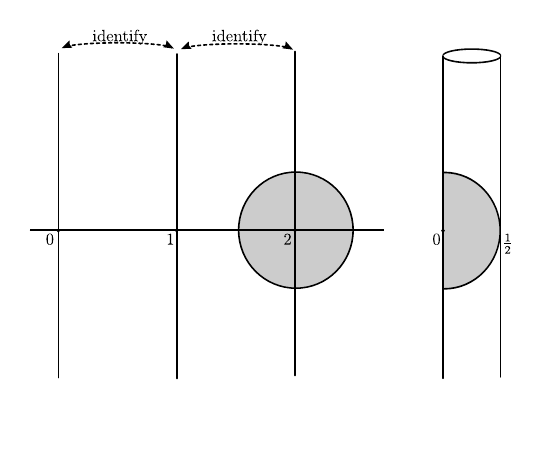}
			\caption{A wind Finslerian cylinder $(S^1\times \R, \Sigma)$. The shaded region represents the c-ball $\hat B_\Sigma^+((0,0),1/2)$}\label{cylinder2}
		\end{figure}
		
	\end{exe}

	That is, extremizing geodesics are either global minimizers of $\ell_F$ or global maximizers of $\ell_{F_l}$ on $C^{\Sigma}_{x_0,x_1}[a,b]$ except when a curve $\alpha$ as above appears. However, one can check that such a curve cannot appear among nearby
	geodesics in the following sense. 
	
	\begin{prop}\label{casesextrgeo}
		Let $ \gamma\colon[a,b]\to M$ be a unit extremizing geodesic between $x_0$ and $x_1$. 
		Then, one of the following exclusive alternatives holds:
		\begin{enumerate}[(i)]
			\item    $\ell_F(\gamma)<\ell_{F_l}(\gamma)$ and $x_1\not\in B^+_{\Sigma}(x_0, \ell_F(\gamma))$ or, equally, $\ell_F(\gamma)=b-a<\ell_{F_l}(\gamma)$.  Then $\gamma$
			minimizes the length functional of $F$ between the curves $\psi_s$ defined by  any     wind  variation $\psi: (-\epsilon, \epsilon)\times [a,b]\rightarrow M$  of $\gamma$ for  $|s|$ sufficiently small ($0\leq |s|<\epsilon'$ for some 
			$\epsilon'\leq \epsilon$).  In this case $\gamma$ will be called a {\em minimizing unit geodesic}.
			\item  $\ell_F(\gamma)<\ell_{F_l}(\gamma)<+\infty$ and $x_1\not\in B^+_{\Sigma}(x_0, \ell_{F_l}(\gamma))$ or, equally, $\ell_F(\gamma)<b-a=\ell_{F_l}(\gamma)$. Then, 
			$\gamma$ maximizes  the length functional of $F_l$ in a sense analogous to (i) above. In this case, $\gamma$
			will be called a {\em  maximizing unit geodesic}.
			\item  $\ell_F(\gamma)=\ell_{F_l}(\gamma)$ (necessarily equal to $b-a$). Then, both \eqref{mini} and \eqref{maxi} hold, and  the velocity of $\gamma$ lies in  $A_E\setminus A_l$. 
			In this case, $\gamma$    will be called a {\em  boundary   unit geodesic}.
		\end{enumerate}
		Moreover,   the restriction of $\gamma$ to  any  closed  subinterval $[a',b']$ of $[a,b]$  also satisfies the same type of extremizing property (i), (ii), (iii) as above 
		(minimization, maximization or velocity in $A\setminus A_l$ as the original $\gamma$). 
	\end{prop}
	\begin{proof}  The distinction of cases comes from Proposition~\ref{disjunction}.  
		
		{\it (i)} 
		Assume, by contradiction,  that there exists a  wind  variation of $\gamma$ such that for some sequence $s_n\to 0$, $\ell_F(\psi_{s_n})<\ell_F(\gamma)$,
		for each $n\in\N$. By Lemma~\ref{lpanocha}, $\ell_{F_l}(\psi_{s_n})=b-a$ and, thus, 
		$F_l(\dot\psi_{s_n})\equiv 1$.  As 
		$\{\dot\psi_{s_n}(t)\}_n \rightarrow \dot\gamma(t)$ for all $t$, and $F_l$ is continuous away from $\mathbf{0}$ (but, there, its value is equal to $1$ too) then $F_l(\dot\gamma)\equiv 1$. 
		As a consequence, a contradiction with the inequality in the lengths  of $\gamma$ appears. 
		
		{\it (ii)}  Analogous to part $(i)$.

		For {\it (iii)} use part (2) of Remark~\ref{remlengthbis}. For the last assertion, recall  the parts $(i)$ and $(ii)$ of Proposition~\ref{disjunction}.

	\end{proof}

	\begin{defi}\label{ex3}
		We  say  that an extremizing geodesic (or more generally,  pregeodesic) is {\em minimizing}, {\em maximizing} or {\em boundary} if it can be reparametrized as a unit extremizing geodesic satisfying respectively
		$(i)$, $(ii)$ or $(iii)$ in Proposition~\ref{casesextrgeo}.
	\end{defi}
	
	\begin{exe}\label{exboth3}
		Notice that the same geodesic can admit two reparametrizations, one as a  minimizing unit geodesic and the other as a maximizing one, so that the possibilities 
		$(i)$ and $(ii)$ are not  exclusive. 
		For example, this will happen for all the   $\Sigma$- admissible  straight lines of a  strong  wind Minkowskian structure  regarded as a wind Finslerian structure.
		In fact, the straight lines starting at the origin and tangent to the indicatrix determine the boundary geodesics, where both equalities \eqref{mini} and \eqref{maxi} hold, as the $F$- and $F_l$-lengths coincide for each one of them. 
		The straight lines inside this cone also satisfy \eqref{mini} and \eqref{maxi}, even though these lengths are now different, and they consequently admit two different parametrizations as unit geodesics,  one with $F(\dot\gamma)\equiv 1$ 
		(minimizing) and the other one with $F_l(\dot\gamma)\equiv 1$ (maximizing). Notice also that a more classical approach also shows that such
		lines minimize locally for {\em any} conic Finsler norm $F$ (see \cite[Section 3.4]{JavSan11}), and 
		an analogous reasoning shows  that they maximize locally for  {\em any}  Lorentzian norm.
	\end{exe}
	
	Finally, we arrive at the following definition of geodesic.

	\begin{defi}\label{windgeodesic}
		Let $I\subset \R$ be an interval.  We say that a curve $\gamma:I\rightarrow M$ is {\em a unit geodesic  of the wind Finslerian structure $(M,\Sigma)$} if, locally, it is a unit extremizing geodesic, namely, for every $t\in I$
		there exists $\varepsilon>0$   such  that    $\gamma|_{[t-\varepsilon,t+\varepsilon]\cap I}$ is a unit extremizing geodesic. We will say that $\gamma$ is a {\em geodesic (resp. pregeodesic)} of the wind Finslerian  manifold  
		$(M,\Sigma)$ if it is an affine (resp. arbitrary)  reparametrization of a unit geodesic. An {\em exceptional geodesic} is a   constant curve $\gamma_{x_0}$ which is locally an extremizing exceptional geodesic
		(according to Proposition~\ref{disjunction}$(iii)$). \end{defi}
	Notice that when the interval $I$ is open,  $\epsilon>0$  can  be chosen such that $[t-\varepsilon,t+\varepsilon] \subset I$ (in agreement with Definition~\ref{extremizing}), while if $I$ is compact, the intersection with $I$ must 
	be taken properly in the endpoints.

	Example~\ref{exboth3}  stresses that
	a  (non-boundary extremizing) geodesic  can satisfy simultaneously 
	both \eqref{mini} and \eqref{maxi} 
	for different ``radii'' 
	\[r_1(t,\varepsilon) = \ell_F(\gamma|_{[t-\varepsilon,t+\varepsilon]\cap I})\text{ and } r_2(t,\varepsilon)=\ell_{F_l}(\gamma|_{[t-\varepsilon,t+\varepsilon]\cap I}).\] Thus, 
	the names {\em boundary}, and {\em locally minimizing} or {\em maximizing} can be used only as non-exclusive possibilities.


	
	Proposition~\ref{pgc} suggests the following general definition of convexity.
	\begin{defi}\label{strongconvex}
		A wind Finslerian structure $(M,\Sigma)$ is  \emph{w-convex}
		if for  any $x_0\in M$ and $r>0$,  both  $\hat{B}_\Sigma^+(x_0,r)$ and $\hat{B}_\Sigma^-(x_0,r)$ are closed.  Moreover, we  say that a wind Finslerian structure is \emph{forward (resp. backward) complete} if the domain of every inextendible  geodesic is an interval of the type $(a,+\infty)$ with $a\geq -\infty$ (resp. $(-\infty,b)$ with $b\leq +\infty)$).
	\end{defi}
	\begin{prop}
		The reverse Finsler structure $\tilde\Sigma=-\Sigma$ satisfies:
		$$ B^+_{\Sigma}(x_0,r) = B^-_{-\Sigma}(x_0,r),
		\qquad \qquad \hat B^+_{\Sigma}(x_0,r) = \hat B^-_{-\Sigma}(x_0,r). $$
		So, it is w-convex iff so is $\Sigma$ and it is  forward complete iff  $\tilde\Sigma$ is backward complete.
	\end{prop}

	\subsection{Link with geodesics of  conic pseudo-Finsler  metrics }
	We will say that a conic pseudo-Finsler metric is {\em non-degenerate} when the fundamental tensor defined in \eqref{fundten} is non-degenerate. In particular, by Proposition~\ref{possibwind}, the conic pseudo-Finsler metrics $F$ and $F_l$ 
	associated with a wind Finslerian structure are non-degenerate.
	Our aim will be to justify that the non-boundary geodesics  coincide with the geodesics for $F$ or $F_l$. Now, on the one hand, the fundamental tensor of $F_l$ is not positive definite and, on the other, the domains of $F$ and $F_l$ are only conic. 
	So, we will make a brief study before arriving at Theorem~\ref{extregeo}.

	\begin{defi}
		Let $F$ be a non-degenerate conic pseudo-Finsler metric on $M$ with conic domain $A$. The Cartan tensor of $F$ is defined as   	
		\[
		C_v(w_1,w_2,w_3)=\frac 14\left. \frac{\partial^3}{\partial s_3\partial s_2\partial s_1}F^2\left(v+\sum_{i=1}^3 s_iw_i\right)\right|_{s_1=s_2=s_3=0}
		\]
		for  $v\in A$  and $w_1,w_2,w_3\in T_{\pi(v)} M$.
	\end{defi}
	Because of the non-degeneracy of $g_v$, it makes sense to consider the Chern connection and, thus,  the formal Christoffel symbols that yield the geodesic equations.
	However, following \cite{Matthi80},  it is especially convenient to study it as a family of affine connections associated with $F$-admissible vector fields  (recall Definition~\ref{sigmadmissible}-(i)):
	\begin{defi}
		Let $F$ be a \bw non-degenerate \ew conic pseudo-Finsler. Given an $F$-admissible vector field $V$ on an open subset $\Omega\subset M$, we define  $\nabla^V$ as the unique  affine connection on  $\Omega$ such that it is
		\begin{enumerate}
			\item {\it torsion-free}, namely,
			\[\nabla^V_XY-\nabla^V_YX=[X,Y]\]
			for every  smooth  vector fields $X$ and $Y$ on $\Omega$,
			\item and  {\it almost $g$-compatible}, namely,
			\[X( g_V(Y,Z))=g_V(\nabla^V_XY,Z)+g_V(Y,\nabla^V_XZ)+2 C_V(\nabla^V_XV,Y,Z),\]
			where $X$, $Y$ and $Z$ are  smooth  vector fields on $\Omega$.
		\end{enumerate}
	\end{defi}
	This approach to Chern connection is very suitable to compute the variations of the length and the energy functional as it was shown in \cite{Jav13,Rad04b}.
	Let us describe it.    
	\begin{defi}
		Given  a chart  $\varphi:\Omega\subset M\rightarrow \varphi(\Omega)\subset\R^m$,
		$\Omega$ open, $\varphi(p)=(x^1(p),x^2(p),\ldots,x^m(p))$,
		we define the Christoffel symbols associated with $\varphi$ and to  the $F$-admissible  vector field $V$, $\gam{V}$, by means of the equation
		\[
		\nabla^V_{\frac{\partial}{\partial x^i}}\left(\frac{\partial}{\partial x^j}\right)=\sum_{k=1}^ m\gam{V} \frac{\partial}{\partial x^k},
		\]
		for $i,j=1,\ldots,m$.
	\end{defi} 
	
	Observe that $\gam{V}$ in $p\in \Omega$ depends only on $V_p$ and not on the extension $V$ (see for example \cite[Proposition 2.6]{Jav13}) and therefore $\Gamma^k_{\,\,ij}$ is a real function defined  on $T\Omega\cap A$.  Moreover,  for any positive function $\lambda$  on $[a,b]$ we have $\gam{V}=\gam{\lambda V}$,
	(see for example \cite[Remark 2.4]{Jav13}).  
	So the following definition becomes consistent:
	\begin{defi}
		Let $\gamma:[a,b]\rightarrow M$ be a curve and $V$ be an $F$-admissible vector field along $\gamma$. The {\em covariant derivative of a vector field $X$ along $\gamma$  with reference $V$} is defined,  (a) when the curve is contained in the domain of a  coordinate chart $(\Omega,\varphi)$,  as
		\begin{equation}\label{DV}
			D^V_\gamma X:= \sum_{i=1}^m\frac{{\rm d}X^i}{{\rm d}t}\frac{\partial}{\partial x^i}+ \sum_{i,j,k=1}^m X^i\dot\gamma^j\gam{V}\frac{\partial}{\partial x^k},
		\end{equation}
		where $(X^1,\ldots,X^m)$ and $(\dot\gamma^1,\ldots,
		\dot\gamma^m)$ are respectively the coordinates of $X$ and $\dot\gamma$ in the coordinate basis of $\varphi$,   (b) in the general case, cover the curve $\gamma$ with a finite number of coordinate charts and define $D^V_{\gamma}X$ in every interval contained in one of these charts as in \eqref{DV}  (the fact that $D^V_\gamma X$ in \eqref{DV} does not depend on the chart used to compute it guarantees that the covariant derivative is well-defined).  Moreover, $\gamma$ is a {\em geodesic} of $(M,F)$ if it is a  (smooth)  $F$-admissible  curve satisfying the equation
		\begin{equation}
			\label{geoF}
			D_\gamma^{\dot\gamma}\dot\gamma=0.
		\end{equation}
	\end{defi}
	As in the standard Finsler case, geodesics (resp. pregeodesics)  are always critical points of the energy  (resp. length)  functional. Nevertheless, in order to ensure that a  piecewise smooth curve which is a critical point of the energy  
	(resp. length)  functional becomes a geodesic  (resp. pregeodesic),  one should require that the Legendre transform is injective (the non-degeneracy of $g_v$ implies that the Lagrangian $L=F^2/2$ is regular and thus,  its Legendre transform  
	is locally  injective \cite[Definition 3.5.8 and Proposition 3.5.10]{AbrMar78}), but global injectivity is naturally required to avoid problems in the breaks, see \cite{JavSoa13}). \soutE{Nevertheless} \bw Anyway, \ew this  always holds  in
	our case, as the following refinement of \cite[Lemma 3.1.1]{Shen01} shows.
	Recall that the Legendre transform of $F$ is defined as the fibre derivative of
	$L$. By homogeneity, it is shown that it coincides with the map
	$\mathscr{L}_F:A\to T^\ast M,$ such that for every $v\in A$,
	$\mathscr{L}_F(v)$ is given by $\mathscr{L}_F(v)(u)=g_v(v,u)$, $u\in T_{\pi(v)}M$.  
	\begin{prop}
		Let $F: A\rightarrow \R$ be a conic Finsler or a Lorentzian Finsler metric on a manifold $M$ such that $A_p\cup \{0\}$ is a convex set for all $p\in M$. Then, its Legendre transform is injective (and, thus, a diffeomorphism onto its image).
		
		In particular, this happens for the conic Finsler metric and the Lorentzian Finsler metric associated with any wind Finslerian structure $\Sigma$.
	\end{prop}
	\begin{proof}
		Recall that by the hypotheses on $A$,  the indicatrix of $F$ at $p$, $(\Sigma_F)_p=\{v\in T_p M: F(v)=1\}$, is a strongly convex hypersurface in $T_pM$ (when $F$ is a Lorentzian Finsler metric, this is understood in the sense that the opposite 
		normal direction has been chosen in the computation of the second fundamental form).  Hence, if $F$ is conic Finsler (resp. Lorentzian Finsler), the  set $C_p=\{v\in A_p: F(v)\leq 1\} \cup \{0_p: \text{if }A_p=T_pM\setminus \{0\}\}$   
		(resp.  $C_p=\{v\in A_p:F(v)\geq 1\}$) is convex.
		Assume  by contradiction  that there exist two  different  vectors $v_1,v_2\in A$,  such that  $\mathscr{L}_F(v_1)=\mathscr{L}_F(v_2)$.  Clearly, by homogeneity, $v_1$ and $v_2$ cannot be collinear.   Then,  the non-extreme points 
		of the segment  joining $v_1/F(v_1)$ and $v_2/F(v_2)$ are contained in  the interior of $C_p$ and the vector $v_1/F(v_1)-v_2/F(v_2)$  points  outwards  $C_p$ in
		$v_1/F(v_1)$ but  inwards  in $v_2/F(v_2)$.  This implies that $\mathscr{L}_F(v_1/F(v_1))(v_2/F(v_2)-v_1/F(v_1))$ and
		$\mathscr{L}_F(v_2/F(v_2))(v_2/F(v_2)-v_1/F(v_1))$ have different signs and so is, by homogeneity,  for $\mathscr{L}_F(v_1)(v_2/F(v_2)-v_1/F(v_1))$ and
		$\mathscr{L}_F(v_2)(v_2/F(v_2)-v_1/F(v_1))$, a contradiction.
	\end{proof}
	
	\begin{lemma}\label{criticallength}
		Assume that $(M,F)$ is a non-degenerate pseudo-Finsler manifold, such that its  Legendre transform $\mathscr{L}_F$ is one-to-one. Then a  curve
		$\gamma\in  \Omega^A_{x_0,x_1}$ is a geodesic of $(M,F)$ if and only if it is a critical point of the length  functional  and  $F(\dot\gamma)$ is constant.
	\end{lemma}
	\begin{proof}
		Following the same lines as in \cite[Proposition 3.1]{JavSoa13}  we can deduce that a curve  $\gamma \in \Omega^A_{x_0,x_1}$ is a critical point of the length functional  if  and only if it satisfies the equation
		\begin{equation}\label{geoeq}
			D_\gamma^{\dot\gamma}\left(\frac{\dot\gamma}{F(\dot\gamma)}\right)=0,
		\end{equation}
		on the interval $[a,b]$. Then any reparametrization $\sigma$ of $\gamma$ such that $F(\dot \sigma)=\mathrm{const.}$ must be smooth and satisfy  equation  \eqref{geoF}.
	\end{proof}
	\begin{thm}\label{extregeo}
		Let $(M,\Sigma)$ be a wind Finslerian  manifold  and $\gamma:[a,b]\rightarrow M$ be    an   $F$-admissible curve.  If  $\gamma$  is a  unit geodesic of $(M,\Sigma)$ then it is a unit  geodesic  of   one  of the two  
		conic pseudo-Finsler metrics associated with $\Sigma$.
	\end{thm}
	\begin{proof}  
		Let us show that $\gamma$  either  minimizes $\ell_F$ or maximizes $\ell_{F_l}$ locally. Being $\gamma$ $F$-admissible and   a  unit  geodesic of $(M,\Sigma)$,  either $(i)$ or  $(ii)$   
		of Proposition~\ref{casesextrgeo}   holds locally.  Then for each $t\in [a,b]$, there exists $\varepsilon$, depending on $t$, such that $\gamma|_{[a,b]\cap[t-\varepsilon,t+\varepsilon]}$   either  minimizes  $\ell_F$ or  
		maximizes $\ell_{F_l}$,    for any  fixed endpoint variation  {\em wind} variation of  $\gamma|_{[\tilde a,\tilde b]}$, where $\tilde a$, $\tilde b$ are the endpoints of the interval $[a,b]\cap[t-\varepsilon,t+\varepsilon]$.   In the first case (the reasoning in the second case is analogous), assume by contradiction  that there exists a  variation (non necessarily a wind one) $\psi \colon (-\epsilon,\epsilon)\times [\tilde a, \tilde b]\to M$, such that $\ell_F(\psi_{s_n})<\ell_F(\gamma)$, for some sequence $s_n\to 0$.   Being $\gamma$ $F$-admissible, also $\psi_{s_n}$ are so  and, then, by Proposition~\ref{prop220} they can be reparametrized as  wind curves on the interval $[\tilde a, \tilde b]$. Moreover, $\ell_F(\gamma|_{[\tilde a,\tilde b]})=\tilde b-\tilde a<\ell_{F_l}(\gamma|_{[\tilde a, \tilde b]})$. Arguing as in the proof of Proposition~\ref{casesextrgeo}-$(i)$, we then get $\ell_{F_l}(\gamma|_{[\tilde a, \tilde b]})=\tilde b-\tilde a$, a contradiction. Therefore, $\gamma|_{[\tilde a, \tilde b]}$ must minimize $\ell_F$ for any variation $\psi$ and by 
		Lemma~\ref{criticallength}, this  implies that $\gamma$ satisfies \eqref{geoeq} for $F$ or $F_l$ on $[a,b]\cap[t-\varepsilon,t+\varepsilon]$ and therefore, being $t$ arbitrary, on all $[a,b]$.
		Indeed, observe that by  Proposition~\ref{disjunction} $(i)$  (case $(b)$),  when $\gamma$ minimizes $\ell_F$, then $F(\dot\gamma)\equiv 1$ and when $\gamma$ maximizes $\ell_{F_l}$,  $F_l(\dot\gamma)\equiv 1$. As both subsets, $\{s\in [a,b]:F(\dot\gamma(s))=1\}$ and $\{s\in [a,b]:F_l(\dot\gamma(s))=1\}$, are closed and disjoint (because $\dot\gamma$ belongs to $A$) then one of them coincides with $[a,b]$ and the other is empty.  
	\end{proof}
	\subsection{Wind Riemannian structures}\label{windRiemann}
	Let us focus now on a particularly important case of wind Finslerian structures.
	\begin{defi}  A {\it wind Riemannian structure} is
		a wind Finslerian structure $\Sigma$ in $TM$ such that
		$\Sigma_p=\Sigma\cap T_pM$ is a (real non-degenerate) ellipsoid for every $p\in M$.
	\end{defi}
	\begin{prop}\label{windyfermat}
		Any  wind  Riemannian structure $\Sigma$ can be constructed univocally
		as the displacement of the indicatrix of a smooth Riemannian metric $g_R$  along a vector field $W$.\end{prop}
	\begin{proof}
		From Proposition~\ref{pzerm}, the field $W$ of the centers of the ellipsoids $\Sigma_p$ is smooth and, by Proposition~\ref{ptf}, the translated hypersurface $\Sigma-W$ is a wind Riemannian structure with centers at $0_p$, for each $p\in M$.
		By  Proposition~\ref{windConseq}, it defines a smooth Riemannian metric $g_R$  on $M$. Hence, $\Sigma$ is defined by the equation $g_R(v-W,v-W)=1$ in $TM$. Clearly, if for any other Riemannian metric $h_0$ and vector field $V$, 
		$\Sigma$ is defined by the equation $h_0(v-V,v-V)=1$ then, necessarily, $V$ must be the field of the centers of the ellipsoids $\Sigma_p$ and then equal to $W$, so that $h_0$ must be equal to  $g_R$. 
	\end{proof}
	In addition to the previous characterization, the definition of a wind Riemannian structure as  a structure of ellipsoids suggests a second characterization in terms of the zeroes of a pointwise polynomial of degree two (defined up to a
	pointwise smooth non-vanishing factor). This second viewpoint will be interpreted in the next section in terms of the conformal class
	of an \sstk splitting,  which will allow us to obtain a powerful characterization of the geometry of wind Riemannian structures.
	
	The following elements equivalent to $g_R, W$ will be used in the remainder and  will be well  adapted to the case of \sstk splittings.
	
	\begin{defi}\label{ctriple}
		Given a  wind Riemannian structure determined by a  Riemannian metric $g_R$  and a vector field $W$   the {\em associated  triple}  $(g_0,\omega,\Lambda)$ is the triple composed by  $g_0=g_R$ and  $\omega, \Lambda$ are the one-form and the function defined as $\omega=-g_0(\,\cdot\, ,W), \Lambda=1-g_0(W,W)$.
	\end{defi}
	Thanks to Proposition~\ref{windyfermat}, we  will also \bw simply say \ew  that $\Sigma$ is the translation of a Riemannian metric,  as in the case of  Zermelo's navigation problem. In fact, in the case  $g_0(W, W)<1$, the so-obtained $\Sigma$ yields 
	a Randers metric $Z$, that is,  $Z(v)=\alpha(v)+ \beta(v)$ for every $v\in TM$,  where $\alpha(v)=\sqrt{\tilde h(v,v)}$,
	being $\tilde h$ a Riemannian metric on $M$ and $\beta$ a one-form such that  its norm with respect to $\tilde h$ satisfies  $\|\beta\|_{\tilde h}<1$ at every point. Indeed, Randers metrics are characterized by this property 
	(see \cite[Section 1.3]{BaRoSh04},  \cite[\S 2.2]{CheShe12} or the computations below). Let us determine all the cases that appear when considering wind  Riemannian  structures, refining Proposition~\ref{possibwind}.
	\begin{prop}\label{Riemannclass}
		Let $(M,\Sigma)$ be a wind Riemannian structure. At each point $p\in M$, one of the following three exclusive cases holds for some one-form $\beta$,  some  scalar product $\tilde h$ of index $0$ or $m-1$, and $\alpha(v)=\sqrt{|\tilde h(v,v)|}$:
		\begin{enumerate}[(i)]
			\item  if the zero vector $0_p$ belongs to the  open unit ball   $B_p$, then $\Sigma_p$ determines a  {\em Randers norm,}   i.e.,
			$F(v)=\alpha(v)+\beta(v)$, where $\tilde h$ is positive  definite  and $\|\beta\|_{\tilde h}<1$ on $A_p=T_pM$;
			\item if $0_p$ lies in $\Sigma_p$, then $\Sigma_p$ determines a
			{\em Kropina norm}, i.e.,
			\[F(v)=\frac{\alpha(v)^2}{\beta(v)},\]
			for  a nowhere vanishing   $\beta$ and  $\tilde h$   positive definite, defined on  $A_p=\{v\in T_pM:  \beta(v)>0\}$;
			\item if $0_p$ does not lie in $B_p\cup \Sigma_p$, then $\Sigma_p$ determines   a {\em proper  wind Riemannian  structure}, i.e., a pair of conic  pseudo-Minkowski norms:
			\begin{align*}
				F(v)&=-\alpha(v)+\beta(v),\\
				F_l(v)&=\alpha(v)+\beta(v),
			\end{align*}
			defined on
			\[A_p=\{v\in T_pM:  \tilde h(v,v)>0\text{ and } \beta(v)>0\},\]
			where $\tilde h$ has index $m-1$ and $\beta$ satisfies $\beta(v)^2>\tilde h(v,v)$,  for  all  $v\in T_pM\setminus \{0\}$, that is $\beta\otimes\beta-\tilde h$ is positive definite.
		\end{enumerate}
		Moreover, in all the three cases the converse holds.
	\end{prop}
	\begin{proof}
		Let $(g_0,\omega,\Lambda)$ be the triple associated with $\Sigma$ according to Proposition~\ref{windyfermat} and Definition~\ref{ctriple}.
		The conic pseudo-Finsler metrics $F$ and $F_l$ associated with $\Sigma$ (see Proposition
		\ref{windConseq})  are both determined by the equation
		\[g_0\left(\frac{v}{Z(v)}-W,\frac{v}{Z(v)}-W\right)=1\]
		(recall \eqref{zerm}) which is equivalent to
		\begin{equation}\label{eq:segundograu}
			g_0(v,v)+2\omega(v) Z(v)-\Lambda Z(v)^2=0,
		\end{equation}
		and, whenever $\Lambda\neq 0$,
		\[Z(v)=\frac{\omega(v)\mp \sqrt{\Lambda g_0(v,v)+\omega^2(v)}}{\Lambda}.\]
		We are interested only in the solutions that make $Z(v)$ positive.
		
		Case $(i)$. If $\Lambda(p)>0$  ($0_p\in B_p$), then the unique positive value of $Z(v)$ is:
		\begin{equation}\label{randersone}
			F(v)=\frac{\omega(v)+\sqrt{\Lambda g_0(v,v)+\omega^2(v)}}{\Lambda}
		\end{equation}
		and the required $\tilde h, \beta$ are then:
		\begin{equation}\label{eab}\tilde h= \frac{g_0}{\Lambda} +  \frac{\omega}{\Lambda}\otimes \frac{\omega}{\Lambda}, \qquad \quad \beta =\frac{\omega}{\Lambda}.\end{equation}
		Conversely,  if $F=\alpha+\beta$, with $\alpha$ the norm of a Riemannian metric $\tilde h$ and $\|\beta\|_{\tilde h}<1$, we can reconstruct $g_0$, $\omega$ and $\Lambda$  from $\tilde h, \beta$,  just by using \eqref{eab} and defining
		\begin{equation}\label{elambda}
			\Lambda = \frac{1}{1+\Delta} , \qquad \qquad \hbox{where} \quad \Delta:= (\tilde h-\beta\otimes\beta) (\beta^\sharp,\beta^\sharp )
		\end{equation}
		being $\beta^\sharp$  the vector metrically equivalent to $\beta$ for the metric $\tilde h-\beta\otimes\beta$.  The restriction $\|\beta\|_{\tilde h}<1$ forces $ \Delta \geq 0 $, i.e., $0< \Lambda \leq 1$, which  ensures the 
		consistency of the reconstruction of $g_0, W$ and $\Lambda$ from  $\alpha$ and $\beta$ (in fact, a posteriori, $\beta^\sharp=-W$ and $\Lambda = 1-g_0(W,W)$).
		
		However, in order to understand better this reconstruction   for later referencing, notice also that the vector $B$ which is $\tilde h$-equivalent to $\beta$ is  proportional to $\beta^\sharp$ (if $\tilde h(B,v)=\beta(v)=0$ then 
		$(\tilde h-\beta^2)(B,v)=0$). More precisely, $\beta^\sharp= aB$ with $a=1/(1-\tilde h(B,B))$ and  $1-\tilde h(B,B)>0$ (this follows equating the expressions $\Delta = \beta(\beta^\sharp)= a\beta(B)=a \tilde h(B,B)$ and  
		$\Delta=a^2(\tilde h-\beta^2)(B,B)$, the latter greater than $0$ whenever $B\neq 0$ and equal to $a^2\tilde h(B,B)(1-\tilde h(B,B))$). Then, putting $\Delta = a\tilde h(B,B)$ in \eqref{elambda} one also has:
		\begin{equation}
			\label{el}
			\Lambda=1-\tilde h(B,B)
		\end{equation}
		(compare with \cite[\S 1.3]{BaRoSh04} and also \cite[Proposition 3.1]{BiJa11}).
		
		Case $(iii)$.
		If $\Lambda(p) <0$, at $p\in M$, then there are two  solutions,
		one given by \eqref{randersone} and the other by
		\begin{equation}\label{randersone2}
			F_l(v)=\frac{\omega(v)- \sqrt{\Lambda g_0(v,v)+\omega^2(v)}}{\Lambda}
		\end{equation}
		both defined in the open domain
		\[A_p=\{v\in T_pM: -\omega (v) >0, \ \Lambda g_0(v,v)+\omega^2(v)>0\}.\] The required $\tilde h, \beta$ are obtained by using again  the  expressions in \eqref{eab}.  Observe that in this case, $\tilde h$ is negative definite
		in the kernel of $\omega$  and  $\tilde h(W,W)=\frac{g_0(W,W)}{\Lambda^ 2}>0$,  which  implies  that $\tilde h$ has index $m-1$.

		For the converse, recall first that, if $B$ is the vector $h$-metrically related to $\beta$, then by the hypotheses, $B$ cannot be 0 and $(\tilde h-\beta^2)(B,B)<0$, i.e.:
		\begin{equation}
			\label{ebb} \tilde h(B,B)(1-\tilde h(B,B)) < 0. 
		\end{equation}
		Moreover, $\tilde h(B,B)$ cannot be negative. Indeed, otherwise, $B$ would be a spacelike vector for the Lorentzian scalar product $-\tilde h$. So, we could take a timelike vector $v$ for $-\tilde h$ orthogonal to $B$ and we would 
		have $(\tilde h-\beta^2)(v,v)=\tilde h(v,v)>0$, in contradiction with the hypotheses on $\tilde h, \beta$. Therefore, \eqref{ebb} forces $\tilde h(B,B)>1$. This ensures that $\Lambda, \omega, g_0$ can be reconstructed from \eqref{el}
		and \eqref{eab} with $\Lambda(p)<0$.

		Case $(ii)$. Now $\Lambda(p)=0$ and from \eqref{eq:segundograu},  we obtain only one metric
		\begin{equation}\label{krop}
			F(v)=-\frac{g_0(v,v)}{2\omega(v)},
		\end{equation}
		which is of the type in $(ii)$  with  $\alpha$ the norm associated with $g_0$ and $\beta=-2\omega$.
		For the converse, choose   $g_0=\frac{\tilde h}{4 \tilde{h}(B,B)}$ and $W=\frac 12 B$, and recall that  $\tilde{h}(B,v)=\beta(v)$ for every $v\in TM$. 
	\end{proof}
	Observe that  the  analysis in Proposition~\ref{Riemannclass} was accomplished in each single tangent space,  while  a wind Riemannian structure in a manifold $M$ can attain all  the three possible types. The standard expressions of 
	the metrics given in the proposition do not allow us to give a  unified  expression on $TM\setminus\mathbf{0}$ of the metric $F$ which can be achieved instead as follows:
	\begin{prop}\label{randerskropinagen}
		Let $(M, \Sigma)$ be a wind Riemannian structure with associated triple $(g_0,\omega, \Lambda)$ (according to Definition~\ref{ctriple}). Then, the conic  Finsler metric pointwise determined by Proposition~\ref{Riemannclass} is equal to
		\begin{equation}\label{randerskropina}
			F(v)=\frac{g_0(v,v)}{-\omega(v)+ \sqrt{\Lambda g_0(v,v)+\omega^2(v)}},
		\end{equation}
		defined, up to the zero section, in the {\em interior} of
		$$\{v\in TM\setminus\mathbf{0} :-\omega(v)+ \sqrt{\Lambda g_0(v,v)+\omega^2(v)}>0\},$$
		on all $TM$. Moreover, $F_l$ on $A_l$ is equal to  the expression above with a minus sign before the root.
	\end{prop}
	\begin{proof}
		Observe that the expressions in \eqref{randersone} and \eqref{krop} coincide with \eqref{randerskropina}.
		In fact, the last inequality is fulfilled for all $v\in T_pM\setminus \{0\}$ whenever $\Lambda(p)>0$, it reduces to $\{v\in T_pM: -\omega(v)>0\}$ whenever $\Lambda(p)=0$, and it includes implicitly the restrictions 
		$\Lambda(p)g_0(v,v)+\omega^2(v)>0$ plus $-\omega(v)>0$ when $\Lambda(p)<0$.   For $F_l$, it is enough to notice that  \eqref{randersone2} is equal to
		\[-\frac{g_0(v,v)}{\omega(v)+ \sqrt{\Lambda g_0(v,v)+\omega^2(v)}}\]
		for  any $v\in A_l$.
	\end{proof}
	The  case $\Lambda\geq 0$ ($g_0(W,W)\leq 1$)
	makes possible a simple description of the wind Riemannian structure  $\Sigma$, as it
	determines a unique conic Finsler metric $F$, which adopts either the Randers or the Kropina form in Proposition~\ref{Riemannclass}.
	\begin{defi}\label{dranderskropina}
		A {\em Randers-Kropina} metric on a manifold $M$ is any wind Riemannian structure  $\Sigma$ such that $0_p\in \bar B_p$ for all $p\in M$ so that
		$\Lambda\geq 0$ and the associated conic Finsler metric $F$ is given by \eqref{randerskropina} with domain
		\[A=\bigcup_{p\in M}A_p, \text{ where } A_p=
		\begin{cases}T_pM\setminus\{0\}, &\text{ if $g_0(W_p,W_p)<1$,}\\
			\{v\in T_pM\colon g_0(W_p,v)>0\}, &\text{ if $g_0(W_p,W_p)=1$.}
		\end{cases}
		\]
	\end{defi}
	When the wind is strong $\Lambda<0$ ($g_0(W,W)> 1$) or, simply, when one restricts to the region  $M_l$, a specific property of the wind Riemannian case holds, namely, the Lorentzian Finsler metric $F_l$ can be described formally 
	in terms of $F$. In fact, notice that the expression \eqref{randersone2} can be obtained from \eqref{randersone} just by applying $F$ to $-v$ and reversing the sign, and analogously this happens with the expression of $F_l$ and $F$ 
	in part $(iii)$ of Proposition~\ref{Riemannclass}. Summing up, we have:
	
	\begin{prop}\label{pformalreverse}
		Let $M_l$ be the strong wind region of a wind Riemannian structure, and $F, F_l$ its associated conic pseudo-Finsler metrics. Then there exist a one-form $\beta$ and a Lorentzian metric $-\tilde h$ such that $\beta\otimes \beta-\tilde h$ 
		is Riemannian satisfying:
		\begin{enumerate}[(i)]
			\item the domain of $F$ and $F_l$ is
			\[A_l=\{v\in TM_l: \tilde h(v,v)>0 \text{ and }  \beta(v)>0 \}.\]
			
			\item $F=-\alpha+\beta$.
			
			\item  $F_l=-F^{\hbox{{\tiny frev}}}$, where $F^{\hbox{{\tiny frev}}}$ is the {\em formal reverse of $F$}, defined by:
			$$F^{\hbox{{\tiny frev}}}(v)=F(-v) , \qquad\qquad\qquad \forall v\in A_l ,$$
			and $F(-v)$ is obtained by applying the expression (ii) to the vectors of $-A_l$.
		\end{enumerate}
	\end{prop}
	
	\begin{rem}\label{rformalreverse}
		Taking into account also the expressions for the conic pseudo-Finsler metrics $\tilde F, \tilde F_l$ with domains associated with the reverse  wind  Riemannian   structure $\tilde \Sigma$ we can  write:
		$$ \tilde F(v)=F(-v),\qquad \qquad \tilde F_l(v)= F_l(-v)=-F(v), \qquad \qquad\forall v\in -A,$$
		where  $F(v)$ is computed  by applying the expression $(ii)$ in the previous proposition. So, the formal expression of $F$ allows us to write easily  $\tilde F$, $F_l$ and $\tilde F_l$ ---these considerations can be extended naturally 
		to the bigger domains $A_E$ in  Convention~\ref{caestar}. This simplifies notations and makes clear that a piece of the indicatrix determines  all of them. Even though we will \bw usually work \ew  with $F$ and $F_l$, some formulae  will be written  
		conveniently by using above expressions.
	\end{rem}
	\section{Fermat structures for \sstk splittings}\label{windFermat}
	In a series of papers \cite{CaJaMa, GHWW, CaJaMa10, CapJavSan10, CapGerSan12,
		FlHeSa13},  it has been developed a detailed correspondence between the geometric properties of
	Randers spaces and the conformal structure of stationary
	spacetimes, including: variational principles for geodesics of a Finsler metric vs.  Fermat's principle for lightlike and timelike geodesics  \cite{CaJaMa}, links between the curvatures of Randers and stationary spaces \cite{GHWW}, 
	Morse theory for Finsler geodesics vs. Morse theory for lightlike and timelike geodesics \cite{CaJaMa10, CaJaMa13},   Finslerian distances and
	geodesics vs. causal structure \cite{CapJavSan10}, convexity of
	hypersurfaces vs. visibility and gravitational lensing
	\cite{CapGerSan12},  Busemann plus Gromov boundaries vs. causal
	boundaries \cite{FlHeSa13}  and almost isometries vs. conformal maps \cite{JaLiPi11}.  
	As pointed out in \cite{CapJavSan10}, such a correspondence would be extendible to obtain further properties of general Finslerian manifolds suggested by the spacetime viewpoint, yielding so a broader
	relation between Lorentzian and Finslerian geometries. 
	In the next section we consider the class of spacetimes that allows us to extend this relation to wind Riemannian structures.
	%
	\subsection{Spacetimes with a space-transverse Killing field}\label{prelimin}
	We will follow \cite{BeEhEa96} and \cite{MinSan08}
	for the general background on spacetimes and  causality. In
	particular,  if $(L,g)$ is  an $(m+1)$-dimensional Lorentzian manifold (with signature
	$(-,+,\dots,+)$) we say, following \cite{MinSan08}, that a tangent
	vector $v\in TL$ is {\em timelike} (resp. lightlike; causal;
	spacelike; non-spacelike) if $g(v,v)<0$ (resp. $g(v,v)=0$ and
	$v\neq 0$; $v$ is either timelike or lightlike; $g(v,v)>0$;
	$g(v,v)\leq 0$).  A {\em spacetime} is a connected time-oriented
	Lorentzian manifold, which  is also denoted $(L,g)$; the time
	orientation  continuously selects a  causal cone at each tangent space and it
	makes possible  to distinguish between {\em future-pointing} causal vectors
	(namely, those in the selected cone) and {\em past-pointing} ones.
	We say that two points $p,q\in L$ are {\it chronologically
		related} ($p$ is chronologically related to $q$ or $p$ lies in the
	chronological past  of  $q$), denoted $p\ll q$, if there exists a
	future-pointing timelike curve (i.e. its tangent vectors are always causal future-pointing) from $p$ to $q$. {\it The
		chronological future} of $p$ is defined as the subset
	$I^+(p)=\{q\in  L:  p\ll q\}$ and analogously the chronological past
	as $I^-(p)=\{q\in  L:  q\ll p\}$. Moreover, we say that $p,q$ are
	{\it strictly causally related} (resp. {\it causally related}),
	denoted $p< q$ (resp. $p\leq q$), if there exists a
	future-pointing causal curve from $p$ to $q$ (resp. either $p<q$
	or $p=q$).  The causal {\it future and past} of $p$ are defined
	respectively as $J^+(p)=\{q\in   L:  p\leq q\}$ and $J^-(p)=\{q\in  L: 
	q\leq p\}$.
	
	\begin{rem} \label{rll}  A well-known property to be used later is that, whenever $p\leq q\leq r$ ($p, q,r\in L$), either $p\ll r$ or the unique non-spacelike curves from $p$ to $r$ are null pregeodesics (with no conjugate points except,
		at most, the endpoints); other  properties  such as $p\leq q\ll r \Rightarrow p\ll r$ and the fact that the relation $\ll$ is open  are also well known.  \end{rem}
	Now, let us focus on the class of spacetimes relevant
	for our approach.
	\begin{defi} \label{dsstk} A spacetime $(L,g)$ is {\em standard
			with a space-transverse Killing vector field (\sstk )} if it
		admits a (necessarily non-vanishing)  complete Killing vector
		field $K$ and a spacelike hypersurface $S$ transverse to $K$  which
		is crossed exactly once by all the integral curves of $K$. 
	\end{defi}
	\begin{prop}\label{psstk} A spacetime is \sstk
		if and only if it is isometric to a product manifold $\R\times M$
		endowed with a Lorentzian metric $g$ of the form
		\begin{equation}\label{lorentz}
			g=- (\Lambda \circ \pi) \df t^2+\pi^*\omega\otimes \df t+\df
			t\otimes \pi^*\omega+\pi^*g_0,
		\end{equation}
		where $\Lambda$, $\omega$ and $g_0$ are, respectively, a smooth
		real function, a one form and a Riemannian metric  on $M$,
		$\pi:\R\times M\rightarrow M$ is the natural projection, and
		$\pi^*$ the pullback operator, satisfying the following relation:
		\begin{equation}
			\label{lorentzian} \Lambda +\|\omega\|^2_{0}>0,
		\end{equation}
		being $\|\omega\|_0$ the pointwise $g_0$ norm of $\omega$.
		In this case,   the projection  $t:\R\times M\rightarrow \R$ satisfies that  $-\nabla t$ is a timelike vector field, which can be assumed future-pointing (i.e. time-orientating the spacetime) with no loss of generality.
		
	\end{prop}
	\begin{proof}
		Notice first that the bilinear form $g$ given in \eqref{lorentz}
		is a Lorentzian metric if and only if \eqref{lorentzian} is
		fulfilled at each $x\in M$. In fact, let $e_1,e_2,\ldots,e_m$ be
		an orthonormal basis for $(T_xM,g_0)$ such that
		$\omega(e_1)=\|\omega\|_{0}$ and $\omega(e_i)=0$, for each
		$i=2,\ldots, m$. Let
		$B=\{(1,0),(0,e_1),(0,e_2),\ldots,(0,e_m)\}$  be  the corresponding
		basis of $(\R\times T_xM,g)$ and $M_B(g)$ the matrix
		representation of $g$ in $B$. The only non-diagonal elements
		different  from $0$ in this matrix come from the product of the first
		two elements of $B$ and, thus
		\begin{equation}
			\label{elor}  {\rm det}\,  M_B(g)=-\Lambda-
			\omega(e_1)^2=-\Lambda-\|\omega\|^2_{0}, 
		\end{equation} which must be negative
		to ensure the Lorentzian signature.
		Clearly, \eqref{lorentz} defines an \sstk with $K=\partial_t$ and
		$S$ equal to any slice  $S_{t_0}:= t^{-1}(t_0)$.
		
		Conversely, given any non-vanishing
		Killing vector field  $K$  on a Lorentzian manifold  $(L,g)$, and any
		choice of a spacelike hypersurface $S$ transverse to $K$, a local
		expression of the metric as in \eqref{lorentz} holds on some
		neighborhood  $U=(a,b)\times U_0\subset L$, $(a,b)\subset \R$, $U_0\subset S$, with
		$K$ identifiable to $\partial_t$. However, the global assumption
		on $S$ plus the completeness of $K$ ensure that the local
		expression can be obtained globally just by moving $S$ (which
		would be identified to the slice $\{0\}\times M$) with the flow of
		$K$.
		
		For the last assertion,  observe that $\nabla t$ is timelike
		because it is orthogonal to the spacelike slices $\{t_0\}\times M$
		and it does not vanish, as $g(\nabla t,\partial_t)=1$. Finally,
		it is not a restriction that $-\nabla t$ time-orientates the
		spacetime as, otherwise, the change of $K$ by $-K$ (or $t$ by
		$-t$) yields the expression \eqref{lorentz} with $\omega$ changed
		by $-\omega$.
	\end{proof}
	
	\begin{rem}\label{increasing} (1)
		A {\em temporal
			function} on a spacetime is a smooth function $t$ with
		past-pointing timelike gradient $\nabla t$, so that $t$  is in particular a
		{\em time function}, i.e. a continuous function that increases on any
		future-pointing causal curve  (see, e.g. \cite{BeEhEa96} and \cite{MinSan08}). 
		The existence of the latter for a
		spacetime can be chosen as a definition of the step {\em stable
			causality} in the so-called  causal ladder or  hierarchy  of
		spacetimes \cite{MinSan08} (in particular, these spacetimes are {\em strongly causal}, i.e. all the causal curves that leave a fixed  neighbourhood of a point cannot return arbitrarily near the same point). 
		
		The previous proposition shows that the constructed function $t$
		is a temporal one and, so,  \sstk  spacetimes are always stably
		causal. In Theorem~\ref{kropinaLadder} we will see that, whenever
		$\Lambda\geq 0$, they are also {\em causally continuous} (the
		subsequent  step in the causal ladder which holds intuitively when,
		additionally, the chronological future $I^+(p)$ and past $I^-(p)$
		of any point $p\in  L$ vary continuously with $p$).
		
		(2) The previous characterization of \sstk spacetimes can be
		refined for the case of {\em stationary} spacetimes, i.e. those
		spacetimes which admit a {\em timelike} Killing vector field $K$.
		It is known \cite{JaSan08} that such a spacetime is {\em standard
			stationary} (i.e., an \sstk splitting with $\Lambda>0$) iff $K$ is a
		complete vector field and the spacetime is distinguishing (i.e.,
		$p\neq q$ implies  $I^+(p)\neq I^+(q)$ and $I^-(p)\neq I^-(q)$).
		The reader can check that all our approach is widely simplified
		for standard stationary spacetimes and agrees with
		\cite{CapJavSan10}.
	\end{rem}
	
	\begin{convention}  \label{convention4}  (1) 
		Except if otherwise specified, in what follows  we will assume
		that the Killing vector field $K$ and the spacelike hypersurface
		$S$ of an \sstk spacetime are prescribed and, so, an {\em \sstk splitting}
		will mean the product manifold $\R\times M$ endowed with the
		metric $g$ in \eqref{lorentz} and the (future) time-orientation
		provided by $-\nabla t$. 
		When different  splittings obtained by changing  the
		hypersurface $S$ will be  taken into account (as in Subsection~\ref{ss6.2}), we will point it out explicitly. 
		
		(2)  When there is no possibility of confusion,  we will
		write a tangent vector to a point $(t_0,x_0)\in\R\times M$  simply as $(\tau,v)\in \R\times TM$,  since the metric $g$ is independent of the time coordinate $t$. 
	\end{convention}
	\subsection{Associated wind Riemannian structure} Next, our goal is to associate a natural wind Riemannian structure  with any \sstk as in previous convention.
	\begin{prop}\label{ppol}
		Let $\Lambda, \omega$ and $g_0$ be a function, a one-form, and a Riemannian metric on $M$. Then, the set $\Sigma\subset TM$  of solutions of
		\begin{equation}-\Lambda+2\omega(v)+g_0(v,v)=0\label{convex}\end{equation}
		constitutes a wind Riemannian structure if and only if the inequality \eqref{lorentzian} holds. In this case, putting $\Omega= \left(\Lambda + \| \omega \|^2_0 \right)^{-1}$, this wind Riemannian structure is the displacement of the 
		indicatrix of the Riemannian metric $g_R=\Omega g_0$ along the vector field $W$ which is $g_0$-metrically equivalent to the one-form $-\omega$.
	\end{prop}
	\begin{proof}
		Notice first that, as $g_0$ is Riemannian, at each $p\in M$, $\Sigma\cap T_pM$ must be
		either the empty set or a point or an ellipsoid, and the last possibility holds if and only if \eqref{lorentzian} holds at $p$.
		In this case, the transversality of $\Sigma$ is automatically satisfied. To check this plus the last assertion, multiply \eqref{convex} by $\Omega$  and observe that $\Sigma$ is  the displacement of the indicatrix of $g_R$  by $W$ 
		(in fact,  $\Omega \Lambda=1-g_R(W,W)$ so that \eqref{convex} becomes equivalent to $g_R(v-W,v-W)=1$); in particular,  Proposition~\ref{ptf} applies.
	\end{proof}
	\begin{lemma}\label{prepwf}
		If a tangent vector $(a,v)\in \R\times TM$ is lightlike, then
		$a\neq 0$ and, in this case, it is future-pointing iff $a>0$.
	\end{lemma}
	\begin{proof}
		Straightforward from the fact that the slices $t=\mathrm{const.}$ are
		spacelike and $t$ is a temporal function.
	\end{proof}
	\begin{prop}\label{pwf}
		The set $\Sigma$ of all the vectors $v \in TM$ such that $(1,v)$
		is a future-pointing lightlike vector in  $\R\times TM$ becomes
		a wind  Riemannian  structure on $M$. Moreover,  the reverse wind Riemannian structure  $\tilde \Sigma =-\Sigma$
		contains all  the vectors $w\in TM$ such that $(-1,w)$ is  a
		past-pointing lightlike vector in $\R\times TM$.
	\end{prop}
	\begin{proof} From the expression of the metric, $(1,v)$ is a lightlike vector iff $v$ satisfies \eqref{convex}.
		Proposition~\ref{ppol}  and
		the consistency of the causal characters in Lemma~\ref{prepwf} yield
		the first assertion; the last  one  follows from Definition~\ref{dreversewind} and the fact that $(-1,-v)$ is lightlike iff so
		is $(1,v)$.
	\end{proof}
	Notice that lightlike vectors are preserved by  all the metrics
	pointwise conformal to $g$.
	When a conformal factor $\Omega >0$ is invariant by the flow of $K=\partial_t$, it induces naturally a function also denoted by $\Omega$ which multiplies the three elements $\Lambda,  \omega,  g_0$.
	\begin{defi}\label{dfermatstructure}
		The {\em Fermat structure} associated with (the conformal class of)
		an \sstk splitting is
		the wind Riemannian structure $\Sigma$
		obtained in Proposition~\ref{pwf}.
	\end{defi}
	\begin{thm}\label{tfermatSSTK}The following statements hold:
		\begin{itemize}
			\item[(i)]
			Any wind Riemannian structure $\Sigma$
			is the
			Fermat structure associated with the  conformal class of an \sstk
			spacetime with a representative $(g_0,\omega, \Lambda)$,
			$\Lambda=1-g_0(W,W)$, $W$ the vector field $g_0$-metrically equivalent to $-\omega$.
			
			Conversely, given the Fermat structure $\Sigma$ associated with the
			conformal class of an \sstk splitting,  there exists a unique
			representative $(g_0,\omega,\Lambda)$ of the class such that the
			vector field $W$, $g_0$-associated with $-\omega$ satisfies
			$\Lambda=1-g_0(W,W)$ and $\Sigma$ is the wind Riemannian  structure
			defined by $g_0$ and the displacement $W$.
			\item[(ii)]
			Two Fermat
			structures $\Sigma, \Sigma'$ associated with two \sstk splittings
			determined by  the  data $(g_0,\omega,\Lambda)$,
			$(g'_0,\omega',\Lambda')$ on $M$ are equal if and only if the two
			spacetimes are pointwise conformal, i.e., there exists some
			function $\Omega>0$  on $M$  such that $(g'_0,\omega',\Lambda')=(\Omega
			g_0,\Omega\omega,\Omega\Lambda)$.
		\end{itemize}
	\end{thm}
	\begin{proof}
		$(i)$  It is an immediate consequence of Proposition~\ref{ppol}.
		
		$(ii)$ Recall first that two pointwise conformal \sstk splittings as in \eqref{lorentz} must differ in a conformal factor invariant by the flow of $\partial_t$ and, so,  they will induce a positive function $\Omega$ on $M$. 
		So, use simply that two spacetimes are pointwise conformal iff they have the same lightlike vectors with the same time-orientations.
	\end{proof}
	
	The regions of strong and weak wind can be easily determined
	(see Fig.~\ref{dis1}).
	\begin{figure}[h]
		\includegraphics[scale=0.7, center]{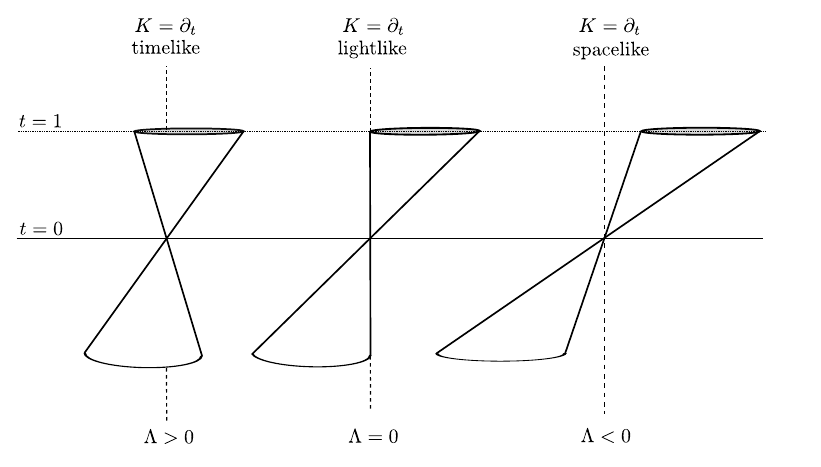}
		\caption{The time cone in an \sstk  splitting}\label{dis1}
	\end{figure}
	\begin{prop}
		Let $(M,\Sigma)$ be a  Fermat structure, $p\in M$ and $0_p\in T_pM$
		the zero vector:
		
		(i) $0_p\in B_p$ iff $K_p$ is timelike ($\Lambda(p)>0$). In this case, $\Sigma_p$
		determines a Randers norm.
		
		(ii) $0_p\in\Sigma_p$ iff $K_p$ is lightlike ($\Lambda(p)=0$). In this case, $\Sigma_p$
		determines a Kropina norm.
		
		(iii) $0_p\not\in \bar B_p$ iff $K_p$ is spacelike ($\Lambda(p) < 0$). In this case,
		$\Sigma_p$ defines a strong  wind     Minkowskian structure.
	\end{prop}
	
	\begin{proof}
		It is straightforward from the facts that $0_p$ satisfies
		\eqref{convex} iff $\Lambda(p)=0$ and    the unit ball $B_p$ defined by $\Sigma_p$  is obtained  by  replacing the equality in
		\eqref{convex}   with  the inequality $<$. 
	\end{proof}
	\subsection{Lightlike vectors and link with Zermelo metrics}
	Next, let us describe in a precise way the lightlike vectors of an
	\sstk splitting and write the Finslerian elements of the Fermat
	structure  in terms of $\Lambda, \omega, g_0$.
	\begin{prop}\label{plightvectorsSSTK}
		Given an \sstk splitting determined by $\Lambda, \omega, g_0$ as in
		\eqref{lorentz}, define, for each $x\in M$:
		\[A_{x}=\begin{cases}
			T_{x}M \setminus \{0\} &\text{if $\Lambda(x)>0$},\\
			\{v\in T_{x}M\colon -\omega(v)>0,\
			\Lambda(x)g_0(v,v)+\omega(v)^2>0\}
			&\text{if
				$\Lambda(x)\leq 0$}.
		\end{cases}
		\]
		Let $M_l=\{x\in M: \Lambda (x)< 0\}$ and put
		$$
		A=\bigcup_{x\in M}A_{x}, \qquad \qquad A_l=\bigcup_{x\in
			M_l}A_{x},
		$$
		as well as $A_E$
		(as defined in Definition~\ref{ae}). Define $F$ and $F_l$ as
		\begin{equation}\label{randers-kropina}
			F(v) = \frac{g_0(v,v)}{-\omega(v)+\sqrt{\Lambda
					g_0(v,v)+\omega(v)^2}},\quad\quad \quad \forall v\in  A,
		\end{equation}
		where, when $\Lambda(x)=0$, the previous expression is understood
		as
		\begin{equation}\label{fermat-kropina}
			F(v)= -\frac{g_0(v,v)}{2\omega(v)}, \qquad \qquad \forall v\in
			\{w\in T_{x}M\colon -\omega(w)>0\},
		\end{equation}
		and
		\begin{equation}\label{rk2}
			F_l(v) =(-F^{\hbox{{\tiny frev}}}(v):=)
			-\frac{g_0(v,v)}{\omega(v)+\sqrt{\Lambda
					g_0(v,v)+\omega(v)^2}},\quad \quad \forall v\in  A_l,
		\end{equation}
		and extend them to
		$A\cup  A_E$ as  in Convention~\ref{caestar}.
		
		A tangent vector $(\tau,v)\in \R\times TM$ is a
		future-pointing lightlike vector if and only if  $\tau>0, v\in 
		A\cup A_E  $  
		and one of the three following cases holds:
		\begin{enumerate}
			\item[(i)] When $\Lambda(x)>0$,
			then $(\tau,v)=(F(v),v)$.
			
			\item[(ii)] When $\Lambda(x)=0$, then
			\begin{itemize}
				\item $(\tau,v)=(\tau,0_{x})$, or
				
				\item $v\in A_{x}$ and $(\tau,v)=(F(v),v)$.
				
			\end{itemize}
			
			\item[(iii)] When $\Lambda(x)<0$, (necessarily,
			$A_x\subsetneqq (A_E)_x$), then
			\begin{itemize}
				\item $(\tau,v)=(F(v),v)$,   iff $\tau\Lambda(x)-\omega(v)\geq 0$,
				
				\item $(\tau,v)=(F_l(v),v)$, iff $\tau\Lambda(x)-\omega(v)\leq 0$,
				
				\item $(\tau,v)=(F(v),v)=(F_l(v),v)$ iff
				$\tau\Lambda(x)-\omega(v)=0$.
				
			\end{itemize}
			Moreover, $0<F(v)\leq F_l(v)$ and the equality holds 
			iff $v\in (A_E)_x\setminus A_x$.
		\end{enumerate}
	\end{prop}
	\begin{proof}
		This can be computed directly by imposing that $(\tau,v)$ must be
		lightlike, i.e., $-\Lambda\tau^2+2\omega(v)\tau+g_0(v,v)=0$ and
		thus
		\begin{equation}\label{tau}
			\tau=\frac{\omega(v)\pm\sqrt{\Lambda
					g_0(v,v)+\omega(v)^2}}{\Lambda}  =
			\frac{g_0(v,v)}{-\omega(v)\pm\sqrt{\Lambda g_0(v,v)+\omega(v)^2}},
		\end{equation}
		the first equality whenever  $\Lambda\not=0$ and the last one
		valid even if $\Lambda$ vanishes whenever $v\neq 0$.  So, the
		result follows from a straightforward discussion of
		cases.
	\end{proof}
	
	The last part of this proposition characterizes precisely all the lightlike  vectors of the spacetime. However, it will be useful to know exactly which are all the causal vectors that project on a given tangent vector to $M$ (recall Fig.~\ref{dis2}). A straightforward discussion of cases yields the following possibilities.

	\begin{cor}   [Future-pointing causal vectors looked from $M$]\label{raclarations} Let
		$(t_0,x_0)\in \R\times M$ and  $v\in  T_{x_0}M\setminus\{0\}$.   Then,  the
		following cases  can occur:
		\begin{enumerate}[(a)]
			\item Case $\Lambda(x_0)>0$. The vector  $(F(v),v)$ tangent at $(t_0,x_0)$  (with $F$ computed indistinctly
			either from \eqref{randers-kropina} or from the first expression in \eqref{tau}  with the positive sign) is a future-pointing lightlike vector; moreover, 
			all future-pointing lightlike vectors in $\R\times T_{x_0}M$  can be written in
			this way. 
			The vector $(\tau, v)$ tangent at $(t_0,x_0)$ is  future-pointing timelike  iff $F(v)<\tau$; moreover, 
			all future-pointing timelike vectors in $\R\times T_{x_0}M$ can be written either in this way or as $(\tau, 0)$ with $\tau>0$.
			\item Case $\Lambda(x_0)=0$. When $v\in A_{x_0}$, the
			vector  $(F(v),v)$ (with $F$ computed from
			\eqref{fermat-kropina}) is lightlike and future-pointing; moreover, all  future-pointing lightlike vectors  can be written either in this way or as $(\tau, 0)$ with $\tau>0$. 
			The vector $(\tau, v)$ is a future-pointing timelike vector iff $F(v)<\tau$; moreover, all  future-pointing timelike vectors can be written in this way.
			\item Case $\Lambda(x_0)<0$. One of the following
			exclusive alternatives occurs:
			\begin{enumerate}[(a1)]
				\item[(c1)] $v\in A_{x_0}$. Then, there are exactly two
				future-pointing lightlike vectors $(F(v),v)$, $(F_l(v),v),
				F(v)<F_l(v)$, in $\R\times T_{x_0}M$ (computed from
				\eqref{randers-kropina} and \eqref{rk2}) that project onto $v$. The tangent vector $(\tau, v)$ is a future-pointing timelike vector iff $F(v)<\tau<F_l(v)$; moreover, all the future-pointing timelike vectors in the case (c) can be written in this way.
				\item[(c2)] $v$ belongs to  $(A_E)_{x_0}\setminus A_{x_0}$. Then,
				there is exactly one future-pointing lightlike vector in
				$\R\times T_{x_0}M$  and no timelike vector  that projects onto $v$.  The first component of this  lightlike
				vector can be computed by using formally any of the two
				expressions \eqref{randers-kropina} and \eqref{rk2}, as they agree
				when computed on such a $v$ (recall also that, as in the previous sub-case,
				$-\omega(v)>0$ necessarily).
				\item[(c3)] $v$ does not belong to $(A_E)_{x_0}$.
				Then, no future-pointing lightlike  nor timelike vector  in $\R\times M$
				projects onto $v$.
			\end{enumerate}
		\end{enumerate} 
	\end{cor}

	\begin{rem}
		In the standard stationary case, $A=TM\setminus\mathbf{0}, A_l=\emptyset$ and $F$ can
		be safely computed from any of the expressions in \eqref{tau} just
		by choosing the sign $+$. So, $F$ becomes a  classical Finsler metric,
		the {\em Fermat metric} of the standard stationary spacetime, and
		the corresponding results can be checked in \cite{CapJavSan10}.
	\end{rem}
	\begin{prop}\label{fermatzermelo}
		Let $\Sigma$ be  the Fermat structure associated with an
		\sstk splitting. Then  the conic Finsler metric $F$ and the Lorentzian
		Finsler metric $F_l$ associated with $\Sigma$ are those determined
		in Proposition~\ref{plightvectorsSSTK}.
	\end{prop}
	\begin{proof}
		Taking into account the definition of the Fermat structure, if $(\tau,v)\in$ $\R\times TM$ is a future-pointing lightlike vector, then $\tau>0$ and $v/\tau\in \Sigma_x$.  So, it is enough to use Proposition~\ref{plightvectorsSSTK} with $\tau=1$ (notice that
		the expressions for the conic pseudo-Finsler metrics $F$ and $F_l$ in \eqref{randers-kropina} and \eqref{rk2} are invariant under the conformal change $(g_0,\omega,\Lambda)\mapsto (\Omega g_0, \Omega\omega,\Omega\Lambda)$, $\Omega\colon M\to (0,+\infty)$).
	\end{proof}
	The equivalences in Theorem~\ref{tfermatSSTK} plus  Propositions~\ref{Riemannclass}, \ref{ppol} and \ref{fermatzermelo} 
	extend the well-known ones existing between Randers, Zermelo and stationary metrics, \cite[Proposition 3.1]{BiJa11},
	and they are summarized in Fig.~\ref{rela}.
	\begin{figure}[t]
		\includegraphics[scale=0.61, center]{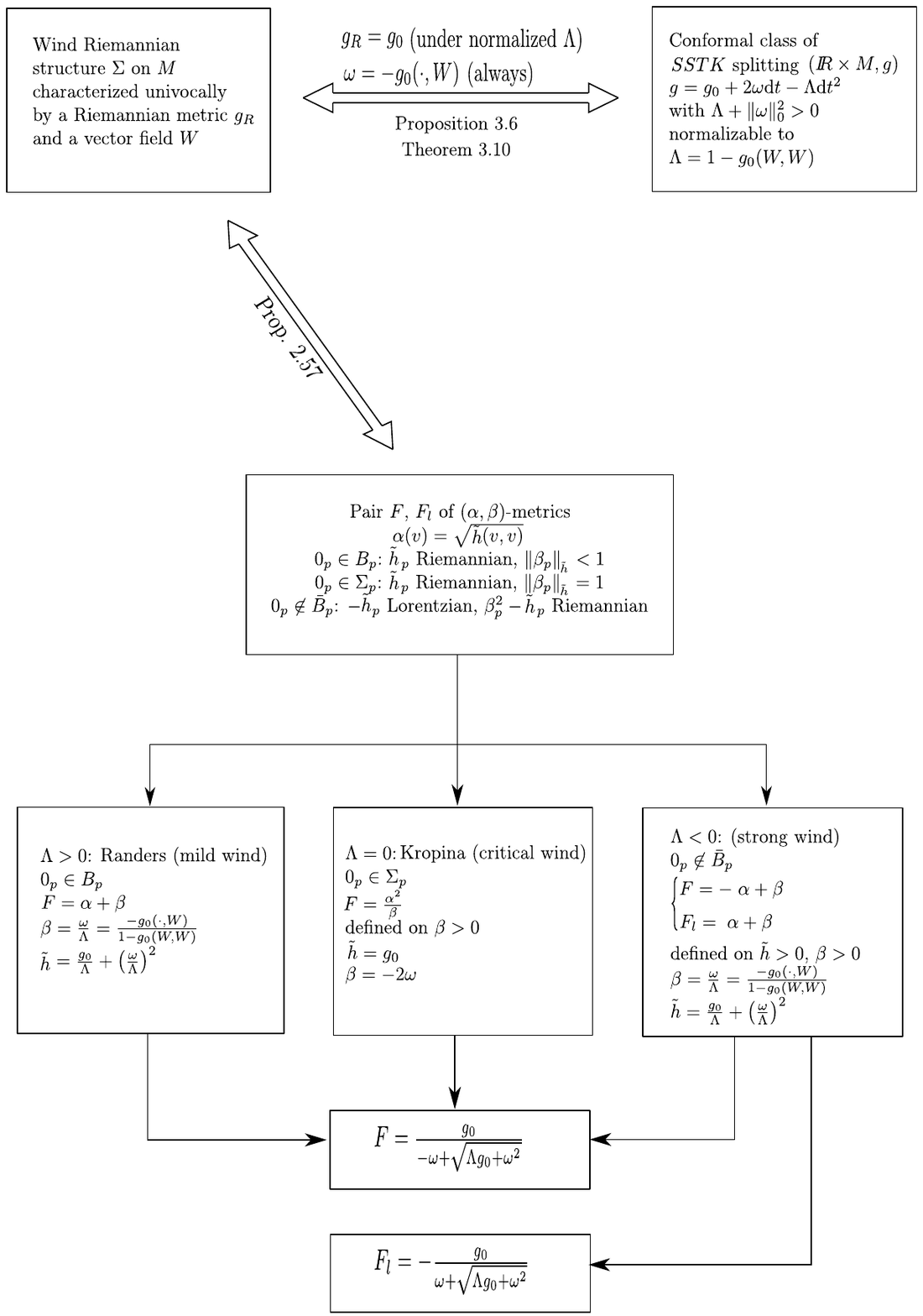}
		\caption{The equivalence between \sstk  splittings, wind Riemannian and Zermelo structures}
		\label{rela}\end{figure}
	\begin{figure}[t]
		\includegraphics[scale=1,center]{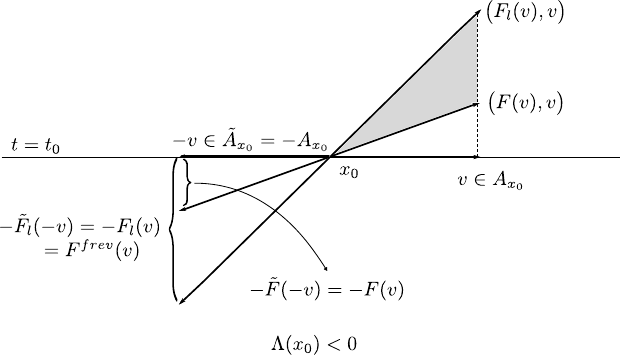}
		\caption{The shaded region represents, in the case where $\Lambda(x_0)<0$, all the future-pointing timelike vectors $(\tau, v)$, for $v\in A_{x_0}$. The lightlike vectors $(F(v),v)$ and $(F_l(v), v)$ yield the boundary of such a region.}
		\label{dis2}
	\end{figure}

	\bigskip
	
	\noindent  The case of past-pointing causal vectors and its relation with the reverse wind Finslerian structure (see Fig.~\ref{dis2}) can be summarized as follows. 
	Recall that in Theorem~\ref{tfermatSSTK}, one assumes implicitly  that $-\nabla t$ is future-pointing (Convention~\ref{convention4}). If we consider an \sstk splitting determined by the triple $(g_0,\omega,\Lambda)$  
	in \eqref{lorentz}, and construct a new spacetime just reversing the time-orientation, the transformation $t\mapsto -t$ would allow
	one  to express this second spacetime as an \sstk with data  $(g_0,-\omega,\Lambda)$. Clearly, the Fermat structure of the latter will be the reverse $\tilde \Sigma$ of the original one.
	Then, the possibilities analogous to  Corollary~\ref{raclarations}  for  lightlike vectors can be summarized as follows.

	\begin{cor}[Past lightlike vectors and time reversal]\label{rtimereversal} Let
		$(t_0,x_0)\in \R\times M$ and  $v\in  T_{x_0}M\setminus\{0\}$.   Then,  the
		following cases  can occur:
		\begin{enumerate}[(a)]
			\item Case $\Lambda(x_0)>0$. The  tangent at $(t_0,x_0)$ vector  $(-F(-v),v)$ (recall $\tilde F(v)=F(-v)$)   is a past-pointing lightlike vector; moreover, 
			all past-pointing lightlike vectors in $\R\times T_{x_0}M$  can be written in
			this way. 
			\item Case $\Lambda(x_0)=0$. When $v\in -A_{x_0}$, the
			vector  $(-F(-v),v)$,
			is lightlike and past-pointing; moreover, all past-pointing lighlike vectors   can be written either in
			this way or as $(-\tau, v)$ with $\tau>0$.
			\item Case $\Lambda(x_0)<0$. One of the following exclusive
			alternatives occurs:
			\begin{enumerate}[(a1)]
				\item[(c1)] $v\in -A_{x_0}$. Then, there are exactly two
				past-pointing lightlike vectors $(-F(-v),v)$, $(-F_l(-v),v),
				-F_l(-v)<-F(-v)$, in $\R\times T_{x_0}M$  (recall that $F(-v)=\tilde F(v)$ and, formally,
				$F_l(-v)=\tilde F_l(v)=-F^{\hbox{\tiny frev}}(-v) =-F(v)$)  that project onto $v$.
				\item[(c2)] $v$ belongs to $-((A_E)_{x_0}\setminus A_{x_0})$. Then,
				there is exactly one past-pointing lightlike vector in
				$\R\times T_{x_0}M$, namely,  $(-F(-v),v)$ ($F(-v) = F_l(-v)$)  that projects onto $v$.
				\item[(c3)] $v$ does not belong to $-(A_E)_{x_0}$.
				Then,  no past-pointing lightlike vector in $\R\times M$
				projects onto $v$.
			\end{enumerate}
		\end{enumerate}
	\end{cor} 
	
	\begin{proof} 
		A vector $(-\tau,v)$ is past-pointing and lightlike if and only if $\tau>0$ and $(\tau,-v)$ is 
		future-pointing and lightlike. So, one should apply Corollary~\ref{raclarations} (or the last part of   Proposition~\ref{plightvectorsSSTK})  replacing $v$ with  $-v$. This change of sign transforms the assertions on $F$ and $F_l$ in 
		assertions on their reverse metrics $\tilde F$ and $\tilde F_l$ defined on $\tilde A=-A$ and $\tilde A_l=-A_l$ (and extendible to
		$\tilde A_E=-A_E$) as asserted in Proposition~\ref{reversewind}; moreover, notice that 
		the
		metric $F$ determines  $\tilde F$, 
		which can be used to give expressions only in terms of $F$ (instead of
		the quadruple $F, F_l, \tilde F, \tilde F_l$, see Proposition~\ref{pformalreverse} and Remark~\ref{rformalreverse}). 
	\end{proof}
	
	Finally, using the SSTK viewpoint, we will characterize the vectors in the indicatrix of $\Sigma$ that correspond to abnormal geodesics. Recall that, at each point of strong wind, the indicatrices $S_F^{m-1}$ and $S_{F_l}^{m-1}$ of   $F$
	and $F_l$ cover all $\Sigma$ but the abnormal ones
	(see Fig.~\ref{dis4}).  
	
	\begin{prop}[The common boundary of the indicatrices of $F$ and $F_l$ on $M_l$] Let $\Sigma$ be a wind Riemannian structure with associated triple $(g_0,\omega,\Lambda)$ and $x_0\in M$ such that $\Lambda(x_0)<0$. The intersection $S_0^{m-2}$ between the indicatrix $\Sigma_{x_0}$ and the boundary of the conic domain $(A_l)_{x_0}$ is  characterized by the equations: 
		\begin{align*}
			g_0(v,v)&=-\Lambda(x_0) , & \omega(v)&=\Lambda(x_0), & v\in T_{x_0}M,
		\end{align*}
		which define a  $(m-2)$-dimensional sphere obtained as the transversal  intersection of  a round
		$g_0$-sphere and  a hyperplane.
	\end{prop} 
	\begin{proof}
		The boundary of
		$(A_l)_{x_0}$ in $TM\setminus\mathbf{0}$ is given by the vectors
		$v\in T_{x_0}M$ such that the expressions  for $F$ and $F_l$ agree (i.e. the square root in \eqref{randers-kropina} and \eqref{rk2} vanishes)  
		and $-\omega(v)>0$. Moreover, a tangent vector $v$ belongs to the indicatrix iff the vector $(1,v)\in
		\R\times T_{x_0}M$ in the associated SSTK splitting  is lightlike (Fig.~\ref{dis4}). These two conditions yield: 
		\begin{align*}
			\Lambda(x_0)  g_0(v,v)+\omega(v)^2&=0,&
			g_0(v,v)+2\omega(v)-\Lambda(x_0) &=0
		\end{align*}
		which are equivalent to the required equations.  Transversality holds because of the Lorentzian restriction \eqref{lorentzian}.
	\end{proof}
	
	\begin{figure}[h]
		\includegraphics[scale=1,center]{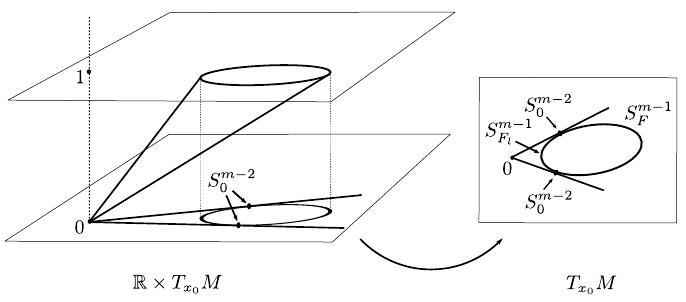}
		\caption{The intersection between the future-pointing light cone
			at $(t_0,x_0)$ and the slice $\{1\}\times T_{x_0}M$ projects onto
			the wind Finslerian structure $\Sigma_{x_0}=S_F^{m-1}\cup
			S_0^{m-2}\cup S_{F_l}^{m-1}$, ($m=2$)}\label{dis4}
	\end{figure}
	\subsection{Projection on  $\partial_t^\perp$ and interpretation of  $A_l$,
		$\tilde{A_l}$}\label{sh} The fact that the radicand in the
	expressions of $F$ and $F_l$ (formulas \eqref{randers-kropina} and
	\eqref{rk2}) may not be automatically positive  has  been
	interpreted above (possibilities $(c2)$  and $(c3)$  in  Corollaries 
	\ref{raclarations} and \ref{rtimereversal}). Let us go a step
	further by analysing the metric tensor in this radicand, that is,
	\begin{equation} \label{eh} h:=\Lambda
		g_0+\omega\otimes\omega,
	\end{equation}
	on $M$.
	The metric $h$ admits the following interpretation on the open
	subset  $M_{\Lambda\neq 0}=\{x\in M: \Lambda\neq 0\}$ where $\partial_t$ is not lightlike.
	\begin{prop}\label{gh} Let
		$p_\R^\perp : \R\times TM_{\Lambda\neq 0} \rightarrow \R\times
		TM_{\Lambda\neq 0}$ the natural projection on the bundle
		$\partial^\perp_t$, $g$-orthogonal  to $\partial_t$.
		Then, for any $v\in T_xM, x\in M_{\Lambda\neq 0}$:
		$$h(v,v)=\Lambda g(p_\R^\perp (0,v),p_\R^\perp (0,v)).$$
		Moreover, $h(v,v)/\Lambda^2= (F+F^{\hbox{\tiny frev}})^2(v)/4$ and, thus, $h/\Lambda^2$ is conformally invariant. 
	\end{prop}
	\begin{proof}  For the first assertion, apply $g$ to  $p_\R^\perp
		(0,v)=(\omega(v)/\Lambda,v)$. For the last one, use \eqref{randers-kropina} and  \eqref{rk2} (or \eqref{tau}) and recall that Fermat metrics are conformally invariant, Theorem~\ref{tfermatSSTK}. 
	\end{proof}
	That is, the metric $h$ on $M$ can be identified with a metric conformal
	to the original one $g$ restricted to
	$\partial_t^\perp$, being the conformal factor $\Lambda$ such that
	$h$ makes sense even in the limit case when $\partial_t$ is
	lightlike.  Recall that $\Lambda$ and $\omega$ satisfy \eqref{lorentzian},  then   we immediately get: 
	\begin{cor}\label{hcases} The metric
		$h$ on $M$ is Riemannian when $\Lambda>0$, degenerate when
		$\Lambda=0$ and it has coindex 1 (i.e., $-h$ is Lorentzian) when
		$\Lambda<0$.
	\end{cor}
	In the region $M_l$ where $\Lambda <0$, the expressions of $F$ and
	$F_l$ (resp. $\tilde F$ and $\tilde F_l$)  have been  well-defined on
	all  $A_E$ (resp. $\tilde A_E=- A_E$)  in Convention~\ref{caestar}.
	Then, one has directly the following  characterizations of the lightlike
	vectors for $-h$ on $M_l$. 
	
	\begin{prop}\label{htimeoriented} 
		For any $v \in
		TM_l\setminus\mathbf{0}$,  the following properties are equivalent: 
		
		(i) $h(v,v)=0$, i.e., $v$ is lightlike for the Lorentzian metric $-h$.
		
		(ii) $v$ belongs to the boundary of $A_E\cup \tilde A_E$ in $TM_l\setminus\mathbf{0}$,
		
		(iii) either $v\in A_E$ and $F(v)=F_l(v)$ or $v\in \tilde A_E$ and
		$\tilde F(v)=\tilde F_l(v)$, and
		
		(iv) either $v\in A_E$ and $F(v)=-\frac{g_0(v,v)}{\omega (v)}$ or
		$v\in \tilde A_E$ and $\tilde F(v)=\frac{g_0(v,v)}{\omega (v)}$.
		
		Consistently, the Lorentzian metric $-h$ is time-oriented so that a lightlike vector
		$v$ for $-h$ will be defined as future-pointing 
		if it
		belongs to the boundary of $A_E$. 
	\end{prop}
	Therefore, $A_l$ (resp. $\tilde A_l$) can be interpreted as the set of
	all the future-pointing (resp. past-pointing) timelike vectors for
	$-h$; analogously, $A_E$ (resp. $\tilde A_E$) is the set of all
	the future-pointing (resp. past-pointing) causal vectors.
	

	The following property of lightlike geodesics of $(\R\times M,g)$ in terms
	of $h$ will be useful later. 
	Notice first that 
	if $\gamma=(\zeta,\sigma)$ is a lightlike curve in $(\R\times M, g)$ then 
	\begin{equation}\label{eextra}
		g(\dot\gamma,\partial_t)=-\Lambda \dot \zeta +\omega(\dot\sigma)=\left\{
		\begin{array}{lll}
			\omega(\dot \sigma) & & \hbox{if} \, \Lambda=0 \\
			\mp \sqrt{\Lambda g_0(\dot \sigma,\dot \sigma)+\omega(\dot \sigma)^2} & &
			\hbox{if} \, \Lambda\neq 0,
		\end{array}
		\right.
	\end{equation}
	(take into account \eqref{tau}). In particular, the lightlike vectors for $-h$ in the region $M_l$ can be  interpreted as the projection of lightlike vectors  of the spacetime orthogonal to $\partial_t$, in addition to the characterizations in  Proposition ~\ref{htimeoriented}. 
	Now, recall that if $\gamma$ is a geodesic
	of $(\R\times M,g)$, then $g(\dot\gamma,\partial_t)$ is constant along
	$\gamma$ (as $\partial_t$ is Killing).  
	
	\begin{lemma}\label{lightgeo}
		For any lightlike geodesic $\gamma=(\zeta,\sigma)$ of $(\R\times M,g)$,
		the constant $C=g(\dot\gamma,\partial_t)$ satisfies $C^2=h(\dot
		\sigma,\dot \sigma)$. Moreover, if $C=0$ either (i) $\sigma$ is constant, $\sigma\equiv x_0\in M$, and the integral curve of $K$ that projects onto $x_0$ is a lightlike pregeodesic (that is, $d\Lambda$ vanishes on the kernel of $\omega_{x_0}$)
		or (ii)  $\sigma$ remains in the closure $\bar M_l$ of $M_l$ (in particular, if $M_l=\emptyset$ this case cannot hold), $\sigma$ can reach the boundary $\partial M_l$ only at isolated points  
		(where the $g_0$-acceleration $D^{g_0}\dot \sigma/ ds$ does not vanish) and,  whenever $\sigma$ remains in $M_l$ ($-h$ is Lorentzian on  $\sigma$), $\sigma$ is a lightlike geodesic of
		$h/\Lambda$. 
	\end{lemma}
	\begin{proof}  For the first assertion,  recall
		that  the expression \eqref{eextra} is equal to $\sqrt{h(\dot\sigma,\dot
			\sigma)}$ up to a sign, i.e. $h(\dot \sigma,\dot \sigma)=C^2$.  
		
		For the last
		assertion, if $\sigma$ is  constantly equal to $x_0$,  then $\gamma$ can be a geodesic
		if and  only if the corresponding orbit of $\partial_t$ can be reparametrized as a geodesic. This happens if and only if the  gradient $\nabla^g\Lambda$ of $\Lambda$ with respect to the metric $g$, 
		when projected on $TM$ by using the differential of $\pi: \R\times M \rightarrow M$, is  $0$.  In fact,  since for  a  Killing vector field   $K$, $\nabla^g (g(K,K))=-2\nabla_K K$, an   orbit $\gamma(t)=(t,x_0)$ 
		of $\partial_t$ is a pregeodesic if and only if there exists a function $\lambda$ such that $\nabla^g \big(g(\partial_t,\partial_t)\big )|_{(t,x_0)}= \lambda(t) \partial_t|_{(t,x_0)}$, for all $t\in \R$. 
		Recalling that $ -  g(\partial_t,\partial_t)(t,x_0)=\Lambda(x_0)$, for all $t\in\R$, and that $\omega_{x_0}$ is the one-form on $M$ $g_0$-equivalent to the $g$-orthogonal projection of $\partial_t$ on $TM$, this equation
		is equivalent to the condition that $\de\Lambda(x_0)$ is proportional to $\omega_{x_0}$, as required.

		Finally,
		observe that   a lightlike geodesic $\gamma=(\zeta,\sigma)$ with
		$\sigma$ non-constant can have $C=0$ 
		only in $\bar M_l$ and, if there exists some $s_0$ such that $\sigma(s_0):=x_0\in \partial M_l$, then: (i) $\dot \sigma(s_0)=0$ (otherwise, $\gamma$ cannot be both, lightlike and orthogonal to $\partial_t$), (ii) 
		$d\Lambda_{x_0}$ cannot be proportional to $\omega_{x_0}$ (otherwise, by uniqueness of geodesics $\gamma$ would be \bw a reparametrization of \ew an integral curve of $\partial_t$), and (iii) $\left(D^{g_0}\dot \sigma/ds)\right)(s_0)\neq 0$.
		To see this, observe that $\gamma$ is a geodesic of $(\R\times M,g)$ if and only if
		\begin{align*}
			&\frac{D^{g_0}}{ds}\dot \sigma= -\ddot \zeta\,\omega^\sharp+\dot \zeta\,\hat\Omega(\dot \sigma)-\frac 12 \dot \zeta^2\,\nabla^{g_0} \Lambda,\\
			&\omega(\dot \sigma)-\Lambda\, \dot \zeta= {\rm const.}\nonumber 
		\end{align*}
		where $D^{g_0}/ds$ and $\nabla^{g_0}$ are respectively the covariant derivative of $(M,g_0)$ along $\sigma$ and the gradient with respect to $g_0$, $\omega^\sharp$ is the vector field $g_0$-equivalent to 
		$\omega$ and $g_0(w,\hat\Omega(v))={\rm d}\omega(v,w)$ for every $v,w\in TM$. The above equations can be obtained for example using that the geodesics are the critical points of the energy functional  (or using the explicit formulas  (13) and (14) in \cite{ijgmmp}, which are valid for arbitrary $\Lambda$).  
		Then $\left(\frac{D^{g_0}}{ds}\dot \sigma\right)(s_0)\not=0$ because otherwise $\ddot \zeta(s_0)\,\omega^\sharp_{x_0}=-\frac 12 \dot \zeta^2(s_0)\,\nabla^{g_0} \Lambda(x_0)$.  As $\dot \zeta^2(s_0)\neq 0$, because $\gamma$ is lightlike, we conclude that  $\sigma$ is constant by  $(i)$.  
		Now, in the region $M_l$,  
		the map
		$\pi:(\R\times M_l,g)\rightarrow (M_l,h/\Lambda)$ is a
		semi-Riemannian submersion (see Proposition~\ref{gh}) and therefore,   lightlike  geodesics orthogonal  to
		the fibers project into  (and  are  all the lifts  of)    lightlike  geodesics of $(M_l,h/\Lambda)$.  
	\end{proof}
	Stationary spacetimes have been studied      in many mathematically oriented     papers, see e.g. \cite{FoGiMa95, MasPic98, JaSan08, ijgmmp, CaJaMa, GHWW, CapJavSan10, CapGerSan12}, while the case $\Lambda \equiv 0$, which includes global Brinkmann decompositions \cite{BlSaSe13}, has been considered recently in \cite{BaCaFl}, where the authors study geodesic connectedness.
	
	\subsection{Fundamental tensors for $F, F_l$}
	Recall that when $\Lambda>0$, $F$ is a Randers metric with a
	well-known positive definite fundamental tensor (see for example
	\cite[Corollary 4.17]{JavSan11}) and, when $\Lambda=0$, then
	$F$ is a  Kropina metric  with also a
	positive-definite fundamental tensor (see \cite[Corollary
	4.12]{JavSan11}). Next, we focus on the region $\Lambda<0$ and the
	domain $A_l$ of $TM$.
	
	The
	fundamental  tensors  of $F$ and $F_l$ can be
	computed explicitly from the expressions  \eqref{randers-kropina} and \eqref{rk2} by
	taking into account that these metrics can be regarded as
	canonical $(F_0,\omega)$-ones, i.e. they can be written as
	$F_0\cdot \phi(\omega/F_0)$ being $F_0$ the root of $g_0$ and
	$$\phi(z)=\frac{1}{-z \pm\sqrt{z^2+\Lambda}},$$ see
	\cite[\S 4.2.1]{JavSan11}. In any case, a simplified
	computation can be accomplished.
	In fact, it is enough  to study the fundamental
	tensor of the Randers type metrics:
	$$
	F^\epsilon =   -\epsilon \sqrt{\tilde h} + \beta \quad \quad
	\hbox{with} \; \tilde h = \beta \otimes \beta -\bar g_0 ,
	$$
	where $\beta$ is the one-form  $\omega/\Lambda$,
	and $\bar g_0$ the  Riemannian metric $g_0/|\Lambda|$,   
	so that $\tilde h$
	is Lorentzian with coindex $1$ and $\beta(v)>0$ on $A_l$.  The
	value of $\epsilon$ is $1$ for $F$ and $-1$ for $F_l$.
	\begin{prop} 
		Let  $G^\epsilon, \epsilon=\pm 1$,  be the
		fundamental tensor of $F^\epsilon$, i.e.,  of $F$ for $\epsilon=1$
		and $F_l$ for $\epsilon=-1$. Then, with the above notation for
		$\tilde h$ and  $\beta$:
		\begin{equation}\label{G}
			G^\epsilon_v(w,w)= -\epsilon F^\epsilon(\tilde{v}) \left(
			\tilde h(w,w) -\tilde h(\tilde{v},w)^2 \right) +\left( -\epsilon \tilde h(\tilde{v},w)+\beta(w)\right)^2
		\end{equation}
		for all $v\in  A_l$ and $w\in T_{x_0}M$, where
		$\tilde{v}=v/\sqrt{\tilde h (v,v)}$ on $A_l$.
	\end{prop}
	
	\begin{proof}
		It is enough to observe that  \cite[Prop. 4.10]{JavSan11}, with $\phi(s)=-\epsilon+s$,  holds
		also in this case.
	\end{proof}
	\begin{rem}
		Observe that \eqref{G} can be used to prove directly that $F$ and
		$F_l$ are, respectively, conic Finsler and Lorentzian Finsler in
		their  domains $A$ and $A_l$.  Focusing on
		$A_l$,  clearly $G^\epsilon_v(v,v) =F^\epsilon(v)^2>0$; 
		therefore, the
		space of the vectors $u\in T_{x_0}M$ which are $G^\epsilon_v$-orthogonal to
		$v$ is transversal to $v$ and has dimension $m-1$.  Moreover,
		setting $v\in A_l$:
		$$
		G^\epsilon_v(v,u)=0 \quad \Leftrightarrow \quad \beta(u)= \epsilon
		\tilde h(\tilde{v},u) \quad \quad \quad \left(\Leftrightarrow \beta(u)=- \frac{ \epsilon\bar g_0(\tilde{v},u)}{1 -\epsilon\beta (\tilde{v})}\right).
		$$
		By using repeatedly this equivalence,  if $u$ is $G_v$-orthogonal
		to $v$, \eqref{G} becomes:
		\[G^\epsilon_v(u,u)= -\epsilon F^\epsilon(\tilde{v})\left(
		\tilde h(u,u) -\tilde h(\tilde{v},u)^2 \right)
		=  \epsilon F^\epsilon(\tilde{v})\bar g_0(u,u),
		\]
		and the result follows as  $F^\epsilon(\tilde v)>0$.
	\end{rem}
	\section{The case of causal $K$: Randers-Kropina metrics}\label{kroran}
	Next,  we focus on the
	case of an \sstk splitting when $K$ is causal (i.e., $\Lambda\geq  0$), so that its Fermat structure becomes a Randers-Kropina
	metric $F$ according to Definition~\ref{dranderskropina}. 
	In particular,  
	$A_l=\emptyset$,  
	and $A_x=T_xM\setminus\{0\}$ iff $\Lambda(x)>0$ while $A_x$ is an open half-space in $T_xM$ iff $\Lambda(x)=0$ (recall  Definition~\ref{ae} and Proposition~\ref{plightvectorsSSTK}).  
	Therefore, $F_l$ and $A_E$ will not be used  and we will consider  $F$-admissibility rather than  notions as wind curves ---the wind balls will be also treated in a way similar to the classical Finslerian one.  
	Our aim is to show that, on the one hand, $F$ can be used to describe the causality of
	the spacetime $(\R\times M,g)$ and, on the other hand, the known properties on causality of spacetimes allow  us to obtain properties of the associated Finslerian separation $d_F:M\times M\rightarrow [0,+\infty]$.
	
	\subsection{Characterization of the chronological relation}\label{Fsepa}  Following \cite[Definitions 3.6, 3.8]{JavSan11},  for any conic pseudo-Finsler metric  $F\colon A\subseteq TM\to [0,+\infty)$  on $M$ and any $x,y\in M$, 
	one says that {\em $x$ $F$-precedes $y$}, written $x\prec y$,
	if there is an $F$-admissible curve from $x$ to $y$.
	Here,  a curve  is said {\em $F$-admissible} consistently with   Definition~\ref{sigmadmissible}, i.e., if
	its velocity lies in the  domain $A$ of $F$, so that $x\prec y$  iff $\Omega^A_{x,y}\neq \emptyset$.
	We recall  that the Finslerian separation $d_F$ (Definition~\ref{from41})  is non-negative  and  satisfies a triangle inequality, namely,  $d(x,z)\leq d(x,y)+d(y,z)$ for all $x,y,z\in M$, but it is non-symmetric  and $d_F(x,x)$ can be positive  \cite[Proposition 3.9]{JavSan11};    
	in the case that a standard Finsler metric is regarded as the conic Finsler metric of  a wind  Finslerian  structure $\Sigma$,
	the balls  $B_F^+(x,r)$ and $B_F^-(x,r)$  introduced in Definition~\ref{from41}  agree
	with the wind balls $B_{\Sigma}^+(x,r)$ and $B_{\Sigma}^-(x,r)$ of
	$\Sigma$ (see  Definition~\ref{sigmaballs}  and notice that $F_l=+\infty$ in this case).    Both of them are open subsets and constitute a basis for the topology  of $M$   in the standard Finsler case and  these properties are  
	generalizable to the wind  Finslerian  case (recall  Proposition~\ref{riemlowerbound}). 
	However,  the closures $\bar B_F^\pm(x,r)$ of these balls cannot be obtained merely replacing the strict equalities by \bw non-strict \ew ones (see  Corollary~\ref{cclballs} below). 
	
	Now, let us focus on the   $F$-separation
	$d_F$ of the conic Finsler metric $F$ associated with an \sstk splitting with $K$ causal.
	The chronological relation can be characterized in a simple way.
	\begin{prop}\label{bolas} For any \sstk splitting with causal $K$:
		$$(t_0,x_0)\ll (t_1,x_1) \quad  \Leftrightarrow \quad d_F(x_0,x_1)<t_1-t_0 ,$$
		for every $x_0,x_1\in M$ and $t_0,t_1\in\R$. Therefore:
		\begin{equation}\begin{array}{l}
				\label{bolas0} I^+(t_0,x_0)= \{(t,y): d_F(x_0,y)<t-t_0\}, \\
				I^-(t_0,x_0)= \{(t,y): d_F(y,x_0)<t_0-t\}.
		\end{array}\end{equation}
		Equivalently, considering $d_F$-forward and backward balls
		\begin{equation*}I^+(t_0,x_0)= \cup_{s> 0}\{t_0+s\}\times B_F^+(x_0,s), \quad
			I^-(t_0,x_0)= \cup_{s> 0}\{t_0-s\}\times B_F^-(x_0,s).
		\end{equation*}
	\end{prop}
	\begin{proof}
		Recall that a vector $(\tau,v)\in\R\times TM$ is timelike and future-pointing if and only if
		$\tau>F(v)$ (see Corollary~\ref{raclarations} cases (a) and (b)).  
		
		If $(t_0,x_0)\ll (t_1,x_1)$, then there exists a future-pointing timelike curve $\gamma=(t,\sigma):[0,1]\rightarrow \R\times M$ from $(t_0,x_0)$ to $(t_1,x_1)$ such that $\dot t>F(\dot x)$  and, perturbing the curve when needed,
		we can assume that $\dot x(s)\not=0$ for every $s\in [0,1]$\footnote{\label{foot4.1}This is necessary as we are assuming here that the speed of an $F$-admissible curve does not vanish (this is somewhat different to our approach 
			in \cite{CapJavSan10}). However, it is easy to check when the dimension $m+1$ of the spacetime is $\geq 3$ that $\sigma$ can be chosen with always non-vanishing speed (for ex., see the proof of \cite[Prop. 3.2]{FlSan08})). 
			For the case $m=1$, this is obvious as piecewise smooth curves can be used here.}  Then by integration, we get
		$t_1-t_0> \ell_F(\sigma)$, i.e.  $d_F(x_0,x_1)<t_1-t_0$.
		
		Conversely, if $d_F(x_0,x_1)<t_1-t_0$, choose an $F$-admissible   curve $\sigma:[0,1]\rightarrow M$ from $x_0$ to $x_1$ such that $d_F(x_0,x_1)\leq \ell_F(\sigma)<t_1-t_0$. Then the curve $(t,\sigma):[0,1]\rightarrow \R\times M$,
		where $t(s)=t_0 + \ell_F(\sigma|_{[0,s]})+\varepsilon s$ and $\varepsilon=t_1-t_0 -\ell_F(\sigma ) $,  is a timelike future-pointing  curve from $(t_0,x_0)$ to $(t_1,x_1)$.
		
		The remainder is then straightforward.
	\end{proof}

	\subsection{Continuity of the Finslerian separation for Randers-Kropina  spaces}
	Let us start with a general result.
	
	\begin{prop} \label{pls} The $F$-separation  associated with any conic
		pseudo-Finsler metric $F$ is upper semi-continuous, i.e., if $
		x_n\rightarrow x$, $ y_n\rightarrow y$, then $$\limsup_n
		d_F(x_n,y_n) \leq d_F(x,y).$$ In particular, if $x\prec y$, then $x_n
		\prec y_n$ for large $n$.
	\end{prop}
	\begin{proof}Assume that $d_F(x,y)<+\infty$ and, by  contradiction,
		$$
		d_F(x_n,y_n) > d_F(x,y) + 3\varepsilon$$ for some $\varepsilon>0$ and  for some subsequences still denoted by $x_n$ and $y_n$. 
		Choose a curve $\gamma$ from $x$ to $y$ with
		$\ell_F(\gamma)<d_F(x,y)+\varepsilon$ and choose $\bar x, \bar y$ on
		$\gamma$ such that $x\in B_F^-(\bar x, \varepsilon)$ and $y\in B_F^+(\bar y, \varepsilon)$.
		By \cite[Prop. 3.9]{JavSan11} these two balls are
		open and so, for large $n$, they contain all $x_n$ and $y_n$. So,
		the curve $\rho$ obtained by concatenating an $F$-admissible curve of
		length smaller than $\varepsilon$ from $x_n$  to $\bar x$ with the piece
		of $\gamma$ from $\bar x$ to $\bar y$ and with another $F$-admissible
		curve of length smaller than $\varepsilon$ from $\bar y$ to  $y_n$
		yields the required contradiction $d_F(x_n,y_n)<d_F(x,y)+3\varepsilon$.
	\end{proof}
	Although lower semi-continuity may not hold even  in the conic Finsler case
	(and even with points $x,y$ at a finite  $F$-separation, see
	\cite[Example 3.18]{JavSan11}), we will check that this semi-continuity does hold in the Randers-Kropina case.
	
	Notice first  the following straightforward consequence
	of Proposition~\ref{bolas}.
	
	\begin{prop}\label{ctf}
		For any \sstk  splitting with causal $K$,  the function  
		$$
		\tau_F:
		M\times M \rightarrow [0,+\infty], \quad \quad \tau_F(x,y):=
		\inf\{t\in\R : \, (0,x)\ll (t,y)\}$$ is equal to the $F$-separation
		function $d_F$.
	\end{prop}
	The function $\tau_F$ will be called the  {\em (future) arrival time function} and its definition on $M$ instead of $\R\times M$ uses implicitly the invariance of the metric with $t$. 
	
	Next, we will prove the  lower semi-continuity  of $d_F$  by using results of spacetimes (which, in particular,
	extend those in \cite{sanche99}).  To this aim, we will use a well-known result on limit curves. The latter are defined as follows (see \cite[Definition 3.28]{BeEhEa96}).
	\begin{defi}\label{limitcurvedef}
		A curve $\gamma$ in a spacetime $(L,g)$ is a limit curve of a sequence of curves $\{\gamma_k\}$, if there exists a subsequence $\{\gamma_m\}$ such that for all $p$ in the image of $\gamma $, any neighborhood of $p$ intersects all but a finite number of the curves in $\{\gamma_m\}$. 
	\end{defi}
	A standard result says that any sequence $\{\gamma_k\}$ of causal, future-pointing, future-inextendible causal curves whose images \soutE{has} \bw have \ew an accumulation point $p$ admits a limit curve through $p$ which is also causal, future-pointing and future-inextendible (see \cite[Proposition 3.31]{BeEhEa96}; the same  holds replacing  ``future'' with ``past'' in the previous statement).  Let us remark that a  limit curve is not necessarily piecewise smooth, but causal continuous (for the definition of a causal continuous curve see 
	\bw the beginning of \S 3.2 in \cite{BeEhEa96}\ew). Moreover, observe that to be causal continuous on an interval $I$ is equivalent to be locally absolutely continuous with future-pointing causal derivative  a. e. in  $I$ (see \cite[Theorem A.1]{CFS08}).
	\begin{thm}\label{tcontdf}
		The $F$-separation $d_F: M\times M \rightarrow [0,+\infty]$
		associated with any Randers-Kropina metric is continuous away from the
		diagonal $D=\{(x,x): x\in M\}\subset M\times M$.
	\end{thm}
	\begin{proof}
		From Proposition~\ref{ctf} and Proposition~\ref{pls}, it is enough to prove the
		lower semi-continuity  of $\tau_F$ for the corresponding \sstk  splitting.
		
		Let $\{x_n\}$ and $\{y_n\}$ be converging sequences,
		$x_n\to x, y_n\to y\neq x$,  and assume by
		contradiction that  there exist subsequences, denoted again by $\{x_n\}$, $\{y_n\}$, such that  $\{\tau_F(x_n,y_n)\}$ converges with
		$$
		T_0:=\lim_n \tau_F(x_n,y_n) < \tau_F(x,y).$$ Choose
		$T_1\in (T_0,\tau_F(x,y))$, and define $q_n:=(T_1,y_n)$, $p_n:=((T_1-T_0)/2,x_n)$.
		Each line $l_{y_n}=\{(s,y_n):s\in\R\} $ is causal and,  since  $\tau_F(x_n,y_n)$ is finite  for large $n$,  necessarily $p_n\ll
		q_n$ for $n$ big enough (recall $l_{y}\neq
		l_x$ and Remark~\ref{rll}).  Thus, we can take a
		sequence of past-pointing timelike curves $\{\gamma_n\}$
		connecting each $q_n$ with $p_n$.  Moreover,  these curves are assumed to be inextendible  to the past, by  prolonging them with the lines $l_{x_n}$.  As $\{q_n\}\rightarrow q_\infty :=(T_1,y)$,  the sequence $\{\gamma_n\}$ admits
		a causal, past-inextendible and past-pointing   limit curve $\gamma$ starting at $q_\infty$.   Necessarily,  $\gamma$ must leave  $l_y$  at
		some point $Q$; otherwise, as $\gamma$ is inextendible, it must
		run all $(t,q_\infty )$ when $t\rightarrow -\infty$ (but this is absurd
		because the points in  $\gamma$  lie in the closure of the set of
		images of all $\gamma_n$, and the piece of such curves with
		$t\circ\gamma_n\leq (T_1-T_0)/2$ lie in $l_{x_n}$,  which accumulate at  $l_x\neq l_y$).  Notice also that if $\gamma$ arrived at the limit $p_\infty:=((T_1-T_0)/2,x)$ of $\{p_n\}$, a contradiction with the definition of $\tau_F(x,y)$
		would be obtained. Now, choose $T_2\in
		(T_1,\tau_F(x,y))$ and
		any point $Q'$ on  $\gamma$  away from  $l_y$; necessarily $Q' < Q < (T_2,y)$
		and, then, $Q'\ll (T_2,y)$ (recall Remark~\ref{rll}). But $Q'$ lies in the closure of the
		images of the set of all $\gamma_n$ and, thus, up to a subsequence, some point $Q'_n$  on each $\gamma_n$
		satisfies $Q'_n \ll (T_2,y)$ for large $n$. Therefore, we can assume
		$$ p_n (\ll Q'_n) \ll (T_2,y),$$
		and choose  a future-inextendible timelike curve $\alpha_n$ from
		$p_n$ to $(T_2,y)$ and equal to $l_y$ beyond this point. Consider the limit curve $\alpha$ of the sequence $\alpha_n$ 
		departing from $p_\infty$. Reasoning as above, $\alpha$ leaves  $l_x$  at some point, and any point  $Q''$ on $\alpha$ away from  $l_x$  satisfies $p=(0,x)\ll Q''$. Choose a point $Q''_n$, in a curve $\alpha_n$, close enough to $Q''$ 
		such that $p\ll Q''_n$. This concludes that
		\[p\ll Q''_n \ll (T_2,y),
		\]
		in contradiction with the definition of $\tau_F(x,y)$.
	\end{proof}
	The necessity of the exception on the diagonal $D$ in the previous
	theorem comes from the following fact.  
	\begin{prop} 
		The $F$-separation $d_F$ associated with  a Randers-Kropina metric is discontinuous at $(x_0,x_0)$ if $d_F(x_0,x_0)>0$.  Moreover, 
		\begin{itemize}
			\item[(i)] the property $d_F(x_0,x_0)>0$ occurs if there exists a neighborhood $U$ of  
			$x_0$ such that no admissible loop contained in $U$ exists, i.e.  $y\not\prec_U y$, for all $y\in U$;
			in particular for any Kropina metric $F=\alpha^2/\beta$ such that  the kernel of $\beta$ is locally integrable,  i.e. $\beta\wedge d\beta=0$; 
			\item[(ii)]for any Kropina norm on a vector space, $d_F(x,x)=\infty$ for all $x\in V$.
		\end{itemize}
	\end{prop}
	\begin{proof} For the first assertion, choosing any $F$-admissible curve $\gamma$ starting at $x$, one constructs
		trivially a sequence  $\{x_n\}$, $x_n\rightarrow x$ of points on $\gamma$ with
		$d_F(x,x_n)\rightarrow 0$.
		For (i),  notice that any $F$-admissible loop starting at $x$ (and leaving necessarily $U$) will have a length greater than some $\varepsilon>0$; to check this,
		notice that
		one can always obtain   a Finsler metric $F_0$ in any compact neighborhood of $x$ such that $F_0$ is smaller than $1$ on the indicatrix of $F$.
		In the Kropina case, the assumption $\beta\wedge d\beta=0$ implies  that  $\beta|_U=\Omega df$ on a small  neighborhood $U$ of $x_0$,  for a positive function $\Omega$ and $f$ with no critical points on $U$ (as any Kropina metric is defined only on the open region of the manifold $M$ where $\beta$  is nowhere vanishing).  So $f$ is strictly increasing on any admissible curve on $U$  (recall that $\beta(v)>0$ for $v\in A$); thus,  $y\not\prec_U y$ for all $y\in U$. 
		Finally in the case of a Kropina norm,  being $\beta$ a constant one-form,  $\beta=df$ on $V$    and therefore $d_F(x,x)=\infty$ for all $x\in V$.
	\end{proof}

	\begin{rem}
		The previous proposition shows that explicit examples of discontinuous $d_F$ can be constructed easily. It also shows that the possible discontinuity on the diagonal would not be removed if $d_F(x,x)$ were redefined as $0$ (namely, regarding
		$\tau_F(x,y) $ as the infimum of the set $\{t\in \R : (0,x) \leq
		(t, y )\}$ and defining $d_F(x,y)$ as  the  new $\tau_F(x,y)$).
		On the other hand, it is trivial to check that for any Randers-Kropina
		metric $F$, if  $K$ is timelike
		at $x\in M$ then $d_F(x,x)=0$ and $d_F$ is continuous at $(x,x)$.
	\end{rem}
	Finally notice the  discontinuity of $d_F$ at the diagonal yields the following subtlety (consistent with Definition~\ref{sigmaballs} and Proposition~\ref{pclosurecballs}). 
	\begin{cor}\label{cclballs}
		The closed forward (resp. backward) $d_F$-balls, defined as the closures of the corresponding open balls, satisfy,  for $r>0$: 
		\[\bar B_F^+(x,r)= \{y\in M:
		d_F(x,y)\leq r\}\cup \{x\}\]
		\[\hbox{(resp.} \; \bar B_F^-(x,r)= \{y\in M: d_F(y,x)\leq r\}\cup
		\{x\}).
		\]
	\end{cor}
	\begin{proof}
		The proof of the first assertion in the previous proposition shows  that $\{x \}$ belongs to the closure of the ball. So, just apply the continuity of $d_F$ outside the diagonal. 
	\end{proof}
	\subsection{Ladder of causality and properties for Randers-Kropina separation}
	Next, we can go further into the causal structure of our class of
	spacetimes. 
	The following  relation between the position of the spacetime
	in the causal ladder and the properties of the Randers-Kropina
	metric appears.
	\begin{thm}\label{kropinaLadder} Consider an \sstk  splitting $(\R\times M,g)$ as
		in \eqref{lorentz} with $K$ causal and associated Randers-Kropina metric $F$ on $M$. Then, $(\R\times M,g)$  is causally
		continuous, and
		\begin{enumerate}[(i)]  \item the following assertions are
			equivalent:
			
			(i1) $(\R\times M,g)$ is causally simple i.e., it is causal  (which means that no closed smooth causal curve exists)  and  the sets $J^+(p)$, $J^-(p)$ are
			closed for all $p\in \R\times M$).
			
			(i2) $(M,F)$ is {\em convex}, in the sense that  for
			every $x,y\in M$, $x\neq y$,  with $d_F(x,y)<+\infty$, there exists a geodesic
			$\gamma$ from $x$ to $y$ such that $\ell_F(\gamma )=d_F(x,y)$.
			
			(i3) $J^+(p)$ is closed for all $p\in \R\times M$.
			
			(i4) $J^-(p)$ is closed for all $p\in \R\times M$.
			
			\item $(\R\times M,g)$ is globally hyperbolic (i.e. it is causal and all the intersections $J^+(p) \cap  J^-(q)$ are compact) if and only if  $\bar B_F^+(x,r_1)\cap \bar B_F^-(y,r_2)$ is compact (or empty) for every $x,y\in
			M$ and
			$r_1,r_2>0$.
			\item The following assertions are
			equivalent:
			
			(iii1) A slice $S_t=\{(t,x): x\in \R\times M\}$ (and, then all the slices) is a
			spacelike Cauchy hypersurface i.e., it is crossed exactly once
			by any
			inextendible timelike curve (and, then, also by any causal one).
			
			(iii2) the closures $\bar{B}_F^+(x,r)$, $\bar{B}_F^-(x,r)$ are compact
			for all $r>0$ and  $x\in M$.
			
			(iii3)   $F$ is forward and backward geodesically complete.
		\end{enumerate}
	\end{thm}
	\begin{proof}
		First observe that, as $t:\R\times M\rightarrow \R$, $(t,x)\mapsto
		t$, is a temporal function, then $(\R\times M, g)$ is stably causal
		and, in particular, distinguishing (for the elements of causality
		to be used here, see \cite{MinSan08} or \cite{BeEhEa96, BerSan08}). So, to
		prove causal continuity, it is enough to show that $(\R\times M,
		g)$ is future and past reflecting (see for example
		\cite[Definition 3.59, Lemma 3.46]{MinSan08} or \cite[Theorem
		3.25, Proposition 3.2]{BeEhEa96}). Let us see that it is past
		reflecting (the other case is analogous),  that is, $I^+(p)\supset
		I^+(q)$ implies $I^-(p)\subset I^-(q)$ for any $p=(t_0, x)$ and
		$q=(t_1,y)$. We can assume $x\neq y$ (otherwise it is obvious), and the inclusion $I^+(p)\supset I^+(q)$ implies that
		$d_F(x,y)\leq t_1-t_0$. This is  a consequence of the
		continuity  of $d_F$  away from $D$  proven in Theorem
		\ref{tcontdf}. In fact,
		consider a sequence $\{q_n=(t_1+\varepsilon_n,y_n)\}$ contained in
		$I^+(q)$ and converging to $q$ so that $\varepsilon_n \searrow 0,
		y_n\rightarrow y$.
		By \eqref{bolas0}, $d_F(x,y_n) <
		(t_1-t_0)+\varepsilon_n$, and, by the continuity of $d_F$, the required
		inequality holds. But $d_F(x,y)\leq t_1-t_0$ implies directly
		$I^-(p)\subset I^-(q)$ (use again \eqref{bolas0}  and the triangle inequality for $d_F$), as required.
		
		Equivalences in   $(i)$ and  $(ii)$ can be proved formally as in the  stationary  case
		\cite{CapJavSan10}.
		The proof of the equivalences in $(iii)$ has some  differences  with respect to the stationary case due to the lack of a Hopf-Rinow theorem for Randers-Kropina metrics.  The reader can check, however, that both, the equivalence 
		between {\em (iii1)} and {\em (iii3)}, and the implications {\em (iii1)} $\Rightarrow$ {\em (iii2)} $\Rightarrow$ {\em (iii3)} hold by means of simple modifications of the  arguments in \cite[Theorems 4.3 and 4.4]{CapJavSan10}. 
		In any case, a full proof can be obtained as a particular case of the most general Theorem~\ref{generalK} below.
	\end{proof}
	As a straightforward consequence of Theorem~\ref{kropinaLadder}  and the  implications from causality theory
	{\em (iii1)} $\Rightarrow$ global hyperbolicity $\Rightarrow$ {\em (i1)}, one has the following version of Hopf-Rinow theorem.
	\begin{cor}\label{cRandersKropinaHopfRinow}
		For any Randers-Kropina metric $F$ on a manifold $M$, the
		forward (resp. backward) geodesic completeness of $d_F$ is equivalent to the
		compactness of the forward closed balls $\bar B^+_F(x,r)$ (resp. backward closed balls $ \bar B^-_F(x,r)$)  for every $x\in M, r>0$.  Moreover,  any of these properties implies the compactness of the intersection between any pair 
		of forward and backward closed balls. Finally,  the  last property  implies the convexity of $(M,F)$, in the sense of Theorem~\ref{kropinaLadder}.
	\end{cor}
	\begin{exe}\label{ex_ppwave}
		The so-called Brinkmann spaces are defined by the existence of a complete
		parallel lightlike vector field $K$, and they include many physical examples of interest, as wave-type spacetimes (plane waves, pp-waves etc.); a detailed study of these spaces is carried out in \cite{BlSaSe13}.  Under very general hypotheses, they are strongly causal \cite{FlSan08} and become an \sstk \cite[Th. V.11]{CosFlo};  so, they determine   a Kropina metric where all the previous results are applicable.
		We mention that  they have been considered recently in \cite{BaCaFl}, where their geodesic connectedness is studied.
	\end{exe}
	\section{The case of arbitrary $K$: general wind Riemannian structures}\label{generalcase}
	In this section we consider the general case of a Killing vector field $K$  with no restriction on its pointwise  causal character. 
	\subsection{Causal futures and lightlike geodesics} For the study of the causality of a general \sstk  splitting, we will use its  Fermat structure $\Sigma$ in Definition~\ref{dfermatstructure}  and  the notation for causal elements 
	in Section~\ref{prelimin}. 
	
	We start with a characterization of the chronological relation which generalizes  the one  obtained  in  Proposition~\ref{bolas} when $K$ is  causal;    notice that  there is  now  no  natural Finslerian separation \soutE{$d_F$} \bw that fully describes the causal properties of $(\R\times M, g)$. \ew  This problem will be circumvented by means of a description of the causal futures and pasts, which makes apparent an interpretation of the c-balls. 
	Recall that the time coordinate  of the \sstk splitting is a temporal function and every causal curve can be parametrized with the time.
	\begin{prop}\label{bolas2} Let $(\R\times M,g)$ be an \sstk  splitting. Then:
		\begin{align*}
			&I^+(t_0,x_0)= \cup_{s> 0}\{t_0+s\}\times B_{\Sigma}^+(x_0,s),\\
			&I^-(t_0,x_0)= \cup_{s> 0}\{t_0-s\}\times B_{\Sigma}^-(x_0,s),\\
			&J^+(t_0,x_0)= \cup_{s\geq  0}\{t_0+s\}\times \hat{B}_{\Sigma}^+(x_0,s),\\
			&J^-(t_0,x_0)= \cup_{s\geq  0}\{t_0-s\}\times \hat{B}_{\Sigma}^-(x_0,s).
		\end{align*}
	\end{prop}
	\begin{proof}
		Taking into  account Corollary~\ref{raclarations} and Convention~\ref{caestar}, a vector $(\tau,v)\in  \R\times  (TM\setminus\mathbf{0})$    with $\tau>0$ 
		is causal and future-pointing if and only if
		$ v\in A_E \cup A \setminus\mathbf{0}$ and
		\begin{equation}\label{emenor}
			F(v)\leq \tau\leq F_l(v); \end{equation} moreover,  it is timelike
		and future-pointing if and only if  both inequalities
		hold  strictly in \eqref{emenor}  (and, thus, $v\in A$).  Accordingly,  a $t$-parametrized  piecewise smooth
		curve $(t,x):[t_0,t_1]\rightarrow \R\times M$   is causal (resp.
		timelike) and future-pointing  if and only if  $x$ is  $\Sigma$-admissible and 
		\begin{equation}\label{controlcausal}
			F(\dot x(t))\leq  \dot t\equiv 1  \leq  F_l(\dot x(t))
		\end{equation}
		(resp. $x$ is  $F$-admissible and
		\begin{equation}\label{controltime}
			F(\dot x(t))<  \dot t \equiv 1   <
			F_l(\dot x(t)))
		\end{equation}
		thus, in particular, $x$ is a wind curve.  Now, reasoning for the future and the inclusions $\subset$, observe that if $(t_1,x_1)\in J^+(t_0,x_0)$, then there exists a future-pointing causal curve $\gamma=(t,x)$ joining $(t_0,x_0)$ and $(t_1,x_1)$, 
		with $x$ being $\Sigma$-admissible  (recall footnote~\ref{foot4.1})  \footnote{  In the particular case when $x_0=x_1$,  $\dot x$ may be forced to vanish when the vertical line on $x_0$ is a  lightlike pregeodesic  
			but then  $\Lambda(x_0)=0$ (i.e. $0_{x_0}\in\Sigma_{x_0}$), and $x_0\in \hat B^+_\Sigma(x_0, r)$, for all $r\geq 0$,}.  So, integrating  in  \eqref{controlcausal}  
		one has:  if
		$(t_1,x_1)\in J^+(t_0,x_0)$ then $x_1\in \hat
		B^+_{\Sigma}(x_0,t_1-t_0)$ (resp. if $(t_1,x_1)\in I^+(t_0,x_0)$,
		then $x_1\in B^+_{\Sigma}(x_0,t_1-t_0)$), as required.
		
		For the converse $\supset$, in the case $J^+$,  
		choose any   $x_1\in \hat
		B^+_{\Sigma}(x_0,s)$ and  a     wind 
		curve  $x:  [0,s]  \rightarrow M$ from $x_0$ to $x_1$ such that
		$\ell_F(x)\leq s\leq \ell_{F_l}(x)$ (which exists by
		definition of $\hat B_\Sigma(x_0,s)$).  From \eqref{controlcausal}, the curve $[t_0,t_0+s]\ni t \mapsto (t,x(t-t_0))$ is the required causal curve from $(t_0,x_0)$ to $(t_0+s,x_1)$.  Moreover, to check the inclusion $\supset$ for $I^+$, notice that if $x_1\in 
		B^+_{\Sigma}(x_0,s)$ then the inequalities \eqref{controlcausal} hold  strictly  
		at some point. If both of them  hold at some point $\bar t$, then 
		the causal curve becomes timelike at  $(\bar t,x(\bar t-t_0))$ ---so that the points $(t_0,x_0)$ and $(t_0+s,x_1)$ can be 
		connected by means of a timelike curve, see for example \cite[Proposition 10.46]{O'neill}.   Otherwise, there must exist two disjoint intervals $[\bar t_1,\bar t_2]$, $[\bar t_3,\bar t_4] \subset ]0,s] $ such that in $[\bar t_1,\bar t_2]$ does hold the second strict inequality and in $[\bar t_3,\bar t_4]$ it holds the first  one. Assume that $\bar t_2<\bar t_3$  (the other case is analogous)  and define the  function  
		\[\rho(\mu)=
		\begin{cases}
			0 & 0\leq \mu \leq\bar t_1\\
			\varepsilon (\mu-\bar t_1)  &  \bar t_1 \leq \mu \leq \bar t_2\\
			\varepsilon (\bar t_2-\bar t_1) & \bar t_2 \leq \mu \leq \bar t_3\\
			\varepsilon  \left(\bar t_2-\bar t_1+  \frac{\bar t_2-\bar t_1}{\bar t_4-\bar t_3}(\bar t_3-\mu)\right)  & \bar t_3 \leq \mu \leq \bar t_4\\
			0 & \bar t_4 \leq \mu \leq  s
		\end{cases}
		\]
		If $\varepsilon>0$ is small enough, the curve $[0,s]\ni \mu\mapsto (\mu+\rho(\mu), x(\mu))\in\R\times M$ is a causal curve from $(t_0,x_0)$ to $(t_0+s,x_1)$ which is timelike in some point. Then applying again \cite[Proposition 10.46]{O'neill} we conclude. 
	\end{proof}

	We recall that two points $p$ and $q$  in a spacetime are said {\em horismotically related} if $q\in J^+(p)\setminus I^+(p)$.   We will give a characterization of these points after the following lemma.

	\begin{lemma}\label{tparam}
		Let $I\subset \R$ be an interval and $\rho\colon I\rightarrow   \R  \times  M$ be a lightlike future-pointing pregeodesic of an \sstk splitting $(\R \times M,g)$. Then
		$\rho$ can be reparametrized as  $s\mapsto(s,x(s))$  on  $[t(\rho(a)),t(\rho(b))]$ and
		the function $C_\rho(s):=g(\partial_t, \dot{\rho}(s))$ either has a definite sign  on $I$ or it vanishes  everywhere. 
	\end{lemma}
	\begin{proof}
		The possibility of the reparametrization  follows because the projection $t: \R\times M\rightarrow \R$ is a temporal function
		(see part (1) of Remark~\ref{increasing}). 
		Moreover, since $\partial_t$ is a Killing vector field, $g(\partial_t,\dot\gamma)$ is constant for any geodesic $\gamma$ of $(\R\times M,g)$, which implies that $g(\partial_t, \dot\rho(s))$ will preserve the sign in $s\in I$, as $\rho$ is  a 
		reparametrization  of  a  geodesic of $(\R\times M,g)$.
	\end{proof}
	
	\begin{cor}\label{horismos}
		Two distinct points $(t_0,x_0), (t_1,x_1)\in \R\times M$ are horismotically related
		if and only if
		$x_1 \in  \hat B^+_\Sigma(x_0,t_1-t_0)\setminus  B^+_\Sigma(x_0,t_1-t_0)$.
		
		In this case,  there exists a lightlike pregeodesic $\rho:[t_0,t_1]\rightarrow \R\times M$, $\rho(s)=(s,x(s))$ from $(t_0,x_0)$ to $(t_1,x_1)$ and 
		such that $x$ is a  unit extremizing 
		geodesic   of $\Sigma$  from $x_0$ to $x_1$ with $\ell_F(x)=t_1-t_0$ or $\ell_{F_l}(x)=t_1-t_0$ (or both).
		Moreover, when $x$ is a constant curve
		(i.e. an extremizing exceptional geodesic),  necessarily $\Lambda (x_0)=0$ 
		with $\de \Lambda(\mathrm{Ker}\, \omega_{x_0})\equiv 0$; when $x$ is not constant then it is   regular  (in the sense of Definition~\ref{sigmadmissible}-(iii)). 
		
	\end{cor}
	\begin{proof}
		The first equivalence is straightforward from Proposition~\ref{bolas2}.  More precisely, two horismotically related points are connected by a lightlike geodesic $\gamma$ (see e.g. \cite[Proposition 10.46]{O'neill}) and, 
		applying Lemma~\ref{tparam}, we can  reparametrize $\gamma$ as $\rho(s)=(s,x(s))$. 
		Now, $x$ is a  wind  curve connecting $x_0$ with $x_1$, 
		and horismoticity  
		implies
		$x(s)\in \hat B^+_{\Sigma}(x_0, s-t_0)\setminus  B^+_{\Sigma}(x_0, s-t_0)$, for all $s\in (t_0,t_1]$, so that $x$ is a unit extremizing geodesic (Definition~\ref{extremizing}). The last assertions follow from Lemmas~\ref{lightgeo} 
		and \ref{tparam}. 
	\end{proof}
	
	The next result  characterizes  the  lightlike  geodesics of an \sstk spacetime in terms of the Finslerian elements.
	But, first,  the following lemma points out some simple technical properties. Recall that a  Lorentzian manifold $(M,g)$ admits a convex neighborhood $U$ at every point $p\in M$ (i.e., $U$ is a normal neighborhood of all its points), \cite{Whiteh32}.
	\begin{lemma}\label{stronly}
		Given an \sstk splitting $(\R\times M,g)$ and $z_0=(t_0,x_0)\in\R\times M$ there exists a convex neighborhood $U$ of $z_0$,  a neighborhood $V$ of $z_0$ contained in $U$  and some small $\varepsilon>0$  such that
		$J^+(z)\cap \{(t,x)\in \R\times M:  t\in [t(z), t_0+\varepsilon) \}\subset U$ for every $z\in V$. 
	\end{lemma}
	\begin{proof}
		Consider a chart $(U, y^0, y^1,\ldots,y^m)$ around $z_0$ such that no causal curve starting at $U$ will leave and return to $U$ (this can be obtained as the \sstk spacetime is strongly causal,  Remark~\ref{increasing} (1)) and with the 
		coordinates adapted to the product structure ($y^0=t$, $y^1,...y^m$ coordinates on $M$). Choosing a smaller $U$,  such that it is convex in $(\R\times M,g)$,  define a (flat) Minkowski metric $g^{\text{\tiny flat}}$ in these coordinates such 
		that $\partial_i,\  i\in\{1,\ldots,m\},$  span a spacelike hyperplane and the timecones of $g^{\text{\tiny flat}}$ are wider than those of $g$ (this can be obtained obviously at the point $z_0$ and, by continuity, in some small neighborhood). 
		Then, the required property for $g$ holds as it does trivially for $g^{\text{\tiny flat}}$.
	\end{proof}
	\begin{thm}\label{existenceofbolasNO} 
		Let $I\subset \R$ be an interval and   $\rho(s)=(s,x(s)),\ s\in I$, be  a  (piecewise smooth)  curve  in an \sstk splitting $(\R\times M,g)$.  Then $\rho$  is a  future-pointing  lightlike pregeodesic of $(\R\times M,g)$ if and only if  
		its projection $I\ni s\to x(s)\in M$ is a  unit geodesic of $(M,\Sigma)$.  Moreover, in this case: 
		\begin{enumerate}[(i)]
			\item  $C_\rho<0$  iff  $x$ is a  unit geodesic of $F$ ($F(\dot x)\equiv 1$;   $x$ is $F$-admissible).  
			
			\item  $C_\rho>0$ iff  $x$ is a  unit geodesic of  $F_l$ ($F_l(\dot x)\equiv 1$;    $x$ is $F$-admissible). 
			
			\item $C_\rho=0$   iff one of the following two possibilities occurs: 
			
			(a) $\rho$ is an  integral curve of $K$ 
			which projects onto some $x_0$ with  $\Lambda(x_0)=0$ and  $\de \Lambda(\mathrm{Ker}\, \omega_{x_0})\equiv 0$ (so that the projection   is an exceptional geodesic,  Definition~\ref{windgeodesic}), 
			or 
			
			(b) $x$ is contained in $\overline{M}_l$; whenever it remains included in $M_l$, $x$ is 
			a lightlike  pregeodesic of $-h$  parametrized with $F(\dot x)\equiv  F_l(\dot x)\equiv 1$, and $x$ can reach $\partial M_l$ only at  isolated points $s_j\in I, j=1,2...$, where $\Lambda(x(s_j))=0$, 
			$\dot x(s_j)=0$ and $\left(D^{g_0}\dot x/ds \right)(s_j)\neq 0$. 
		\end{enumerate}
	\end{thm}
	\begin{proof}
		Assume that $\rho$ is a lightlike pregeodesic.
		For each $s_0\in I$ (different from its endpoints, and with straightforward modifications otherwise),  there exists $\varepsilon>0$ 
		such that $[s_0-\varepsilon,s_0+\varepsilon]\subset I$ and  $\rho(s)\in J^+(\rho(s_0-\varepsilon))\setminus I^+(\rho(s_0-\varepsilon))$, for all $s\in [s_0-\varepsilon,s_0+\varepsilon]$ (recall,  for example,   \cite[Proposition 5.34]{O'neill} 
		and use strong causality).  Thus, Corollary~\ref{horismos} can be applied locally (recall Lemma~\ref{stronly}), and
		$x|_{[s_0-\varepsilon, s_0+\varepsilon]}$ is  a unit extremizing geodesic  of the Fermat structure $(M,\Sigma)$. 
		
		Conversely, if   $x: I\rightarrow M$ is a unit geodesic of $(M,\Sigma)$, then it is locally a unit extremizing geodesic (recall Definitions~\ref{windgeodesic}, \ref{extremizing}). So, by Proposition~\ref{bolas2},
		every $s_0$ (as above) admits  an $\varepsilon>0$ such that  $[s_0-\varepsilon,s_0+\varepsilon]\subset I$  and the curve
		$[s_0-\varepsilon,s_0+\varepsilon]\ni s\rightarrow (s,x(s))\in \R\times M$ is contained in $J^+( s_0-\varepsilon ,x(s_0-\varepsilon)) \setminus I^+( s_0-\varepsilon, x(s_0-\varepsilon))$; therefore,  it is a lightlike pregeodesic 
		(see \cite[Proposition 10.46]{O'neill}).
		
		For the last part, first notice that
		$C_\rho=-\Lambda(x(s))+\omega(\dot x(s))$ and
		$C^2_\rho=h(\dot x, \dot x)$ (the latter follows as in the first part of the proof of Lemma~\ref{lightgeo}). Thus, when $C_\rho \neq 0$,
		$\dot x(s)$ belongs to $A_{x(s)}$, for all $s\in I$. Hence,   $x$ is a unit and $F$-admissible geodesic of $(M, \Sigma)$ and, then, from Theorem~\ref{extregeo},  a geodesic of $F$ or $F_l$. Precisely, from part $(iii)$ of 
		Proposition~\ref{plightvectorsSSTK}, $F(\dot x)=1$ iff $-C_\rho=\Lambda(x(s))-\omega(\dot x(s))\geq 0$ and $F_l(\dot x)=1$ iff $-C_\rho=\Lambda(x(s))-\omega(\dot x(s))\leq 0$,  i.e. 
		$x$  is a unit $F$-geodesic iff $C_\rho<0$ and  a unit $F_l$-geodesic  iff $C_\rho>0$. Finally,  $(iii)$ follows from Lemma~\ref{lightgeo}.
	\end{proof}
	As a straightforward consequence of  Lemmas~\ref{lightgeo}, \ref{tparam} and  Theorem~\ref{existenceofbolasNO} we get:
	\begin{cor}\label{lightgeo2}
		Let $\gamma\colon I\to\R\times M$, $\gamma(s)=(\zeta(s),\sigma(s))$ be a  (piecewise smooth) curve  in an \sstk splitting $(\R\times M,g)$, with $\sigma$ non-constant.
		Then $\gamma$ is a future-pointing lightlike geodesic if and only if $\sigma$ is a pregeodesic of $(M,\Sigma)$ parametrized with $h(\dot\sigma,\dot\sigma)=\mathrm{const.}$,  and one of the following three exclusive possibilities holds
		\begin{enumerate}[(i)]
			\item  $C_\gamma<0$  and  $\sigma$ is a  pregeodesic of $F$ and  $\zeta(\bar s_0)-\zeta(s_0)=\ell_{F}(\sigma|_{[s_0,\bar s_0]})$ for any $s_0,\bar s_0\in I$, $s_0<\bar s_0$;
			\item  $C_\gamma>0$ and  $\sigma$ is a  pregeodesic of  $F_l$ and $\zeta(\bar s_0)-\zeta(s_0)=\ell_{F_l}(\sigma|_{[s_0,\bar s_0]})$  for any $s_0,\bar s_0\in I$, $s_0<\bar s_0$;
			\item $C_\gamma=0$,  $\sigma$  
			is  smooth\footnote{Notice that smoothness follows if we assume just that it is  twice differentiable  at the points where it touches $\partial M_l$.},  it 
			is included  in $\overline{M}_l$,   it  touches $\partial M_l$ at most at 
			isolated points  $s_j$, $j=1,2,\ldots,$ such that $\dot\sigma(s_j)=0$,  $\sigma$  is a lightlike  geodesic  of $h/\Lambda$ whenever it remains in $M_l$,  and also  $\zeta$ is determined by 
			\[\zeta(\bar s_0)-\zeta(s_0)=\ell_{F_l}(\sigma|_{[s_0,\bar s_0]})=\ell_{F}(\sigma|_{[s_0,\bar s_0]})\] 
			for any $s_0,\bar s_0\in I$, $s_0<\bar s_0$. 
		\end{enumerate}
		%
		%
	\end{cor}
	\subsection{Characterization of the causal ladder}
	The following technical property concerns limit curves  (recall Definition~\ref{limitcurvedef})
	in connection with time functions. The role of these functions in the limit process is not usually taken into account (indeed, it is rather trivial in the case of Cauchy temporal functions). So, we write it for any stably causal spacetime, which may have interest in its own right.
	
	\begin{lemma}\label{limitcurve} Let $(L,g)$ be a spacetime endowed
		with a time function $t:L\rightarrow \R$.
		\begin{enumerate}[(i)]
			\item   Consider a sequence  of inextendible causal curves $\{\gamma_n\}$ parametrized by the time $t$ and assume that there exists a convergent sequence $\{t_n\}$ such that $\gamma_n(t_n)$ converges to $z_0$. 
			Then there exists an (inextendible, causal) limit curve $\gamma$ 
			through $z_0$ parametrized by the time $t$,  and   a  subsequence  $\gamma_{n_k}$  
			such  that,  whenever    the intersection
			of $\gamma$ with the slice $S_{t_0}:=\{z\in L: t(z)=t_0\}$ is  not empty for   $t_0\in \R$,
			then all the curves $\gamma_{n_k}$ but a finite number intersect $S_{t_0}$ and  $\gamma(t_0)=\lim_k\gamma_{n_k}(t_0)$.
			\item  Let  $\gamma_n$ be a sequence of causal
			curves   and, for each $n\in\N$,  $z_n\leq w_n$ be two points on  $\gamma_n$.  If $z_n\to z$, $w_n\to w$,  $z\neq w$, and the    intersection of the slice
			$S_{t_0}$ with the images of all $\gamma_n$  lies in a
			compact subset for any $t_0 \in   (t(z),t(w)) $,  then any (inextendible) limit
			curve $\gamma$ of the sequence starting at   $z$   arrives at   $w$. 
		\end{enumerate}
	\end{lemma}
	\begin{proof}
		$(i)$  The existence of the  limit  curve follows from \cite[Proposition 3.31]{BeEhEa96}.  Let $\{\gamma_{n_k}\}$ be any  subsequence that converges to
		$\gamma$ uniformly on compact subsets for some auxiliary complete
		Riemannian metric (up to a reparametrization, according to
		\cite[Lemma 14.2]{BeEhEa96})  and, so,   such that some sequence
		$\{p_k:=\gamma_{n_k}(t_{n_k})\},   t_{n_k}\in \R$,
		converges to $p:=\gamma(t_0)$; in particular $t_{n_k}\rightarrow t_0$.
		Choose open neighborhoods $V, U$ of $p$, with $V\subset U$, $U$ convex
		and $V$  having compact closure included in $U$ and  being globally
		hyperbolic with Cauchy hypersurface $S_{t_0}$.\footnote{Such a neighborhood $V$
			can be constructed easily by taken the Cauchy development of a small neighborhood in $S_{t_0}$ of $p$, see \cite[Theorem 2.14]{MinSan08}.}
		Due to the convergence to $p$, all $\gamma_{n_k}$ but a finite number  will enter in $V$ and cross $V\cap S_{t_0}$ at a single point $q_k$.
		Reasoning by contradiction, if $\{q_k\}$ does not converge to $p$ then, up to a subsequence, $\{q_k\}$  converges to some $q\in U\setminus\{p\}$.
		Assume  that, up to a new subsequence,  $p_k\leq q_k$ (otherwise we could assume    $q_k\leq p_k$ up to a subsequence, and the reasoning would be analogous).  By the convexity of $U$,
		$p\leq q$ but, as $p$ and $q$ lie in the acausal set\footnote{\label{acausal} A subset $\mathcal A$ of a spacetime $V$ is said {\em acausal} if no $p, q\in \mathcal A$ are causally related in $V$.} $V\cap S_{t_0}$, one obtains the absurd $p=q$.
		
		$(ii)$   Let $t^*\in (t(z),t(w))$ and let us reparameterize   all the curves with $t$. Assume that $\gamma:
		[t(  z ),t^*)\rightarrow L$  cannot be extended to $t^*$. Let
		$\{\gamma_{n_k}\}$ be any  subsequence that converges to $\gamma$
		as  in part $(i)$  and such that $\{\gamma_{n_k}(t^*)\}$ converges to some
		point $z^*\in  S_{t^*}$, the latter property by the assumption on
		compactness. Up to a subsequence, $\{\gamma_{n_k}\}$ admits a
		limit curve  starting at $z^*$, say, $\rho: (t^*-\varepsilon,
		t^*]\rightarrow L$ for some $\varepsilon>0$. Now,
		by  part $(i)$,
		necessarily  $\{\gamma_{n_k}(t)\}$  converges  to both
		$\gamma(t)$ and $\rho(t)$,  for each $t\in (t^*-\varepsilon, t^*)$.  So,
		$\gamma$ admits  $z^*$ as a  future limit point and it is then extendible, a
		contradiction.   Thus, $\gamma$ is defined on $[t(z),t(w))$ and since $w_n\to w$ necessarily it arrives at $w$.  
	\end{proof}
	As a first consequence, we obtain characterizations of some causal
	properties, which will be related to the possible reflectivity and
	causal
	simplicity of the spacetime. 
	\begin{prop}\label{lreflect} For any $p=(t_0,x_0), q=(t_1, x_1)$ in an \sstk splitting:
		\begin{enumerate}[(i)]
			\item $I^+(p)\supset I^+(q)$ if and only if $x_1\in
			\bar{B}^+_{\Sigma}(x_0,t_1-t_0)$, and
			\item  $I^-(p)\subset I^-(q)$ if and only if $x_0\in
			\bar{B}^-_{\Sigma}(x_1,t_1-t_0)$.
		\end{enumerate}
		\smallskip
		\noindent Moreover,
		\begin{align*}
			&\bar J^+(t_0,x_0)= \left(\cup_{s> 0}\{t_0+s\}\times \bar{B}_{\Sigma}^+(x_0,s)\right)\,\cup\,\{(t_0,x_0)\}\\
			&\bar J^-(t_0,x_0)= \left(\cup_{s> 0}\{t_0-s\}\times \bar{B}_{\Sigma}^-(x_0,s)\right)\,\cup\,\{(t_0,x_0)\}
		\end{align*}
	\end{prop}
	\begin{proof} We consider the case $(i)$, being part $(ii)$ analogous.
		
		($\Rightarrow$) Choose $\{q_n\}\subset I^+(q)$, converging to $q$ and
		inextendible future-pointing timelike curves $\gamma_n$
		through $p$ and $q_n$. From part $(i)$ of Lemma~\ref{limitcurve}, there
		exists a subsequence $\gamma_{n_k}$ that  cuts the slice $S_{t_1}$
		in a sequence of points  $(t_1, y_{n_k})$, such that $y_{n_k}\to
		x_1$. By Proposition~\ref{bolas2},  $y_{n_k}\in B^+_{\Sigma}(x_0,
		t_1-t_0)$ and, then, $x_1\in \bar B^+_{\Sigma}(x_0,t_1-t_0)$.
		
		($\Leftarrow$) As $x_1\in \bar{B}^+_{\Sigma}(x_0,t_1-t_0)$, take a
		sequence $\{\tilde{y}_n\}$  in
		$B^+_{\Sigma}(x_0,t_1-t_0)$ converging to $x_1$.  The sequence
		$q_n=(t_1, \tilde{y}_n)$ converges to $q$ and, by  Proposition~\ref{bolas2}, is contained in $I^+(p)$. As the chronological relations
		are open, given $r\in I^+(q)$, then $r\in I^+(q_n)$ for $n$ big
		enough. This implies that $r\in I^+(p)$ and then $I^+(p)\supset
		I^+(q) $, as required.
		
		For the last assertion, recall first that,  in any spacetime $\bar J^\pm(p)=\bar I^\pm(p)$
		(see, e.g. \cite[Lemma 14.6]{O'neill})
		hence the  inclusions
		$\supset$ hold trivially from Proposition~\ref{bolas2}. For the
		converse in the case of $\bar J^+$,
		let $q=(t_1,x_1) \in \bar J^+(p)=\bar I^+(p)$, and take  $q_n\in I^+(p)$ such that $q_n\to q$.  Apply part $(i)$ of
		Lemma~\ref{limitcurve} to obtain a sequence $(t_1, y_{n_k})\in
		I^+(p)$ such that $y_{n_k}\to x_1$, and conclude again from
		Proposition~\ref{bolas2} that  $y_{n_k}\in B^+_{\Sigma}(x_0,
		t_1-t_0)$, so that $x_1\in \bar B^+_{\Sigma}(x_0,t_1-t_0)$.
	\end{proof}
	Now, we can study the causal ladder of any \sstk splitting,
	extending the Randers-Kropina case in Theorem~\ref{kropinaLadder}.
	\begin{thm}\label{generalK} Consider an \sstk  splitting $(\R\times M,g)$ as
		in \eqref{lorentz} with associated Fermat structure $\Sigma$ on $M$. Then, $(\R\times M,g)$  is stably causal and
		\begin{enumerate}[(i)]
			\item  $(\R\times M,g)$ is causally continuous if and only if
			$\Sigma$ satisfies the following property: given any pair  of points
			$x_0,x_1$ in $M$ and $r>0$, $x_1\in \bar{B}^+_{\Sigma}(x_0,r)$ if
			and only if $x_0\in \bar{B}^-_{\Sigma}(x_1,r)$.
			
			\item  $(\R\times M,g)$ is causally simple
			if and only if $(M,\Sigma)$ is  w-convex  (according to Definition~\ref{strongconvex}).
			\item  The following assertions are equivalent:
			\begin{enumerate}
				\item[(iii1)] $(\R\times M,g)$ is globally hyperbolic.
				\item[(iii2)]$\hat{ B}^+_{\Sigma}(x,r_1)\cap
				\hat{B}^-_{\Sigma}(y,r_2)$ is compact for every $x,y\in
				M$ and $r_1,r_2>0$.
				\item[(iii3)]  $\bar{ B}^+_{\Sigma}(x,r_1)\cap \bar{B}^-_{\Sigma}(y,r_2)$
				is compact for every $x,y\in M$ and
				$r_1,r_2>0$.
			\end{enumerate}
			\item The following assertions are
			equivalent:
			\begin{enumerate}
				\item[(iv1)] A slice $S_t$ (and, then every slice) is a
				spacelike Cauchy hypersurface.
				\item[(iv2)] All the c-balls $\hat B_{\Sigma}^+(x,r)$ and
				$\hat B_{\Sigma}^-(x,r)$, $r>0$, $x\in M$, are compact.
				\item[(iv3)] All the (open) balls $B_{\Sigma}^+(x,r)$ and
				$B_{\Sigma}^-(x,r)$, $r>0$, $x\in M$, are precompact.
				\item[(iv4)]   $\Sigma$ is forward and backward geodesically complete (according to  Definition~\ref{strongconvex}).
			\end{enumerate}
		\end{enumerate}
	\end{thm}
	\begin{proof}
		$(i)$ As the natural projection $t:\R\times M\rightarrow
		\R$  is a temporal function,   Remark~\ref{increasing},   the spacetime is
		stably  causal.  Thus,  causal continuity becomes equivalent to
		past and future  reflectivity. However, by Proposition~\ref{lreflect},
		past reflectivity (i.e., $I^+(p)\supset I^+(q)$ implies $I^-(p)\subset
		I^-(q)$) becomes equivalent to the property $$x_1\in \bar
		B^+_{\Sigma}(x_0,t_1-t_0) \Rightarrow x_0\in
		\bar{B}^-_{\Sigma}(x_1,t_1-t_0),$$ and future reflectivity is
		equivalent to the converse.
		
		$(ii)$  Assuming that  $(\R\times M,g)$ is causally simple, for
		any  $x_0\in M$ and $t_0<t_1$, the intersections $S_{t_1} \cap
		J^+(t_0,x_0)$ and $S_{2t_0-t_1}\cap J^-(t_0,x_0)$ must be closed.
		By Proposition~\ref{bolas2}, these intersections are equal to
		$\{t_1\}\times \hat{B}_\Sigma^+(x_0,t_1-t_0)$ and
		$\{2t_0-t_1\}\times \hat{B}_\Sigma^-(x_0,t_1-t_0)$, respectively,
		which means that $\Sigma$ is w-convex. For the converse, just
		apply the last assertion of Proposition~\ref{lreflect}, plus
		Proposition~\ref{pclosurecballs} and Proposition~\ref{bolas2}.

		$(iii)$ ($(iii1)\Rightarrow (iii2)$) Assume that $(\R\times M,g)$ is globally hyperbolic and consider the points
		$(r_1,x)$ and $(-r_2,y)$. By Proposition~\ref{bolas2},
		$$
		\{0\}\times
		\left(\hat{ B}_{\Sigma}^+(x,r_1)\cap \hat{
			B}_{\Sigma}^-(y,r_2)\right)= (\{0\}\times M) \cap J^+(-r_1,x)\cap
		J^-(r_2,y),$$ and the right-hand side is compact by global
		hyperbolicity.

		($(iii2)\Rightarrow (iii3)$) By Proposition~\ref{pclosurecballs},
		it is enough to prove that the property of compactness of the
		intersections implies  the closedness of the c-balls. Reasoning by
		contradiction, if, say, $z\in \bar{B}_{\Sigma}^+(x,r_1)\setminus
		\hat{B}_{\Sigma}^+(x,r_1)$, as $z$ \bw always belongs \ew to some (open)
		ball $ {B}_{\Sigma}^-(y,r_2)$ (recall that one can take any
		$\Sigma-$ admissible curve through $z$ in order to choose appropriate $y$
		and $r_2$), necessarily
		$\hat{B}_{\Sigma}^+(x,r_1)\cap
		\hat{B}_{\Sigma}^-(y,r_2)$ cannot be compact.
		
		($(iii3)\Rightarrow (iii1)$)
		For any  $(t_0,x_0)$, $(t_1,x_1)$ in $\R\times M$,  recalling
		Propositions~\ref{bolas2} and \ref{pclosurecballs}, we get 
		\begin{align*}
			\lefteqn{J^+(t_0,x_0)\cap J^-(t_1,x_1) \subset}&\\
			&\cup_{s\in(0,t_1-t_0)}\{t_0+s\}\times\left(\bar{B}_{\Sigma}^+(x_0,s)\cap\bar{B}_{\Sigma}^-(x_1,t_1-t_0-s)\right)\,\cup\,\{(t,x_0), (t_1,x_1)\}
		\end{align*}
		and we have to check that the \bw left-hand \ew side is compact. Indeed,
		for any sequence $\{z_n\}\subset J^+(t_0,x_0)\cap J^-(t_1,x_1)$ we can
		take a  sequence of causal curves  $\gamma_n$ from $(t_0,x_0)$ to $(t_1,x_1)$
		passing through $z_n$. By the hypothesis on the closures, part $(ii)$ of Lemma
		\ref{limitcurve} is applicable, and there exists a limit curve
		$\gamma$ of $\{\gamma_n\}$ with the same endpoints. So, some
		subsequence $\{\gamma_{n_k}\}$ converges in the $C^0$ topology to
		$\gamma$ (see \cite[Proposition 3.34]{BeEhEa96}) and, thus,
		$\{z_{n_k}\}$ lies in a compact subset, admitting so a convergent
		subsequence to a point in the image of $\gamma$.
		
		$(iv)$ ($\it (iv1) \Rightarrow (iv2)$) By Proposition
		\ref{bolas2},
		$$ J^+(0,x) \cap S_r = \{r\} \times \hat B_{\Sigma}^+(x,r), \quad
		J^-(0,x) \cap S_{-r} = \{-r\} \times \hat B_{\Sigma}^-(x,r)
		$$
		and the \bw left-hand \ew sides are compact as $S_r, S_{-r}$ are Cauchy
		hypersurfaces (otherwise, the limit curve of the
		sequence of causal curves obtained by connecting $(0,x)$ with a
		diverging sequence of points would not cross the corresponding
		Cauchy hypersurface).
		
		($\it (iv2) \Rightarrow (iv3)$) Just apply  Proposition~\ref{pclosurecballs}.
		
		($\it (iv3) \Rightarrow (iv1)$) 
		By using the one-parameter group of isometries generated by the complete Killing field $\partial_t$, one easily sees that if a slice $S_{t_0}$ is a Cauchy hypersurface then all the slices $S_t$ are Cauchy hypersurfaces. 
		Thus, by contradiction, let us assume that $S_0$ is not Cauchy.  Hence, there will
		exist some inextendible timelike curve $\rho:[0,t_0)\rightarrow \R\times M$, $\rho(s)=(s-t_0,x_\rho(s))$ or $\rho:(-t_0,0]\rightarrow \R\times M$, $\rho(s)=(s+t_0,x_\rho(s))$, which does not cross it.  We recall that any 
		timelike vertical line $s\mapsto (s, \bar x)$, $\bar x\in M$, always crosses once $S_0$ and that, for any subinterval $[s_1,s_2]\subset [0,t_0)$ (or $\subset(-t_0,0]$)  containing   some point $\bar s$ where $\dot x_{\rho}(\bar s)=0$,
		$\rho_{|[s_1,s_2]}$ can be replaced by a timelike future-pointing curve $\tilde \rho(s)=(s,x_{\tilde \rho}(s))$,
		such that $\dot x_{\tilde \rho}(s)\neq 0$, for each $s\in [s_1,s_2]$  (see footnote~\ref{foot4.1}).  Summing up, we can assume, without losing generality, that $\dot x_\rho(s)\neq 0$, for all $s$.
		Thus, being  $\rho$ timelike and future-pointing, we have (recall \eqref{controltime}) $\ell_F(x_\rho)<t_0<\ell_{F_l}(x_\rho)$.  Then if $\epsilon>0$ is small enough, $x_\rho(t_0-\epsilon)\in B_\Sigma^+(x_\rho(0),t_0)$  or
		$x_\rho(-t_0+\epsilon)\in B_\Sigma^-(x_\rho(0),t_0)$. As $\rho$ cannot remain in a
		compact region of the spacetime (otherwise it would be extendible), either $B_{\Sigma}^+(x_\rho(0),t_0)$ in the first case, or $B_{\Sigma}^-(x_\rho(0),t_0)$ in the second one, is not precompact.

		($\it (iv1) \Leftrightarrow (iv4)$)
		As the slices $S_t$ are
		closed, spacelike  and acausal,  each one  will be Cauchy if and only if it is crossed by any  future-pointing 
		inextendible lightlike pregeodesic $\rho$  (see \cite[Lemma 14.42 and Corollary 14.54]{O'neill}). So, let $\rho$ be any inextendible future-pointing lightlike pregeodesic  and let us parametrize it  as $\rho(s)=(s,x_{\rho}(s))$ (recall Lemma~\ref{tparam}). 
		Hence, $S_t$ will be Cauchy
		if and only if $\rho$ is defined on $\R$.
		From the first part of Theorem~\ref{existenceofbolasNO},
		this property is equivalent to saying that  $(M,\Sigma)$ 
		is forward and backward geodesically  complete. 
	\end{proof}
	\begin{rem}
		(1) As suggested by   the equivalence $(i)$  above,  wind Finslerian
		structures where    $x_1\in \bar B^+_{\Sigma}(x_0,r)$ does not
		imply that $x_0\in \bar B^-_{\Sigma}(x_1,r)$ do exist  (in
		contrast with the Randers-Kropina case);  in fact, it is   not
		difficult to construct explicit examples, as the one in
		Fig.~\ref{dis6}.
		
		(2) In the  comparison with the Randers-Kropina case,  notice   that Theorem~\ref{kropinaLadder} was
		stated by using only balls for the Finslerian separation $d_F$, as
		this notion  had  familiar similarities with a distance.  However,   the
		results stated here in terms of  c-balls   are more accurate and refine   those
		in that theorem.
		\begin{figure}[h]
			\includegraphics[scale=1,center]{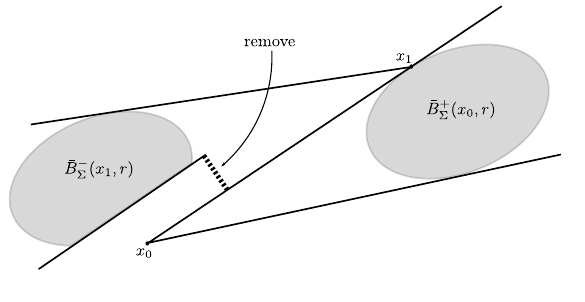}
			\caption{A wind Finslerian structure on  $\R^2$ minus the bold dashed segment,  where $x_1\in \bar B^+_{\Sigma}(x_0,r)$ but $x_0\not \in \bar B^+_{\Sigma}(x_1,r)$}\label{dis6}
		\end{figure}
	\end{rem}
	\section{Applications to wind Riemannian structures and navigation}\label{causaltoRiemann}
	As emphasized in Section~\ref{windFermat}, every wind Riemannian
	structure can be seen as the Fermat structure of a certain \sstk  splitting  (canonically chosen in a conformal class,
	see  Theorem~\ref{tfermatSSTK}), which
	will be referred  to  as the {\em \sstk splitting associated with the
		wind Riemannian structure}.
	
	Along this section, some results on wind Riemannian structures will be obtained by using the associated \sstk splitting.  Some of these properties \soutE{might} \bw can \ew  be generalizable to any wind Finslerian structure, \soutE{but this might be not 
	straightforward and  would require the extension of  the techniques of classical (Lorentzian) spacetimes to more general Lorentz-Finsler
	spacetimes (about this possibility, see \cite{JavSan13})} \bw see Section~\ref{conclusions}. \ew
	
	\subsection{Characterization of geodesics and Hopf-Rinow properties}
	\label{ss6.1}  One of the most relevant difficulties of arbitrary wind
	Finslerian structures in comparison with standard Finsler (or
	Riemannian) metrics is that the exponential map is not
	necessarily defined in all directions. Indeed, the second order
	differential equations defining the extremizing geodesics of the
	conic  non-degenerate pseudo-Finsler metrics $F$ and $F_l$
	associated with $\Sigma$  have coefficients $\Gamma^k_{\,\,ij}$ defined
	on a  subset ($A$ or $A_l$) of $TM$ that, in general, does not
	contain a punctured neighborhood of the zero section (the most we could say
	was Theorem~\ref{extregeo}).
	In wind Riemannian structures, however, this difficulty can be
	overcome easily by using the associated \sstk  splitting.  Indeed, recall  that  a triple $(g_0,\omega,\Lambda)$ is associated with  any wind Riemannian structure $(M,\Sigma)$    (Proposition~\ref{windyfermat} and Definition~\ref{ctriple}) and then  an \sstk spacetime  with Fermat structure $\Sigma$ (Theorem~\ref{tfermatSSTK}).  
	As a first result in this direction we have the following.
	\begin{prop}
		Let $(M,\Sigma)$ be a wind Riemannian manifold, then its  wind balls are open.
	\end{prop}
	\begin{proof} From Proposition~\ref{bolas2}, 
		$ B_{\Sigma}^\pm (x_0,s)$ is homeomorphic to  $I^\pm (t_0,x_0) \cap  S_{t_0+s}$ and, in  any spacetime,  $I^\pm (t_0,x_0)$ is open. 
	\end{proof}
	Notice that    the  Lorentzian result \cite[Proposition 10.46]{O'neill}) claimed in the proof of 
	Proposition~\ref{bolas2}  has been  crucial. 
	
	We pass now to study geodesics of a wind Riemannian structure. 
	As any lightlike geodesic is locally horismotic, Corollary~\ref{horismos} plus the crucial Lemma~\ref{lightgeo} on lightlike geodesics of an \sstk splitting imply:
	\begin{prop}\label{rappendix}
		All the non-exceptional  geodesics of a wind Riemannian manifold  $(M,\Sigma)$ are regular.   Moreover, 
		if $x$ is a   non-exceptional  geodesic of $\Sigma$  then it is a smooth curve, its velocity can be zero only at isolated points and its acceleration  (for one auxiliary Riemannian metric and, then, for any of them) does not vanish at those zeroes.  
	\end{prop}	
	
	The following theorem characterizes  wind Riemannian geodesics and refines Theorem~\ref{extregeo}.  
	\begin{thm}\label{minimizers}
		Let $(M,\Sigma)$ be a wind Riemannian structure. For any
		$\Sigma$-admissible  (piecewise smooth) curve $x:[a,b]\rightarrow M$,
		the following conditions are equivalent: 
		\begin{enumerate}[(i)]
			\item $x$ is a  unit  geodesic of $\Sigma$,
			\item
			$x$ satisfies one of the following three exclusive possibilities: 
			
			(a) $x$ is  a unit  $F$-admissible geodesic of $(M,F)$ and, then, locally, it  minimizes  the $F$-length  
			
			(b) $x$ is  a unit $F$-admissible geodesic of $(M_l,F_l)$  and, then, locally, it  maximizes  the $F_l$-length,
			
			(c) $x$ is  a smooth curve  contained in $\overline{M}_l$  and  either (c1)  $x$ is constant,  $\Lambda(x)=0$ and $d\Lambda({\rm ker}\, \omega_x)\equiv 0$   or (c2) whenever it remains included in $M_l$, $x$ is 
			a lightlike  pregeodesic of the Lorentzian metric $-h$  in \eqref{eh}  parametrized with $F(\dot x)\equiv  F_l(\dot x)\equiv 1$ ($x$ is a boundary geodesic), and $x$ can reach $\partial M_l$ only at  isolated points $s_j\in I, j=1,2...$, where  
			$\dot x(s_j)=0$,  $d\Lambda({\rm ker}\, \omega_{x(s_j)})\neq 0$  and its second derivative (in one and then any coordinates)\footnote{With natural identifications; this condition can be also formulated in terms of the 2-jet of $x$ at each $s_j$.} 
			does not vanish.  
		\end{enumerate}
	\end{thm}
	\begin{proof}
		The implication $(i) \Rightarrow (ii)$ follows   by applying  Theorem~\ref{existenceofbolasNO}  to the \sstk splitting $(\R\times M,g)$ associated with $(M,\Sigma)$.  Moreover,     
		as in  Theorem~\ref{extregeo},  it can be proved that the extremal properties hold for any variation. 
		For the converse,   in the cases $(a)$ and $(b)$, choose $s_0\in [a,b]$  and   take the   future-pointing lightlike  geodesic  $\gamma$ of the associated \sstk spacetime $(\R\times M,g)$, 
		satisfying the  initial conditions  $\big((s_0, x(s_0)), (1,\dot x(s_0))\big)$. 
		Let  us reparametrize $\gamma$ as 
		$\rho(s)=(s,x_\rho(s))$. From Theorem~\ref{existenceofbolasNO},
		$x_\rho$ is a  unit geodesic of $(M,\Sigma)$  and, as the vector $\dot x(s_0)\in A_{x(s_0)}$,
		$C_\rho\neq 0$. Thus, $x_\rho$ is, according to the sign of $C_\rho$, a unit  $F$-admissible geodesic  for $F$ or  $F_l$ 
		which  coincides with $x$ by existence and uniqueness of geodesic of a conic pseudo-Finsler metric. 
		In the remaining case 
		$(c)$, whenever $x$ is not constant, 
		$j=1,2,\ldots,$ such that $\Lambda(x_\rho(s_j))=0$ and $\dot x_\rho(s_j)=0$,  
		the curve $\rho(s)=(s,x(s))$ is  orthogonal to the Killing vector field (recall part $
		(iii)$ of Proposition~\ref{plightvectorsSSTK}) and then   a lightlike pregeodesic 
		whenever $x$ is contained in $M_l$ (recall,  from Proposition~\ref{gh}, that 
		$\pi:(\R\times M_l,g)\rightarrow (M_l,\frac{1}{\Lambda}h)$ is  a   semi-Riemannian
		submersion, and also that  
		lightlike pregeodesics  were preserved by conformal changes of the metric); moreover, as $x$ 
		is smooth and its derivative vanishes at the points where it touches the boundary $\partial 
		M_l$, we conclude that $\rho$ is globally a lightlike pregeodesic. Then, from Theorem~\ref{existenceofbolasNO}-(iii), 
		$x$ is a unit geodesic of $(M,\Sigma)$. Finally, if 
		$x\equiv x_0\in M$, $\Lambda(x_0)=0$ 
		and $d\Lambda({\rm ker}\, \omega_{x_0})= 0$, then $x$ is an exceptional geodesic of $(M,\Sigma)$. 
	\end{proof}
	\begin{lemma}\label{existenceofbolas}
		For every neighborhood $W_0$ of  $x_0\in M$, there exists
		another neighborhood $U_0\subset W_0$ and some $\varepsilon>0$ such that
		$\hat{B}^\pm_\Sigma(x,r)$ is compact and contained in $W_0$ for
		every $r<\varepsilon$ and $x\in U_0$.
	\end{lemma}
	\begin{proof}
		The proof is a refinement of Lemma~\ref{stronly} obtained by taking into account that,  given $W_0$ and considering the  \sstk splitting $(\R\times M,g)$
		associated with the wind Riemannian structure, then  the convex neighborhood $U$ provided by that lemma around $z_0=(0,x_0)$ can be easily chosen satisfying the following properties:  (i) $U$ is included in $\R\times W_0$, (ii) $U$ is precompact and (iii) 
		$U$ contains the intersections $J^\pm(z)\cap ([-\varepsilon,\varepsilon]\times M)$ for some $\varepsilon>0$  and all $z$ in a smaller neighborhood $V\subset U$.  Then, put $U_0:=V\cap S_0 \subset W_0$ 
		and observe that, for any $r\in (0,\varepsilon)$, the set $J^+(0,x)\cap (\{\pm r\}\times M)$
		is compact for all $x\in U_0$ (observe that it is a closed set, as $U$ is normal, included in a compact set, also by hypothesis on $U$) and,   by Proposition~\ref{bolas2},   it projects homeomorphically into $\hat B^\pm_{\Sigma}(x_0, r)$
		and this projection is contained in $W_0$ as required.
	\end{proof}
	The following local properties can also be  proved by using the spacetime viewpoint.  They are equivalent to saying that the exponential  maps  of the conic pseudo-Finsler metrics associated with a wind Riemannian metric are defined in some small  cone. 
	\begin{prop}
		Let $(M,\Sigma)$ be a wind Riemannian structure and $x_0\in M$, then there exists $\varepsilon >0$ such that 
		the geodesics of $F$  (resp. $F_l$)  departing from
		$x_0$  and  parametrized by the  arc length  are defined on
		$[0,\varepsilon)$  and they are  extremizing unit geodesic and therefore minima (resp. maxima) of $\ell_F$ (resp. $\ell_{F_l}$) with respect to any variation. 
	\end{prop}
	\begin{proof}
		Let $\varepsilon>0$ as in Lemma~\ref{existenceofbolas} and consider the \sstk splitting associated with $(M,\Sigma)$. Any  geodesic $x\colon [0,a_x]\to M$, $a_x>0$ of $F$ or $F_l$, starting at $x_0$ and parametrized by the arc length, defines  
		a lightlike pregeodesic $(s, x(s))$ (recall Theorems~\ref{existenceofbolasNO}  and \ref{minimizers})  which must be defined on $[0,r]$, for any $r<\varepsilon$. In fact,   $(s,x(s))\in J^+(0, x_0)\cap (\{s\}\times M)$ and, for 
		$s\in [0,r]$,  $J^+(0, x_0)\cap (\{s\}\times M)$  is contained in a precompact convex neighborhood of $(0,x_0)$ (see the proof of  Lemma~\ref{existenceofbolas}).  Therefore, $(s,x(s))\in J^+(0,x_0)\setminus I^+(0,x_0)$ and  $x$
		must also be an extremizing unit geodesic for
		$\Sigma$  by Corollary~\ref{horismos}  and then, being   $F$-admissible,   it must be
		a local minimum of $\ell_F$ or  a local maximum of $\ell_{F_l}$  (recall Theorem~\ref{minimizers}). 
	\end{proof}
	Finally, the following result becomes straightforward from
	Theorem~\ref{generalK} and plays the role of Hopf-Rinow theorem for wind Finslerian structures (then generalizing Corollary~\ref{cRandersKropinaHopfRinow}).
	\begin{prop}\label{c63}
		Let $(M,\Sigma)$ be a wind Riemann structure.
		\begin{enumerate}[(i)]
			\item The following properties are equivalent:
			\begin{enumerate}[(a)]
				\item $\Sigma$ is geodesically complete, \item $B^+_\Sigma(x,r)$
				and $B^-_\Sigma(x,r)$ are precompact for every $x\in M$ and $r>0$.
				\item $\hat B^+_\Sigma(x,r)$ and $\hat B^-_\Sigma(x,r)$ are
				compact for every $x\in M$ and $r>0$.
			\end{enumerate}
			In particular, if $M$ is compact then $\Sigma$ is  geodesically  complete.
			\item The following properties are equivalent and imply the
			w-convexity   (Definition~\ref{strongconvex}) of   $(M,\Sigma)$:
			\begin{enumerate}[(a)]
				\item  $\hat{ B}^+_{\Sigma}(x,r_1)\cap \hat{B}^-_{\Sigma}(y,r_2)$
				is compact for every $x,y\in M$ and $r_1,r_2>0$.
				\item   $\bar{ B}^+_{\Sigma}(x,r_1)\cap \bar{B}^-_{\Sigma}(y,r_2)$
				is compact for every $x,y\in M$ and $r_1,r_2>0$.
			\end{enumerate}
			Moreover, these conditions  hold whenever the previous ones in $(i)$
			are satisfied.
			
			\item If $(M,\Sigma)$ is  w-convex, then   $x_1\in
			\bar{B}^+_{\Sigma}(x_0,r)$ if and only if $x_0\in
			\bar{B}^-_{\Sigma}(x_1,r)$ for any  $x_0,x_1 \in M$ and $r>0$.
		\end{enumerate}
	\end{prop}
	\begin{proof}
		Apply Theorem~\ref{generalK} to the  associated \sstk splitting (see Theorem~\ref{tfermatSSTK}),
		and use the causal implications: existence of a Cauchy
		hypersurface $\Rightarrow$ global hyperbolicity $\Rightarrow$
		causal simplicity $\Rightarrow$ causal continuity.
	\end{proof}
	The relations between lightlike geodesics on an \sstk  splitting  and geodesics of the associated Fermat structure are summarized in Fig.~\ref{geos}. 
	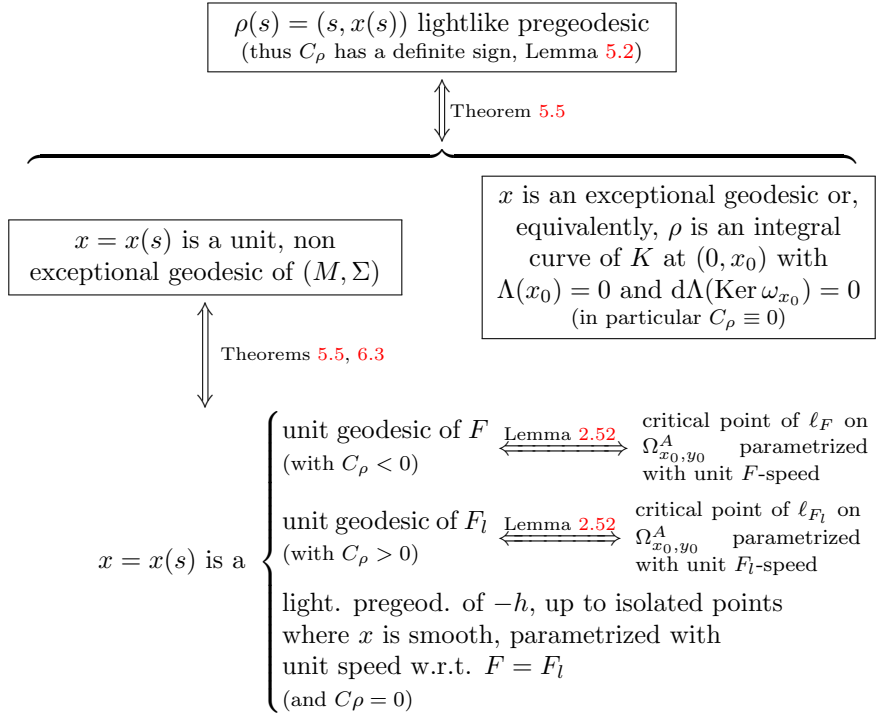
\begin{figure}[h]
		\[
		\xymatrix{\fbox{\begin{minipage}{6cm}\centering $\rho(s)=(s,x(s))$ lightlike pregeodesic \\ \footnotesize (thus $C_\rho$ has a definite sign, Lemma~\ref{tparam})\end{minipage}}\ar@{<=>}[d]^-{\text{Theorem~\ref{existenceofbolasNO}}}\\
			\protect\overbrace{\makebox[11cm]{}}}\]\vspace{-0.3cm}
		\[
		\xymatrix{\fbox{\begin{minipage}{5cm}\centering
					$x=x(s)$ is a unit,   non exceptional  geodesic of $(M,\Sigma)$\end{minipage}}\ar@{<=>}[d]^-{\text{  Theorems~\ref{existenceofbolasNO}, \ref{minimizers}  }}&
			\fbox{\begin{minipage}{5cm}\centering $x$ is an exceptional geodesic or, equivalently, $\rho$  is an   integral curve of $K$  at  $(0,x_0)$ with $\Lambda(x_0)=0$ and $\de \Lambda(\mathrm{Ker}\,\omega_{x_0})=0$\\
					\footnotesize(in particular $C_\rho\equiv 0$)
			\end{minipage}}\\
			&}\]\vspace{-0.3cm}
		\[\quad\quad\quad\quad \quad x=x(s) \text{ is a }\left\{\mbox{\begin{minipage}{11cm}\mbox{\begin{minipage}{2.75cm}unit geodesic of $F$\\ {\footnotesize (with $C_\rho<0$)}\end{minipage}}
				$\xLeftrightarrow{\text{Lemma~\ref{criticallength} 
				}}\,$
				\mbox{\begin{minipage}{3cm} \footnotesize critical point of $\ell_F$  on $\Omega^A_{x_0,y_0}$  parametrized with unit $F$-speed \end{minipage}} \vspace{0.2cm} \\ \mbox{\begin{minipage}{2.75cm}unit geodesic of $F_l$\\ {\footnotesize (with $C_\rho>0$)}\end{minipage}}
				$\xLeftrightarrow{\text{Lemma~\ref{criticallength}}}\,$
				\mbox{\begin{minipage}{3cm} \footnotesize critical point of $\ell_{F_l}$  on $\Omega^A_{x_0,y_0}$  parametrized with unit  $F_l$-speed \end{minipage}} \vspace{0.2cm} \\ light. pregeod. of $-h$, up to isolated points \\  where $x$ is smooth,   
				parametrized with \\unit speed w.r.t. $F=F_l$
				\\  {\footnotesize (and  $C\rho=0$)}\end{minipage}}\right.\]
		\caption{The relations between future-pointing lightlike pregeodesics of an \sstk  splitting and geodesics of the associated
			Fermat structure. Here $\rho$ is a $t$-parametrized curve in the \sstk splitting  ($s\in [a,b], \ x_0=x(a),\ y_0=x(b)$).
			Moreover, maximizing lightlike
			pregeodesics (i.e., whose points are horismotically related) which are not
			reparametrizations of flow lines of $K$ correspond to extremizing geodesics of the
			Fermat structure, Corollary~\ref{horismos}.}\label{geos}
	\end{figure}
	
	\subsection{The role of the different splittings of an \sstk spacetime.}\label{ss6.2}
	Observe that
	given an \sstk  spacetime, for every spacelike hypersurface which
	intersects all the orbits of the Killing  field there  will exist
	a different splitting \eqref{lorentz} as an \sstk   (with the same
	Killing vector field). Let us characterize when a transversal
	hypersurface is spacelike in terms of the  Fermat structure.
	\begin{lemma}\label{fsplitting}
		Let $(\R\times M,g)$ be an \sstk  splitting with $g$  given by
		\eqref{lorentz} and $f:M\rightarrow \R$, a smooth function. Then
		$S^f=\{(f(x),x)\in \R\times M: x\in M\}$ is a spacelike
		hypersurface if and only if one of the following two
		exclusive possibilities occurs: either
		\begin{equation}\label{spaceineq}
			\df f(v)<F(v) \text{ for every $v\in  A \cup  A_E $,}
		\end{equation}
		or the Killing field $K=\partial_t$ is spacelike everywhere
		and
		\begin{equation}\label{spaceineq2}
			\df f(v)>F_l(v) \text{ for every $v\in   A\cup  A_E$}.
		\end{equation}
	\end{lemma}
	\begin{proof}
		Observe that the tangent space to $S^f$ at $(f(x),x)\in S^f$ is given by
		\[T_{(f(x),x)}S^f=\{(\df f(v),v):v\in T_xM\},\]
		and, then,
		\[g((\df f(v),v),(\df f(v),v))=g_0(v,v)+2\omega(v) \df f(v)-\Lambda \df f(v)^2,\]
		so that $S^f$ is spacelike if and only if
		\begin{equation}\label{spacecond}
			g_0(v,v)+2\omega(v) \df f(v)-\Lambda \df f(v)^2>0
		\end{equation}
		for every  $v\in TM\setminus  \mathbf{0}$.  Now,  if $\Lambda(x)>0$,   \eqref{spacecond} is equivalent to
		$-\tilde{F}(v)<\df f(v)<F(v)$ for every $v\in T_xM \setminus\{0\}$, and this is
		equivalent to $\df f(v)<F(v)$ (because $-\tilde F(v)=-F(-v)$ for
		every $v\in T_xM\setminus\{0\}$, recall  Remark~\ref{rformalreverse}).  If $\Lambda(x)=0$ then,  \eqref{spacecond} is
		equivalent to  $\df f (v)<F(v)=-\frac{g(v,v)}{2\omega(v)}$,  for all
		$v\in A_x=\{v\in T_xM: -\omega(v)>0\}$, 
		since it is satisfied for each $v\neq 0$ belonging to the kernel of
		$\omega$,  while on
		$-A_x$ it becomes $\df f (v)> -\frac{g(v,v)}{2\omega(v)}=-F(-v)$.
		Hence, we  conclude that,  when $\Lambda(x)\geq 0$,
		\eqref{spacecond} is satisfied if and only if \eqref{spaceineq}
		holds on $A$.
		
		If $\Lambda(x)<0$, \eqref{spacecond} is satisfied
		away from $\overline{A_x\cup (-A_x)}\setminus\{0\}$, while on $(A_E)_x(=\{v\in T_x M:
		-\omega(v)>0,\ \Lambda(x)g_0(v,v)+\omega(v)^2\geq 0\})$ it is
		equivalent either  to
		\begin{equation}\label{espatial}
			\df f(v)<F(v) \qquad \hbox{or  to } \qquad \df f(v)>F_l(v). \end{equation}
		On $-(A_E)_x$,  the required  conditions are satisfied iff they  are
		satisfied on  $(A_E)_x$,  so that \eqref{espatial} suffices.
		
		By a simple continuity argument, it follows that both conditions
		in \eqref{espatial} cannot hold for different tangent vectors (at
		the same or at different points). Then, if the second inequality
		holds, $\Lambda<0$ in $M$ and the Killing field is spacelike
		everywhere.
	\end{proof}
	\begin{rem}  Geometrically, the meaning of the two
		possibilities in the lemma is the following. When the tangent
		space $T_xS^f$ is naturally included in   $T_{(f(x),x)}L$, 
		$(L=\R\times M)$ the latter is divided in two open half spaces. If
		\eqref{spaceineq} holds, then the future-pointing vectors and the
		Killing $\partial_t$ lie in the same half space, but when
		\eqref{spaceineq2} holds they lie in different ones (see Fig.~\ref{semispazio}).
		In the latter case, we can follow the proof of Proposition~\ref{psstk} and
		choose $S^f$ as the spacelike hypersurface $S$ which allows us  to
		write the spacetime as an \sstk splitting. Then the corresponding
		projection $t^f: L\rightarrow \R$ still satisfies that $-\nabla
		t^f$ is timelike, but we cannot assume that it is future-pointing 
		(as the time-orientation had already been prescribed). Indeed,
		$-t^f$ (no $t^f$) is a temporal function now.
	\end{rem}
	\begin{figure}[h]
		\includegraphics[scale=1,center]{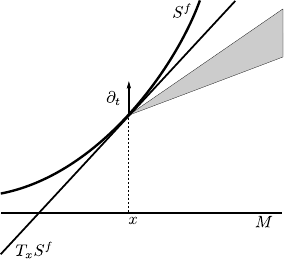}
		\caption{A hypersurface $S^f$ in an \sstk splitting $L=\R\times M$ satisfying condition \eqref{spaceineq2}.  The shaded region represents the future-pointing causal vectors in $T_{(f(x),x)} L$}\label{semispazio}
	\end{figure}
	Now, let $(\R\times M,g)$ be an \sstk  splitting and fix some
	$f:M\rightarrow\R$ under the hypotheses  \eqref{spaceineq} or
	\eqref{spaceineq2} of the lemma. Define the spacelike hypersurface
	$S^f$ of $(\R\times M,g)$ and denote the new \sstk  splitting
	as $(\R\times M,g^f)$, where
	\begin{equation}\label{gf}
		g^f((\tau,v),(\tau,v))=g((df(v)+\tau,v),(df(v)+\tau,v))
	\end{equation}
	for $(\tau,v)\in \R\times TM$.
	\begin{convention}\label{sigmaf}
		According to the remark above, $\Sigma^f$ will denote the
		Fermat structure associated with $(\R\times M,g^f)$ when $f$
		lies in the case \eqref{spaceineq} and its reverse Fermat
		structure (see comment before Corollary~\ref{rtimereversal})  when $f$ lies in the case \eqref{spaceineq2}.
	\end{convention}
	\begin{prop}\label{changedf}
		With the above notation,
		\begin{enumerate}[(i)]
			\item if $F$ and $F_l$ are the conic pseudo-Finsler  metrics associated with
			$\Sigma$, then $F^f=F-df$ and $F_l^f=F_l-df$ are the  conic pseudo-Finsler  metrics
			associated with $\Sigma^f$,
			\item $\Sigma^f$ has the same geodesics
			as $\Sigma$ up to parametrization.
		\end{enumerate}
	\end{prop}
	\begin{proof}
		Observe that
		$g^f((\tau,v),(\tau,v))=g_0^f(v,v)+2\omega^f(v)\tau-\Lambda^f
		\tau^2$, where
		\begin{equation}\label{eauxi0} \begin{array}{rll} g_0^f(v,v) & = & g((df(v),v),(df(v),v))=g_0(v,v)+2\omega(v)
				df(v)-\Lambda df(v)^2,
				\\
				\omega^f(v) & = & g((df(v),v),(1,0))=\omega(v)-\Lambda
				df(v)\end{array}
		\end{equation} and $\Lambda^f=\Lambda$; in particular, the metric
		$h$ in \eqref{eh} remains invariant:
		\begin{equation}\label{eauxi} \Lambda
			g_0^f(v,v)+\omega^f(v)^2=\Lambda g_0(v,v)+\omega(v)^2.
		\end{equation}
		
		$(i)$ When $\Lambda(x)=0$, the equality $F^f=F-df$ follows
		directly from \eqref{eauxi0} (recall \eqref{fermat-kropina}). For
		the case $\Lambda\not=0$, just notice that the expressions of $F,
		F_l$ in \eqref{randers-kropina}, \eqref{rk2} can be rewritten as
		in the first  identity of \eqref{tau}, and use \eqref{eauxi}.
		
		$(ii)$ By Theorem~\ref{minimizers}, the geodesics
		of $\Sigma^f$ are either geodesics of $F^f$,  geodesics of
		$F_l^f$, lightlike pregeodesics of $\Lambda
		g_0^f(v,v)+\omega^f(v)^2$  (except at isolated points) or constant curves with $\Lambda^f(x)=0$ and $d\Lambda^f({\rm Ker}\, \omega^f_x)\equiv 0$. In the two last cases,  they are pregeodesics
		of $\Sigma$  by  \eqref{eauxi}  and because $d\Lambda({\rm Ker}\, \omega_x)\equiv 0$ if and only if $d\Lambda^f({\rm Ker}\, \omega^f_x)\equiv 0$ in the points where $\Lambda^f(x)  (=\Lambda(x))=0$.  In the other cases, the
		length functionals of $F^f=F-df$ and $F$ (resp. $ F_l^f$ and
		$F_l$), when defined on the space of $F$-admissible  (or equivalently $F^f$-admissible)  curves connecting
		two prescribed points, are the same up to a
		constant. Therefore, both functionals 
		have the same critical points and these critical points are
		pregeodesics of  $(M,\Sigma)$ by Theorem~\ref{minimizers}.
	\end{proof}
	Previous interpretations allow  us  to refine   the conclusions of
	Proposition~\ref{c63} in the Randers-Kropina case.
	\begin{cor}\label{conectedness}
		Let $(M,F)$ be a Randers-Kropina metric. If
		the intersection of any closed forward ball and any closed backward one is compact 
		then there exists a new Randers-Kropina metric $F^f$ which is
		geodesically complete and has the same
		pregeodesics of $F$.
	\end{cor}
	\begin{proof}
		As the associated \sstk  splitting $(\R\times M,g)$ is globally
		hyperbolic by part  $(ii)$ of Theorem~\ref{kropinaLadder},   there exists  a
		smooth  spacelike Cauchy hypersurface $S$ (see \cite{BerSan05}).
		As the integral curves of $\partial_t$ are causal, they must 
		intersect $S$ transversely  and, so, $S$ can be written as a
		graph $S^f$.  By Proposition~\ref{changedf}, the associated
		splitting $(\R\times M,g^f)$ has a Fermat structure $\Sigma^f$
		which must be associated with some Randers-Kropina metric
		$F^f=F-df$, and will have the same pregeodesics as $F$. By part
		$(iii)$ of Theorem~\ref{kropinaLadder},  $\Sigma^f$ is geodesically
		complete.
	\end{proof}
	\begin{rem}
		In the stationary case $\Lambda>0$, the function $f$ can be
		explained physically as a (new) synchronization of the
		``observers'' travelling along the integral curves of
		$\partial_t$.  In the case $\Lambda\geq 0$,   Corollary~\ref{conectedness}  extends
		\cite[Theorem 5.10]{CapJavSan10}  valid  for Randers manifolds (namely, if
		$R$ is a Randers metric and the closed symmetrized balls are
		compact then there exists a function $f$ such that $R-df$ is also
		Randers, plus complete and with the same pregeodesics as  $R$). As
		suggested in that reference and proved in \cite{Matvee13}, such
		a result can be extended  from Randers to  any Finsler   manifold. \soutE{It would be
		interesting to know if Corollary~\ref{conectedness} could also be
		extended to more general wind Finslerian  
		structures. In this direction,} \bw Anyway, notice that \ew in the case of strong wind an additional difficulty occurs in the proof of Corollary~\ref{conectedness}: as the integral curves of $\partial_t$ become spacelike, it is not guaranteed that they will cross the Cauchy hypersurface $S$.
	\end{rem}
	
	\subsection{ Precedence relation and solution to Zermelo problem}\label{ss6.3}
	In \cite{JavSan11}, the authors introduced a notion of
	{\em precedence, $\prec$,} for a conic pseudo-Finsler metric (see Section~\ref{Fsepa}).  This can be
	easily extended  to any  wind Finslerian  structure $(M,
	\Sigma)$   and, moreover, also a less restrictive relation $\preceq$ appears naturally, so that $\prec$ and $\preceq$ resemble, respectively, the chronological and the causal  relations  in a Lorentzian manifold. Namely,   
	for  any $x,y\in M$, we say that $x\prec y$ (resp. $x\preceq y$) if there exists an  $F$-wind (resp. wind)  curve connecting $x$ to $y$ (i.e $C^A_{x,y}\neq\emptyset$, resp. $C^\Sigma_{x,y}\neq\emptyset$).  Moreover,  for any $x\in M$,
	the {\em $F$-future} (resp. {\em $\Sigma$-future}, {\em $F$-past}, {\em $\Sigma$-past}) of $x$ is
	the set  $I_\Sigma^+(x)$ (resp. $J_\Sigma^+(x)$, $I_\Sigma^-(x)$, $J_\Sigma^-(x)$) defined  as
	$I_\Sigma^+(x)=\{y\in M:  x\prec y\}$ (resp. $J_\Sigma^+(x)=\{y\in M:  x\preceq y\}\cup\{x\}$,
	$I_\Sigma^-(x)=\{y\in M:  y\prec x \}$, $J_\Sigma^-(x)=\{y\in M:  y\preceq x \}\cup\{x\}$).
	The following result summarizes the relations between the above-defined $I^\pm_\Sigma, J^\pm_\Sigma$, and the corresponding future or past sets for the Lorentzian metrics $g$ in $\R\times M$ and $-h$ in $M_l$.
	\begin{prop}\label{esalva}
		For a wind Riemannian structure with associated  \sstk splitting
		$(\R\times M,g)$ and natural projection $\pi: \R \times
		M\rightarrow M$:
		\begin{align*}
			&I^+_\Sigma(x) = \pi\left(I^+(0,x)\right), &I^-_\Sigma(x)=\pi\left(I^-(0,x)\right),\\
			&J^+_\Sigma(x) = \pi\left(J^+(0,x)\right), &J^-_\Sigma(x)=\pi\left(J^-(0,x)\right),
		\end{align*}
		where  $x\in M, (0,x)\in \R\times M$. In particular,
		$I^\pm_\Sigma(x)$
		are open subsets.
		
		In the case of strong wind  ($M=M_l$), $I^+_\Sigma(x)$ and
		$I^-_\Sigma(x)$ coincide, resp., with the chronological future and
		past of $x$ for the Lorentzian metric $-h$
		on $M$ endowed with a natural time-orientation.
	\end{prop}
	\begin{proof}
		We shall prove the proposition only for the future sets $I^+_{\Sigma}$,  being the proof for $I^-_{\Sigma}$ completely analogous.   Let $y\in I^+_{\Sigma}(x)$,  $\sigma \in C^A_{x,y}$.  As
		$\dot\sigma^\pm(t)\in A$,  then
		$F\big(\dot\sigma^{\pm}(t)\big)<F_l\big(\dot\sigma^{\pm}(t)\big)$  and $y\in
		B^+_{\Sigma}(x,r)$, $r\in (\ell_F(\sigma),\ell_{F_l}(\sigma))$. Thus, Proposition~\ref{bolas2} ensures that $(r,y)\in
		I^+(0,x)$. For the converse inclusion, just  recall that if $(0,x)\ll (r,y)\in \R\times M$ then there exists a future-pointing timelike curve $\gamma(s)=(t(s),\sigma(s))$ between $(0,x)$ and $(r,y)$ with
		$\dot\sigma(s)\neq 0$ for all $s$, see footnote~\ref{foot4.1}.   Reparameterizing $\gamma$ with respect to $t$ gives an $F$-wind curve  between
		$x$ and $y$ and $y\in I^+_{\Sigma}(x)$. 
		
		The proof of the inclusions for  $J^{\pm}_\Sigma$  are analogous except for a slight difference. In fact,  for the inclusion $\pi\left(J^+(0,x)\right)\subset J^+_{\Sigma}(x)$, observe  that if $(0,x)< (r,y)$, the existence of a causal curve 
		between $(0,x)$ and $(r,y)$, such that $\dot\sigma(s)\neq 0$ for all $s$, is guaranteed except when $x=y$ and $\gamma$ is  a lightlike pregeodesic, but in that case $y\in J^+_{\Sigma}(x)$, by definition. 
		For the last assertion,
		recall Proposition~\ref{htimeoriented}. 
	\end{proof}
	\begin{exe}
		Even though  $\pi: \R\times M \rightarrow M$ is an open map (in agreement with the fact that $I^\pm_\Sigma(x)$ must be open as so is $I^\pm(0,x)$) if,  say, $J^+(0,x)$ is closed, then $J^+_\Sigma(x)$ is not necessarily closed. 
		Typically, this happens for ``black hole regions'' (recall Section~\ref{ss6.5})  even in globally hyperbolic \sstk splittings. Namely, the causal future of a point $(0,x)$ inside the black hole is closed (as the spacetime is globally hyperbolic) but its causal future may approach the horizon and never touch it. So, the projection  $J^+_\Sigma(x)$ of $J^+(0,x)$ admits as boundary points the projection of points of the horizon, but these points do not belong to $J^+_\Sigma(x_0)$.
	\end{exe}
	
	One could define the $F$-separation $d_F$ in an analogous   way  as
	in the Randers-Kropina case, by using the infimum of the lengths of the
	$F$-admissible (or $\Sigma$-admissible) curves connecting each two
	points, as well as a Lorentzian separation $d^{F_l}$ by taking the
	supremum of the $F_l$-lengths. In fact, if $F$-admissible curves are
	taken, then $d_F$ lies again in the case of the $F$-separation
	defined for any conic Finsler structure in \cite{JavSan11}.
	However, as a difference with the Randers-Kropina case, now
	discontinuities of $d_F$ may appear in non-trivial cases.
	\soutE{(typically,  when $x\preceq y$ but $x\not\prec y$; such
	discontinuities would remain if
	$\Sigma$-admissible curves were used instead of $F$-admissible ones). As this would affect possible results involving boundary geodesics, we prefer not to
	follow this approach here.  On the contrary,} \bw Even though some properties can be proven in this case by using  a kind of extended Finsler separation (see Section 9.2), we will \ew ensure  directly the
	existence of geodesics and their extremizing properties by using
	the previously introduced notions.

	
	\begin{thm}\label{compactcase}
		Let $(M,\Sigma)$ be a w-convex  wind Riemannian structure and  let  $x_0,y_0 \in M $ such that $y_0\in
		J_\Sigma^+(x_0)\setminus{\{x_0\}}$.  Then:
		\begin{enumerate}[(i)]
			\item  There exists a  global minimum $\sigma$ 
			on $C^\Sigma_{x_0,y_0}$ of the length functional $\ell_F$ which is a pregeodesic of $(M,\Sigma)$,   and,  when
			$y_0\not\in I_\Sigma^+(x_0)$,  it is  a lightlike
			pregeodesic of the Lorentzian metric $-h$   in \eqref{eh},   up to isolated points where its derivative vanishes. 
			\item If $ R(x_0,y_0):= \sup \{ r>0:  y_0\in  \hat B_\Sigma^+(x_0,r) \}<+\infty$,    
			there exists a global maximum $\sigma$ on $C^\Sigma_{x_0,y_0}$ of the length
			functional $\ell_{F_l}$ which is a pregeodesic of $(M,\Sigma)$,  and, when
			$y_0\not\in I_\Sigma^+(x_0)$,   it is  a lightlike pregeodesic of the Lorentzian metric
			$-h$  (with non-vanishing  derivative). 
		\end{enumerate}
	\end{thm}
	\begin{proof} $(i)$
		Consider the associated \sstk splitting $(\R\times M, g)$ and recall that
		w-convexity implies its causal simplicity  (see Theorem
		\ref{generalK}). Given $x_0, y_0$ as above, denote $ r(x_0,y_0): =\inf \{ r>0:
		y_0\in \hat B_\Sigma^+(x_0,r)\}$;
		notice that $0< r(x_0,y_0)<+\infty$ (the first inequality follows from  Proposition~\ref{bolas2} and the acausality of $S_0$,  the second one trivially follows from  $y_0\in J^+_{\Sigma}(x_0)$). 
		The definition of $ r(x_0,y_0)$ and Proposition~\ref{bolas2} imply that
		$(r_n,y_0)\in J^+(0,x_0)$ for some sequence $r_n\rightarrow  r(x_0,y_0)$
		with  $r_n\geq  r(x_0,y_0)$ and, moreover, $(r',y_0)\not\in J^+(0,x_0)$ if  $r'< r(x_0,y_0)$.
		So, $( r(x_0,y_0),y_0)$ lies in the boundary of  $J^+(0,x_0)$ and
		causal simplicity implies that this boundary is contained in
		$J^+(0,x_0)$. Thus, $(0,x_0)$ and $(r(x_0,y_0),y_0)$ are horismotically related  and
		any connecting causal curve from $(0,x_0)$ to $(r(x_0,y_0),y_0)$ must be a
		lightlike pregeodesic. The projection $\sigma$ on $M$ of such a
		pregeodesic  is an extremizing  pregeodesic of $(M,\Sigma)$ (recall
		Corollary~\ref{horismos});  moreover, $\sigma$ must be a global minimum of $\ell_F$ on $C^\Sigma_{x_0,y_0}$ otherwise a curve $\sigma_1\in C^\Sigma_{x_0,y_0}$ should exist such that $\ell_F(\sigma_1)<\ell_F(\sigma)=r(x_0,y_0)$. As $\ell_F(\sigma_1)\leq \ell_{F_l}(\sigma_1)$, $y_0\in \hat B_\Sigma^+(x_0,\ell_F(\sigma_1))$ in contradiction with the definition of $r(x_0,y_0)$.     In the case $y_0\not\in I_\Sigma^+(x_0)$, 
		the velocity of any
		connecting $\Sigma$-admissible curve must lie in $A_E\setminus A$
		at some point and, thus,   $\sigma$ must belong to case (ii)-(c) of  Theorem~\ref{minimizers} that  implies the  last conclusion. 
		
		$(ii)$  Notice that the additional hypothesis $R(x_0,y_0)<+\infty$
		allows us  to use the same technique as in the previous part in order to obtain a maximizing pregeodesic $\sigma$ of $(M,\Sigma)$. However, now the velocity of $\sigma$ cannot vanish at some (isolated) point because, in this case, $K$ would be lightlike at that point and one could concatenate an arbitrary segment of integral curve of $K$ at that point. Thus, $R(x_0,y_0)=+\infty$, a contradiction. 
	\end{proof}
	\begin{exe}
		Let us observe that if $y_0\in I^+_\Sigma (x_0)$ in the part $(i)$ of last theorem, this does not necessarily imply that the solution is a pregeodesic of $F$. This can happen for example  in  an \sstk whose associated $-h$ is Lorentzian and the slice is compact, as $(\R\times T^2,g)$ with $g= dt^2+dx^2+dy^2-\sqrt{2}(dxdt+dtdx)$ (we consider the  torus  $T^2$ as the quotient of $\R^2$ with the identifications $(x,y)\sim(x+1,y)$ and $(x,y)\sim(x,y+1)$). In this case, $h=dx^2-dy^2$ and $(T^2,-h)$ is totally vicious, so we have that $y_0\in I^+_\Sigma (x_0)$ for all $x_0,y_0\in T^2$. But the spacetime $(\R\times T^2,g)$ is globally hyperbolic and there always exists solution to the associated Zermelo problem, given by a boundary geodesic  in some cases as when we consider $x_0=(0,0)\in T^2$ and $y_0=(s,s)\in T^2$,  for small $s>0$.  
	\end{exe}
	From Proposition~\ref{c63}-$(ii)$, we have that w-convexity is satisfied if $M$ is compact, and then  we get immediately:
	\begin{cor}\label{compactcase2}
		Let $\Sigma$ be the wind Riemannian structure determined by a compact Riemannian manifold $(M,g_R)$ and  a vector field $W$ on $M$.  Then for any couple of points $x_0, y_0\in M$, $x_0\neq y_0$, there exists a curve  in $C^\Sigma_{x_0,y_0}$ which is a minimum of $\ell_F$  provided that at least a  wind curve from $x_0$ to $y_0$ exists.
	\end{cor}

	Theorem~\ref{compactcase}   gives the last step in the solution to Zermelo's
	navigation theorem under any type of  (time-independent) wind $W$ in a Riemannian background $(M,  g_R )$. 
	As far as we know, the description of this problem as a Finslerian  geodesic connectedness problem appeared first in \cite{Sh03}, under the assumption that  the wind is  mild (apart \soutE{of} \bw from \ew its time-independence). The case of a wind which is everywhere critical  was considered \soutE{recently} in \cite{YosSab12}.
	Because of its importance, we summarize and discuss the general solution now.

	\begin{cor}[Summary of the solution to Zermelo's problem] \label{rsummaryZ}   Let $\Sigma$ be the wind Riemannian structure determined by a Riemannian manifold $(M,g_R)$ and  a vector field $W$ on $M$ and  let  $x_0\neq y_0 \in M $: 
		\begin{enumerate}[(i)]
			\item  If there exists a solution $\sigma$ to Zermelo navigation connecting $x_0$ to $y_0$, that is,  
			a regular
			wind curve from $x_0$ to $y_0$ which is a
			global minimum of the length functional $\ell_F$ 
			on 
			$C^\Sigma_{x_0,y_0}$, then $x_0\preceq y_0$ and $\sigma$ is a pregeodesic of $(M,\Sigma)$. 
			
			Moreover,  $\sigma$ is  either a pregeodesic for the conic Finsler metric $F$ (and, thus, $x_0\prec y_0$), or 
			a lightlike
			pregeodesic of the Lorentzian metric $-h$   in \eqref{eh},   up to isolated points where its derivative vanishes. 
			\item  If $x_0\preceq y_0$ and the wind Finslerian structure is w-convex,  then there exists a solution to Zermelo navigation from $x_0$ to $y_0$. 
		\end{enumerate}
	\end{cor}
	\begin{proof} 
		Recall that, by Definition~\ref{sigmadmissible}--(iii), $\sigma$ is a piecewise smooth, $\Sigma$-admissible curve whose left and right derivatives can vanish only at a finite number of instants $\sigma(s_j)$ (being also $0_{\sigma(s_j)}\in \Sigma_{\sigma(s_j)}$ and, by Convention~\ref{caestar}, $F(0_{\sigma(s_j)}=1$). Since $\ell_F(\sigma)<+\infty$, we    can assume that  $\sigma$ is reparametrized by using its $F$-length. 
		Then, $\sigma$ must be a unit extremizing $\Sigma$-geodesic (according to Definition~\ref{extremizing}) because, otherwise, there would be a connecting wind curve of strictly smaller length. So,  the result follows from the classification of $\Sigma$-geodesics in Theorem~\ref{minimizers}.
		For (ii), just use Theorem~\ref{compactcase}--(i).  
	\end{proof}
	In order to apply these results in a practical way, the following comments are in order (see also Figure~\ref{geos}): 
	
	\begin{enumerate}
		\item An obvious necessary condition for the existence of a Zermelo solution between $x_0, y_0$ $(x_0\neq y_0)$ is the possibility to travel from the first to the second point, that is, $y_0 \in J^+_\Sigma(x_0)$. This is a vacuous condition  when the wind is mild.  When the wind is strong (on all $M$), the question is
		reduced to study the causal future of $x_0$ for the Lorentzian metric $-h$; this  is a non-trivial but typical computation in spacetimes (see, e.g., Example~\ref{ex_con} and Proposition~\ref{always} below). When there are points with critical wind, then  the precedence relation must be studied specifically there; moreover, the possibility of travelling between two critical points by moving freely in the region of mild wind must be also taken into account. Even though, in principle, this may be done directly, from the spacetime viewpoint it becomes  equivalent to the existence of a future-pointing causal curve  
		connecting $(0,x_0)$ to $l_{y_0}=\R\times \{y_0\}$). 
		\item In the case that the trip is possible,   the possible solutions to Zermelo's navigation must be found in the set of geodesics for the conic Finsler metric $F$ and in the set of  piecewise smooth lightlike pregeodesics for $-h$ with $C^1$ zero velocity at the breaks  and non-vanishing second derivative there.

		The possibility of the existence of these last geodesics was pointed out by 
		Caratheodory  in \cite[\S 282]{Carath67}, who studied a time-independent wind on a plane.  Indeed, he  mentioned  the possibility of the existence of solutions which are limits  of maximal and minimal ones and called ``anomalous'' their possible velocity  vector fields.    In a more modern language, these solutions are called {\em abnormal extremals } (compare, e.g., with \cite[p.54-55]{userres}, where the Zermelo navigation problem on a plane is analysed). 
		
		We  have interpreted  abnormal extremals in three equivalent ways: (a) boundary geodesics of $(M,\Sigma)$, (b) lightlike pregeodesics (up to a finite number of instants where the velocity vanishes) of  $-h$, and (c) projections of certain lightlike geodesics in the associated SSTK spacetime.  
		
		From the spacetime viewpont, all  Zermelo solutions are projections of first arriving future-pointing lightlike pregeodesics  
		connecting $(0,x_0)$ to $l_{y_0}=\R\times \{y_0\}$. 
		
		\item  For the existence of maximizing geodesics,  assuming the obvious necessary condition $R(x_0,y_0)<+\infty$ (apart from  $y_0 \in J^+_\Sigma(x_0)$), 
		the possible maximizing curves must be found in the set of geodesics for the Lorentz-Finsler metric $F_l$ and in the set of  (smooth) lightlike pregeodesics of $-h$. 
		
		Notice that the maximizing geodesics, if they exist, must be entirely contained in the region of strong wind: otherwise, when one crosses a point of non-strong wind, one can concatenate a wind curve segment  so that the curve remains close to this point along an arbitrarily long time, before arriving at $y_0$.   
		
		From the spacetime viewpoint, all  maximizing curves are projections of last arriving future-pointing lightlike pregeodesics  
		connecting $(0,x_0)$ to $l_{y_0}=\R\times \{y_0\}$.

		\item In order to ensure the existence of extremals, the condition of  w-convexity  becomes the natural one: (a) it holds when $\Sigma$ is complete or any of the  conditions in parts $(i)$ and $(ii)$ of  Proposition~\ref{c63} holds, (b) it generalizes the classical notion of convexity for domains of Riemannian and Finslerian manifolds, and 
		(c) as in these geometries, it is related to the convexity of the boundary of the domain (see Theorem~\ref{lake} below).

		From the spacetime viewpoint, w-convexity becomes equivalent to the  causal simplicity of the spacetime, a standard causality condition. As in the previous cases, the interpretation in the SSTK has a double interest: (i) it may be easier to check, and (ii) it provides the 
		arrival times of the extremal geodesics (namely,  the $t$-coordinate at  boundary points in $l_{y_0}$ of $J^+(0,x_0)\cap l_{y_0}$).
		
		\item As in Riemannian Geometry, one can wonder at what extent all the $\Sigma$-pregeodesics from $x_0$ to $y_0$ are critical points for some length functional. The answer to this question (Theorem~\ref{fp}) is postponed to the study of a new general Fermat's principle for spacetimes and its adaptation for SSTK spacetimes in Section~\ref{further1}. 
	\end{enumerate}

	\begin{exe}\label{ex_con} The obvious connectivity condition $y_0\in
		J_\Sigma^+(x_0)\setminus{\{x_0\}}$ in Zermelo's problem may fail even if $M$ is compact. Indeed, consider a sphere $S^2$ with the natural metric induced by the Euclidean one and a smooth vector field which is given in spherical coordinates $(\theta, \phi)$ by 
		\[
		W(\phi,\theta)=\begin{cases}0&\theta \in [0,\pi/6]\cup[\pi/2,\pi]\\
			f(\theta)e_\theta&\theta\in(\pi/6, \pi/2)
		\end{cases}
		\]  
		where $e_\theta$ is the unit vector field associated with the latitude coordinate $\theta$ and the function $f\colon [\pi/6, \pi/2]\to [0,+\infty)$ is smooth, non-negative, equal to $0$ at $\theta=\pi/6,\pi/2$,  strictly increasing in $[\pi/6,\pi/4]$,  strictly decreasing otherwise and such that $f(\pi/4)>1$. Any wind curve from each point $x_0$ in the hemisphere containing the south pole cannot connect points in the open region containing the north pole and having as boundary the parallel of latitude $\bar\theta$, where $\bar\theta\in (\pi/4,\pi/2)$ is equal to $f^{-1}(1)$. In fact, along the line $p_\theta=\{(\phi,\theta):\theta=\bar\theta\}$, the set of admissible velocities for wind curves is included in  the tangent half-space containing $e_\theta$ plus the zero vectors and therefore any wind curve starting on the region $\Omega_S=\{(\phi,\theta):\theta>\bar\theta\}$ must turn back into $\Omega_S$ when arriving to a point on $p_\theta$ (compare with Proposition~\ref{always} below).  
	\end{exe}
	
	\subsection{ Further  results on existence of geodesics }
	Our methods can also be applied to find a solution of Zermelo navigation problem in quite a few interesting cases. 
	Let us start   considering an open subset of a wind Riemannian manifold $(M, \Sigma)$ whose boundary satisfies a convexity assumption. 
	%
	We recall first  some notions and results about convexity of the boundary of an open subset which have been studied in  \cite{BaCaGS11}. Let $D$ be  an open subset of a Finsler manifold $(M,F)$  with smooth boundary.  We say that $D$  has {\em locally  convex} boundary if  for each $x\in \partial D$ there exists some neighborhood of
	$0$ in $T_x\partial D$ whose images by both  the exponential maps of $F$ and of its reverse Finsler metric $\tilde F(v):=F(-v)$ do not intersect $D$.   This condition is equivalent to the infinitesimal convexity of $\partial D$    (related to its normal curvature   at any point $x\in\partial D$)   and will be referred here just as the (extrinsic) {\em convexity} of $\partial D$.  
	If $\partial D$ is   convex   and $x\in \partial D$ then, \cite[Lemma 3.5]{BaCaGS11}, there exists   a small enough convex ball (of the metric $F$) $B^+(x,\delta)$ such that  for each $x_1,x_2\in D\cap B^+(x,\delta)$ the (unique) geodesic in $B^+(x,\delta)$ which connects $x_1$ with $x_2$ is included in $D$. The following lemma is a refinement of that result.
	\begin{lemma}\label{refine}
		Let $D$ be  an open subset of a Finsler manifold $(M,F)$  with smooth,   convex boundary  and   $x\in \partial D$. Then for all  $x_1,x_2\in \bar D\cap B^+(x,\delta)$, $\delta>0$ small enough,  the (unique) geodesic in $B^+(x,\delta)$ which connects $x_1$ with $x_2$ is included in $\bar D$. In particular, if $x_1,x_2\in \partial D$ then it is either contained in $\partial D  \cap B^+(x,\delta)$ or it is contained in  $D\cap B^+(x,\delta)$, except for its endpoints,   and it is not tangent to $\partial D$ in the endpoints.   This last case always happens if at least   one of the points $x_1$, $x_2$   belongs to $D$. 
	\end{lemma}
	\begin{proof}
		Choosing $\delta$ as in the discussion above, only the case when at least one of the points $x_1, x_2$ belongs to $\partial D$ must be taken into account. 
		Take two sequences of points $\{x^1_k\},\ \{x^2_k\}\subset D\cap B^+(x,\delta) $ converging resp. to $x_1$ and $x_2$. Consider the geodesics $\gamma_k$ connecting $x^1_k$ with $x_k^2$ which are contained in $D\cap B^+(x,\delta)$,  \cite[Lemma 3.5]{BaCaGS11}.  By smooth dependence of geodesics in a convex neighborhood from endpoints, $\gamma_k$ converges (in the $C^2$-topology) to the geodesic $\gamma$ connecting   $x_1$ and $x_2$   in $B^+(x,\delta)$. Thus,  $\gamma$ is contained in $\bar D$   and it is tangent to $\partial D$ when it touches the boundary away from the endpoints. By the definition of local convexity, this easily implies that $\gamma$ is  either  contained in the boundary or it touches the boundary transversally at most in the endpoints.  
	\end{proof}
	By using the above lemma and the correspondence between \sstk spacetimes and wind Riemannian structures, we can prove the existence of a  solution to Zermelo navigation problem in an open subset $D\subset M$ such that $\partial D$ is compact. To this end, we need to consider wind curves whose image is contained in $\bar D$ and we will recall this  by  using the symbol $\Sigma|_{\bar D}$. For example, given $x_0,y_0\in \bar D$, $C^{\Sigma|_{\bar D}}_{x_0,y_0}$ denotes the subset of wind curves from $x_0$ to $y_0$ with image in $\bar D$. 
	\begin{thm}\label{lake}
		Let $(M,\Sigma)$ be a  wind Riemannian structure, $D\subset M$ be a precompact, open subset with   smooth  boundary $\partial D$, and let $x_0, y_0 \in \bar D$ such that $y_0\in
		J_{\Sigma|_{\bar D }}^+(x_0)\setminus{\{x_0\}}$. Assume that the wind is mild on $\partial D$ and that this boundary is convex for $F$.    Then  there exists a global minimum $\sigma$
		on $C^{\Sigma|_{\bar D}}_{x_0,y_0}$ of the length functional $\ell_F$   and it fulfils one of the following two possibilities:
		\begin{enumerate}[(a)]
			\item $\sigma$ is fully contained in $\partial D$
			and, thus, it is a geodesic of both, $F$ and the Finsler metric induced by $F$ on $\partial D$; 
			\item $\sigma$ is contained in $D$ except, at most, its endpoints. 
		\end{enumerate}
		In particular, this last case happens when   one of the points $x_0$, $y_0$   belongs to $D$.   Moreover if (b) occurs: (i) $\sigma$ is  a pregeodesic of $(M,\Sigma)$ and  (ii) when $y_0\not\in I_{\Sigma|_{\bar D}}^+(x_0)$, then $\sigma$ is also  
		a lightlike pregeodesic of the Lorentzian metric $-h$   in \eqref{eh},    up to isolated points where its derivative vanishes.  
	\end{thm} 
	\begin{proof}
		As $y_0\in
		J_{\Sigma|_{\bar D }}^+(x_0)\setminus{\{x_0\}}$, the set of wind curves between $x_0$ and $y_0$ whose image is contained in $\bar D$ is not empty. We want to find a curve $\sigma\in  C^{\Sigma|_{\bar D}}_{x_0,y_0}$ which attains the infimum 
		\[T_0(x_0,y_0):=\inf_{\sigma\in C^{\Sigma|_{\bar D}}_{x_0,y_0} }\ell_F(\sigma).\] 
		Consider an \sstk spacetime $(\R\times M,g)$ associated with $(M,\Sigma)$ (see part $(i)$ of  Theorem~\ref{tfermatSSTK}) and the subset  of curves
		\begin{multline*} C^{cc|_{\bar D}}_{x_0,y_0}=\{\gamma:[a_\gamma,b_\gamma]\rightarrow \R\times \bar D\subset\R\times M: \text{$\gamma$ future-pointing causal continuous,}\\ \quad\quad\quad \quad\quad\quad \gamma(a_\gamma)=(0,x_0),\ \pi\big(\gamma(b_\gamma)\big)=y_0,\ a_\gamma<b_\gamma\}
		\end{multline*}
		(recall Definition~\ref{limitcurvedef} and the paragraph below it for the notion of causal continuous curve). Now define 
		\[T_1(x_0,y_0):=\inf_{\gamma\in  C^{cc|_{\bar D}}_{x_0,y_0} }T(\gamma),\] 
		where $T(\gamma)$ is the arrival time, namely, the first coordinate of $\gamma(b)$. Observe that $T_1(x_0,y_0)\leq T_0(x_0,y_0)$, since each curve in $C^{\Sigma|_{\bar D}}_{x_0,y_0}$ can be lifted to a future-pointing lightlike curve $\gamma(t)=(t,\sigma(t))$ such that $\ell_F(\sigma)=T(\gamma) $ (see the proof of Proposition~\ref{bolas2}). It is enough to prove that the infimum $T_1(x_0,y_0)$ is attained by a future-pointing lightlike pregeodesic which, by Theorem~\ref{existenceofbolasNO}, projects into a pregeodesic $\sigma$ of $(M,\Sigma)$ with $\ell_F(\sigma)=T_1(x_0,y_0)\leq T_0(x_0,y_0)$. Take a sequence of curves $\{\gamma_k\}$ in $C^{cc|_{\bar D}}_{x_0,y_0} $ such that $\lim_k T(\gamma_k)=T_1(x_0,y_0)$, which can be assumed parametrized by the first coordinate, namely, $\gamma_k(t)=(t,\sigma_k(t))$. Then by Lemma  
		\ref{limitcurve}, there exists a limit curve $\gamma(t)=(t,\sigma(t))$ defined on $[0,T_1(x_0,y_0)]$ which is future-pointing causal continuous. Let us see that $\gamma$ is a future-pointing lightlike geodesic: 
		
		Case (i): let us first consider  an instant $t_0\in (0,T_1(x_0,y_0))$ such that $\sigma(t_0)\in D$. Then there exists $\varepsilon>0$ such that $\gamma([t_0-\varepsilon,t_0+\varepsilon])\subset D$. Moreover, if $\gamma|_{[t_0-\varepsilon,t_0+\varepsilon]}$ is not a future-pointing lightlike pregeodesic, there exists a smooth causal curve $\beta$ from $\gamma(t_0-\varepsilon)$  to $(t_0+\varepsilon-\epsilon,\sigma(t_0+\varepsilon))$ for some $\epsilon>0$. In order to prove the existence of $\beta$ recall that by definition of causal continuous curve we can find a piecewise smooth causal curve close to $\gamma|_{[t_0-\varepsilon,t_0+\varepsilon]}$. Then by \cite[Proposition 10.46]{O'neill}, if it is not a future-pointing lightlike pregeodesic, we can find a future-pointing timelike curve from $\gamma(t_0-\varepsilon)$  to $\gamma(t_0+\varepsilon)$ and the conclusion follows using that the chronological relation is open.  Now consider the concatenation $\tilde\gamma=\gamma|_{[0,t_0-\varepsilon]}\ast\beta\ast \bar\gamma$, where $\bar\gamma(t)=(t-\epsilon,\sigma(t))$ is defined in $[t_0+\varepsilon,T_1(x_0,y_0)]$. It turns out that $T(\tilde\gamma)=T_1(x_0,y_0)-\epsilon$, a contradiction.   
		
		Case (ii): assume now that $\gamma(t_0)\in \partial D$ for $t_0\in (0,T_1(x_0,y_0))$ and consider a ball $B^+(\gamma(t_0),\delta)$ as in Lemma~\ref{refine}. There exists $\varepsilon>0$ such that $\gamma|_{[t_0-\varepsilon,t_0+\varepsilon]}$ is contained in $B^+(\gamma(t_0),\delta)$. Moreover, $\sigma|_{[t_0-\varepsilon,t_0+\varepsilon]}$ minimizes the $F$-length from $\sigma(t_0-\varepsilon)$ to $\sigma(t_0+\varepsilon)$, because otherwise if there exists  a shorter curve $\tilde\sigma$, one can construct a causal continuous curve by concatenation as in case (i), using $\beta(t)=(t,\tilde\sigma(t))$ and $\epsilon=\ell_F(\sigma)-\ell_F(\tilde\sigma)$. Being $\sigma|_{[t_0-\varepsilon,t_0+\varepsilon]}$ $F$-minimizing, we conclude by Lemma~\ref{refine} that it is a geodesic contained in $\partial D$. 
		
		Case (ii) also implies  conditions (a) and (b), and the final statement follows from part (i) of Theorem~\ref{compactcase}.

	\end{proof} 
	%
	%
	%
	\begin{rem}\label{unboundomain}
		The last result of convexity can be extended  with the same technique to more general cases. For example,  when    the wind is mild outside a compact subset $K\subset D$, $\partial D$ is  convex, and $D\setminus K$ is forward (or backward) complete, in the sense that any $F$-geodesic $\gamma:[0,b)\rightarrow D\setminus K$  which is  inextendible to $b$ in $D\setminus K$, either converges to some point in $\partial D \cup K$ or satisfies $b=\infty$. 
		Recall also that another characterizations of this property can be easily obtained from Theorem  2.1 (a),(b) and (e) in \cite{CapJavSan10}   (this  includes the compactness of the closed forward balls,  which makes possible  to reduce the non-compact case to the solved one in Theorem~\ref{lake}). 
		
	\end{rem}

	Now, let us focus on the case of strong wind. i.e., $M=M_l$. In the compact case, the unique condition for the existence of Zermelo solutions is 
	the assumption of precedence, and it is easy to find conditions ensuring   it for any two points.  
	
	\begin{prop}\label{always} 
		Let $(M, \Sigma)$ be a wind Riemannian structure with strong wind and compact $M$. If the Lorentzian metric $-h$  is endowed with a timelike conformal Killing vector field then $y_0\in I^{\pm}_\Sigma(x_0)$, for any couple of points $x_0, y_0\in M$. Therefore, Zermelo navigation problem has always a solution for any couple of points.  
		
		In particular, such a vector field exists  if the strong wind data $(g_R,W)$ satisfy that $W$ is Killing for $g_R$. 
	\end{prop}
	\begin{proof}
		By  \cite[Th. 1]{sanche06} a compact Lorentzian manifold $M$ with a timelike conformal Killing vector field is totally vicious, i.e. the chronological future and past of every point $x_0\in M$ are equal to $M$. Since for a strong wind  the chronological relation on $(M, -h)$ coincides with the precedence relation on $(M, \Sigma)$ (recall last part of Proposition~\ref{esalva}) the result  follows (the  assertion on Zermelo follows from Theorem~\ref{compactcase}-(i)).   
		For the last assertion,
		just notice that $W$ must be then also timelike and Killing for $-h$ (use \eqref{eh} with $g_R=g_0$,  $\Lambda=1-g_R(W,W)$ and $\omega=-g_R(\cdot, W)$). 
	\end{proof}
	

	
	For the existence of  maximizing
	geodesics,  the  (obviously necessary) hypothesis on $R(x_0,y_0)$  in  Theorem~\ref{compactcase}(ii), never can hold  if $y_0\not\in M_l$.  However, the
	next lemma provides a  natural sufficient condition. 
	\begin{lemma}\label{lolo}
		Let $x_0,y_0\in M$ and  assume that there exists $\bar r>0$ such that $y_0\in  \hat B_\Sigma^+(x_0,\bar r)$. Then $R(x_0,y_0)<+\infty$ 
		whenever (i)
		the wind is strong (i.e. $M=M_l$), and 
		(ii) the metric $-h$ is globally hyperbolic.
	\end{lemma}
	\begin{proof} As global hyperbolicity is preserved by conformal
		changes, $(M,\frac{1}{\Lambda}h)$ and $(M,-\frac{1}{\Lambda^2}h)$
		are also globally hyperbolic  (recall that $\Lambda<0$ if $M=M_l$). 
		Since $y_0\in \hat B^+_\Sigma(x_0,\bar r)$, $x_0$ and $y_0$ are causally related in $(M,-h)$. Thus, by global hyperbolicity,  the  length w.r.t. the Lorentzian metric $-\frac{1}{\Lambda^2}h$
		of all  the future-pointing $-h$-causal curves between them is  bounded  (see \cite[Lemma 4.5]{BeEhEa96}). Moreover, since   they  are causal curves,  the length w.r.t. the Riemannian metric  $-\frac{1}{\Lambda}g_0$  is bounded as well
		(see \cite[p. 76]{BeEhEa96}). 
		As, 
		$\frac{1}{\Lambda^2}\omega(v)^2=\frac{1}{\Lambda^2}
		h(v,v)-\frac{1}{\Lambda}g_0(v,v)$, then
		\[ \big|\frac{1}{\Lambda}\omega(v)\big|\leq \sqrt{\frac{1}{\Lambda^2} h(v,v)}+\sqrt{-\frac{1}{\Lambda}g_0(v,v)}\]
		for every causal $v\in TM$, which implies that the $F_l$-length of  all the 
		$\Sigma$-admissible curves between those two points is bounded  (recall that $F_l=\sqrt{h/\Lambda^2}+\omega/\Lambda$, Fig.~\ref{rela})  and
		consequently $R(x_0,y_0)<+\infty$.
	\end{proof}
	
	\begin{exe}\label{rolo} The global hyperbolicity of $-h$ is {\em not} implied by the global
		hyperbolicity of the \sstk  splitting:  
		a counterexample is any \sstk splitting with compact slices such that $K$ is spacelike; indeed,  the SSTK spacetime is  globally hyperbolic (apply part $(iii)$ of Theorem~\ref{generalK}), but
		$-h$ can never be globally hyperbolic (as  compactness implies that it admits closed timelike curves, i.e., $-h$ is not chronological). An explicit counterexample is the Lorentzian  cylinder $\R\times S^1$, $g=dt^2
		-4dtd\theta +d\theta^2$. 
		However, the next
		lemma shows that the converse holds.
	\end{exe}
	
	\begin{lemma}\label{lolo2}
		In the case of strong wind, if $-h$ is globally hyperbolic then
		the associated \sstk  splitting is also globally hyperbolic.
	\end{lemma}
	\begin{proof}
		Since $\pi:(\R\times M,g)\rightarrow (M,\frac{1}{\Lambda}h)$ is a
		Lorentzian submersion,   one can easily check that  a  lift of any Cauchy hypersurface on
		$(M,\frac{1}{\Lambda}h)$ is also a Cauchy hypersurface of
		$(\R\times M,g)$.
	\end{proof}
	\noindent  The previous lemmas yield a nice result on the existence of maximizing geodesics. 
	\begin{thm}\label{maxZermelo}
		Let $(M,\Sigma)$ be a wind Riemannian structure with strong wind
		such that the Lorentzian metric $-h$ is globally hyperbolic. For
		any $x_0\in M$, if  $y_0\in J_\Sigma^+(x_0)\setminus \{x_0\}$  then there exists a
		global maximum on $C^\Sigma_{x_0,y_0}$ of the length functional
		$\ell_{F_l}$.
	\end{thm}
	\begin{proof}
		Apply Theorem~\ref{compactcase} by taking into account that  Lemma
		\ref{lolo2} ensures w-convexity  and Lemma~\ref{lolo} ensures that
		the hypothesis in the last part of that theorem is fulfilled.
	\end{proof}

	Finally, let us give an application to the existence of closed geodesics\footnote{As in the classical Riemannian case, {\em closed} is understood here in the sense of  smooth {\em periodic}. In the Lorentzian case, 
		closed non-periodic geodesics can exist (they are necessarily lightlike and incomplete) but, clearly, this possibility cannot occur for wind Riemannian structures.}.  
	The differences between the causal properties of the Lorentzian metric $-h$ on all the manifold $M$, and the metric $g$ of the associated \sstk splitting, were stressed in Example~\ref{rolo}.   They are exploited now to prove the following result.
	\begin{thm}\label{prop:closedgeo}
		Let $(M,\Sigma)$ be a wind Riemannian structure with strong wind. If $M$ is compact then $(M,\Sigma)$ admits a closed 
		(non-constant)  geodesic.
	\end{thm}
	\begin{proof}
		Consider the associated \sstk splitting and define:
		$$T_0=\inf_{x\in M}\{T_0(x)\} \qquad \hbox{where} \qquad
		T_0(x)=\inf\{T>0: (0,x)\ll (T,x)\}.$$
		Notice that $T_0<+\infty$. In fact, any closed timelike curve $\sigma:[0,1]\rightarrow M$  for $-h$ provides an \sstk-timelike curve 
		$\tilde\sigma(s)=(\eta(s),\sigma(s)), s\in [0,1]$, $2\eta(s)=\int_0^s(F(\dot\sigma(\bar s))+F_l(\dot\sigma(\bar s)) )d\bar s$ from $(0,\sigma(0))$ to $(T=\eta(1),\sigma(0))$, so that $T_0(\sigma(0))<+\infty$.
		Notice also that,   as the associated \sstk splitting is strongly causal,  
		$T_0(x)>0$ for each $x\in M$ and, whenever    $T_0(x)<+\infty$,   $(0,x)$ and $(T_0(x),x)$ can be joined by a  lightlike geodesic $\gamma_x$
		(since the associated \sstk splitting is  globally hyperbolic,   $J^+(0,x)$ is  closed, so  $(T_0(x),x)\in J^+(0,x)\setminus I^+(0,x)$). 
		Now, consider a sequence  $\{x_n\} \subset M$ such that $T_0(x_n)\rightarrow T_0$ and, with no loss of generality, assume that $x_n\rightarrow x_0$. We claim that
		$T_0>0$  (see below). Thus, a  limit curve  $\gamma_0$ (not reduced to a point) of the corresponding sequence of curves $\{\gamma_{x_n}\}$ will exist and connect $(0,x_0)$ to $(T_0, x_0)$ (so that $T_0=T_0(x_0)$)   
		and it must be a lightlike pregeodesic too. Hence, being $\partial_t$ spacelike,  its image   cannot be  $l_{x_0}$  and  by Theorem~\ref{existenceofbolasNO},   its projection $\sigma_0=\pi\circ \gamma_0$   must be a  pregeodesic  of $(M,\Sigma)$
		with endpoints equal to $x_0$.
		In order to check that $\sigma_0$  must be {\em closed}, assume that $\sigma_0$ and $ \gamma_0$ are parametrized on $[0,1]$, extend $\sigma_0$ periodically, and extend $\gamma_0$ accordingly by making it invariant under the 
		translation $(t,x)\mapsto (t+T_0,x)$. If the velocities  $\dot\sigma_0(0), \dot\sigma_0(1)$  do not match, then the points $x_-=\sigma_0(1-\varepsilon), x_+=\sigma_0(\varepsilon)$ on $\sigma_0$ for some $0<\varepsilon<1/2$  satisfy, 
		by well known local causality properties of a Lorentzian manifold (applied to $(M,-h)$)
		$x_-\preceq x_0 \preceq x_+$. Accordingly, $\gamma_0(1-\varepsilon)\leq \gamma_0(1)\leq \gamma_0(1+\varepsilon)$. As $\gamma_0$ is a  lightlike geodesic broken at $\gamma_0(1)$, we can modify it into a causal curve $\rho$ by putting   
		a timelike segment from $\gamma_0(1-\varepsilon)$ to $ \gamma_0(1+\varepsilon)$ and making $\rho$ invariant under the translation $(t,x)\mapsto (t+T_0,x)$ too. Then,  there exists a point $\bar x\in M$ such that
		$\rho(0)=(0,\bar x)$,   $\rho(1)=(T_0,\bar x)$ and $\rho_{|[0,1]}$ is not a lightlike pregeodesic. Therefore, for $\varepsilon$ small enough $(0,\bar x)\ll (T_0-\varepsilon, \bar x)$, in   contradiction with the definition of $T_0$.
		
		\smallskip
		{\em Claim}:  $T_0>0$. 
		
		Assume, by contradiction, that  $T_0(x_n)\rightarrow 0$, $x_n\rightarrow x_0$, and  choose a neighborhood $W_0\subset M$  of $x_0$ such that $-h$ restricted to $W_0$ is causal. By Lemma~\ref{existenceofbolas}, there exists a
		new neighborhood $U_0\subset W_0$ and some $\varepsilon>0$ such that any  $\Sigma$-admissible closed curve  starting at any $y\in U_0$ and leaving $W_0$ will leave $\hat B^+_{\Sigma}(y, r )$, for each $r\in [0,\varepsilon)$. 
		Nevertheless, for large $n$ one has $x_n\in U_0$ and $T_0(x_n)<\varepsilon$. So, the projection $\sigma_n$ of the lightlike geodesic $\gamma_{x_n}$ will be a $\Sigma$-admissible loop based at $x_n$  (thus leaving $W_0$) with $F$-length  smaller than $\varepsilon$.
		As  
		$\ell_{F}(\sigma_n)\leq \ell_{F_l}(\sigma_n)$,   each point of $\sigma_n$ must belong to $B^+_{\Sigma}(x_n, r)$ for some  $r\in [0,\varepsilon)$, a contradiction.
	\end{proof}
	\begin{rem}
		The closed geodesic in Theorem~\ref{prop:closedgeo} corresponds, in the associated \sstk splitting, to  a  future-pointing, lightlike geodesic  which has  closed component $x$ and connects the points $(0,x_0), (T_0, x_0)\in \R\times M$, 
		for each $x_0$ belonging to the  support of $x$ (these  geodesics are called {\em $T_0$-periodic trajectories} in \cite{Candel96, sanche99, BiJa11}).  This  extends to the case of a spacelike Killing vector field $K$ (using the only 
		topological assumption that $M$ is compact) results on the existence of at least one geometrically non-trivial,  lightlike, $T$-periodic trajectory, when $K$ is timelike,  obtained in \cite{Candel96,MasPic98}. Indeed, the proof of 
		Theorem~\ref{prop:closedgeo} is inspired by a well known result by Tipler   \cite{Tipler79}  as well as results on $T$-periodic trajectories in \cite{sanche99} --even though the reader will find quite a few  non-trivial differences.
		Notice that if the wind were not strong at some point (non-spacelike  $K$), the previous proof would fail as  $T_0$ would be equal to $0$. Nevertheless, in the case of mild wind (timelike $K$) a closed (non-constant) geodesic must exist. 
		In fact, this is known for any compact Finsler manifold \cite{LyuFet51} and, thus, for Randers ones.  \soutE{The question} \bw Whether \ew a closed geodesic must exist when $K$ is allowed to be lightlike at some point may deserve a further study.
		\bw The case when $K$ is lightlike everywhere has been recently studied in \cite{CGMS}. \ew
	\end{rem}
	
	To end this section, let us digress on some links to other variational problems on curves in the literature and possible prospects. 
	
	(1)  It is worth to stress that  the dynamics of Zermelo's navigation is  {\em not}  equivalent to the one of a system  defined by a Riemannian metric $g_R$ and a one-form $\omega$ on a manifold $M$ (a particular case of the so-called {\em magnetic geodesic flow}) except,  obviously,  if $\|\omega\|_{g_R}<1$ because   in this case   both problems are then described by standard Finsler geometry.\footnote{In fact, at the points $p\in M$ where $\|\omega_p\|_{g_R}>1$, the  Lagrangian   for magnetic trajectories   $L(v)=\sqrt{g_R(v,v)}+ \omega(v)$, $v\in TM$, has ``indicatrix'' $\{v\in T_pM: L(v)=1\}$ which is an unbounded hypersurface.     Observe also that the Lagrangian $F$ in the Zermelo navigation problem is formally identical to  $L$ in the regions of  mild and strong wind,  but the
		metric under the square root in $F$, namely $\tilde h$, is a signature-changing metric: specifically, it changes from
		Riemannian to Lorentzian (with signature
		$(+,-,\dots,-)$).  }   
	
	(2)  Nevertheless,   one can find  a parallelism between Zermelo navigation problem  and the problem of the existence of timelike or causal curves $\rho$ connecting two events $z_0, w_0$ in a spacetime $L$ which are critical for the action functional associated with an electromagnetic field for some prescribed charge-to-mass ratio $q/m$ (so that the timelike critical curves are solutions to the corresponding  Euler-Lagrange  equation, i.e., the Lorentz force equation for $q/m$,  see e.g. \cite{CapMin, min, MinSan06}).  These curves can be interpreted, say, in the case when $q/m>0$, 
	as those lightlike geodesics for a Kaluza-Klein spacetime $L\times \R$ 
	which locally minimize the natural arrival coordinate at $\{w_0\}\times \R$, see \cite[Theorem 4.1]{MinSan06}.  Even though the natural projection $y:L\times \R\rightarrow \R$ is not a  temporal function for the Kaluza Klein metric, the natural vector field $\partial_y$ is Killing. The similarities  between the \sstk approach for 
	Zermelo navigation and the Kaluza-Klein for electromagnetism suggests \soutE{to  pose}   the following navigation problem, whose electromagnetic analogous was solved in \cite{MinSan06}:
	
	{\em  Consider  classes of wind curves from $x_0$ to $y_0$ which are homotopic through $F$- (resp. $\Sigma$-) wind  curves, and  determine when such a class admits  an  $F$-length  minimizing (or $F_l$-length  maximizing) curve.} 
	
	Even though, in principle,  our techniques would allow one to ensure the existence of $\Sigma$-wind curves (under background hypotheses such as completeness or $w$-convexity), subtleties would appear for the existence of critical $F$-wind curves when $x_0 \prec y_0$. This  problem has a  parallelism with the existence of critical points for the electromagnetic action in timelike or causal homotopy classes, where  very precise results (which ensure the existence of timelike critical curves, not only lightlike ones) have been obtained, \cite[Theorems 4.2 and 5.1]{MinSan06}. 
	Although such techniques seem translatable here, they would require the developing of further notions on wind geodesics (such as cut points)
	and, so, this will not be studied here. 

	\section{Fermat's principle and generalized Zermelo navigation problem.}\label{further1}
	
	\subsection{A new problem: Fermat's principle for arbitrary arrival curves}\label{Fermatprinciple1}
	In optics, Fermat's  principle  is a variational principle for light rays.  Its  formulations in general relativity, as e.g. in \cite{Kovner90, Pe90, GiMaPi02, Perlic04}, involve a prescribed event (a point $p$ on the Lorentzian manifold $L$
	modeling the spacetime, which represents the event of light   reception), the world-line
	of a  light source  (a timelike  injective curve  $\alpha\colon (a,b)\to L$), intersecting the causal past of $p$,  an {\em ``arrival time functional''} (the proper time in which the light source emitted the light signals). \bw If \ew some matter
	distribution (one or more galaxies, dark matter, etc., encoded in the spacetime metric $g$) \bw is present \ew between the source and the observer, \soutE{responsible for} \bw it can cause \ew the bending of light \soutE{and, thus,} \bw producing then the so-called {\em gravitational lensing} \ew \soutE{for the multiple images  effect}.\footnote{Some causality conditions, 
		as  global hyperbolicity or stable causality are also usually assumed for Morse theory of light rays, \bw often used in the study of gravitational lensing by variational methods, \ew compare e.g. \cite{Uhlenb75, GiMaPi98}.}  From a geometric point of view,
	this configuration is equivalent to the case to be considered here when the light source emits the signals at a given event $p$ and the timelike curve  $\alpha$, intersecting the causal future of $p$, is  the world line of an observer
	(so that, the name ``arrival time functional'' is justified).  From the properties of chronological and causal futures, it is easy to check that
	if there exists a lightlike (or causal) curve where the absolute maximum  or minimum arrival time to $\alpha$ is achieved, this curve  must be a lightlike pregeodesic. This is a consequence of the fact that  the arrival point would be 
	horismotically related to $p$ and this would hold even if $\alpha$ is not timelike. However, Fermat's principle states that the set of critical points of the arrival time is equal to the set of lightlike pregeodesics connecting $p$ and $\alpha$ 
	and the timelike character of $\alpha$ is required then.  As emphasized by V. Perlick, no external notion of time is necessary, but just the arrival   instant   $\bar t$  with respect to the proper time parametrization of $\alpha$. 
	
	Nevertheless, in an \sstk splitting, it is natural to consider the  temporal function $t$ 
	and then to ask if  Fermat's  principle holds  when considering a point and an integral line of the Killing field $\partial_t$  (we recall that   an  integral  curve  $l_{x_1}$, $x_1\in M$, of $\partial_t$  can be parametrized
	with the  temporal function  $t$ of the \sstk splitting,  so that the  arrival time functional makes still sense).  We emphasize 
	that $l_{x_1}$ {\em can be   spacelike or lightlike}  now 
	(the case when the Killing vector field is timelike is well known, see  \cite{FoGiMa95,CaJaMa})  and these possibilities have a clear interpretation in  
	Zermelo's navigation problem  since the  travel time can be identified, up to an initial constant,  with the arrival time functional of the associated \sstk splitting. Indeed, Fermat's principle  can be viewed as a 
	variational principle for a generalized Zermelo's navigation  problem, namely:  look not only for the quickest  navigation paths but also for any path  that makes  critical the time of navigation between two given points $x_0,x_1\in M$. 
	These  paths  will be the projections on $M$ of the  paths in  
	$\R\times M$  which are critical for   the arrival
	time
	functional at  $l_{x_1}$  defined on the set of the piecewise smooth future-pointing lightlike curves connecting $p_0=(t_0,x_0)$ to $l_{x_1}$.
	
	Summing up, {\em a Fermat's principle in SSTK spacetimes valid for lightlike curves  from $p_0$ to any (timelike, spacelike or lightlike) $l_{x_1}$  would allow an interpretation of all the geodesics of a wind Riemannian structure as critical points for the arrival functional}. However, this problem poses a   general mathematical question for Fermat's principle:  
	when one considers the critical lightlike curves  starting at a point $p_0$ and arriving at a curve $\alpha$ in an arbitrary spacetime, {\em is it necessary to impose a prescribed causal character or any other hypotheses on $\alpha$?}
	
	At first glance, two  natural  restrictions appear: (a) a connecting lightlike curve $\gamma$ that is orthogonal  to $\alpha$ at the arrival point $z_\alpha$ would not be permitted (otherwise, $\gamma$ could be a lightlike  pregeodesic  which coincided with $\alpha$ in a left neighborhood of the end-point  and, thus, $\gamma$ would not be a critical point) and (b) $\alpha$ must be an embedded curve, that is, on the one hand its velocity must not be zero (in order to avoid bothering  requirements on the variations); on the other, 
	self-intersections or cases where the induced topology on the image is coarser than the one coming from an immersion  should be excluded (in order to define properly the Fermat functional as  the value of the parameter of $\alpha$   at the arrival point). Recall that these conditions are  assumed  in  previous   results, 
	in fact,  the requirements  (a)  and (b), the latter at least locally,   hold  classically as $\alpha$ is assumed to be a timelike curve;  nevertheless in all the  proofs we are aware of, the   timelike character of $\alpha$ appears to be a fundamental assumption (compare \cite[Lemma 2.5]{Uhlenb75}, \cite[Lemma 3]{Pe90}, \cite[Lemma 2.1]{AntPic96}, \cite[Remark 3.3]{GiMaPi02}). 
	
	In the next subsection we will check that, amazingly, these two conditions are enough for a consistent general Fermat's principle in spacetimes  (Theorem~\ref{lema:fermatprinc}). Moreover, in the last subsection we will  prove   that the particular structure of \sstk spacetimes makes unnecessary  assumption (a)  when $\alpha$ is a line $l_{x_1} (\not\ni p_0)$, (of course, in this case, (b) is satisfied) and,  even more,  sharper conclusions can be obtained  (Theorem~\ref{fp}).
	\subsection{A general Fermat's Principle}\label{Fermatprinciple2}
	Our aim is to  establish  a  Fermat's principle between a point $p_0$ and an arbitrary smooth embedded  curve $\alpha$ for an arbitrary spacetime $(L,g)$. 

	Given a vector field $\xi =\xi(s)$ along a curve $\gamma$, we will denote by $\xi'$  its covariant derivative $\frac{D\xi}{d s}$ with respect to the Levi-Civita connection of $g$.  Let us introduce the variational approach for Fermat's principle.
	
	\begin{defi}\label{def:var_fer}
		Let $(L,g)$ be any spacetime,  $\alpha:(\bar{a},\bar{b})\rightarrow L$ a smooth embedded curve in $L$, $p_0\in L$. Fix an interval $[a,b]$ (eventually normalized to $[0,1]$), and put 
		\begin{multline} \label{mline}
			\mathcal N_{p_0, \alpha}=\{ \gamma:[a,b]\rightarrow L: \text{ $\gamma$ piecewise smooth,  future-pointing  lightlike,}\\ \text{ and $\gamma(a)=p_0,\ \gamma(b)\in {\rm Im} (\alpha)$}\}
		\end{multline}
		
		(i)  The {\em arrival  functional} $T: \mathcal N_{p_0, \alpha}\rightarrow \R$ is defined as 
		\begin{equation}
			\label{earrivalf} T(\gamma)=\alpha^{-1}(\gamma(b))
		\end{equation} 
		
		(ii) An {\em admissible variation} $\chi$  
		of $\gamma\in \mathcal N_{p_0, \alpha}$ is
		a  $C^1$   map $\chi \colon (-\varepsilon,\varepsilon)\times [a,b]\to L$ which has,  at least,  continuous second order mixed derivatives $\frac{\partial^2\chi}{\partial w\partial s},\ \frac{\partial^2 \chi}{\partial s\partial w}$  on $(-\varepsilon, \varepsilon)\times [s_{j}, s_{j+1}]$ (so that, there, $\frac{D}{d w}\frac{\partial \chi}{\partial s}=\frac{D}{d s}\frac{\partial \chi}{\partial w}$),  where
		$a=s_1<s_2<\ldots <s_n=b$ is a subdivision of the interval $[a,b]$, and
		such that $\gamma_w:=  \chi_w= \chi (w,\cdot)\in  \mathcal N_{p_0,\alpha}$, for all $w\in (-\varepsilon,\varepsilon)$, and $\chi_0=\gamma$. 
		The  {\em variational vector field} $Z$  associated with an admissible variation is     
		$Z(s)=\frac{\partial \chi}{\partial w}(0,s)$. 
		
		(iii)  A curve $z\in \mathcal N_{p_0, \alpha}$  is a {\em critical point of the arrival functional} if, for every  variation  $\gamma_w$, we have that $\frac{\de}{\de w}T(\gamma_w)|_{w=0}=0$.
	\end{defi}
	A basic fact in the proof of the next Lemma is that  any $m$-dimensional spacetime can be locally described 
	as a  product manifold $(c,d)\times \Omega$, where  $\Omega$ is   an open,  $m-1$ dimensional,  smooth manifold  
	endowed with the metric
	\begin{equation}\label{one}
		g((\tau,v),(\tau,v))=-\Lambda^t(x)\tau^2+2 \omega^t(v)\tau+g_0^t(v,v)
	\end{equation}
	where $(\tau,v)\in T_{(t,x)}((c,d)\times \Omega)\equiv \R\times T_x\Omega$ and for every $t\in (c,d)$, $\Lambda^t$, $\omega^t$ and $g^t$ are respectively a function, a one-form and a Riemannian metric  in $\Omega$ (under the analogous to \eqref{lorentzian}).   Observe that the coordinate $t\in(c,d)$   also determines a  temporal  function of $((c,d)\times \Omega,g)$ as  for an \sstk spacetime.  Thus, 
	this time-orientation and a notation consistent with the previously introduced  one  for the \sstk case
	will be used;   for example,  the line $l_{x_1}=\{ (s,x_1): s\in (c,d)\} $  for every $x_1\in \Omega$.
	In particular, a   vector field $Z$  along  a curve $\gamma$ with image in  $(c,d)\times \Omega  $  will be denoted by $Z(s)=(Y(s),W(s))$.  
	
	\begin{lemma}
		\label{claim}
		Let $(L,g)$  be a spacetime, $p_0\in L$ and 
		$\alpha:(\bar a,\bar b)\rightarrow L$ be any smooth,  embedded  curve in $L$. 
		Assume that   $\gamma: [a,b]\rightarrow L$ 
		is a piecewise smooth  future-pointing  lightlike curve from $p_0$ to  ${\rm Im}(\alpha)$, such that $\dot \gamma(b)$ is not  orthogonal  to $ \dot\alpha(T(\gamma))$. Then there exists a partition $s_0=a<s_1<\ldots<s_n=b$ of the interval $[a,b]$ and $n$  
		open subsets $U_1,U_2,\ldots, U_n\subset L$   such that $\gamma$ is smooth 
		in $[s_j,s_{j+1}]$ with $\gamma([s_j,s_{j+1}])\subset U_{i+1}$ for $i=0,\ldots,n-1$.  Moreover,	each $U_j$ is of the type $(-\varepsilon_j, \varepsilon_j)\times \Omega_j$ and the metric $g$ is   written as in  \eqref{one}  with   $\partial_t$  nowhere orthogonal to $\gamma$,  and $\alpha$  an integral curve for $\partial_t$ or $-\partial_t$ in $U_n$.  
	\end{lemma}
	\begin{proof}
		As a first step, recall that   there exists a vector field $X$ defined on all $M$ 
		such that $X$ is not orthogonal (thus, neither tangent)  to $\gamma$ at any point and $\alpha$ is the integral curve of $X$ on a neighborhood $W$ of $q=\gamma(b)=\alpha(T(\gamma))$. 
		Indeed, take any future-pointing timelike vector field $X_1$ on $M$,  which is necessarily  non-orthogonal to $\dot\gamma$, as the latter is lightlike,  and a second vector field $X_2$ on a neighborhood $W$ of $q$, so that $\alpha$ is the integral curve  of $X_2$ through $q$  (this can be done for example by considering adapted coordinates to $\alpha$).  Assume that $g(\dot\gamma(b),X_2)<0$ (otherwise, replace $X_2$ with $-X_2$) and reduce the neighborhood $W$ where $X_2$ is defined in such a way that $g(\dot\gamma(s),X_2)<0$ whenever $\gamma(s)\in W$.   Then, choose a bump 
		function $\mu$ with support 
		in $W$ and $\mu=1$ on a smaller neighborhood $W'$ of $q$, and 
		put $X=(1-\mu)X_1+\mu X_2$ 
		on all $M$. 
		
		For every $s\in [a,b]$, choose a spacelike hypersurface $\Omega_s$ transverse to $X$ and containing $\gamma(s)$ with $X$ and $\dot\gamma(s)$ in the same side.  Moving $\Omega_s$ with the flow of 
		$X$ one obtains a 
		neighborhood $U_s$ of $
		\gamma(s)$. Now, taking a Lebesgue number for the covering $U_s$ of ${\rm Im}(\gamma)$  (using an auxiliary Riemannian metric in $L$) one obtains the required sequence $a=s_0<\dots <s_n=b$.  Finally, we can add the possible breaks of $\gamma$ to the subset  $s_0,s_1,\ldots,s_n$, dividing every subinterval $[s_j,s_{j+1}]$ in a finite number of intervals and considering the subset $U_{j+1}$ in all of them. 
	\end{proof}
	\begin{lemma}\label{lem:exist}
		Let $(L,g)$, $p_0$ and $\alpha$ as in Lemma~\ref{claim}. 
		Let  $Z$  be  a piecewise smooth  vector field along $\gamma$ with $Z(a)=0$ and $Z(b)$ proportional to $\dot\alpha(T(\gamma))$.   Then $Z$ is the variational vector field of a  variation    by lightlike curves from $p_0$ to ${\rm Im}(\alpha)$ if and only if $g(\dot \gamma,Z')=0$.
	\end{lemma}
	\begin{proof} 
		The implication to the right follows observing that  
		$g(\dot\gamma_w,\dot\gamma_w)=0$, for all $w$,  and  then differentiating both sides of this equality w.r.t. $w$, using that $\frac{D}{d w}\frac{\partial \chi}{\partial s}=\frac{D}{d s}\frac{\partial \chi}{\partial w}$ and evaluating in $w=0$,  we get that $g(\dot \gamma,Z')=0$.  
		For the converse,  let  us make some previous considerations for the case  of a smooth curve $\gamma:[ a_1, b_1]\rightarrow L$ contained in one of  the  local splittings $(c,d)\times  \Omega $ in  Lemma~\ref{claim}. 
		Setting then $\gamma=(\theta,\sigma)$,    $Z=(Y,W)$   and given a smooth curve $x:[a_1, b_1]\rightarrow \Omega$,  
		one has  that 
		$(t,x):[a_1, b_1]\rightarrow (c,d)\times  \Omega $ is a lightlike curve if and only if 
		\[-\Lambda^t(x) \dot t^2+2\omega^t(\dot x) \dot t+g_0^t(\dot x,\dot x)=0.\]
		It follows that  $x$  can be lifted  to a   smooth   lightlike curve in  $((c,d)\times \Omega,g)$  whenever  one of the differential equations  
		\begin{equation}\label{diffeq}
			\dot t=\frac{g_0^t(\dot x,\dot x)}{- \omega^t(\dot x)\pm  \sqrt{\omega^t(\dot x)^2+\Lambda^t(x) g^t_0(\dot x,\dot x)}}
		\end{equation}
		has a solution in $[a_1,  b_1]$ which takes values in $(c,d)$.  Observe that,  as   the initial curve 
		$\gamma(s)=(\theta(s),\sigma (s))$ is  
		future-pointing, i.e., $\dot{\theta}>0$, 
		and nowhere  orthogonal to $\partial_t$, i.e., $\omega^{\theta(s)}(\dot{\sigma}(s))-\Lambda^{\theta(s)}(\sigma(s)) \dot{\theta}(s)\not=0$, for all $s\in [a_1, b_1]$,  
		the same  properties hold for nearby curves
		and we will assume them for the curve constructed from  $x(s)$ and the corresponding $t(s)$ whenever $x(s)$ is a longitudinal curve for some variation of $\sigma$.  
		In particular,  whenever $\Lambda^t$ vanishes,  $\dot x$ cannot vanish and, moreover, \eqref{diffeq} holds with  $\omega^t(\dot x)<0$ and the choice of sign $+$ in the denominator.
		Recall that neither $\dot x$ can vanish when $\Lambda^t \neq 0$ 
		(otherwise $\partial_t$ would be non-null but proportional to the velocity  $(\dot t, \dot x)$ of the lifted lightlike curve).
		In conclusion, 
		the  right-hand side of \eqref{diffeq} is \bw well-defined \ew  and smooth on $[a_1, b_1]$,  
		so that the lift can be carried out on all $[a_1, b_1]$ and its $t$-component remains $C^1$ close to  the  $t$-component of $\gamma$ if $x$ is $C^1$ close enough to $\sigma$ (in particular,  the $t$-component  is contained in $(c,d)$).
		
		Consider now  a piecewise smooth vector field $Z$ along $\gamma$ such that $g(Z',\dot\gamma)=0$.  Assume that $\gamma$   has  breaks at most  
		at $a=s_0<s_1<s_2<\ldots< s_n =b$ and   
		let us consider $n$ open subsets $U_j$, $j\in\{1,\ldots,n\}$ as in Lemma~\ref{claim}.  Let us denote     $\gamma|_{[a,s_1]}$ as $(\theta^{(1)},\sigma^{(1)})$ on $U_1$ and write consistently $Z=(Y_1,W_1)$ along $[a,s_1]$. Consider  a variation  $\bar\chi_1:(-\varepsilon,\varepsilon)\times [a,s_1] \rightarrow \Omega_1$ of $\sigma$  with variational vector field $W_1$  and fixed initial point (the latter can be imposed  as necessarily $W_1(a)=0$). 
		Thus, up to reducing $\varepsilon$,  $\bar\chi_1:  (-\varepsilon,\varepsilon)\times [a,s_1]\rightarrow \Omega_1$  can be  lifted  to  a  (unique)  
		variation $\tilde{\chi}_1:(-\varepsilon,\varepsilon)\times [a,s_1]\rightarrow  (c,d) \times \Omega_1$  by  lightlike curves in $[a,s_1]$  with fixed initial point 
		$p_0$,  and which has  $\tilde Z{_1}=(\tilde {Y}_1,W_1)$ as variational vector field.  The fact that this variation is given by lightlike curves 
		departing from $p_0$ implies that 
		\begin{equation}\label{diffdeeta}
			g( \tilde{Z}_1'  ,\dot \gamma)=0   \text{  with } \;  \tilde{Z_1}(a)=0, \qquad  \mbox{on} \; 
			[a,s_1].
		\end{equation}
		In particular, as $W_1$ was prescribed, the function $\tilde{Y}_1$ is determined by the differential equation  \eqref{diffdeeta}.  In fact, denoting the components of $\tilde Z'_1$ by $(\tilde Y'_1,  W'_1)$, this is the equation
		\begin{equation}
			\big(-\Lambda(\sigma^{(1)})\dot\theta^{(1)}+\omega^{\theta^{(1)}}(\dot \sigma^{(1)})\big)\tilde Y'_1+ \dot\theta^{(1)}\omega^{\theta^{(1)}}( W'_1)+g_0^{\theta^{(1)}}(\dot \sigma^{(1)}, W'_1)=0.\label{diffdeeta2}
		\end{equation}
		As $-\Lambda(\sigma^{(1)})\dot\theta^{(1)}+\omega^{\theta^{(1)}}(\dot \sigma^{(1)})=g(\dot\gamma, \partial_t)|_{[a,s_1]}\neq 0$, for all $s\in [a, s_1]$  and taking into account the expression of the covariant derivative $\tilde{Z}_1'=(\tilde{Y}'_1,W'_1)$,  \eqref{diffdeeta2} can be put in normal form.  As $Y_1$ is also a solution of \eqref{diffdeeta},  we conclude that $\tilde{Y}_1=Y_1$  on  $[a,s_1]$,  and therefore, $Z|_{[a,s_1]}$ is the variational vector of a variation by lightlike curves, as required. 
		Finally, proceed inductively by considering analogously
		$\bar \chi_i:(-\varepsilon,\varepsilon)\times [s_{i-1},s_i]\rightarrow \Omega_i$ with $\bar \chi_i(w,s_i)=\bar \chi_{i-1}(w,s_i)$ and, when $i=n$, with fixed endpoint at $ \sigma(b)$. Recall that these variations are lifted to variations $\tilde{\chi}_i:(-\varepsilon,\varepsilon)\times[s_{i-1},s_i]\rightarrow (c,d)
		\times \Omega_i$ by lightlike curves which match  in the required  one   $\chi:(-\varepsilon,\varepsilon)\times[a,b]\rightarrow L$  
		after $n$ steps, as $\tilde{\chi}_i(w,s_i)=\tilde{\chi}_{i-1}(w,s_i)$ for $w\in(-\varepsilon,\varepsilon)$.
	\end{proof}
	
	Now, we are ready to give the  extension of Fermat's principle; our approach here  has been  inspired by \cite{AntPic96}.  
	\begin{thm}[Generalized Fermat's principle] \label{lema:fermatprinc}
		Let $(L,g)$ be any spacetime 
		and $\alpha:(\bar a,\bar b)\rightarrow L$ a smooth embedded  curve. Assume that $\gamma:[a,b]\rightarrow L$ is a piecewise smooth  future-pointing  
		lightlike curve from $p_0$ to  ${\rm Im}(\alpha)$, such that $\dot \gamma(b)$ is not  orthogonal  to $\alpha$.  
		Then, $\gamma:[a,b]\rightarrow L$   is a critical point of the  arrival functional  $T$  if and only if it is a pregeodesic.
	\end{thm}
	\begin{proof}
		By definition,  $\gamma$ is a critical point of  $T$ if and only if  $\frac{d}{dw}T(\gamma_w)_{|w=0}=0  $   for every    admissible variation   $\gamma_w$.    
		By  Lemma~\ref{lem:exist}, this is equivalent to  $Z(b)=0$ for any  variational vector field $Z$ along $\gamma$   since 
		\[Z(b)= \frac{d}{dw} \gamma_w(b)\mid_{w=0}=\frac{d}{dw}\alpha(T(\gamma_w))\mid_{w=0}  =
		\frac{d}{dw} T(\gamma_w)\mid_{w=0}  \cdot \dot\alpha(T(\gamma)).\]  
		Let  $U$ be a vector field along $\gamma$ which is not orthogonal to $\gamma$  at each point  $\gamma(s)$ and it is proportional to $\dot\alpha(T(\gamma))$ at $s=b$  (recall the  proof of Lemma~\ref{claim}). 
		Observe that, given  any  vector field  $W$ along $\gamma$ with $W(a)=W(b)=0$, we can obtain a vector field 
		corresponding to an admissible variation as
		\begin{equation}\label{xiW}
			Z_W(s)=W(s)+f_W(s) U(s),
		\end{equation}
		where 
		\begin{equation}\label{efw} f_W(s)=-  e^{-\rho(s)}\int_a^s \frac{g(W',\dot \gamma)}{g(U,\dot \gamma)} e^\rho d\mu \quad \hbox{and} \quad \rho(s)=\int_a^s \frac{g(U',\dot \gamma )}{g(U,\dot \gamma)} d\mu \, ,\quad \forall s\in [a,b].\end{equation} 
		Moreover, all the admissible vector fields can be obtained in this way. Indeed, assume that $Z$ is admissible and consider $W(s)=Z(s)- \left(c(s-a)/(b-a)\right)  U$, where $c$ is the constant that satisfies $Z(b)=c U(b)$. Now observe that the difference  $Z_W(s)-Z(s)=(f_W(s) -  c(s-a)/(b-a))U$ is also admissible, but it has to be zero.  The reason is that any  admissible vector field $p\cdot U$, with $p:[a,b]\rightarrow \R$ a smooth function such that $p(a)=0$, has to be zero, since it  must satisfy 
		\[\dot p \cdot g(U,\dot \gamma)+p\cdot g(U',\dot \gamma)=0 \]
		and $g(U,\dot \gamma)$ cannot vanish. So,  $Z=Z_W$ follows,  and $\gamma$ will be  a critical point if and only if $Z_W(b)=0$ for every $W$ as above. This is equivalent to $f_W(b)=0$ or, from the explicit formula \eqref{efw}, 
		\begin{equation}\label{firstEulang}
			\int_a^b \frac{g(W',\dot \gamma)}{g(U,\dot \gamma)} e^\rho d\mu=0.
		\end{equation}
		Now, if $\gamma$ is a critical point, we can choose 
		$W$ such that  $W(s_i)=0$ in all the breaks. Applying integration by parts, 
		\begin{equation}\label{firstEulang2}
			\int_a^b g(W, (\varphi\dot \gamma)') d\mu=0\end{equation}
		where $\varphi =e^\rho/g(U,\dot \gamma)$. Then $(\varphi\dot \gamma)'=0$ outside the breaks, which implies directly that $\gamma$ is a piecewise pregeodesic. 
		Moreover, if there were a break  
		at some  $s_i\in(a,b)$ then, for every $w\in T_{\gamma(s_i)}L$, one could choose a vector field $W$ along $\gamma$ such that $W(s_i)=w$ and $W$ is zero in the other breaks (as well as in the endpoints). Then, consider $Z_W$ given by  \eqref{xiW} and apply integration by parts to \eqref{firstEulang} again in order to obtain
		\[g\left(w, e^\rho(s_i)\left(\frac{\dot \gamma(s_i^+)}{g(U(s_i),\dot \gamma(s_i^+))}-\frac{\dot \gamma(s_i^-)}{g(U(s_i),\dot \gamma(s_i^-))}\right)\right)=0\]
		for all $w\in T_{\gamma(s_i)}L$.   That is,  $\dot \gamma(s_i^+)$ and $\dot \gamma(s_i^-)$ are proportional and $\gamma$ could be reparametrized as a smooth lightlike pregeodesic. 
		
		Conversely, if $\gamma$ is a pregeodesic, we can reparametrize it as a geodesic (with no breaks in the parametrization) as the value of the arrival time functional would remain unchanged. However, for a geodesic the function  $\varphi$   is clearly a constant, 
		which allows us to obtain \eqref{firstEulang2} and, finally, \eqref{firstEulang}. 
		
	\end{proof}


	Some extensions of the Generalized Fermat's Principle  are still possible  
	as for example when one considers timelike curves rather than lightlike ones. 
	This case becomes meaningful if one prescribes a fixed length $c$ for all the timelike curves from $p_0$ to $\alpha$.   
	In order to reduce this case to the lightlike one,  consider the extended  spacetime $(L\times \R, \tilde g)$, $\tilde g=\pi_{L}^*g+ \de u^2$, where $\pi_{L}$ is the 
	canonical projection of $L\times\R$ onto $L$.  It is straightforward to check  that,  for any lightlike geodesic $\tilde \gamma:[a,b]\to L\times \R$, $\tilde \gamma(s)=(\gamma(s),u(s))$, of the metric $\tilde g$, the component $\gamma$  is a causal 
	geodesic of $(L,g)$ and the component $u$ satisfies $\dot u^2\equiv \mathrm{const.}:= c^2/(b-a)^2$.  Thus, for any point $p_0\in L$ and any smooth embedded curve $\alpha$,   Theorem~\ref{lema:fermatprinc} applied to the spacetime $(L\times \R, \tilde g)$, the point $(p_0,0)$ and the curve $\tilde \alpha(s)= (\alpha(s),c)$, $c>0$, gives: 
	\begin{cor}\label{timelikefermat}
		Let $(L,g)$ be a spacetime 
		and $\alpha:(\bar a,\bar b)\rightarrow L$ be a smooth embedded  curve. Assume that $\gamma:[a,b]\rightarrow L$ is a piecewise smooth  future-pointing  
		timelike  curve from $p_0\in L$ to  ${\rm Im}(\alpha)$, with Lorentzian length $\int_a^b\sqrt{-g(\dot\gamma,\dot\gamma)}ds=c$, such  that $\dot \gamma(b)$ is not  orthogonal to $\alpha$.  
		Then, $\gamma:[a,b]\rightarrow L$   is a critical point of the  arrival functional $T$ defined on the set of the piecewise smooth timelike   curves   joining $p_0$ to ${\rm Im}(\alpha)$ and having fixed Lorentzian length $c$ if and only if it is a pregeodesic.
	\end{cor}

	As a final application, observe that the generalized Fermat's principle can be \bw even applied \ew    in a purely Riemannian setting.   Given a Riemannian manifold $(M,h)$,  $x_0\in M$ and $\alpha:(\bar a,\bar b)\rightarrow M$ a smooth  embedded  curve, let us introduce the following two spaces of paths between $x_0$ and ${\rm Im}(\alpha)$:
	\begin{align*}
		&\mathcal L_{x_0, \alpha,c}=\{ x:[a,b]\rightarrow M: \text{ $x$ piecewise smooth}\\
		&\quad\quad\quad\quad\quad\quad \text{ and $x(a)=x_0,\ x(b)\in {\rm Im} (\alpha)$ with $\ell_h(x)=c$}\},\\
		&\mathcal F_{x_0, \alpha}=\{ x:[a,b]\rightarrow M: \text{ $x$ piecewise smooth}\\ &\quad \quad \quad\quad\quad\quad\text{ and $x(a)=x_0,\ x(b)\in {\rm Im} (\alpha)$ with $\ell_h(x)=\alpha^{-1}(x(b))$}\},
	\end{align*}
	where $\ell_h(x)=\int_a^b \| \dot x\|_h ds$, $\| \dot x\|_h= \sqrt{h(\dot x,\dot x)}$,  and  $c$ is any constant greater than the distance between $x_0$ and ${\rm Im}(\alpha)$.  Let the {\em arrival functional} $T$ defined as $
	T(x)= \alpha^{-1}(x(b))$. 
	\begin{cor}\label{cor:appRieFermat}
		Let $(M,h)$ be a Riemannian manifold and $\alpha:(\bar a,\bar b)\rightarrow M$ a smooth  embedded  curve. Let $x:[a,b]\rightarrow M$ be a piecewise smooth curve with  $x(a)=x_0$, $x(b)\in {\rm Im}(\alpha)$.  
		\begin{itemize}
			\item[(i)]  If  $x\in \mathcal L_{x_0, \alpha,c}$ and $\dot x(b)$ is not orthogonal to $\alpha$, then $x$ is a critical curve of the arrival functional  $T$ on the space $\mathcal L_{x_0, \alpha, c}$ if and only if $x$ is a pregeodesic of $(M,h)$.
			\item[(ii)]   If  $x\in \mathcal F_{x_0, \alpha}$ and 
			\begin{equation}\label{eqh} 
				h\left(\frac{\dot  x(b)}{\| \dot x(b)\|_h} ,\dot\alpha(T(x))\right)\not= 1,  
			\end{equation}  
			then $x$ is a critical curve of the arrival functional $T$ on the space $\mathcal F_{x_0, \alpha}$ if and only if $x$ is a  pregeodesic of $(M,h)$.
		\end{itemize}
	\end{cor}
	\begin{proof}
		It follows from Theorem~\ref{lema:fermatprinc} by considering the spacetime $(\R\times M,g)$ with $g((\tau,v),(\tau,v))=-\tau^2+h(v,v)$. Then any curve $x:[a,b]\rightarrow M$ lifts to a unique future-pointing lightlike curve $(t,x):[a,b]\rightarrow \R\times M$ with $t(s)=\int_a^s \sqrt{h(\dot x,\dot x)}d\mu$. 
		Moreover, lift $\alpha$ to the curve $(\bar a,\bar b)\ni  \bar t\rightarrow \tilde{\alpha}(\bar t)=(c,\alpha(\bar t))\in \R\times M$ for statement  (i) and to  $(\bar a,\bar b)\ni  \bar t\rightarrow \tilde{\alpha}(\bar t)=(\bar t,\alpha(\bar t))\in \R\times M$, for $(ii)$.    The conclusions are obtained by observing that $ (\dot t(b),\dot x(b))$ is not  orthogonal to $\tilde{\alpha}$ if and only if,  (i),  $\dot x(b)$ is not orthogonal to $\alpha$ and  (ii) \eqref{eqh} holds.
	\end{proof}
	\begin{rem}\label{eriemann}
		While the result in (i) is  immediately seen as a variational principle for geodesics  with fixed length between a point and a curve,  (ii) might require more explanation. In fact, it    can be interpreted in the following (non-relativistic) way. The curve $\alpha$ is parametrized by a classical time $t$ and then $\alpha$ describes the motion of some target vehicle with arbitrary (but non-vanishing) speed.  The curves in $\mathcal F_{x_0, \alpha}$ are the trajectories followed by some tracker starting at $x_0$. As the length of the trajectories is independent of the parametrization, one can assume (neglecting the curves with speed vanishing at non-isolated points) that the tracker moves at constant speed. As a first approach, this speed can be assumed to be equal to 1 so that  each trajectory $\gamma$ is parametrized by $t$ in the interval $[0,\ell_h(\gamma)]$. 
		Now, the space $\mathcal F_{x_0, \alpha}$ contains all the trajectories such that the tracker catches the target, being the arrival functional $T$ just the exact time (or length of $\gamma$)  necessary for this aim. The corollary asserts that the geodesics coincide with the critical points of $T$ whenever the inequality \eqref{eqh} holds. As an interpretation of this inequality, notice that, if the component of the velocity of the target at the instant of the meeting in the direction of $\dot x(b)$ were equal exactly to $\dot x(b)$ then, even if $x$ were a geodesic, variations in the trajectory of the tracker might remain catching the target in subsequent instants. In particular, when the velocity of $\alpha$ were equal to $\dot x(b)$, these variations could be obtained simply by prolonging $x$ with (a reparametrization of) $\alpha$ beyond $b$.
	\end{rem}
	\subsection{Fermat's principle for \sstk spacetimes}\label{Fermatprinciple3} 
	For \sstk spacetimes, the curve $\alpha$ will be taken just equal to  a line $l_{x_1}$ parametrized with the global time function $t:  \R \times M \rightarrow \R$ so that the space 
	$\mathcal N_{p_0, \alpha}$ in \eqref{mline}
	is written now 
	$\mathcal N_{p_0, l_{x_1}}$ with $p_0=(t_0,x_0)$. Moreover, the arrival functional \eqref{earrivalf} becomes now a true arrival {\em  time} functional
	\begin{equation}\label{functiemp}
		T(\gamma) = t(\gamma(b))
	\end{equation}
	for future-pointing lightlike curves $\gamma \in \mathcal N_{p_0, l_{x_1}}$  ($\gamma\colon[a,b]\rightarrow L=\R\times M$). Notice that $\gamma$ is a critical point for $T$ on ${\mathcal N}_{p_0,l_{x_1}}$ if and only if 
	\begin{equation}\label{ezzz}  \de t_{\gamma(b)}(Z(b))(=Y(b))=0 
	\end{equation} for the variational
	vector field    $Z(s)=(Y(s),W(s))$ of any variation $\gamma_w$ of $\gamma$ in 
	$\mathcal N_{p_0,l_{x_1}}$. We will assume  the non-triviality assumption $p_0\not\in {\rm Im}(\alpha )$  i.e. $x_0\neq x_1$ (see Remark~\ref{rneq}). 
	Now, we are  ready for the general version of Fermat's principle for $\sstk$ spacetimes.
	\begin{thm}\label{fp}
		Let $(\R\times M,g)$ be an $\sstk$ as in \eqref{lorentz}, $x_0,x_1\in M$, $x_0\neq x_1$, $p_0=(t_0, x_0)$ and $\gamma\in \mathcal N_{p_0, l_{x_1}}$, $\gamma(s)=(\zeta(s),x(s))$. Then,
		\begin{itemize}
			\item[(1)]  if  $\gamma$ is a critical point of the
			arrival time functional $T$ on $\mathcal N_{p_0, l_{x_1}}$, then it is a lightlike pregeodesic of $(\R\times M,g)$;
			\item[(2)]   if  $\gamma$ is a
			lightlike geodesic of $(\R\times M,g)$ and $C_\gamma=g(\partial_t, \dot\gamma)$
			then one of the following three
			exclusive
			possibilities occurs:
			\begin{itemize}
				\item[(i)] $C_\gamma<0$, $\dot x$ lies in $A$, $x$ is a
				pregeodesic of $F$ parametrized with $h(\dot x,\dot x)=\mathrm{const.}$,  $\gamma$ is a critical point of $T$  and
				\[
				\zeta(s)=\zeta(a)+\int_a^s F(\dot x)\de \tau.
				\]
				\item[(ii)] $C_\gamma>0$, $\dot x$ lies in $A_l$ (so that $\Lambda < 0$
				on   all   $x$), $x$ is a pregeodesic of $F_l$ parametrized with $h(\dot x,\dot x)=\mathrm{const.}$,   $\gamma$ is a critical point of $T$  and
				\[
				\zeta(s)=\zeta(a)+\int_a^s  F_l(\dot x)\de \tau.
				\]
				
				\item[(iii)] $C_\gamma=0$,  $\dot x$ lies in $A_E\setminus A$ (so that $\Lambda\leq  0$
				on  all  $x$),    whenever it remains in $M_l$, $x$  is a lightlike  geodesic of $h/\Lambda$  such that 
				$-\omega(\dot x)>0$ and $x$ lies in $M_{crit}$ only at the isolated points where $\dot x$ vanishes; 
				moreover,  $\zeta$
				satisfies 
				\[
				\zeta(s)=\zeta(a)-\int_{a}^{s} \frac{g_0(\dot x,\dot x)}{\omega(\dot x)}\de\tau,
				\]
				for all $ s\in [a,b]$.
			\end{itemize}
		\end{itemize}
	\end{thm}
	\begin{proof}(1) 
		Let us distinguish three cases:
		
		Case (a).  Assume that there is an instant $s_0\in (a,b)$ such that $x$ is  $F$-admissible. We can also assume that  $\gamma$ is smooth in $s_0$, otherwise just choose a close instant to $s_0$ where $\gamma$ is still $F$-admissible and smooth.  
		Now, notice that the restriction $\gamma|_{[a,s_0]}$ must be a critical point of the arrival time functional on $\mathcal N_{p_0, l_{x(s_0)}}$. In fact, otherwise, take a variation $\gamma^{(s_0)}_w$ which contradicts the critical character of $\gamma|_{[a,s_0]}$ 
		and put 
		$$t^{(r)}(w)= t(\gamma^{(s_0)}_w(r))-t(\gamma (r)), \quad a\leq r\leq s_0$$
		Taking into account that $\partial_t$ is a Killing vector field, each curve $\gamma^{ (s_0)}_w$ 
		can be concatenated with the curve $s\mapsto (\zeta(s)+ t^{(s_0)}(w), x(s)), s\in[s_0,b]$, in 
		contradiction with the critical character of $\gamma$. 
		Thus, Theorem~\ref{lema:fermatprinc} 
		is applicable to $\gamma|_{[a,s_0]}$, and this piece of $\gamma$ must be a pregeodesic.    
		Now observe that we can consider the reverse problem because, reversing the parametrization of $\gamma$, it becomes a critical curve for the  arrival time functional   between $(t(\gamma(b)),x_1)$ and the line $l_{x(s_0)}$ (notice that all the longitudinal curves of a variation $\gamma_w$ of $\gamma$ can be shifted in $-t^{(b)}(w)$ by the flow of $\partial_t)$. Then one also has that $\gamma|_{[s_0,b]}$ must be a pregeodesic (which matches smoothly with the first piece), as required.

		Case (b).   Assume that $x$ is constant in an open interval $(s_0-\delta,s_0+\delta)$ (so that $\partial_t$ is lightlike at $x(s_0)$). In this case, $(\zeta,x)$ cannot be a critical point of the arrival time. In fact, consider a variational 
		vector field which is  $Z(s)=f(s)\partial_t$ with $f(s)=0$ for every $s\in [a,s_0-\delta]$ and $f(s)=1$ for every $s\in [s_0+\delta,b]$ and the associated variation $x_w=x$ and
		$\zeta_w(s)=\zeta(s)+w f(s)$ for every $w\in (-\varepsilon,\varepsilon)$, being $\varepsilon>0$ small enough in such a way that $\dot\zeta_w(s)=\dot\zeta(s) +w \dot f(s)>0$ for every $s\in (s_0-\delta,s_0+\delta)$. 
		Then ${\rm d}T(Z)={\rm d}t_{\gamma(b)}(Z(b))=1\not=0$.

		Case (c).  The only case left is when $\dot x(s)\in (A_E\setminus A)$ for every $s\in [a,b]$ and it is not zero in any subinterval.  Moreover, we can also assume that $\dot x(s)$ never vanishes and, so, $x(s)$ lies in $M_l$ and it is a piecewise smooth lightlike curve of the Lorentzian metric $-h$. Indeed, if this case is solved, then, 
		for any interval $J=[\bar{a},\bar{b}]\subset [a,b]$ such that $x|_J$ is smooth and strictly regular ($\dot x(s)\not=0$ for every $s\in J$) then  $ \gamma_{|J}$ will be  a lightlike pregeodesic. 
		As by Case (b), the set of zeroes does not contain intervals, then,  the claimed case  implies that $\gamma$ fulfils  the equation of the pregeodesics $D^g \dot\gamma/ds=f\cdot \dot \gamma$ in an open dense subset $D$ of $[a,b]$, for some smooth function $f$ on $D$. Being $\gamma$ piecewise smooth and $\dot \gamma$ non-vanishing,  $f$ can be smoothly extended to all $[a,b]$ except at most to the breaks, and $\gamma$ becomes a piecewise smooth lightlike pregeodesic. Moreover, if a break $s_0\in(a,b)$ appeared, the case $\dot x(s_0^+)=0$ (or $\dot x(s_0^-)=0$) could not hold.
		Indeed, otherwise
		$\Lambda(x(s_0))=0$, and this implies $(A_E\setminus A)\cap T_{x(s_0)}M=\{0\}$. Thus, $\dot x(s_0^-)=0$ and $\dot\gamma(s_0^-)$ and $\dot\gamma(s_0^+)$ become proportional, which implies  that $\gamma$ admits a reparametrization as a smooth geodesic.  Of course, the case  when
		$\dot x(s_0^+),\ \dot x(s_0^-)$ are both different from $0$  can hold and will be taken into account (indeed, the solution in the smooth case would imply that $\gamma$ is a piecewise pregeodesic with $C=0$ and, thus, its projection $x(s)$ would be a  piecewise lightlike pregeodesic of $(M, -h)$, recall Corollary~\ref{lightgeo2}).
		As a technical detail,  the conformal Lorentzian metric    
		$ -\tilde h:=-h/\Lambda^2$ will be used in the remainder (consistently with \eqref{eab}). This is 
		equivalent to the usage of $-h$ as  only lightlike curves and pregeodesics will be 
		concerned, and allows  us   to express easily the associated
		Fermat metrics 
		\begin{equation}\label{effl}
			F=\frac{\omega}{\Lambda} +\sqrt{\tilde{h}} \qquad \qquad  F_l=\frac{\omega}{\Lambda} -\sqrt{\tilde{h}} 
		\end{equation}
		(recall Proposition~\ref{plightvectorsSSTK} and equation \eqref{tau}), where $\Lambda<0$.

		So, assume that $x$ is a piecewise smooth lightlike curve in $(M_l,-\tilde h)$. 
		In particular,  the lightlike curve $\gamma$ is univocally reconstructed from  $x$ plus its initial point (Corollary~\ref{raclarations} (c2) is applicable to $\dot x(s)$) and $g(\dot \gamma, \partial_t)\equiv 0$ (from the interpretation of $h$, see \eqref{eextra}).
		In the case that $x$ is also a (smooth) pregeodesic, Corollary~\ref{lightgeo2} implies that $\gamma$ is a lightlike pregeodesic too. Otherwise, 
		we can find   a variation $x_w$ of $x$ by means of timelike curves  of $-\tilde h$ for every $w\in (0,\varepsilon)$ with variational vector field $\xi$ such that
		$\tilde{h}(\xi',\dot x)>0$ in two cases: when $x$ is smooth but not a pregeodesic and when $x$ is a piecewise pregeodesic   (see  case 2 and last part of the proof of \cite[Proposition 10.46]{O'neill}). 
		Our aim is to lift this variation (up to a subtle choice $\theta$ of the parameter) to a variation of $\gamma$ in the spacetime. Specifically, the variation $\eta_\theta=\gamma_{w(\theta)}$ for $\theta\in (-\varepsilon',\varepsilon')$ will be written as 
		

		\begin{equation}\label{gammaw}
			\gamma_{w(\theta)}(s)= (\zeta_{w(\theta)}(s),x_{w(\theta)}(s))
		\end{equation}
		and $\zeta_{w(\theta)}$ is defined on $[a,b]$  as
		\begin{equation}\label{zetaw}
			\zeta_{w(\theta)}(s)=
			\begin{cases}
				\zeta(a)+\int_{a}^{s} F(\dot x_{w(\theta)}) \de\tau & \mbox{ if } \theta \in [-\varepsilon',0],\\
				\zeta(a)+\int_{a}^{s} F_l(\dot x_{w(\theta)}) \de\tau & \mbox{ if } \theta \in [0,\varepsilon'],
			\end{cases}
		\end{equation}
		for all $s\in [a,b]$. Notice that both expressions agree for $\theta=0$, and the longitudinal curves at constant  $\theta$ are lightlike. 
		
		The reparametrization $w(\theta)$ will be crucial because otherwise $\frac{\partial}{\partial w}\zeta_{w}(s_0)|_{w=0}$ might make no sense. 
		Indeed,  choose any $s_0\in (a,b)$ and, for small $w\geq 0$,  put: 
		\[\theta(w)=\int_0^w \frac{d\bar w}{\sqrt{\tilde{h}(\dot x_{\bar w}(s_0),\dot x_{\bar w}(s_0))}}\]
		whenever $w\in(0,\varepsilon)$, which is well-defined  and it can  be extended continuously at $w=0$ since  $\tilde{h}(\dot x_w(s_0),\dot x_w(s_0))>0$, for $w>0$, and 
		\[\frac{\partial}{\partial w}\tilde{h}(\dot x_w(s_0),\dot x_w(s_0))|_{w=0}=2\tilde{h}(\xi'(s_0),\dot x(s_0))>0.\] 
		In fact,  the latter implies 
		\[\sqrt{\tilde{h}(\dot x_w(s_0),\dot x_w(s_0))}\geq c\sqrt{w},\]
		for some constant $c>0$  and $0<w\leq \varepsilon$,   and consequently,
		\[\theta(w)=\int_0^w \frac{d\bar w}{\sqrt{\tilde{h}(\dot x_{\bar w}(s_0),\dot x_{\bar w}(s_0))}}\leq \int_0^w \frac{d\bar w}{ c\sqrt{\bar w}}= \frac{2}{c}  \sqrt{w} .\]
		So, put $\theta(0)=0$ and  let $w(\theta)$, $\theta\in [0, \varepsilon')$, be the inverse function of $\theta(w)$, $w\in [0,\varepsilon)$. Observe that  $\dot w(\theta)=\sqrt{\tilde{h}(\dot x_{w(\theta)}(s_0),\dot x_{w(\theta)}(s_0))}$, for $\theta>0$, and $\lim_{\theta\rightarrow 0^+}\dot w(\theta)=0$. Thus,  $w$ can be $C^1$-extended evenly, that is, we write $w(-\theta)=w(\theta)$,  on $(-\varepsilon',\varepsilon')$. 

		Once defined this (non-injective) function $w(\theta)$, our aim is to check the appropriate smoothness of the variation as well as to compute its variational  vector field. 
		As a previous technical computation, let us check that the function 
		\begin{equation}\label{edominada}
			(0, \varepsilon']\times [a,b] \ni (\theta,s)\mapsto  \frac{\dot w(\theta)}{\sqrt{\tilde{h}(\dot x_{w(\theta)}(s ),\dot x_{w(\theta)}(s))}} =
			\sqrt{\frac{\tilde{h}(\dot x_{w(\theta)}(s_0),\dot x_{w(\theta)}(s_0))}{\tilde{h}(\dot x_{w(\theta)}(s),\dot x_{w(\theta)}(s))}}, 
		\end{equation}
		is bounded 
		so that  Lebesgue's theorem of dominated convergence can be used in the integrals  below.  Indeed,  taking into account that, by assumption, $\tilde h(\dot x,\xi')>0$ on $[a,b]$, consider the smooth function 
		\begin{equation*}
			u(s):=
			\sqrt{\frac{\tilde{h}(\dot x(s_0),\xi'(s_0))}{\tilde{h}(\dot x(s) ,\xi'(s))}} (>0), \qquad \forall s\in [a,b]. 
		\end{equation*}
		Now, applying L'Hopital's rule for fixed $s\in [a,b]$ in the radicand of \eqref{edominada}: 
		\begin{equation}\label{us} 
			\lim_{\theta\rightarrow 0^+}\frac{\dot w(\theta)}{\sqrt{\tilde{h}(\dot x_{w(\theta)}(s),\dot x_{w(\theta)}(s))}}=
			\lim_{\theta\rightarrow 0^+}\sqrt{\frac{\tilde{h}\big(\dot x_{w(\theta)}(s_0),\frac{\tilde D \dot x_{w}}{\de w}|_{w=w(\theta)}(s_0)\big)\dot w(\theta)}{\tilde{h}\big(\dot x_{w(\theta)}(s),\frac{\tilde D  \dot x_{w}}{\de w}|_{w=w(\theta)}(s)\big)\dot w(\theta)}}=u(s) 
		\end{equation}
		and, up to consider a smaller $\varepsilon'$,  the boundedness of \eqref{edominada} follows easily. Indeed, observe that the assumption $\tilde h(\dot x,\xi')|_J>0$ implies that the  function $(\theta,s)\in [0, \varepsilon')\times J\mapsto \frac{\tilde{h}\big(\dot x(s_0),\frac{D \dot x_{w}}{\de w}|_{w=w(\theta)}(s_0)\big)}{\tilde{h}\big(\dot x(s),\frac{D \dot x_{w}}{\de w}|_{w=w(\theta)}(s)\big)}$ is  bounded on $[0,\varepsilon'']\times J$, for $0<\varepsilon''$ small enough. 
		Thus, by Cauchy's mean value theorem 
		we also have that  the function 
		$(\theta,s)\in [0,\varepsilon'']\times J\mapsto \frac{\dot w(\theta)}{\sqrt{\tilde{h}(\dot x_{w(\theta)}(s ),\dot x_{w(\theta)}(s))}}$ is  bounded. 
		
		Consider now $\eta_\theta=\gamma_{w(\theta)}$ with $\theta \in (-\varepsilon',\varepsilon')$  as defined in \eqref{gammaw}, \eqref{zetaw},
		with $F, F_l$ as in \eqref{effl},  recalling that  $\Lambda(x(s))\neq 0$, for  all $s\in [a,b]$. Then, 
		\begin{align*}\lefteqn{\lim_{\theta\rightarrow 0^+}\frac{\zeta_{w(\theta)}(s)-\zeta_{w(0)}(s)}{\theta}}&\\
			&\quad\quad=\lim_{\theta\rightarrow 0^+}\int_{a}^s \left.\frac{\partial}{\partial w}F_l(\dot x_w)\right|_{w=w(\theta)}\dot w(\theta)\de \tau\nonumber\\
			&\quad\quad=\lim_{\theta\rightarrow 0^+}\int_{a}^s\Big(\left.\frac{\partial}{\partial w}\frac{\omega(\dot x_w)}{\Lambda(x_{w})}\right|_{w=w(\theta)}+\left.\frac{\tilde{h}( \frac{\tilde D \dot x_w}{\de w} ,\dot x_w)}{\sqrt{\tilde{h}(\dot x_w,\dot x_w)}}\right|_{w=w(\theta)}\Big)\dot w(\theta)\de \tau\nonumber\\
			&\quad\quad=\int_{a}^s\tilde{h}(\xi',\dot x)\, u\de \tau
		\end{align*}
		where  $\lim_{\theta\to 0} \dot w(\theta)=0$ and  
		\eqref{us} are used in the last equality.  Analogously, 
		\begin{align*}\lefteqn{\lim_{\theta\rightarrow 0^-}\frac{\zeta_{w(\theta)}(s)-\zeta_{w(0)}(s)}{\theta}}&\nonumber\\
			&\quad\quad=\lim_{\theta\rightarrow 0^-}\int_{a}^s \left.\frac{\partial}{\partial w}F(\dot x_w)\right|_{w=w(\theta)}\dot w(\theta)\de \tau\nonumber\\
			&\quad\quad=-\lim_{\theta\rightarrow 0^+}\int_{a}^s \left.\frac{\partial}{\partial w}F(\dot x_w)\right|_{w=w(\theta)}\dot w(\theta)\de \tau\nonumber \\
			&\quad\quad=-\lim_{\theta\rightarrow 0^+}\int_{a}^s\Big(\left.\frac{\partial}{\partial w}\frac{\omega(\dot x_w)}{\Lambda(x_{w})}\right|_{w=w(\theta)}-\left.\frac{\tilde{h}( \frac{\tilde D \dot x_w}{\de w} ,\dot x_w)}{\sqrt{\tilde{h}(\dot x_w,\dot x_w)}}\right|_{w=w(\theta)}\Big)\dot w(\theta)\de \tau\nonumber\\
			&\quad\quad=\int_{a}^s\tilde{h}(\xi',\dot x)\, u\de \tau.
		\end{align*}
		That is, it follows: 
		\[
		\left.\frac{\de \zeta_{w(\theta)}}{\de \theta}\right|_{\theta=0}(s)=
		\int_{a}^s\tilde{h}(\xi',\dot x)\, u\de \tau \qquad \qquad \forall s\in [a, b]
		\]
		Moreover, as $\left.\frac{\de x_{w(\theta)}}{\de \theta}\right|_{\theta=0}(s)=\dot w(0)\xi(s)=0$, we conclude that the variational vector field defining $\eta_\theta$ is
		$Z=\big (\left.\frac{\de \zeta_{w(\theta)}}{\de \theta}\right|_{\theta=0}, 0\big)$.  Notice also that the corresponding variation $\gamma_{w(\theta)}$ (recall \eqref{gammaw}) has continuous  second order mixed derivatives  on $(-\varepsilon,\varepsilon)\times [a,b]$.
		Thus, it is an admissible variation but 
		\[dt_{\gamma(b)}(Z(b))= r \int_{a}^{b}\tilde{h}(\xi',\dot x)\, u\de \tau= \sqrt{\tilde h(\xi'(s_0),\dot x(s_0))}\int_{a}^{b}\sqrt{\tilde{h}(\xi',\dot x)}\de \tau>0, 
		\]
		in  contradiction with \eqref{ezzz}. 
		
		Therefore, at the  interval $J$ where $x$ is $(-\tilde h)$-lightlike, $x|_J$ cannot be neither a smooth curve that is not a $-\tilde{h}$-lightlike pregeodesic nor a broken  lightlike  pregeodesic of $-\tilde{h}$, i.e., $x|_J$ has to be a  lightlike pregeodesic of $\tilde{h}$ and this concludes Case (c). 

		(2) Recall that, from Corollary~\ref{lightgeo2},  if $\gamma$ is a future-pointing
		lightlike geodesic  with $C_\gamma\neq 0$, then  $x$ is a pregeodesic  of $(M, F)$ or $(M, F_l)$ according to $C_\gamma<0$ or $C_\gamma>0$. Now, any  variation $\gamma_w=(\zeta_w,x_w)$ of $\gamma$, must satisfy  $g(\partial_t,\dot \gamma_w)<0$,
		in the first case, and $g(\partial_t,\dot \gamma_w)>0$,  in the second one, on all the  interval $[a,b]$ and for $w$ small enough. Hence,   $x_w$ defines  an $F$-admissible  variation of $x$.
		As $\zeta_w(s)=t_0+\ell_F\big((x_w)|_{[a,s]}\big)$ (resp. $\zeta_w(s)=t_0+\ell_{F_l}\big((x_w)|_{[a,s]}\big)$) we get that  $\gamma$ is a critical point of  $T$  (recall Lemma~\ref{criticallength}).  Finally, the case when $C_\gamma=0$ follows from 
		Corollary~\ref{lightgeo2}-(iii). 
	\end{proof}

	\begin{rem}\label{rneq}
		(1) Comparing Theorems~\ref{lema:fermatprinc} and  \ref{fp}, one realizes that the more restrictive ambient of the latter makes possible  both, an accurate description of the critical points and also to remove the condition of non-orthogonality at the endpoint in Theorem~\ref{lema:fermatprinc}. Nevertheless, a condition of non-triviality $x_0\neq x_1$ was assumed in Theorem~\ref{fp}. The role of this condition is apparent because if $x_0\in l_{x_1}$ and this line is lightlike, then the  case (b) in the proof of Theorem~\ref{fp} shows that even if this curve is a geodesic it will not be a critical point. If $l_{x_1}$ is a lightlike curve then last-point non-orthogonality should be assumed as in Theorem~\ref{lema:fermatprinc}  and if it is not lightlike then the hypothesis can be removed.     
		
		(2)   Observe that the variation obtained in Case (c) of the above proof is not necessarily $C^2$.  Indeed, if one tries  to compute the second partial derivative with respect to $\theta$, some denominators tending to $0$ appear. In any case,  it is an admissible variation (according to Definition~\ref{def:var_fer})  because the second order mixed derivatives exist and are continuous. 
		
		Moreover, even though we have used just the first derivative of $w(\theta)$,  one can check that \soutE{it} \bw $w$ \ew is $C^2$.  As a matter of fact,  denoting by $\frac{\tilde D \dot x_{w}}{\de w}(s)$ the covariant derivative along the curve 
		$w\mapsto x_w(s)$ associated with the Levi-Civita connection of  $\tilde h$ on $M_l$,  we have   for the chosen $s_0\in (a,b)$ and  each $\theta>0$:
		\[\ddot w(\theta)=\frac{\tilde h\big (\dot x_{w(\theta)}(s_0),\frac{\tilde D \dot x_{w}}{\de w}|_{w=w(\theta)}(s_0)\big)}{\sqrt{\tilde{h}\big (\dot x_{w(\theta)}(s_0),\dot x_{w(\theta)}(s_0)\big)}}\dot w(\theta)=\tilde h\big (\dot x_{w(\theta)}(s_0),\tfrac{\tilde D \dot x_{w}}{\de w}|_{w=w(\theta )}(s_0)\big).\]
		As the limit at $\theta=0$ of the right-hand side is well-defined,
		L'Hopital's rule  
		ensures that $w(\theta)$ is a $C^2$ function on all $(-\varepsilon',\varepsilon')$ (recall that $w(-\theta)=w(\theta)$ and then $\ddot w(-\theta)=\ddot w(\theta)$). 
		
		%
	\end{rem}
	
	
	Since the lightlike geodesics in $(\R\times M, g)$, that connect  a point $x_0$ with a line $l_{x_1}$ when at least one of the two points $x_0, x_1$ belongs to the region of  mild wind, \soutE{where $\Lambda>0$}  are those projecting on 
	pregeodesics of $F$ (recall Corollary~\ref{lightgeo2}), from Theorem~\ref{fp} we immediately get: 
	\begin{cor}
		Let $x_0,x_1\in M$ be such that at least one of the two points belongs  to the mild wind region.  Then the critical points of the arrival time functional on $\mathcal N_{(t_0,x_0),l_{x_1}}$ are all and only the  future-pointing lightlike 
		curves connecting  $(t_0,x_0)$ to $l_{x_1}$, whose projections on $M$ are pregeodesics of $(M, F)$ and, vice versa, all the  pregeodesics of $(M,F)$ connecting $x_0$ to $x_1$, when lifted  to  $(\R\times M, g)$ as lightlike curves 
		starting at $(t_0, x_0)$, are critical points of $T$ on $\mathcal N_{(t_0,x_0),l_{x_1}}$.
	\end{cor}
	As in Corollary~\ref{timelikefermat}, we can obtain a result for timelike geodesics by considering the extended spacetime $(\R\times M\times \R_u, \tilde g)$. 
	Notice that \eqref{lorentzian} is enough to ensure that 
	$(\R\times M\times \R_u, \tilde g)$ is also an \sstk splitting and  the canonical projection $t\colon\R\times M\times \R_u\to \R$ is  a temporal function.
	The  Fermat structure $\Sigma_1$ on $M\times \R_u$  carries  two  pseudo-Finsler metrics $F_1$ and $(F_1)_l$  given by \eqref{randers-kropina} 
	and \eqref{rk2} with $g_0$ replaced by the Riemannian metric on $M\times\R_u$, $g_1:= \pi_{M,u}^* g_0+\de u^2$, where $\pi_{M,u}$ is the canonical projection of $M\times \R_u$ on $M$. Clearly also the  domains $A_1$ and  $(A_1)_E$ 
	follow trivial modifications according to Proposition~\ref{plightvectorsSSTK}.
	Then,  Theorem~\ref{fp} applied to $(\R\times M\times \R_u, \tilde g )$ with its  Fermat structure $\Sigma_1$ and  the arrival  time functional  $T$, from a  point $(p_0,0)\in  \R\times M\times \R_u$  to a    line $l_{(x_1,\eta)}=\R\times \{(x_1,\eta)\}$ provides:
	\begin{cor}\label{fptimelike}
		If the curve  $\gamma_1$ is a critical point of the
		arrival time functional $T$ on $\mathcal N_{(p_0,0), l_{(x_1,\eta)}}$, then, its projection $\pi_{\R\times M}\circ\gamma_1$ is a timelike  pregeodesic of $(\R\times M,g)$ with length $\eta$. 
		Conversely,    if  $\gamma=(\zeta,x)\colon [0,1]\to  \R\times M$ is a
		timelike  geodesic of $(\R\times M,g)$ (of length $\eta=\sqrt{-g(\dot \gamma, \dot \gamma)}$) then, the lightlike geodesic of $(\R\times M\times \R_u, \tilde g)$, $\gamma_1(s)=(\gamma(s),\eta s)$, satisfies  
		$C=g(\partial_t, \dot\gamma)=\tilde g(\partial_t,\dot\gamma_1)$ and one of the following three
		exclusive cases holds
		\begin{itemize}
			\item[(i)] $C<0$, $\dot x$ lies in $A$, $(x(s),\eta s)$ is a
			pregeodesic of $F_1$,  $\gamma_1$ is a critical point of $T_1$  and
			\[
			\zeta(s)=\zeta(a)+\int_a^s F_1\big((\dot x(\bar s),\eta )\big)\de \bar s.
			\]
			\item[(ii)] $C>0$, $\dot x$ lies in $A_l$ (so that $\Lambda < 0$
			on   all   $x$), $(x(s),\eta s)$ is a pregeodesic of $(F_1)_l$, $\gamma_1$ is a critical point of $T_1$  and
			\[
			\zeta(s)=\zeta(a)+\int_a^s  (F_1)_l\big((\dot x(\bar s),\eta)\big)\de \bar s.
			\]
			\item[(iii)] $C=0$,   $x$ remains in $\bar M_l$ 
			and, whenever $\Lambda<  0$ on  $x$, necessarily $(x(s),\eta s)$ is a lightlike geodesic
			of the Lorentzian metric $h_1/\Lambda$ on $M_l\times \R_u$, where $h_1=\pi_{M,u}^* h+\de u^2$
			$-\omega(\dot x)>0$ and $\zeta$
			satisfies
			\[
			\zeta(s)=\zeta(a)-\int_a^s \frac{g_1\big((\dot x(\bar s),\eta),(\dot x(\bar s),\eta)\big)}{\omega(\dot x(\bar s))}\de \bar s.
			\]
		\end{itemize}
	\end{cor}
	\section{Further applications}\label{further2}
	Next, we give some simple applications to spacetimes (which, eventually, could be developed further in concrete cases of physical interest) in Subsections~\ref{develops} and \ref{ss6.5}. An application to the differentiability of the Randers-Kropina separation in Subsection~\ref{s8.2} is also provided. 
	\subsection{Cauchy developments}\label{develops}
	The description of  the causal properties of an \sstk  splitting in terms of  its  Fermat structure  allows us to obtain also information about   Cauchy developments
	(see \cite[Ch. 14]{O'neill} for background and conventions used here). The notion of Cauchy development makes sense for any subset $\mathcal A$ of a
	spacetime $L$ that is {\em achronal} i.e. no $x,y\in \mathcal A$ are chronologically related (we will only consider  subsets included in a slice of an \sstk splitting that are always acausal  too, see footnote~\ref{acausal}).
	For such an $\mathcal A$, the {\it future (resp. past)
		Cauchy development of $\mathcal A$}, denoted by $D^+(\mathcal A)$
	(resp. $D^-(\mathcal A)$) is defined as the subset of the  points $y$
	such that every past-inextendible (resp. future-inextendible)
	causal curve through $y$ meets $\mathcal A$. The union of both
	$D^+({\mathcal A})\cup D^-({\mathcal A})$ is simply called the
	{\it Cauchy development of $\mathcal A$}  and it will be denoted by  $D(\mathcal
	A)$. The {\em future
		(resp. past) Cauchy horizon} $H^+({\mathcal A})$ (resp. $H^-({\mathcal A})$) is defined as
	\[H^{\pm}({\mathcal A})=\{  q  \in \overline{D}^\pm ({\mathcal A}):  I^\pm(q )\cap D^\pm ({\mathcal A}) =\emptyset\}.\]
	It is helpful to think that $D({\mathcal A})$ is the region of the spacetime predictable from data in $\mathcal A$ (in fact, the interior of $D(\mathcal A)$ is globally hyperbolic when non-empty)
	and the horizon $H({\mathcal A})=
	H^+({\mathcal A})\cup H^-({\mathcal A})$ can be thought as the boundary of this region.   As an immediate consequence of the definition,
	$H^{+}(\mathcal A)=\overline{D}^+  (\mathcal A)\setminus I^-(D^+(\mathcal A))$  and analogously for  $H^{-}(\mathcal A)$.
	\begin{prop}\label{cauchydevhor} Let $(\R\times M,g)$ be an
		\sstk splitting as in \eqref{lorentz},  ${\mathcal A}\subset M$
		and ${\mathcal A}_{t_0}=\{t_0\}\times {\mathcal A}$ the
		(necessarily achronal) subset of $S_{t_0}$. Then 
		\begin{align*}
			D^+({\mathcal A}_{t_0})& \subset\,   \big\{(t,y)\in \R\times M: t \geq  t_0  \text{ and }\hat{B}^-_\Sigma(y,t-t_0)\subset {\mathcal A} \big\},\\
			D^-({\mathcal A}_{t_0})& \subset\,   \{(t,y)\in \R\times M: t \leq  t_0  \text{ and }\hat{B}^+_\Sigma(y,t_0-t)\subset {\mathcal A} \}.
		\end{align*}
		Moreover, if the spacetime is globally hyperbolic with $S_0$
		a Cauchy hypersurface, then the reverse inclusions hold  and: 
		\begin{align*}
			H^+({\mathcal A}_{t_0}) = &\big\{(t,y)\in \R\times M:  t\geq  t_0,\ y\in \bigcup_{x\in \partial\mathcal A}\hat{B}^+_\Sigma(x,t-t_0)\setminus \bigcup_{x\notin {\mathcal A}} B^+_\Sigma(x,t-t_0)\big\},\\
			H^-({\mathcal A}_{t_0})= &\big\{(t,y)\in \R\times M:t\leq  t_0, \ y\in
			\bigcup_{x\in \partial\mathcal
				A}\hat{B}^-_\Sigma(x,t_0-t)\setminus \bigcup_{x\notin {\mathcal
					A}} B^-_\Sigma(x,t_0-t)\big\}.
		\end{align*}
	\end{prop}
	\begin{proof}
		Reasoning always for the $^+$ case, let $(t,y)\in D^+(\mathcal A_{t_0})$. As the   time function of the \sstk splitting is decreasing on past-pointing causal curves, $t\geq t_0$. If $x\in \hat{B}^-_\Sigma(y,t-t_0)$, 
		from Proposition~\ref{bolas2}, $(t_0,x)\in J^-(t,y)$ and there exists a  past-pointing causal curve from $(t,y)$ to $(t_0, x)$.  Again by monotonicity of the time function, $(t_0,x)$ is its unique  point 
		of intersection with $\mathcal A_{t_0}$, so that $x\in \mathcal A$.
		
		Now let us assume that $S_{t_0}$ is a Cauchy hypersurface. 
		
		Let $(t,y)\in \R\times M$ such that $\hat B^-_\Sigma(y,t-t_0)\subset \mathcal A$. From Proposition ~\ref{bolas2}, any  past-inextendible causal curve through $(t,y)$  intersects $S_{t_0}$ in a  point $(t_0,x)$  with $x\in \hat B^-_\Sigma(y,t-t_0)$,  so that,  $x\in
		\mathcal A$  and consequently $(t,y)\in D^+(\mathcal A_{t_0})$. 
		
		For the Cauchy horizons, consider
		the case $t>t_0$ and let us show  the inclusion $\subset$.
		Let $(t,y)\in H^+(\mathcal A_{t_0})$.  Assume that there exists
		$x\not\in \mathcal A$ such that $y\in B^+_{\Sigma}(x, t-t_0)$.
		From Proposition~\ref{bolas2}   $(t_0,x)\ll (t,y)$; thus,
		$I^+(t_0,x)$ is a neighborhood of $(t,y)$ which does not intersect
		$D^+(\mathcal A_{t_0})$, which is absurd. Therefore, $y\not\in
		\cup_{x\not\in\mathcal A}B^+_{\Sigma}(x,t-t_0)$.
		Now, let $\{(t^+_n,y^+_n)\}$ be a sequence in $I^+(t,y)$
		converging to $(t,y)$ such that, for  each $n\in\N$,  there exists a
		future-pointing causal curve $\gamma_n$ which does not cross
		$\mathcal{A}_{t_0}$. However, as $S_{t_0}$ is Cauchy, $\gamma_n$
		will cross $S_{t_0}$ at some $(t_0, x_n)$ with  $x_n\not \in
		\mathcal A$.
		The  limit  curve $\gamma$  of the sequence $\{\gamma_n\}$ passing through $(t,y)$ will also cross $S_{t_0}$ at some point $(t_0, x)$ and,  by $(i)$ in Lemma~\ref{limitcurve}, $x=\lim_n x_n$. 
		Therefore, $x\in \overline{M\setminus \mathcal A}$ and  $y\in \hat B^+_{\Sigma}(x, t-t_0)$.
		Moreover,   $x$ cannot belong to the interior of $M\setminus {\mathcal A}$. Otherwise, if $V$ is a neighborhood of $x$ in $M\setminus {\mathcal A}$, then,  recalling that $J^-(t,y)= \bar I^-(t,y)$,  $I^-(t,y)$ 
		would intersect $\{t_0\}\times V$ and, reasoning as above, $(t,y)\not\in\overline{D}^+(\mathcal{A}_{t_0})$.
		
		For the inclusion $\supset$, notice first that, if $(t,y)\not\in \overline{D}^+(\mathcal{A}_{t_0})$, then there exists
		$(t',y')$, with $t_0<t'<t$, such that $\overline{
			D}^+(\mathcal{A}_{t_0}) \not\ni (t',y')\ll (t,y)$. Taking an
		inextendible past-pointing causal  curve starting at $(t',y')$
		which does not cross $\mathcal{A}_{t_0}$ (but which will  cross
		$S_{t_0}$ necessarily), there exists $\bar x \in M\setminus
		\mathcal{A}$ such that $(t_0,\bar x)\le (t',y')\ll (t,y)$ and, so,
		$y\in B^+_\Sigma(\bar x,t-t_0)$. Thus, one has just to prove only
		for points $(t,y)\in \overline{ D}^+(\mathcal{A}_{t_0})$ that
		$(t,y)$ belongs to $H^+(\mathcal{A}_{t_0})$ whenever $y\in
		\hat{B}^+_\Sigma(\tilde x,t-t_0)$, for some $\tilde x\in
		\partial \mathcal A$.  This hypothesis implies that  $(t,y)\in J^+(t_0,\tilde x)$ and, so, any
		$(t',x')\gg (t,y)$ also satisfies $(t',x')\gg (t_0,\tilde x)$. As
		$\tilde x$ lies in $\partial\mathcal A$, necessarily  $(t',x')\in I^+(S_{t_0}\setminus \mathcal{A}_{t_0})$, i.e., $(t',x')\not\in
		D^+(\mathcal{A}_{t_0})$, as required.
		
		For the case $t=t_0$, the inclusion $\subset$ is straightforward (the balls $B_\Sigma^+(x,0)$ are empty and a simple local comparison with Lorentz Minkowski shows that (a) if $x$ belongs to the interior of $\mathcal{A}$ then $I^+(t_0,x)$ 
		intersects $D^+(\mathcal{A}_{t_0})$ while (b) if $x$ belongs to the interior of $M\setminus \mathcal{A}$ then $(t_0,x)$ does not belong to the closure of $D^+(\mathcal{A}_{t_0})$). The converse inclusion follows because, clearly, $(t_0,x)$
		belongs to  $\bar D^+(\mathcal{A}_{t_0})$ if $x\in \partial \mathcal A$ (as $\mathcal{A}_{t_0}\subset \bar D^+(\mathcal{A}_{t_0})$) and, if some $(t,y)\in I^+(t_0,x)\cap  D^+(\mathcal{A}_{t_0})$ then a contradiction follows as above.
	\end{proof}
	The previous result extends  \cite[Prop. 4.7]{CapJavSan10}.  When global hyperbolicity is assumed, one can use  indistinctly closed or c-balls, even though  some  extensions in the sense of \cite[Remark 4.8]{CapJavSan10} could be explored.  The following example stresses the role of the Cauchy hypersurface. 
	\begin{exe}
		$\R\times (0,+\infty)\subset \LL^2$ as \sstk spacetime shows that if $S_t$ is not Cauchy, then none on the conclusions for $H^\pm(\mathcal{A}_t)$ holds if one chooses $\mathcal{A}=(0,2)$ (see Fig.~\ref{hor}). Notice that this example is even causally simple. 
		\begin{figure}[h]
			\includegraphics[scale=1.2, center]{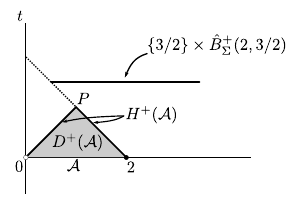}
			\caption{The points of the dashed line belong to $\cup_{t>0}\{t\}\times \hat B^+_{\Sigma}(2,t) \setminus B^+_{\Sigma}(2,t)$ but not to $H^+(\mathcal A)$; the ones of the segment $\overline{0P}$, $0$ excluded, are in  $H^+(\mathcal A)$ 
				but not in $\cup_{t>0}\{t\}\times \hat B^+_{\Sigma}(2,t)$.}\label{hor}
		\end{figure}
	\end{exe}
	\subsection{Differentiability of the Randers-Kropina separation}\label{s8.2}
	For an \sstk with a causal Killing vector field, we can use the Finslerian separation $d_F$ of the associated Randers-Kropina metric  to  describe the future or the past Cauchy horizon of  $\mathcal A_{t_0}=\{t_0\}\times \mathcal A$, at 
	least when $S_0$ is Cauchy.
	Indeed, in this case, from Proposition~\ref{cauchydevhor}, the set:
	\begin{multline*}
		\{(t,y)\in \R\times M : \not\exists
		x'\in M\setminus {\mathcal A}:  d_F(x',y)<t-t_0\text{ and }\\
		\text{either $y\in \partial \mathcal A$, $t\geq t_0$ and  $\Lambda(y)=0$ or }\exists x\in \partial \mathcal A\text{ s. t. }
		d_F(x,y)=t-t_0\}\\
		=\{(t,y)\in \R\times  \bar{\mathcal A} : \, \hbox{inf}_{x\not\in \mathcal{A}}
		d_F(x,y)=t-t_0\}\\ 
		\cup \{(t,x): x\in \partial \mathcal{A},  \Lambda(x)=0,\  t\geq t_0 \, \hbox{and} \not\exists x'\in M\setminus \mathcal{A}: d_F(x',x)<t-t_0\}
	\end{multline*}
	is  equal to 
	the horizon $H^+({\mathcal A}_{t_0})$.
	Notice that,  differently from the stationary case,   $H^+({\mathcal A}_{t_0})$ might also contain achronal  arcs, included in integral curves of $\partial_t$, such that $x\in \partial \mathcal A$, $\Lambda(x)=0$ and $\de\Lambda( \hbox{Ker}\, \omega_x)=0$.  
	\begin{exe}  Consider $\R^2$   endowed with the flat metric $g=\de x^2- (\de x\de t+\de t \de x)/2$  and $\mathcal A=(0,1)\subset \R$,  so that $S_0=\{0\}\times \R$ is spacelike and  the segment connecting the points having coordinates  $(0,1)$ 
		and $(1,1)$ (the first coordinate is $t$) is contained in $H^+(\mathcal A_{0})$. Clearly,   this arc  cannot be described as  made of  points belonging to  the graph
		(in $\R\times M$) of the function $$\varphi:=x\in \bar{\mathcal A} \mapsto d_F(M\setminus \mathcal A, x),$$
		where $d_F(M\setminus \mathcal A, x):=\inf_{y\in M\setminus \mathcal A} d_F(y,x)$.  Modifying this example one can also easily check  that, differently from the Riemannian or the Finslerian case,  $\varphi$ 
		is not continuous on $\partial \mathcal A$, in general. 
	\end{exe}
	
	Following \cite{ChFuGH02}, we  introduce the notion of {\em future horizon} which  encompasses   some of the essential properties possessed by a Cauchy horizon.  A future horizon is a  topological, closed, achronal hypersurface ruled by future 
	inextendible lightlike geodesics. This notion  allows us to remove the assumption that $S_0$ must be a Cauchy hypersurface and, then,  to extend to the Randers-Kropina separation $d_F$ a result about differentiability of the distance function 
	from a closed subset valid for a Riemannian distance \cite[Proposition 11]{ChFuGH02} and for a  distance associated with a Randers metric \cite[Theorem 5.12]{CapJavSan10}.
	
	Let $(M, F)$ be a Randers-Kropina space and $C\subset M$ a closed subset. 
	Let $(\R\times M, g)$ be the   \sstk splitting associated  with   $(M,F)$ and $((-\infty,0)\times (M\setminus C), g)$ the spacetime obtained by considering the open subset $(-\infty,0)\times (M\setminus C)\subset \R\times M$. Let us define 
	$\rho\colon M\setminus C\to [-\infty,0)$, $\rho(x):=-d_F(x,C)$, where, now, $d_F(x, C):=\inf_{y\in C}d_F(x,y)$. 
	Notice that, similarly to Proposition~\ref{ctf},  $\rho$ is equal to the function
	\begin{align*}
		\tau(x)&:=\sup\{t \in \R : \exists y\in C \text{ such that } (t,x) \ll (0,y) \}\\
		&= \sup\{ t \in \R : (\{0\}\times C) \cap I^+(t,x) \neq \emptyset\}\\
		&= \sup\{ t \in \R : (t,x) \in I^-(\{0\}\times C) \}.\end{align*}
	Let us prove that $\rho$ is a continuous function ($[-\infty, 0)$ is endowed with  its natural order topology). Indeed, being $\rho$ defined as minus the infimum of continuous functions, it is lower semi-continuous. Moreover, the following holds too:
	\begin{prop}\label{rocont}
		The function  $\rho\colon M\setminus C\to [-\infty, 0)$ is  upper semi-continuous. 
	\end{prop}
	\begin{proof}
		Assume by contradiction that there exist $x \in M\setminus C$ and two  sequences $\{y_n \} \subset C$ and $\{x_n\}\subset M\setminus C$ such that $x_n \rightarrow x$ 
		and 
		$t_n:=-d_F(x_n,y_n)$ satisfy
		$$
		T_0 := \lim_n t_n > \rho(x).
		$$
		Then, take future-pointing timelike curves $\gamma_n$ from $(t_n-1/n,x_n)$ to $(0,y_n)$ and the limit curve $\gamma$ of the sequence starting at $(T_0,x)$. 
		As $\gamma$ is inextendible and $x\in M\setminus C$, by a reasoning analogous to  that   in the proof of Theorem~\ref{tcontdf}, we deduce that its support cannot be included in the line $l_x=\R\times \{x\}$. So, take a point $Q$ of $\gamma$ 
		away from $l_x$ and  $T_1\in (\rho(x), T_0)$. As the line $l_{x}$ is causal, the segment with endpoints $(T_1, x)$ and $(T_0,x)$ glued with the arc of $\gamma$ between $(T_0,x)$ and $Q$ gives a causal curve which cannot be a lightlike 
		pregeodesic (otherwise the line $l_x$ would be also a lightlike pregeodesic and, at the first point where $\gamma$ leaves $l_x$, uniqueness of geodesics would be violated). Hence, from
		Remark~\ref{rll}, $(T_1,x)\ll Q$. Being $\gamma$ a limit curve and the relation $\ll$ open, there exists  $Q_{\bar n}$,  $\bar n\in \N$,  belonging  to the support of $\gamma_{\bar n}$  such that $(T_1, x)\ll Q_{\bar n}$. Therefore,
		$(T_1,x)\ll (0,y_{\bar n})$ too, and then $T_1\leq \rho(x)$, a contradiction.
	\end{proof}
	For each closed set $C\subset M$ we can construct a future horizon by using the function $\rho$ associated with $C$ as follows.
	\begin{prop}\label{hruled}
		Let $C\subset M$ be a closed  set. Then the hypersurface $H=\{(\rho(x),x): x\in M\setminus C, \rho(x)\neq -\infty\}$  is a future horizon in the spacetime $((-\infty,0)\times (M\setminus C), g)$.
	\end{prop}
	\begin{proof}
		Since $\rho: M\setminus C\to [-\infty,0)$ is continuous we get that $H$ is a topological closed  hypersurface in $((-\infty,0)\times M\setminus C, g)$. Moreover, it is achronal, otherwise a timelike future-pointing curve would connect  
		$(\rho(x_1), x_1)$ to  $(\rho(x_2), x_2)$ and, by taking a sequence of points $\{y_n\}\subset C$ such that  $(\rho(x_2)-1/n, x_2)\ll (0, y_n)$ we would get $(\rho(x_1),x_1)\ll (0,y_n)$, for $n$ big enough. Then,  for $\varepsilon>0$ small enough, 
		$(\rho(x_1)+\varepsilon, x_1)\ll (0,y_{\bar n})$,  for some $\bar n\in \N$,  which, recalling that $\rho(x_1)=\tau(x_1)$, gives a contradiction.

		Let us now prove the existence of a future-pointing, future-inextendible, geodesic $\gamma\colon [0,a)\to (-\infty,0)\times M$ through any point of $H$ and contained in $H$.  Consider a sequence $\{\gamma_n\}$ of timelike future-pointing 
		curves connecting $(\rho(x)-\frac{1}{n}, x)$ to $(0,y_n)$, with $y_n\in C$,  which we can assume parametrized by the time function  $t$.   Such a sequence admits a future-pointing limit curve $\gamma\colon [0,a)\to \R\times M$, 
		such that  $\gamma(0)=(\rho(x), x)$, and we can assume, by taking a smaller $a$ if necessary, that the image of $\gamma$ is contained in $\R\times (M\setminus C)$. 
		Reasoning as in the proof of Theorem~\ref{tcontdf}  and Proposition~\ref{rocont}, $\gamma$ cannot be contained in the line $l_x$ (as $x\in M\setminus C$) and all the points in $\gamma$ are horismotically related (otherwise, a point $Q$ on 
		$\gamma$ would lie in the chronological  future  of $(\rho(x),x)$  and, being $\gamma$ a limit curve, this also would  imply  that $(\rho(x),x)$ is in the chronological past of $C$). 
		From Corollary~\ref{horismos}, $\gamma$ is a (future 
		inextendible) lightlike pregeodesic which can be parametrized with $t$,  i.e., $\gamma(t)=(t,\sigma(t))$,  $t\in [\rho(x),t_1)$ for some $t_1>0$ and some unit minimizing $F$-geodesic $\sigma$. To check that $\gamma$ is included in $H$, 
		parametrize also the converging curves as $\gamma_n(t)=(t,x_n(t))$. Then,  for each $t\in [\rho(x), t_1)$,   $t\leq  \tau(x_n(t))= \rho(x_n(t))$  and by the  continuity of $\rho$  and Lemma~\ref{limitcurve},  $\rho(x_n(t))\rightarrow \rho(\sigma(t))$, 
		hence  $t\leq \rho(\sigma(t))$. If $t< \rho(\sigma(t))$, there would exist a timelike future-pointing curve connecting $\big (t+\rho(\sigma(t))/2, \sigma(t)\big)$ to $C$ and then, being the line $l_{\sigma(t)}$ causal, $(\rho(x),x)$ 
		would be in the chronological past of $C$, which is impossible. \soutE{Hence} \bw Thus, \ew  
		$\rho(\sigma(t))=t$ so that $(t,\sigma(t))\in H$.
	\end{proof}
	
	Actually  the proof of Proposition~\ref{hruled} shows also that the following extension to Randers-Kropina  metrics of \cite[Proposition 9]{ChFuGH02} and \cite[Proposition 5.11]{CapJavSan10} holds. 
	\begin{cor}
		Let $(M,F)$ be a Randers-Kropina space and $C\subset M$ a closed subset. Then every point $x\in M\setminus C$, such that $d_F(x,C)<+\infty$, belongs to at least one  geodesic segment  $\sigma$ which is minimizing, i.e. $d_F(\sigma(t),C)=t$, 
		for all $t$ in a certain interval $(t_0, t_1]$ with $\sigma(t_1)=x$. 
	\end{cor}
	The correspondence between   lightlike geodesics ruling $H$ and $d_F(\cdot, C)$ - minimizing geodesics    allows us to use \cite[Theorem 3.5]{BeeKr98}, which states that the differentiable points of a horizon are the points which are crossed 
	by one and only one lightlike geodesic ruling the horizon, so that we immediately obtain the following.
	\begin{prop}\label{smoothdF}
		Let $(M,F)$ be a Randers-Kropina space and $C$ a closed subset of $M$. A point $x\in M\setminus C$, such that $d_F(x, C)<+\infty$, is  a
		differentiable point of the function $d_F(\cdot, C)$  if and only if 
		there exists a unique minimizing geodesic segment of $(M,F)$ through $x$. 
	\end{prop}
	\begin{rem}
		Another consequence of the results in \cite{BeeKr98} is that  the set of the points in $M\setminus C$ where $d_F(x,C)$ fails to be differentiable is
		included in the set of the {\em $F$-cut points} of $C$, i.e. the points $x\in M\setminus C$ such that a minimizing geodesic through $x$ cannot be extended beyond its beginning at  $x$ as a minimizing geodesic. The $F$-cut points of $C$ correspond 
		to the {\em endpoints} of $H$  i.e. the points in $H$ where the lightlike geodesics ruling $H$ cease to belong to $H$. Moreover, the set of the points where $d_F(\cdot, C)$ is not differentiable corresponds  with the {\em crease set} of $H$ i.e.
		the subset of the endpoints of $H$ belonging to  two or more ruling lightlike geodesics.
	\end{rem}
	\subsection{K-horizons}  \label{ss6.5}
	When the Killing vector field $K$ of an \sstk   spacetime  is
	timelike (i.e., in the stationary case), there are no restriction
	for the admissible curves on the associated Finsler structure; so,
	each point $p\in \R\times M$ and each integral curve of $K$ can be
	joined by means of a timelike curve $\gamma$. Physically, this
	prevents the existence of horizons. In fact, if the spacetime
	admits a sensible notion of future infinity $\mathcal{J}^+$
	(namely, by means of a conformal embedding \cite{Wald} or by using
	the causal boundary \cite{FlHeSa11}), the existence of the
	connecting curves $\gamma$, the invariance of the metric with the
	flow of $K$, and the fact that this flow is composed by timelike
	curves, would imply
	$I^-(\mathcal{J}^+)=\R\times M$. Nevertheless, the situation is different when $K$
	changes from timelike to spacelike. This situation is natural in Mathematical Relativity; recall that this happens, for example, in the
	extension of the Schwarzschild spacetime through its {\em event horizon}
	(the hypersurface $r=2m$, which coincides with the vanishing of $g(K,K)$ for
	its natural Killing $K=\partial_t$)\footnote{Notice that the usual description of Schwarzschild spacetime in coordinates $(t,r,\theta,\phi)$ fails in the hypersurface  $r=2m$ because the slices $t=$constant are forced to be orthogonal to 
		$K=\partial_t$ (i.e., the spacetime is being described as {\em static}); however, such a description can be extended beyond that hypersurface by regarding the spacetime as stationary.} as well as in Kerr spacetime, through the stationary limit
	hypersurface $H$ that serves as a boundary for the ergosphere (the spacetime
	event horizon appears beyond  $H$)\footnote{It is well-known that the usual Boyer-Lindquist coordinates for Kerr spacetime fail for some values $r_-\leq r_+$ of the radial coordinate $r$ (specifically, two distinct values in the case of slow Kerr. i.e., $a^2<m^2$, a single value for extreme Kerr, $a^2=m^2$, and none for fast Kerr $a^2>m^2$). Our approach is directly applicable in the region $r_+<r$, which contains the ergosphere of slow Kerr, as $t$ is a temporal function therein (see \cite[Prop. 2.4.6]{O'neillkerr}). Notice also that the reference \cite{JavSan17}, summarized in \S 9.2, also includes additional simple examples.}.
	
	In general,  an embedded  hypersurface $H$
	invariant by the flow of a Killing vector field $K$ which is tangent and lightlike on $H$ is called a {\em Killing horizon}. The regularity of $H$ depends on the context, typically, $0$ would be a regular  value   of $g(K,K)$, and $H$ would be 
	a connected component  of the  preimage, but one may admit non-smooth $H$ 
	(see \cite[\S 2.5]{ChCoHe12}, \cite{CaMars} and references
	therein).  In the simple case that $H$ is the preimage of a
	regular value, a naive justification of the name  horizon  goes as follows.
	As all the future causal cones must lie on one side of $H$, given
	$(t_0,x_0)\in H$ (regarded as $H=\R \times N_H) (\subset \R\times
	M)$ for some submanifold $N_H$) there is a neighborhood $U$ of
	$x_0$ such that no points $p, q\in \R\times U$ with
	$g(K_p,K_p)>0, g(K_q,K_q)<0$, can be joined by means of a
	future-pointing (or past-pointing, depending \soutE{of} \bw on \ew the time-orientations) timelike curve $\gamma$ from
	$p$ to $q$ entirely contained in $\R\times U$, see Fig.~\ref{khorizon} ($\gamma$ could not cross $H$ maintaining its timelike character).
	\begin{figure}[h]
		\includegraphics[scale=1.2,center]{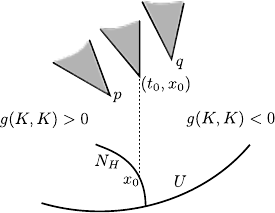}
		\caption{A Killing horizon $H=\R\times N_H$ with the lightlike cones at $(t_0,x_0)\in H$, $p$ and $q$}\label{khorizon}
	\end{figure}
	
	However, even in this case, one can wonder if $p$ and $q$ could be joined by future-pointing causal curves which leave $\R\times U$.
	
	By using  the associated Fermat structure $\Sigma$,
	these considerations can be formulated from a global viewpoint, with independence of the existence of Killing horizons, by means of the following notion.
	
	\begin{defi}  Let $M$ be a manifold endowed with a wind Finslerian structure $\Sigma$, and let ${\mathcal A}\subset M$. The  $K$-{\em horizon} for $\mathcal A$ is the boundary
		$H_\Sigma({\mathcal A})$ of the set  $\{y\in M: \exists x\in \mathcal A \text{ s. t. } C^A_{y,x}\neq\emptyset\}$.
	\end{defi}
	
	\begin{rem} When the  wind  Finslerian structure is a  Randers-Kropina metric $F$ with associated separation $d_F$, the $K$-horizon for $\mathcal A$ can be written as
		$$H_F({\mathcal A}) =  \partial (\{y\in M : d_F(y,{\mathcal A})<+\infty\}).$$
	\end{rem}
	When applied to \sstk   spacetimes, this yields a natural concept:
	
	\begin{defi} \label{dkhorizonsstk} Let $(\R\times M,g)$ be an
		\sstk  splitting and ${\mathcal A}\subset M$.
		The  $K$-{\em horizon} for $\R\times \mathcal A$, is the boundary
		$H({\mathcal A})$ of the set of the points $(t,y)\in \R\times M$ such that there exists a future-pointing timelike curve from $(t,y)$ to $\R\times {\mathcal A}$.
	\end{defi}
	Notice that the $K$-horizons for $\mathcal{A}$ (for both $\Sigma$ and the \sstk splitting) are controlled by the admissible curves and, these, by the open domain $A$ for $\Sigma$.
	The following results ensure the consistency of  the  previous two definitions.
	\begin{lemma} Let $L$  be an \sstk  spacetime   that splits in two different ways $(\R\times M,g),\ (\R\times M^f,g^f)$
		as in Lemma~\ref{fsplitting} and formula~\eqref{gf} for the same  Killing vector field
		$K$, and let $\Sigma$, $\Sigma^f$ be the corresponding wind Finslerian structures (according to Convention~\ref{sigmaf} and Proposition~\ref{changedf}).
		Consider a subset $\mathcal{A}_L\subset L$ invariant by the flow of $K$, and let $\mathcal{A}$ be its  projection on $M$ (by using either of the previous two splittings). Then, $H_\Sigma (\mathcal{A})= H_{\Sigma^f} (\mathcal{A})$.
	\end{lemma}
	\begin{proof}
		Notice that the $K$-horizons for $\Sigma$ and $\Sigma^f$ depends exclusively on the set $A$ of $F$-admissible directions, and these directions are equal for $\Sigma$ and $\Sigma^f$,  since they are the projections of causal vectors to $M$ and $M^f$ (recall that \soutE{you can see} $M^f\equiv S^f$ \bw can be seen \ew as a hypersurface in $\R\times M$ obtained as a graph, see \eqref{gf}).  
	\end{proof}
	\begin{prop}
		Let $(\R\times M,g)$ be an
		\sstk  splitting and ${\mathcal A}\subset M$. Then:
		\begin{enumerate}[(i)]
			\item $H({\mathcal A})= \R\times H_\Sigma({\mathcal A})$.
			
			\item  The K-horizon $H({\mathcal A})$ is included in the  region $g(K,K)\geq 0$.
		\end{enumerate}
	\end{prop}
	\begin{proof}
		The part $(i)$  is a straightforward consequence of the definitions (recall Proposition~\ref{bolas2}).
		For
		$(ii)$, if  $p\in H({\mathcal A})$, $g(K_p,K_p)$ cannot be negative as, otherwise, one would have a stationary region $\R\times U$   around $p$ ($U\ni p$ open and connected) and all pairs of  integral curves of $K$ in this region can be 
		connected by both, future-pointing and past-pointing timelike curves.
	\end{proof}
	\begin{exe}
		({\em Asymptotic flatness}). A natural choice of  a subset  $\mathcal A$ for an \sstk  splitting is the region $\Lambda>0$, so that $H({\mathcal A})$ can be understood as the limit of the region from which one can access to the stationary 
		part $\R\times {\mathcal A}$. In fact, the standard situation of horizons (including Kerr or Schwarzschild spacetimes) is the following: one considers an \sstk  splitting
		which is asymptotically flat, in the sense that, away from a
		compact subset, $M$ has one or more {\em ends}, each one diffeomorphic to $\R^m$ with a ball removed, and the associated wind Finslerian structure becomes a Randers metric $F$
		approaching asymptotically to the natural Euclidean
		metric (see \cite[Sect. 2.4]{CapGerSan12}). In this case, it is
		natural to take $\mathcal A$ as the exterior of a large ball in one of
		the ends. Notice that  Definition~\ref{dkhorizonsstk} gives a natural notion
		of {\em horizon} for that end (in the presence of a Killing field), extendible even
		when  the standard  notions of asymptotic flatness cannot be applied.\footnote{Recall that the classical notion of asymptotic flatness relies on the existence of a Penrose conformal embedding which makes possible to define
			the null infinity $\mathcal{J}^+$; see  \cite[Sect.
			3.4]{FlHeSa11} for extensions of this approach.}
	\end{exe}
	\begin{rem}({\em $K$ horizons vs. Killing horizons}). Notice that  the strict inequality may hold in the case $(ii)$ of the previous proposition. In fact, this happens in Kerr spacetime for the $K$-horizon $H({\mathcal A})$ of the asymptotic  
		region with $\Lambda>0$ (recall that $K$ is spacelike in  Kerr's ergosphere up to the poles). Remarkably,  Kerr's horizon is a $K$-horizon but not a Killing horizon (the stationary limit surface is neither a $K$-horizon nor a Killing horizon).
		
		But even when the equality to 0 holds in $(ii)$, $H({\mathcal A})$ can be strictly included in the region $g(K,K)=0$
		because of several reasons. First, this region may be not a hypersurface  and, for example, it may contain  an open subset; this happens in the case of pp-waves, recall Example~\ref{ex_ppwave}. Moreover, when the zero level of $g(K,K)$ is a 
		Killing horizon, the global behavior of the metric may prevent even the existence of a $K$-horizon, see Fig.~\ref{cylinder}.
		\begin{figure}[h]
			\includegraphics[scale=1,center]{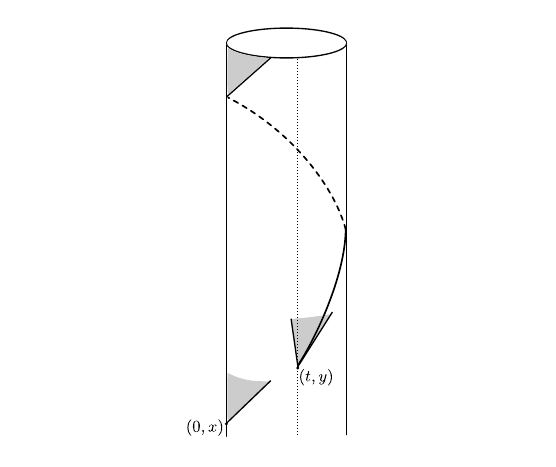}
			\caption{An \sstk cylinder $\R\times S^1$ with a Killing vector field which is timelike everywhere except at the points of the line passing through $(0,x)$ where it is lightlike; the $K$-horizon  $H(\{x\})$ is empty}\label{cylinder}
		\end{figure}
		Summing up, the physical interpretation of the $K$-horizon for $\mathcal{A}$  becomes apparent: $H( \mathcal{A})$ is the limit of the region $R$ so that $\R\times \mathcal{A}$ is not accessible for  particles  starting beyond $R$. So,
		$K$-horizons are always associated  with some concrete $\mathcal{A}$, and are clearly distinct to Killing horizons. Nevertheless, results on spacetimes and Killing horizons may be applicable to $K$-horizons (for example,  to ensure that the 
		Kropina part of a wind Finslerian structure appears on a
		smooth surface  and, eventually, may be a $K$-horizon,
		see Lemma 7 or Proposition 3 in \cite{CaMars}).
	\end{rem}

	\section{Conclusions and  further  developments}\label{conclusions}
	
	\noindent  To conclude, first we spotlight the distinguishing features of our approach. 	
	
	When strong wind is considered, there is an obvious failure of the description of Zermelo problem by means of a Finslerian metric of Randers type. The notion of wind-Riemannian structure or, with more generality, wind Finslerian one $\Sigma$, appears then as a  natural model even though, as far as we know, the only precedents of  such geometric structures are the Kropina-metrics, a particular type of singular Finslerian metrics which would correspond to Zermelo problem with critical wind.
	
	In order to develop the properties of $\Sigma$, a number of new concepts are defined (wind curve, c-ball,  w-convexity, etc.), and  more  standard \bb elements appear: a conic Finsler metric $F$ and a Lorentz-Finsler metric $F_l$ on a cone domain. \eb 
	Notably, geodesics in $\Sigma$  split  in three \bb types \eb of geodesics: locally minimizing $F$-geodesics, locally maximizing $F_l$-geodesics and abnormal geodesics.
	
	For the case of a wind Riemannian structure of dimension  $m$,  a correspondence with the conformal structure for an \sstk spacetime (a general class of Lorentzian manifolds) of dimension  $m+1$  appears. In particular, the cone \bb domain associated with \eb  $\Sigma$ \bb can also be \eb  interpreted \bb in terms of the (future) causal cone \eb for a Lorentz metric  in the region of strong wind (obtained from a sort of projection of the metric in the \sstk  splitting), and the abnormal geodesics for $\Sigma$ become  lightlike geodesics for this \bb Lorentz metric. \eb   
	
	The applications of the correspondence work in both directions. On the one hand,    known results in Lorentzian Geometry can be applied to \sstk spacetimes, and they yield a full description of the geodesics of  $\Sigma$. In particular,  the properties of  Kropina and Randers-Kropina  metrics can be widely developed.
	On the other hand, the conformal geometry of an \sstk spacetime can be characterized by means of a wind  Riemannian structure  in a very precise way. This includes  the \sstk causal ladder and the existence of horizons, as well as other properties of more purely mathematical interest such as the existence of some closed geodesics.

	This  correspondence  allows us to obtain  a fully satisfactory solution of  Zermelo navigation problem with arbitrary wind: roughly, any solution is a geodesic for $\Sigma$, and solutions must exist under w-convexity.

	What is more,   Zermelo  navigation suggests   our general version of Fermat's principle, with independent interest in General Relativity.   
	This  general  version  leads to  a unified global variational description of all the $\Sigma$-geodesics connecting two prescribed points  of an SSTK spacetime as the critical points of  the  time  arrival functional.
	
	
	Direct   applications  of our approach include the modelization of practical situations \cite{Markvo16}, examples of Analogue Gravity \cite[\S 2]{BLV05} and links with fields where cone structures become relevant, such as Finsler spacetimes \cite{JavSan13} or Hamilton-Jacobi equation \cite{FatSic12}.  What is more, further developments  have been carried out by the authors  since the first version of the present article.  Next, we will mention some of them \bb following four main references, \eb hoping that this may serve as a motivation and guide for readers interested in a variety of related topics. 
	
	\subsection{Randers and wind Riemannian manifolds of constant flag curvature }  (Ref. \cite{JavSan17b}.)  The  study  of Randers  manifolds of  constant flag curvature (CFC) is a substantial topic in Finslerian Geometry, which includes as a milestone the local and global classification  by Bao, Robles and Shen \cite{BaRoSh04}.  Essentially, they proved that, for all  such metrics,  the local Zermelo data $(g_R,W)$ on $M$ consist of a Riemannian metric $g_R$ of constant curvature and a homothetic vector field $W$. Globally, however, a striking difference appears  in comparison with the Riemannian case. In the latter, the model spaces (Euclidean, spherical, hyperbolic) appear under the natural assumptions of simple connectedness and {\em completeness}. However, in the Randers case, the fact that the $g_R$-norm of  a  homothetic $W$ may be greater and smaller than 1  in different regions of $M$ leads naturally to {\em incomplete} examples.

	This is solved in \cite{BaRoSh04} by considering   {\em inextensible} CFC Randers spaces;   however, a solution  can be given within our more general framework.  Indeed, following  \cite{JavSan17b}, all the incomplete Randers models can be extended as  {\em complete} wind Riemannian structures by taking into account our notions of geodesic and completeness developed in Sections~\ref{section:windFinsler} and \ref{ss6.1}. Moreover, the notion of CFC can be extended to wind Riemannian structures, being   these spaces  naturally classifiable now. 
	
	Going into detail, we say that a wind Riemannian structure $(M,\Sigma)$ has CFC 
	if its associated conic metrics $F$ and $F_l$ have constant flag curvatures for all the  flagpoles in $A$ and $A_l$, respectively. As a first step, the techniques in \cite{BaRoSh04} remain applicable proving that $(M,\Sigma)$  has constant flag curvature $\kappa \in \R$ if and only if its Zermelo data $(g_R,W)$ satisfy  the following two conditions:
	\begin{itemize}
		\item there exists $\mu\in \R$ such that the wind $W$  is $\mu$-homothetic for $g_R$ (i.e., ${\mathcal L}_W g_R=2\mu g_R$, where $\mathcal L$ is the Lie derivative), and
		\item the Riemannian metric $g_R$ has constant curvature $\kappa+\frac{1}{4}\mu^2$,
	\end{itemize}
	see \cite[Theorem 3.8]{JavSan17b}.  Using this local result, the following  global classification of CFC spaces can be proved, see \cite[Theorem 3.12]{JavSan17b}: 
	\begin{quote}
		{\em  A complete simply connected  wind Riemannian structure $(M,\Sigma)$ of constant flag curvature   with Zermelo data $(g_R,W)$   satisfies  either

			(i) $(M,g_R)$ is a  model space of constant curvature  and $W$ is any of its Killing vector fields, or 
			
			(ii)~$(M,g_R)$ is isometric to 
			$\R^n$  and $W$ is  a properly homothetic (i.e. non-Killing) 
			vector field.}
	\end{quote}	
 Along the proof, the following criterion  on completeness for  $(M,\Sigma)$  is used:  {\em if $g_R$ is complete and $W$ is homothetic, then $\Sigma$ is complete} (moreover, if $W$ is 
	properly homothetic, 
	then $g_R$ is flat), see \cite[Theorem 3.10-(ii)]{JavSan17b}.

	\subsection{Sharp criteria on 
		completeness 
		and  Cauchy hypersurfaces }\label{sub2}  (Ref. \cite{JavSan17}.) 
	From the viewpoint of classical Riemannian or Finslerian Geometries, it was not surprising
	that completeness was essential for  problems of classification; in particular, the   criterion  cited above goes in this standard direction.  However, a major goal in our approach has been to translate  properties of the wind Riemannian structure $(M,\Sigma)$ into  properties of causality of the associated SSTK spacetime  as  in Section~\ref{generalcase}, which opened the possibility of a pletora of relativistic interpretations, including those  in  Section~\ref{further2}. In this setting, it is especially relevant  the equivalence (proven in Theorem~\ref{generalK}-(iv)) between the completeness of $\Sigma$ and the fact that the slices of its SSTK spacetime are Cauchy hypersurfaces.  In \cite{JavSan17},  completeness  is studied  motivated by applications to these hypersurfaces. Indeed,  a natural characterization of completeness is proven   first  and,  then,  criteria to ensure  the Cauchy character of the slices   are  provided,  as well as  applications to explicit spacetimes. 
	In a nutshell, the idea is that, taking into account only the conic Finsler metric $F$  {\em including its continuous extension  to the closure of $A$ in  $TM\backslash \mathbf{0} $},  Hopf-Rinow type properties will characterize completeness. 
	
	More precisely, one considers first an extension of the notion of Finslerian separation  in Definition~\ref{Fsepa}, obtained by considering  the set of wind curves between two given points $C^\Sigma_{x_0,x_1}$ and the lengths  w.r.t. the extension $\bar F$ of $F$ (expressed as in \eqref{randerskropina}) to  $A\cup A_E$ (recall Definition~\ref{ae} and Remark~\ref{rformalreverse}).  The notion of           {\em geodesic}  can be quite naturally extended to  the conic Finsler manifold $(M, \bar F)$ (see \cite[Definition 3.19]{JavSan17}) and, then, also  forward and backward completeness are naturally defined for $(M,\bar F)$. Then, it is proved that   $(M,\Sigma)$ is geodesically complete if and only if $(M,\bar F)$ is  geodesically  complete. Furthermore,   forward (resp. backward) Cauchy completeness of $(M,d_{\bar F})$ becomes equivalent to  its  forward (resp. backward)  geodesic  completeness which is in turn equivalent to the fact that each 
	closed and forward (resp. backward) bounded subset of $(M,d_{\bar F})$ is compact (see \cite[Theorem 3.23]{JavSan17}). 
	Summing up, applying Theorem~\ref{generalK}-(iv): {\em 
		the  slices   $S_t$ in the SSTK spacetime $(\R\times M, g)$ are Cauchy hypersurfaces iff the associated conic Finsler space $(M, \bar F)$ is geodesically complete. }
	
	Starting at this characterization, it is easy to give sufficient conditions for the completeness of $(M,\Sigma)$ arising from the completeness of  some Riemannian or Finslerian metrics whose indicatrix  at each $p\in M$ encloses $\Sigma_p$,  see \cite[Props. 4.3 and  4.4, Example 4.6]{JavSan17}.  Some more refined applications, both relativistic and purely geometric, are also provided. The former include examples of ergospheres and Killing horizons which enhance our study in Section~\ref{ss6.5}. Among the latter, it is worth pointing out an example of  a Randers manifold $(M, R)$ satisfying the following property: even though not all the closed symmetrized balls of $(M, R)$ are
	compact, its universal covering $(\tilde M, \tilde R)$ satisfies that all its closed symmetrized balls are
	compact,  see \cite[Example 2.11]{{JavSan17}}.   We emphasize that, for reversible Finsler metrics, the compactness of the closed symmetrized balls is equivalent to its completeness, as a difference with the non-reversible case.  Such an example  becomes the Finslerian translation of a notable static spacetime obtained by Harris \cite{Harris}; the  underlying ideas in its interpretation come from Section~\ref{ss6.2}.

	\subsection{From wind Finslerian to cones  and Lorentz-Finsler metrics}  (Ref. \cite{JavSan20}.) 
Recall than a wind Finslerian structure $\Sigma$ can always be interpreted using the Zermelo problem, which is now anisotropic in a more general sense that the wind Riemannian case. Indeed, at each point $p\in M$, the subset $\bar B_p$ enclosed by $\Sigma_p=\Sigma\cap T_pM$  determines the velocities that a certain moving object can attain at $p$ (see Remark~\ref{remlength}). 
If we add a non-relativistic time $t$ as a first coordinate the (non-relativistic) spacetime $\R\times M$, the (boundary of the)  velocities  of the allowed trajectories determine a cone structure $\mathcal C$ which at each point $(t,p)$ is given by 
\begin{equation}\label{windcone}
{\mathcal C}_p=\{ \lambda(1,v)\in \R\times T_pM: \lambda>0, v\in \Sigma_p\}.
\end{equation}	
Observe that this cone can be described by the triple $(dt, \partial_t,\Sigma)$,  being $\Sigma_p$  defined in $\ker(dt)\equiv T_pM$,   (see Fig.~\ref{dis1} for an enlightening picture). 

In the above setting, it is very natural to consider a time dependent Zermelo problem, which leads to wind Finslerian structures of the triple $(dt,\partial_t,\Sigma)$ which depend on the coordinate $t$. 
Even if we restrict to the wind Riemannian case, an important difference with our approach in terms of SSTK spacetimes would appear, namely,
 now one cannot expect that the conformal  geometry of the spacetime is codified in a single $t$-slice and, thus, a tidy result such as Theorem~\ref{generalK} cannot hold\footnote{Notice, for example,  that hypotheses as those in \cite[Section 3]{Sa97} (which involve the whole cone structure) are the natural ones to ensure that the $t$-slices are Cauchy.}.
	However, the spacetime viewpoint introduced here will be very fruitful to handle  these $t$-dependent structures in a unified way.   Indeed, to achieve
a generalized correspondence for wind Finslerian structures and also time-dependent,
we will consider strongly convex {\em cone structures}   and Lorentz-Finsler metrics \cite[Defs. 2.7, 3.5]{JavSan20} as the natural generalization of \sstk spacetimes.

	As a first observation on cone structures $\mathcal{C}$, they can always be  obtained from a  class of (anisotropically conformal) Lorentz-Finsler metrics sharing the same pregeodesics, which turn out  to be  the intrinsic cone geodesics of $\mathcal{C}$  \cite[Th. 1.1]{JavSan20}.  To establish the searched   generalization,     start at a wind Finslerian structure $(M,\Sigma)$  and construct the cone structure $\mathcal C$ as in \eqref{windcone}.
	Then $\mathcal{C}$ can also be regarded as the cone structure of a Lorentz-Finsler  metric  $L$ on $\R\times M$ which admits $K=\partial_t$ as a Killing vector field (see some details below). So, we arrive at the natural Finsler generalization of SSTK spacetimes. Recall that 
	the notion of   a {\em smooth}  standard stationary Finsler spacetime has been developed in  \cite[Def. 4.6]{JavSan20}. As in the relativistic case, these spacetimes are also endowed with a standard splitting such that $K=\partial_t$ is  timelike and Killing.\footnote{However, Finslerian stationary and static spacetimes present some subtleties  in comparison with the relativistic ones;  these are related essentially  to the possible lack of smoothness of the Lorentz-Finsler metric along the timelike  Killing  field $K$,  see 
		\cite[Sect. 4.2]{JavSan20} and the study of  non-smooth  stationary and static Finsler spacetimes in \cite{CapSta18, CapSta16}. } Thus, Finslerian SSTK spacetimes arise just dropping the timelike restriction. Summing up:

	\begin{enumerate}
		\item \label{(1)}   There is a natural   Finslerian generalization of SSTK spacetimes, also endowed with a Killing $K=\partial_t$. 
		\item \label{(2)} The  further generalization  which also permits $t$-dependence is obtained by dropping the Killing character of $\partial_t$.
	\end{enumerate}	
	
	 Using the generalization  \eqref{(1)}, the extension of our full approach is also straightforward. 
	Indeed, by using  the intrinsic causal properties of  cone structures, the whole setting of Sections~\ref{kroran} to \ref{further2} and  most of their results  are directly transferable.  In particular,  
	when $K$ is causal, a notion of Finsler-Kropina metric emerges as a natural extension of Randers-Kropina ones and all the results in Section~\ref{kroran} can be extended to this setting.\footnote{ Another type of generalization appears when $K$ is timelike but  the slices $t=0$ are not necessarily spacelike, as studied in \cite{HeJa21}. If the timelike character of  $K$ were dropped, further wind pseudo-Finsler structures would arise.  } 
		About \eqref{(2)}, some applications will be explained in the next subsection, but the following discussion is convenient first.

	 In contrast with the Lorentz case, many issues on  Lorentz-Finsler metrics and their cone structures  have been  developed only recently. A detailed study of this general setting as well as  of the tools to establish links with the  present article is carried out  in  \cite{JavSan20}. 
	Indeed, cone structures are studied both intrinsically and by using a (highly non-unique) cone triple 
	$(\Omega, T, F)$ in a manifold $N$  composed by a 1-form $\Omega$ a   {\em timelike}  vector field $T$ such that $\Omega(T)\equiv 1$  and a Finsler metric $F$ on the kernel of $\Omega$, the latter canonically determined by $\mathcal{C}$ once  $\Omega$ and $T$ have been chosen.  Observe that at each tangent space of $N$, one can obtain a decomposition $\R\times \ker(\Omega)\equiv T_pN$ using the vector field $T$ and then an infinitesimal Zermelo problem using the ball of $F$ in $\ker(\Omega)$ as admissible velocities (compare with \eqref{windcone}). Moreover,  these elements provide a Lorentz-Finsler metric (which  is everywhere smooth but on $T$, where it is smoothable anyway) $L=\Omega^2-F^2$   and the cone structure $\mathcal{C}$ emerges as the tangent vectors $v$ which are lightlike ($L(v)=0$) and future-directed ($ \Omega(v)>0$), see \cite[Th. 1.2]{JavSan20}. When the manifold splits,  the setting of standard stationary spacetimes is naturally  reobtained  by taking $\Omega=dt$, $T=\partial_t$ and $F$ independent of the slice
	 (moreover, a simple description in terms of  non-smooth  static Finsler spacetimes as in \cite{CapSta16} also emerges); obviously, F can also be chosen t-dependent to deal with \eqref{(2)}.
	
	 Notice, however, that the above description from \cite{JavSan20} uses a  Finsler metric $F$ rather than a wind Finsler  structure $\Sigma$.
	Nevertheless,  there would not be any problem to describe  $\mathcal{C}$ by using a triple where $F$ is replaced by $\Sigma$: the latter would  also  be canonically determined when the restriction of being timelike for  $T$ is dropped  (recall again \eqref{windcone}).  
	The restriction to a  timelike $T$ (or a Finsler $F$) can always be done and  gives a  simple  expression for the Lorentz-Finsler metric $L$, but it is not especially relevant. 
	Indeed, there are cases where a non-necessarily timelike vector field $Z$ with $\Omega(Z)\equiv 1$ appears with independence of $\mathcal{C}$  (for example, when $Z$ is Killing  as  in the SSTK case  or when  it represents a priviledged field of observers as in the generalized Zermelo problem below). In these cases, $Z$ may play the role of $\partial_t$ and one can use a wind Finslerian structure $\Sigma$ for the description of $\mathcal{C}$. 
			
	\subsection{Generalizations of Zermelo's and applications to wave propagation}  (Ref. \cite{JPS21}.) 
	In accordance with the items \eqref{(1)} and \eqref{(2)} 
	above, the Zermelo navigation problem and Fermat's principle considered  
	in Sections~\ref{causaltoRiemann} and \ref{further1} admit two natural extensions, namely, when the velocity of propagation is {\em anisotropic} (i.e., depending on the direction, beyond  the  existence of the wind) and time-dependent. Indeed, in \cite[Sect. 6]{JavSan20}, the authors also considered the above Lorentz-Finsler setting, showing that  the  Zermelo problem reduces to Fermat's principle and, then, that the cone geodesics of $\mathcal{C}$ provide the solutions for Zermelo's  (notice that, because of time-dependence, the used techniques are different and many of our sharp conclusions for the projections on slices are dropped).
	Clearly,  this setting is  applicable to interesting situations as, for example, when the flight time of an airship is affected by a time-dependent wind,  eventually strong.  
	
	 More  subtlety, the following application to wave propagation and wildfires, developed in \cite{JPS21} after the work by Markvorsen in \cite{Markvo16, Markvo17}, holds. 
	Consider  a wave  which propagates in an anisotropic medium, the latter moving with respect to an observer. For example, when a sound wave propagates in the air, the variations of the properties of this medium such as  pressure or temperature may yield an anisotropic speed of propagation with respect to this medium. Moreover,  the wind may move the latter and, so, one  should add the velocity of the air with respect to Earth to the previous velocity with respect to the medium (one assumes that the observer of the wave would remain at rest on Earth). Of course, the wind, as well as the pressure or temperature, might vary with the space point and the time. 
	From an abstract viewpoint, this is modelled with a cone structure $\mathcal{C}$, eventually described by a cone triple   $(\Omega=dt, \partial_t, \Sigma)$,   as explained at the previous subsection.  Moreover, in this setting, $\Sigma$ is obtained as the translation of the indicatrix of a Finsler metric $F$ (velocities of the wave without perturbations) by a vector field $W$ (the wind). 
 In this case, $T=W+\partial_t$ 
	represents a sort of comoving field of 
	observers,\footnote{ When the wind is strong, comoving observers  move at a speed bigger than the wave  (recall that our waves propagate in a material medium, and  the relativistic situation of the light propagating in vacuum is different). So, wind Finsler structures are relevant for the computation of the actual arrival time of the wave measured by the rest observers at $Z$, see \cite[Sect. 6]{JPS21}.
	The modelling of wildfires has some particular subtleties, anyway, wind Finsler structures  can be  useful to compute  the evolution of the {\em active} firefront \cite[Section 6.3]{JPS21}. }
	which may be useful for modelling.

	When the Zermelo problem is time-dependent,   it is no longer possible to reduce the study of its solutions to classical Finsler metrics. A  way to solve this problem was firstly introduced by Markvorsen.  Motivated by the problem of wildfire propagation, this author introduced a  Finsler approach for the time-independent case \cite{Markvo16} and, then, he considered  rheonomic Lagrange manifolds and frozen metrics for the time-dependent one \cite{Markvo17}. The latter would correspond to  study   $t$-dependent Finsler metrics in the spirit of (2'), but dropping the spacetime viewpoint. This viewpoint is taken into account in  \cite{JPS21} so that the  ``fastest trajectories'' followed by the wave (which will yield the frontwave at each instant of time) will become lightlike geodesics of the Lorentz-Finsler metric $L$. 
	
	More precisely, the wavefront propagation 
	relies on a  {\em Zermelo problem starting from a submanifold} rather than from a point. Indeed, in order to  compute the evolution of a  wave or wildfire, one needs to solve the Zermelo problem from the present wavefront, which  will be a codimension 2 submanifold $S_0$ (a  hypersurface embedded in a slice $t=t_0$) of the Finsler spacetime. 
	In particular, Theorem~4.8 in \cite{JPS21} shows that  the solutions to this Zermelo problem are given by the lightlike geodesics departing orthogonally from $S_0$, at least locally.
	When the wind is time-independent such geodesics can be computed with a wind Finslerian structure  (which is a Finsler metric in the region of mild wind), see \cite[Section 6]{JPS21}. 
	

   \newpage
	\section*{Appendix: List of some symbols  and conventions}
	In order to avoid heavy notations, we have made some abuses of notation when there was no possibility of confusion. For the convenience of the reader, the main ingredients of the notation are listed here.
	
	\medskip
	\renewcommand{\arraystretch}{1.2}
	\begin{tabular}{ lll } \toprule
		\bf Symbol&\bf Stays for&  \bf Reference\\
		\midrule 
		\multirow{2}{*}{$\Sigma$} & \begin{minipage}{7cm}wind Minkowskian structure on a vector space $V$;\smallskip\end{minipage} & Def.~\ref{min2} \\
		& \begin{minipage}{7cm}wind Finslerian structure  on a manifold $M$ ($\Sigma= \cup_p \Sigma_p$)\end{minipage} &Def.~\ref{windStruct}\\ \hline
		\multirow{2}{*}{$\Sigma_F$} & indicatrix of a  Minkowski norm; & below Def.~\ref{min}\\
		& indicatrix of a Finsler manifold & Remark~\ref{r2.7}\\ \hline
		\begin{minipage}[b][0.02\textheight]{1cm}$\tilde \Sigma$\end{minipage}&reverse wind Finsler structure & Def.~\ref{dreversewind}\\ \hline
		$\MM$& region of mild wind&\multirow{3}{*}{Def.~\ref{ae}}\\ 
		$\MC$& region of critical wind&\\
		$M_l$& region of strong wind &\\ \hline
		\multirow{5}{*}{$F$} & (positively homogeneous) Minkowski norm; & Def.~\ref{min}\\
		&\begin{minipage}{7cm}\smallskip conic  pseudo-Minkowski norm, conic Minkowski norm, Lorentzian  norm;\smallskip\end{minipage}& Def.~\ref{defA}\\
		&\begin{minipage}{7cm} \smallskip Finsler metric; conic  pseudo-Finsler metric, Lorentzian Finsler metric;\end{minipage}& above Def.~\ref{windStruct}\\
		&\begin{minipage}{7cm}\smallskip conic Finsler metric associated 
			to a wind Finslerian manifold;\smallskip\end{minipage}&Prop.~\ref{windConseq}\\ 
		&\begin{minipage}{7cm}\smallskip conic Finsler metric associated 
			to a wind Riemannian manifold\smallskip\end{minipage}&Eq. \eqref{randersone}\\ 
		\hline
		\multirow{3}{*}{$F_l$}&\begin{minipage}{7cm}\smallskip Lorentzian norm associated with a strong wind Minkowskian structure;\smallskip\end{minipage}&Prop.~\ref{possibwind}\\
		&\begin{minipage}{7cm}\smallskip Lorentzian Finsler metric associated to a wind Finslerian manifold;\smallskip\end{minipage}&Prop.~\ref{windConseq}\\
		&\begin{minipage}{7cm}\smallskip Lorentzian Finsler metric associated to a wind Riemannnian manifold\smallskip\end{minipage}&Eq. \eqref{randersone2}\\
		\hline
		\multirow{3}{*}{$A$}&(open) conic domain in a vector space;& Def.~\ref{defA}\\
		&\begin{minipage}{7cm}\smallskip domain of a conic  Minkowski  and a Lorentzian  norm  associated with a wind Minkowskian structure;\smallskip\end{minipage}&Prop.~\ref{possibwind}\\ 
		&\begin{minipage}{7cm}\smallskip (open) domain of a wind  Finslerian 
			structure  ($A=\cup_{p\in M} A_p$)\smallskip\end{minipage}& Def.~\ref{windStruct}\\\hline
		$A_l=\cup_{p\in M_l}A_p$&\begin{minipage}{7cm}\smallskip open  domain  of the conic Finsler metric $F$ and the  Lorentzian Finsler metric $F_l$ of a strong wind Finslerian structure\smallskip\end{minipage}&Def.~\ref{ae} \\ \hline
		$A_E$
		&\begin{minipage}{7cm}\smallskip extended domain of wind Finslerian structures\smallskip
		\end{minipage}&Def.~\ref{ae}\\\hline
		&\begin{minipage}{7cm}{\smallskip extended definition of $F$ and $F_l$ to $A\cup A_E$}\end{minipage}&Conv.~\ref{caestar}\\
		\bottomrule
	\end{tabular}
	
	\begin{tabular}{ lll }\toprule
		\bf Symbol&\bf Stays for&  \bf Reference\\ 
		\midrule
		$B$&\begin{minipage}{7cm}\smallskip (open) unit ball for a wind Minkowskian\\ structure\smallskip\end{minipage}&Def.~\ref{min2}\\ \hline
		$B_p$&unit ball for $\Sigma_p$& Def.~\ref{windStruct}\\ \hline
		\begin{minipage}[b][0.02\textheight]{1cm}$B^\pm_{\Sigma}(x_0,r)$\end{minipage}&forward/backward wind balls;&\multirow{3}{*}{Def.~\ref{sigmaballs}}\\
		\begin{minipage}[b][0.02\textheight]{1cm}$\hat{B}^\pm_{\Sigma}(x_0,r)$\end{minipage}&forward/backward c-balls;&\\
		\begin{minipage}[b][0.02\textheight]{1cm} $\bar{B}^\pm_{\Sigma}(x_0,r)$\end{minipage}& closed forward/backward wind balls&\\\hline
		&$\Sigma$-admissible curve;&\multirow{2}{*}{Def.~\ref{sigmadmissible}-(i)}\\
		&$F$-admissible curve;&\\ 
		&wind curve;&Def.~\ref{sigmadmissible}-(ii)\\
		&(strictly) regular curve&Def.~\ref{sigmadmissible}-(iii)\\ \hline
		\begin{minipage}[b][0.02\textheight]{1cm}$C^{\Sigma}_{x_0,x_1}$\end{minipage}&set of wind curves;&\multirow{3}{*}{below Ex.~\ref{exazero}}\\
		\begin{minipage}[b][0.02\textheight]{1cm}$C^{A}_{x_0,x_1}$\end{minipage}&set of $F$-wind curves;&\\ 
		\begin{minipage}[b][0.02\textheight]{1cm} $\Omega^{A}_{x_0,x_1}$ \end{minipage}&\begin{minipage}{7cm}\smallskip  set of $F$-admissible curves  \smallskip\end{minipage}&\\ 
		\begin{minipage}[b][0.02\textheight]{1.5cm}$C^{\Sigma}_{x_0,x_1}[a,b]$\end{minipage}&set of wind curves with domain $[a,b]$& Def.~\ref{variations}\\ \hline
		$\psi$& wind variation, $F$-wind variation & Def.~\ref{variations} \\ \hline
		$\ell_{F}$&\multirow{2}{*}{wind lengths}&\multirow{2}{*}{Def.~\ref{sigmadmissible}-(iv)}\\
		$ \ell_{F_l}$&&\\ \hline
		$d_{F}$ & Finslerian separation &  Def.~\ref{from41} \\ \hline
		&unit extremizing (pre)geodesic&Def.~\ref{extremizing}\\ \hline
		&\begin{minipage}{7cm}\smallskip minimizing, maximizing, boundary (unit) (pre)geodesic\smallskip\end{minipage}&Def.~\ref{ex3}\\ 
		\hline
		&(pre)geodesic&Def.~\ref{windgeodesic}\\ \hline
		\multirow{2}{*}{$\gamma_{x_0}$}&\multirow{2}{*}{(extremizing) exceptional geodesic at $x_0$}&Prop.~\ref{disjunction}\\
		&&Def.~\ref{windgeodesic}\\ \hline
		&w-convex&Def.~\ref{strongconvex}\\ \hline
		&geodesically convex&Prop.~\ref{pgc}\\ \hline
		$(L,g)$& Lorentzian   ($m+1$)-manifold  or  spacetime&Subsect.~\ref{prelimin}\\ \hline
		$g$& metric of an \sstk&Def. \ref{dsstk}\\ \hline
		$g_R$&Riemannian metric on $M$&\\\hline
		\multirow{2}{*}{ $(g_0,\omega,\Lambda)$}& associated triple to a wind Riemannian structure;&Def.~\ref{ctriple}\\
		& spacelike metric, shift and lapse of an \sstk&Prop.~\ref{psstk}\\ \hline
		\begin{minipage}{1.2cm}\smallskip $S_t =$ $\{t\}\times M$\smallskip \end{minipage}&slice of any \sstk spacetime&below Eq. \eqref{elor}\\ \hline
		$\R\times TM$&\begin{minipage}{7cm}\smallskip tangent space to $\R\times M$ when the $t$-component becomes irrelevant\end{minipage}&Conv.~\ref{convention4}\\
		\bottomrule
	\end{tabular}
	
	\begin{tabular}{ lll }\toprule
		\bf Symbol&\bf Stays for&  \bf Reference\\ 
		\midrule
		
		&\multirow{2}{*}{limit curve}&Def.~\ref{limitcurvedef}\\ &&Lem.~\ref{limitcurve}\\ \hline
		&causal spacetime&Th.~\ref{kropinaLadder}\\\hline 
		&strongly causal spacetime;&\multirow{3}{*}{Rem.~\ref{increasing}}\\ 
		&stably causal spacetime;&\\ 
		&temporal function&\\ \hline
		&causally simple spacetime;&\multirow{3}{*}{Th.~\ref{kropinaLadder}}\\ 
		&globally hyperbolic spacetime;&\\ 
		&Cauchy hypersurface&\\ \hline
		$h$& signature changing metric on $M$& Eq. \eqref{eh}\\ \hline
		\begin{minipage}[b][0.02\textheight]{1cm}$\tilde h$\end{minipage}&\begin{minipage}{7cm}\smallskip Lorentzian metric (of index $m-1$) on $M_l$ equal to $h/\Lambda^2$\smallskip\end{minipage}&Eq. \eqref{eab}\\ \hline
		$M_{\Lambda\neq 0}$&region where $\Lambda$ does not vanish&below Eq. \eqref{eh}\\ \hline
		$\mathcal N_{p_0, \alpha}$&space of lightlike curves 
		from $p$ to the curve $\alpha$&Eq. \eqref{mline}\\ \hline
		$\mathcal N_{p_0, l{_{x_1}}}$&\begin{minipage}{7cm}\smallskip case when $\alpha$ is the line 
			$l{_{x_1}}=\{(t,x_1): t\in\R\}$\smallskip\end{minipage}&below Eq. \eqref{functiemp} \\ \hline
		\multirow{2}{*}{$T$}&arrival functional;&Eq. \eqref{earrivalf}\\
		&arrival time functional&above Eq. \eqref{functiemp}\\ \bottomrule
	\end{tabular}
	
\end{document}